\documentclass[11pt]{article}

\usepackage{amsmath}
\usepackage{epsfig, amssymb, amsfonts,dsfont,bbm}
\usepackage{amsthm}
\usepackage{amstext}
\usepackage{amsopn}
\usepackage{mathrsfs}
\usepackage{subfigure}
\usepackage{graphicx}
\usepackage[left=2.3cm,top=2cm,right=2.3cm,bottom=2cm, nohead,foot=1cm]{geometry}
\usepackage[makeroom]{cancel}
\usepackage{color}
\usepackage{enumerate}
\usepackage{comment}
\usepackage{mathtools}

\DeclareMathAlphabet{\mathpzc}{OT1}{pzc}{m}{it}

\numberwithin{equation}{section}

\DeclareMathOperator*{\bplus}{\scalerel*{+}{\sum}}

\usepackage{scalerel}

\newtheorem{theorem}{Theorem}[section]

\newtheorem{lemma}[theorem]{Lemma}
\newtheorem{conjecture}[theorem]{Conjecture}
\newtheorem{corollary}[theorem]{Corollary}
\newtheorem{proposition}[theorem]{Proposition}

\newtheorem{notation}[theorem]{Notation}
\theoremstyle{definition}
\newtheorem{definition}[theorem]{Definition}
\newtheorem{remark}[theorem]{Remark}

\newcommand{\R}{{\mathbb R}}

\DeclareFontFamily{U}{mathx}{\hyphenchar\font45}
\DeclareFontShape{U}{mathx}{m}{n}{
      <5> <6> <7> <8> <9> <10>
      <10.95> <12> <14.4> <17.28> <20.74> <24.88>
      mathx10
      }{}
\DeclareSymbolFont{mathx}{U}{mathx}{m}{n}
\DeclareMathSymbol{\bigtimes}{1}{mathx}{"91}

\usepackage{relsize}

\date{\vspace{-9ex}}

\title{Continuum models of directed polymers on disordered diamond fractals in the critical case}

\date{  }

  \author{ \textbf{Jeremy Thane Clark}\footnote{ {\tt
jeremy@olemiss.edu}} \vspace{.1cm}  \\  University of Mississippi, Department of Mathematics   }

\begin{document}
\maketitle

\begin{abstract}
We construct and study a family  random continuum  polymer measures $\mathbf{M}_{r}$ corresponding to  limiting partition function laws recently derived in a weak-coupling  regime of polymer models on hierarchical graphs with marginally relevant disorder.  The continuum polymers, which we refer to as \textit{directed paths}, are identified with isometric embeddings of the unit interval $[0,1]$ into a compact diamond fractal with Hausdorff dimension two, and there is a natural `uniform' probability  measure, $\mu$, over the space of directed paths,  $\Gamma$.  Realizations of the random path measures $\mathbf{M}_{r}$ exhibit strong localization properties in comparison to their subcritical counterparts when the diamond fractal has dimension less than two.  Whereas two  paths $p,q\in \Gamma$ sampled independently using the pure measure $\mu$ have only finitely many intersections with probability one, a realization of the disordered product measure $ \mathbf{M}_{r}\times \mathbf{M}_{r}$ a.s.\ assigns positive weight to the set of  pairs of paths $(p,q)$ whose intersection sets are uncountable but of Hausdorff dimension zero.  We give a more refined characterization of the size of these dimension-zero sets using generalized (logarithmic) Hausdorff measures.    The law of the random measure $\mathbf{M}_{r}$  cannot be constructed as a subcritical Gaussian multiplicative   chaos because the coupling strength to the Gaussian field would, in a formal sense,  have to be infinite.

\end{abstract}

\section{Introduction}

A statistical mechanical model is said to be \textit{disorder irrelevant} if introducing a sufficiently small but fixed level of disorder to the system will have a vanishing influence on the  behavior of the model  in a large-scale limit~\cite{Giacomin}.  In other terms, the presence of a weak enough disorder is overpowered by the   entropy as the system grows.  Alternatively, if the perturbative effect of any fixed disorder strength increases   as the system is scaled up,  the model is classified as~\textit{disorder relevant}.  Disorder relevance opens up  the possibility that the system can be driven towards a nontrivial limit  through an appropriate weak-disorder/coarse-graining transformation for which the limit is an attractive fixed point within some space of models~\cite{CSZ2}.  Borderline cases of disorder relevant models are referred to as  \textit{marginally relevant}, and their  renormalization procedures tend to require scalings with slowly-varying functions rather than power laws and to exhibit nonlinear behaviors that are precluded by a more robust form of disorder.

In this article, we will construct and  analyze a one-parameter ($r\in \R$) family of   random continuum polymer measures (RCPMs), $\mathbf{M}_r$, whose laws are derived through a weak-disorder limiting  regime introduced in~\cite{Clark1,Clark3} for  models of random polymers on   hierarchical graphs with marginally relevant disorder.  This scaling limit is similar to the critical regime    for (2+1)-dimensional polymers  studied by Caravenna, Sun, and Zygouras in the article~\cite{CSZ4}, which extends their previous work~\cite{CSZ0,CSZ1,CSZ2,CSZ3} and is related to the recent result~\cite{GQT} by Gu, Quastel, and Tsai on the two-dimensional stochastic heat equation at criticality.\footnote{Caravenna, Sun, and Zygouras have posted a more recent  result~\cite{CSZ5} on the arXiv that determines the uniqueness of the distributional limits in~\cite{CSZ4} and~\cite{GQT}.}  The  weak-disorder regime for (2+1)-dimensional polymers  poses fundamental new challenges from the disorder relevant (1+1)-dimensional polymer model studied by Alberts, Khanin, and Quastel~\cite{alberts,alberts2}, where the convergence of the partition functions can be handled through a term-by-term analysis of polynomial chaos expansions that limit to corresponding  Wiener chaos expansions. The exact renormalization symmetry that is baked into our  hierarchical model allows us to proceed further in developing a theory for the RCPMs $\{\mathbf{M}_r\}_{r\in \R}$ in this setting than has currently been achieved for the marginally relevant (2+1)-dimensional model at criticality, and  the results here suggest some ideas for what to expect in general for similar critical continuum  models.

The random measures $\mathbf{M}_r$ act on a space of directed paths $\Gamma$ crossing a compact diamond fractal $D$ having Hausdorff dimension two.  Each path $p\in \Gamma$ is an isometric embedding of  the unit interval $[0,1]$ into $D$ that bridges nodes $A$ and $B$ on opposite ends of the fractal. An analogous theory for the subcritical case in which the Hausdorff dimension of the diamond fractal is less than two was developed in~\cite{Clark2} for a family  $M_{\beta}$ of RCPMs indexed by $\beta\geq 0$  whose laws arise from weak-disorder  scaling limits of disorder relevant  polymer models on   hierarchical graphs~\cite{US} (with either vertex or edge disorder).  The motivation for~\cite{Clark2} was to define a counterpart to the continuum (1+1)-dimensional polymer~\cite{alberts2} in the setting of diamond hierachical graphs. The random measure $M_{\beta}$ has expectation $\mu$, where $\mu$ is a canonical `uniform' probability measure on the space of paths $\Gamma$, and $M_{\beta}$  can be constructed as a subcritical Gaussian multiplicative chaos (GMC) formally given by 
\begin{align}\label{GMC}
M_{\beta}(dp)\,=\,e^{\beta \mathbf{W}(p)-\frac{\beta^2}{2}\mathbb{E}[\mathbf{W}^2(p)]    }\mu(dp)\hspace{.5cm}\text{for the Gaussian field}\hspace{.5cm}\mathbf{W}(p)\,:=\,\int_0^1W\big(p(t)   \big)dt 
\end{align}   
over $p\in\Gamma$,  where $W\equiv \{W(x)\}_{x\in D}$ is a Gaussian white noise on $D$.   The two-point correlations of $M_{\beta}$ can be expressed as
\begin{align*}
 \mathbb{E}\big[M_{\beta}(dp)M_{\beta}(dq)  \big]\,=\,e^{\beta^2 T(p,q)    }\mu(dp)\mu(dq)  \hspace{.5cm}\text{for}\hspace{.3cm}p,q\in \Gamma  \,,
 \end{align*}
where $T(p,q)$ is the \textit{intersection time}, i.e., a quantity measuring the fractal set of intersection times $I^{p,q}:=\{t\in [0,1]\,|\,p(t)=q(t)     \}  $ for two paths $p$ and $q$.  When the diamond fractal $D$ has Hausdorff dimension $d\in (1,2)$, the set $I^{p,q}$ either has Hausdorff dimension $2-d$ or is finite   for $\mu\times \mu$-a.e.\ pair of paths $(p,q)$.  The above subcritical GMC construction from the pure measure $\mu$ and the Gaussian white noise $W$  breaks down for the critical RCPMs, $\mathbf{M}_r$.  One reason for this constructive limitation is that  the pure product $\mu\times \mu$ is supported on the set of pairs of paths $(p,q)$ having trivial intersection sets (i.e., $I^{p,q}$ is finite) when $D$ has Hausdorff dimension two.

The construction of the  RCPMs, $\mathbf{M}_r$, in this article is a straightforward task using the limiting partition function laws derived in~\cite{Clark3}.  Beyond outlining some of the basic features of  $\mathbf{M}_r$, such as that $ \mathbf{M}_r$ is a.s.\  non-atomic and mutually singular to $\mathbb{E}[\mathbf{M}_r]=\mu$, our main focus is on characterizing the typical size of intersection-times sets $I^{p,q}$ for $(p,q)\in \Gamma\times \Gamma$  chosen according to  a realization of the product measure $\mathbf{M}_r\times \mathbf{M}_r$.  When not finite, the sets $I^{p,q}$ are $\mathbf{M}_r\times \mathbf{M}_r$-a.e.\ uncountably infinite but of Hausdorff dimension zero. A more refined understanding of the size of these sparse sets can be gained using generalized Hausdorff measures~\cite{Hausdorff}  defined in terms of logarithmic dimension functions $h_{\epsilon}(a)=\frac{1}{|\log(1/a)|^\epsilon}$ for exponent $\epsilon>0$  in place of  standard power functions $h_\epsilon(a)=a^\epsilon$; see Definition~\ref{DefLogHaus}.  Generalized Hausdorff measures of this form have been considered, for instance,  in the theory of  Furstenberg-type sets~\cite{Rela,Rela2}.   
The trivial-to-nontrivial difference in the typical behavior of $I^{p,q}$ between the pure product measure $\mu\times \mu$ and realizations of the disordered product measures $\mathbf{M}_r\times \mathbf{M}_r$ is a  localization property that suggests  $\mathbf{M}_r$ is supported on a set of paths restricted to a `small' subset of  $D$.  This effective constriction of the space available to paths  drives them into having richer intersection sets when sampled independently using a fixed realization of $\mathbf{M}_r$.

As $R\searrow -\infty$ the law of the random measure $\mathbf{M}_R$ converges to the deterministic pure measure $\mu$ on paths.  In heuristic terms, a second reason that the RCPM $\mathbf{M}_r$ does not fit into the mold of a subcritical GMC on $\mu$ is that it would require an infinite coupling strength $\beta=\infty$ to a field.  There is, however, a conditional GMC  construction of $\mathbf{M}_r$ from  $\mathbf{M}_R$ for any $R\in (-\infty, r)$, which is explored in~\cite{Clark4}.  To summarize the construction in formal terms,  we  write
\begin{align}\label{ReGMC}
\mathbf{M}_{r}(dp)\,\stackrel{\mathcal{L}}{=}\,e^{\sqrt{r-R}\mathbf{W}_{ \mathsmaller{\mathbf{M}_{R}} }(p)-\frac{r-R}{2}\mathbb{E}[ \mathbf{W}^2_{\mathsmaller{\mathbf{M}_{R}} }(p)  ]    }\mathbf{M}_{R}(dp)\,,
\end{align}
where $\{\mathbf{W}_{ \mathbf{M}_{R} }(p)\}_{p\in \Gamma}$ is a Gaussian field on $( \Gamma, \mathbf{M}_{R} )$ when conditioned on $\mathbf{M}_{R}$ that has  correlation kernel
$$  \mathbb{E}\big[\mathbf{W}_{ \mathbf{M}_{R} }(p)\mathbf{W}_{ \mathbf{M}_{R} }(q)\,\big|\,\mathbf{M}_{R} \big]\,=\,T(p,q)\,,  $$
for an  intersection time  $T(p,q)$ that measures the size of the Hausdorff dimension zero sets $I^{p,q}$.  There is an analogous construction of the subcritical GMC $M_{\beta}$ in~(\ref{GMC}) from $M_{\beta'}$ for $0\leq \beta' <\beta$, however, an obvious difference for the  critical model is that the parameter $R$ is not bounded from below. In particular, the coupling  strength $\beta=\sqrt{r-R}$ in~(\ref{ReGMC}) tends to infinity in the limit $R\searrow -\infty$ wherein the law of $\mathbf{M}_{R}$ approaches $\mu$.  Although the  conditional GMC structure is of mathematical interest in itself,  it also enables an easy proof that the continuum polymer model transitions to strong disorder as $r\nearrow \infty$ in the sense that the total mass, $  \mathbf{M}_{r}(\Gamma)$, converges in probability to $0$.

\subsection{Notation and article organization}\label{SubSecNO}

$\textbf{Notation:}$ We define $\mathbb{N}_0:=\mathbb{N}\cup\{0\}$ and $\overline{\mathbb{N}}:=\mathbb{N}\cup\{\infty\}$, where $\mathbb{N}$ is the set of positive integers.  The set of positive real numbers, $(0,\infty)$, will be denoted by $\R_+$. \vspace{.1cm}

For a metric space $X$, we use $C(X)$ and $\mathcal{M}_X$ to denote, respectively, the set of real-valued continuous functions on $X$ and the set of  finite Borel measures on $X$.  We endow $\mathcal{M}_X$ with the weak topology and the associated Borel $\sigma$-algebra, $\mathcal{B}_{\mathcal{M}_X}$. \vspace{.1cm}
  
   We use $\sqcup_{i\in I}S$ to denote a disjoint union of copies of a set $S$ indexed by a set  $I$, which we interpret as the Cartesian product $I\times S$.  If $\{\eta_{i}\}_{i\in I}$ is a family of measures on a measurable space $(S,\mathcal{S})$, then $\bplus_{i\in I}\eta_i$ denotes the measure on $\sqcup_{i\in I}S $  that assigns sets of the form $\sqcup_{i\in I} A_i   $ for $A_i\in \mathcal{S}  $ the value $\sum_{i\in  I}\eta_i(A_i) $.  Thus $( \sqcup_{i\in I}S , \bplus_{i\in I}\eta_i)$ is a union of disjoint measure spaces equipped with the usual $\sigma$-algebra. \vspace{.3cm}

\noindent $\textbf{Organization:}$ We summarize the content of the remaining sections as follows:
\begin{itemize}
\item Section~\ref{SecMainResults} builds up to the statement of a functional limit theorem (Theorem~\ref{ThmUniversality}) for a family of disordered Gibbsian measures on directed paths crossing diamond graphs in a critical scaling regime.  

\item Section~\ref{SecLimitLawChar} includes the statements of  some  propositions  formulating the elementary properties of the limiting RCPM laws from  Section~\ref{SecMainResults}.

\item  Section~\ref{SecTwoPaths} contains the statements of various results relating to the  typical behavior of the set of intersection times between two paths sampled independently using  a realization of the  limiting RCPM.

\item  Section~\ref{SectionDHLConst} presents a precise formulation for structures defined over the diamond fractal, which are only discussed in general terms in earlier sections. 

\item  Sections~\ref{SecCorrMeas}\,-\,\ref{SectionLast} contain or prepare for the proofs of propositions stated in Sections~\ref{SecMainResults}\,-\,\ref{SecTwoPaths}. 
\end{itemize}

\section{RCPMs  as critical scaling limits of discrete polymer models}\label{SecMainResults}

In the  next two subsections, we will discuss the construction of   hierarchical  diamond graphs, define   a family of Gibbsian random measures on the set of directed paths crossing the diamond graphs, and review a distributional limit theorem from~\cite{Clark3} for the total masses of the random path measures. In Section~\ref{SecLimitThmForMeasures}, we extend the sense of distributional convergence from the total masses of the random measures to the full measures, which will require some terminology related to the space of continuum directed paths on the  diamond fractal, to be summarized in Section~\ref{SubSecDHL}.

\subsection{Construction of the hierarchical diamond  graphs}\label{SecHDG}
With a branching number $b\in \{2,3,\ldots\}$ and a segmenting number $s\in \{2,3,\ldots\}$, we  define the hierarchical diamond  graphs $(D_{n}^{b,s})_{n\in \mathbb{N}}$ inductively as follows:
\begin{itemize}
\item  The first diamond graph $D_{1}^{b,s}$ is defined by $b$ parallel branches connecting two  nodes, $A$ and $B$, wherein each branch is formed by $s$  edges running in series.

\item   The  graph  $D_{n+1}^{b,s}$ is defined from $D_{n}^{b,s}$  by replacing each edge on $D_{1}^{b,s}$ by a nested copy of $D_{n}^{b,s}$.

\end{itemize}
We can extend the definition of $D_n^{b,s}$ consistently to the  $n=0$ case  by defining $D_{0}^{b,s}$ as having a single edge that connects $A$ and $B$. The illustration below depicts the first few diamond graphs  with  $(b,s)=(2,3)$. 
\begin{center}
\includegraphics[scale=.8]{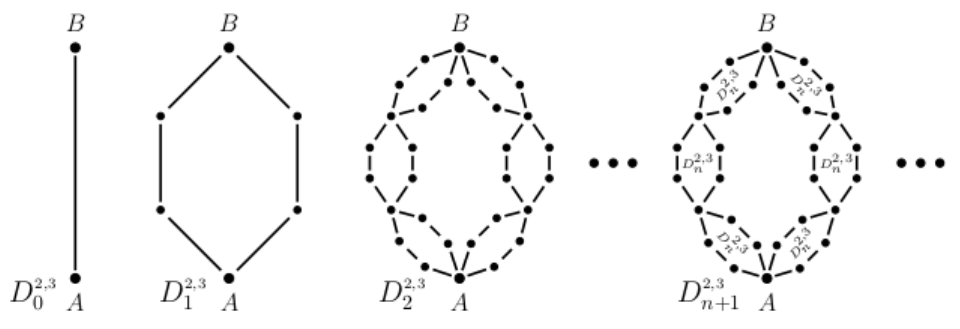}
%\\\small  $D_{n+1}^{2,3}$ is defined through $6=2\cdot 3$ copies of $D_{n}^{2,3}$ that are connected in the formation of $D_1^{2,3}$.  
\end{center}
A \textit{directed path} on $D_{n}^{b,s}$ is defined as a function $p:\{1,\ldots, s^n\}\rightarrow E^{b,s}_n$ for which $p(1)$ is incident to $A$, $p(s^n)$ is incident to $B$, and successive edges $p(k)$ and $p(k+1)$ share a  vertex for $1\leq k<s^n$. Thus the path starts at $A$ and moves progressively up to $B$. The set of directed paths on $D_{n}^{b,s}$  is  denoted by $\Gamma_{n}^{b,s}$.  

 The hierarchical diamond graphs are canonically embedded  on a  compact fractal with Hausdorff dimension $(\log s+\log b )/\log s$, which we refer to as the \textit{diamond hierarchical  lattice} (DHL).    For the remainder of the article, we will focus on the  case when the branching  and segmenting  parameters are equal ($b=s$) and treat $b\in \{2,3,\ldots\}$ as a fixed underlying parameter that will not be appended to notations for objects depending on it, e.g.,  $D^{b,b}_n\equiv D_n$.   For easy reference, we list the following notations relating to the diamond graph $D_n$:
\begin{align*}
\vspace{.3cm}&V_n  \hspace{1.3cm}  \text{Set of vertices on $D_n$}& \\
&E_{n}  \hspace{1.25cm}  \text{Set of edges on  $D_{n}$}&     \\
&\Gamma_n  \hspace{1.3cm}  \text{Set of directed paths on $D_{n}$}&\\  
&[\textbf{p}]_N  \hspace{1cm}  \text{Path in $\Gamma_N$ that is the generation-$N$ coarse-graining of $\textbf{p}\in \Gamma_n $, where $n>N$} &
\end{align*}
It is easily seen that $|E_n|=b^{2n}$.     The following are consequences of the recursive construction of the diamond graphs: when $n>N$,  
\begin{itemize}
\item  $V_N$ is canonically identifiable with a subset of  $V_n$,

\item $E_N$ determines a canonical equivalence relation on $E_n$, and

\item  $\Gamma_N$ determines a canonical equivalence relation on $\Gamma_n$.  

\end{itemize}
Elements of $V_n\backslash V_{n-1}$ are referred to as the \textit{generation}-$n$ vertices.

\subsection{A limit theorem for the partition function}\label{SubSecPartFun}

 Let $\{\omega_{h}\}_{h\in E_n} $ be a family of independent copies of a random variable $\omega$ having mean zero, variance one, and finite exponential moments $\mathbb{E}\big[\exp\{\beta \omega\}\big]$ for $\beta\geq 0$. Given an inverse temperature value $\beta\in [0,\infty)$, we define a random path measure $\mathbf{M}^{\omega}_{\beta, n}$ on the set of generation-$n$ directed paths such that   $\mathbf{p}\in \Gamma_n$ is assigned weight
\begin{align}\label{DefEM}
\mathbf{M}^{\omega}_{\beta, n}(\mathbf{p}) \,  : = \,\frac{1}{|\Gamma_n|} \frac{  e^{\beta  H_{n}^{\omega}(\mathbf{p}) }    }{ \mathbb{E}\big[ e^{\beta  H_{n}^{\omega}(\mathbf{p}) }     \big]   } \hspace{.5cm} \text{for path energy} \hspace{.5cm} H_{n}^{\omega}(\mathbf{p}) \, :=  \, \sum_{h\in \mathbf{p}} \omega_{h}   \, ,  
\end{align}
where  $h\in \mathbf{p}$ means that the edge $h\in E_n$ lies along the path $\mathbf{p}$, i.e., $h\in \textup{Range}(\mathbf{p})$.   We are interested in a joint scaling limit regime for the path measures $\mathbf{M}^{\omega}_{\beta, n}$ in which the number, $n$, of hierarchical layers of the diamond graphs  grows while the inverse temperature $\beta\equiv \beta_n$ vanishes with an appropriate dependence on $n$ determined by~(\ref{VarAsymp}) below; see also Remark~\ref{RemarkBeta}.  Since it is  more technical to define the sense in which the sequence of random path measures $(\mathbf{M}^{\omega}_{\beta, n})_{n\in \mathbb{N}}$ converges in law to a random path measure on the DHL, we will first state the corresponding limit theorem for the \textit{partition function} $W_n^{\omega}(\beta):=\mathbf{M}^{\omega}_{\beta, n} (\Gamma_n)$, i.e., the total mass of the measure $\mathbf{M}^{\omega}_{\beta, n}$.

 The hierarchical symmetry of the model implies that the sequence of partition functions $\big( W^{\omega}_n(\beta) \big )_{n\in \mathbb{N}}$ satisfies the recursive equality in distribution
\begin{align}\label{PartSymm}
W_{n+1}^{\omega}(\beta)\,\stackrel{\textup{d}}{=}\,\frac{1}{b}\sum_{i=1}^b \prod_{j=1}^b W_{n}^{(i,j)}(\beta)\,,
\end{align}
where the random variables $W_{n}^{(i,j)}(\beta)$ are independent copies of $W_n(\beta)$.  The partition functions have mean one, and it follows from (\ref{PartSymm}) that the variances $\varrho_n(\beta):=\textup{Var}\big( W_n(\beta) \big)  $ satisfy the recursive relation $M\big( \varrho_n(\beta) \big)=\varrho_{n+1}(\beta)$ for the polynomial function $M:[0,\infty)\rightarrow [0,\infty)$ given by
$$  M(x)\,:=\,\frac{1}{b}\left[(1+x)^b-1\right]\,.$$
If for fixed $r\in \R$ the sequence $(\beta_{r,n})_{n\in \mathbb{N}}$ is tuned to vanish with large $n$ in such a way that
\begin{align}\label{VarAsymp}
\varrho_0(\beta_{r,n})\,:=\,\textup{Var}\left( \frac{ e^{\beta_{r,n}\omega  } }{\mathbb{E}\big[e^{\beta_{r,n}\omega }\big]   } \right)    \,=\,\kappa^2 \left(\frac{ 1}{ n }\,+\,\frac{\eta \log n   }{ n^2 } \,+\,\frac{  r }{ n^2 }\right) \,+\,\mathit{o}\Big(\frac{1}{n^2}\Big) 
\end{align}
for $\kappa:=\big( \frac{2}{b-1} \big)^{1/2}   $ and $\eta:=\frac{b+1}{3(b-1)}   $, then the sequence of variances $\big(\varrho_n(\beta_{r,n})\big)_{n\in \mathbb{N}}$ converges to a limit $R(r)$ characterized in the following lemma from~\cite[Lemma 1.1]{Clark1}.
\begin{lemma}\label{LemVar}
For $b\in \{2,3,\ldots\}$, there is a unique continuously differentiable increasing function $R:\R\rightarrow \R_+$  satisfying the properties (I)-(III) below.
\begin{enumerate}[(I)]
\item For all $r\in \R$, we have $ M\big(R(r)\big)\,=\, R(r+1) $.

\item As $r\rightarrow \infty$, $R(r)$ diverges to $\infty$.  As $r\rightarrow -\infty$, $R(r)$ has the vanishing asymptotics $  R(r)=-\frac{ \kappa^2 }{ r }+ \frac{ \kappa^2\eta\log(-r) }{ r^2 }+\mathit{O}\Big( \frac{\log^2(-r)}{|r|^3} \Big) $.

\item The derivative $R'(r)$ admits the limiting form $ R'(r)\,=\,\lim_{n\rightarrow \infty}\frac{\kappa^2}{n^2}\prod_{k=1}^n\big(1+R(r-k)\big)^{b-1}      $.

\end{enumerate}
Moreover,   if  a sequence $(x_{r,n})_{n\in \mathbb{N}} $ has the large $n$ asymptotics $x_{r,n}  =  \kappa^2 \big(\frac{ 1 }{n}  +\frac{ \eta\log n}{n^2}+\frac{r}{n^2}\big)+\mathit{o}\big(\frac{1}{n^2}  \big)$, then $ M^{n}(x_{r,n}   )$ converges as $n\rightarrow \infty$ to $R(r) $, where $ M^{n}$ denotes the $n$-fold composition of $M$.
\end{lemma}

\begin{remark}\label{RemarkDerR} Property (III)  is equivalent to stating that $R'(r-n)=\frac{\kappa^2}{n^2}+\mathit{o}\big( \frac{1}{n^2} \big)$ for  $n\gg 1$ since applying the chain rule to the identity (I)  yields that $ R'(r)=R'(r-n) \prod_{k=1}^n\big( 1+R(r-k)\big)^{b-1}      $.  
\end{remark}

In light of Lemma~\ref{LemVar}, it is reasonable to conjecture that the sequence of variables $\big(W_n^{\omega}(\beta_{ r,n})\big)_{n\in \mathbb{N}}$ converges in law, which is the main result of the following theorem from~\cite[Theorem 2.7]{Clark3}. 
\begin{theorem}\label{ThmPrev} Fix $b\in \{2,3,\ldots\}$ and $r\in \R$. The sequence in $n\in \mathbb{N}$ of random variables $W_n^{\omega}(\beta_{ r,n})$ converges in law to a limit distribution $L_r^{(b)}$.  Moreover, the positive integer moments of  $W_n^{\omega}(\beta_{r,n})$ converge with large $n$ to the positive integer moments of $L_r^{(b)}$.
\end{theorem}

\begin{remark}\label{RemarkBeta} The variance asymptotic~(\ref{VarAsymp}) is obtained by  taking the sequence $(\beta_{r,n})_{n\in \mathbb{N}}$ to have the large $n$ asymptotics
\begin{align}\label{BetaForm}
\beta_{ r,n}\, :=\, \frac{\kappa}{\sqrt{n}}\,-\,\frac{ \kappa^2 \tau}{2n}\,+\,\frac{\kappa\eta\log n}{2n^{\frac{3}{2}}}\,+\,\frac{\kappa r+\kappa^3(\frac{5}{4} \tau^2-\frac{7}{12}\tau'- \frac{1}{2} ) }{2n^{\frac{3}{2}}}\, +\,\mathit{o}\left( \frac{1}{n^{\frac{3}{2}}} \right) \,,
\end{align}
where $\tau:=\mathbb{E}[\omega^3]$ and $\tau':=\mathbb{E}[\omega^4]-3$, i.e., the $3^{rd}$ and $4^{th}$ cumulants of the disorder variables. The scaling $\beta_{ r,n}$ occurs in a vanishing window around the critical point $\epsilon=\kappa$ for coarser scalings of the form $\beta_n^{(\epsilon)}=\epsilon/ \sqrt{n}+\mathit{o}(1/\sqrt{n}) $ for a parameter $\epsilon\in[0,\infty)$.  Discussion of the  weak-disorder scaling~(\ref{BetaForm}), which is similar to the critical window scaling for (2+1)-dimensional polymers in~\cite{CSZ4}, can be found in~\cite[Section 2.3]{Clark3}.
\end{remark}

\subsection{The diamond hierarchical lattice and its path space}\label{SubSecDHL}

In the next subsection, we will extend the statement of the limit result in Theorem~\ref{ThmPrev}  for the partition functions $W_n^{\omega}(\beta_{r,n }):=\mathbf{M}^{\omega}_{\beta_{r, n}, n}(\Gamma)$ to a limit theorem for the full random measures $\mathbf{M}^{\omega}_{\beta_{r,n}, n}$. First, we will summarize the terminology and notation related to the DHL and make some elementary observations relating to hierarchical symmetry.  The DHL, denoted by $D$, is a compact metric space on which each directed path $p\in \Gamma$ is an isometric embedding $p:[0,1]\rightarrow D$ with $p(0)=A$ and $p(1)=B$, where $A$ and $B$ are points on opposite ends of $D$.  Thus $D$ is a network of interweaving copies of $[0,1]$, and we can measure the distance between  points on $D$   using a  travel metric.  We make these definitions more precise in Section~\ref{SectionDHLConst}, and, for now, we extend our notations as follows:
\begin{align*}
&V  \hspace{1.73cm}  \text{Set of vertex points on $D$}& \\
&E  \hspace{1.7cm}  \text{Complement of $V$ in   $D $  }& \\
&\Gamma  \hspace{1.8cm}  \text{Set of directed paths on $D$}& \\
&\nu     \hspace{1.81cm} \text{Uniform probability measure on $D$}
\\
&\mu   \hspace{1.81cm} \text{Uniform probability measure on $\Gamma$}&\\
&\mathcal{B}_{\Gamma}   \hspace{1.61cm} \text{Borel $\sigma$-algbra on $\Gamma$}&
\\  
&[x]_n  \hspace{1.46cm}  \text{Edge in $ E_n $ that is the generation-$n$ `coarse-graining' of $x\in E$}&
\\  
&[p]_n  \hspace{1.5cm}  \text{Path in $ \Gamma_n $ that is the generation-$n$ `coarse-graining' of $p\in \Gamma$}&
\end{align*}
The following are some canonical identifications between the  diamond graph structures and subsets of the DHL and its path space: 
\begin{itemize}
\item $V$ is a countable, dense subset of $D$ that   is  identifiable with $\cup_{n=1}^\infty V_n$. 

\item The edge set $E_n$ defines an equivalence relation on $E$ in which each element  $\mathbf{e}\in E_n$ corresponds to a cylinder subset $\mathbf{\overline{e}}$ of $E$.

\item The path set $\Gamma_n$ defines an equivalence relation on $\Gamma$ in which each element $\mathbf{p}\in \Gamma_n$ corresponds to a cylinder subset $\mathbf{\overline{p}}$ of $\Gamma$.

\end{itemize}
  The uniform measures $\nu$ on $D$ and $\mu$ on $\Gamma$ assign weights $ \nu(\mathbf{\overline{e}})=1/|E_n|$ and $\mu(\mathbf{\overline{p}})=1/|\Gamma_n|$ to the cylinder sets corresponding to
  $\mathbf{e}\in E_n$ and  $\mathbf{p}\in \Gamma_n$, respectively. 
  
 Recall from Section~\ref{SubSecNO} that $\sqcup_{i\in I}S$  is alternative notation for the Cartesian product $I\times S$. We make the following series of remarks regarding  hierarchical symmetry.

\begin{remark}\label{RemarkCylinderE} For  $n\in \mathbb{N}$, there is a canonical bijection from $E$ to $E_n\times E$ in which $x \in E $ maps to the pair $\big([x]_n, \langle x\rangle_n\big)$, where
\begin{itemize}
\item $[x]_n\in E_n$ is the $n^{th}$ generation coarse-graining of the point $x$, and

\item $ \langle x\rangle_n \in E$ is the dilated position of $x$ within the embedded subcopy of $D$ corresponding to $[x]_n$.

\end{itemize}
The edge $\mathbf{e}\in E_n$ determines  the cylinder set $\mathbf{\overline{e}}:=\big\{x\in E\,\big|\,[x]_n=\mathbf{e}  \big\}$.
\end{remark}

\begin{remark}\label{RemarkDCopies} The canonical bijection in Remark~\ref{RemarkCylinderE} between $E$ and $\bigsqcup_{e\in E_n} E$ does not quite extend to a bijection between $D$ and $\bigsqcup_{e\in E_n} D$ because the root nodes of the $n^{th}$ generation embedded copies of $D$ can intersect at points in the finite set $V_n\subset D$.  In an abuse of terminology, we will nevertheless speak of $D$ as being canonically identifiable with $\bigsqcup_{e\in E_n} D$  because all measures on $D$ that we consider assign the countable set of vertices $V=D\backslash E$ measure zero.  The uniform measure $\nu$ on $D$ satisfies the obvious decomposition 
\begin{align*}
 \nu\,=\,\,\frac{1}{b^{2n}}\bplus_{e\in E_n} \nu \hspace{.5cm} \text{under the identification}\hspace{.5cm} D\,\equiv\,\bigsqcup_{e\in E_n}  D  \,.
\end{align*}
\end{remark}
  
\begin{remark}\label{RemarkCylinder}
 For $n\in \mathbb{N}$, there is  a canonical bijection  from  $\Gamma$ to $\Gamma_n\times \bigtimes_{k=1}^{b^n}\Gamma    $ in which $p\in \Gamma$  maps to the $(b^n+1)$-tuple $\big([p]_n;\, \langle p\rangle_{n}^1 ,\ldots , \langle p\rangle ^{b^n}_{n}\big)$, where 
\begin{itemize}
\item $[p]_n\in \Gamma_n$ is the $n^{th}$ generation coarse-graining of the path $p$ referred to above, and

\item  $\langle p\rangle_n^{j}\in \Gamma$ is a dilation of the part of the path $p$ crossing the shrunken, embedded copy of $D$ corresponding to the edge $[p]_n(j)\in E_n$.  

\end{itemize} 
 The path $\mathbf{p}\in \Gamma_n$ determines  the cylinder set $\mathbf{\overline{p}}:=\big\{p\in \Gamma\,\big|\,[p]_n=\mathbf{p}  \big\}$.

\end{remark}

\begin{remark}\label{Remarkn=1}  In the case $n=1$ of Remark~\ref{RemarkCylinder},   we prefer to use that $|\Gamma_1|=b$ to identify the canonical bijection  as from $\Gamma$ to $\bigsqcup_{i=1}^b \bigtimes_{j=1}^b  \Gamma$.  Thus each $p\in \Gamma$ corresponds to a tuple $(i;p_1,\ldots,p_b)$ for some $1\leq i\leq b$ and $p_1,\ldots, p_b\in \Gamma$.
\end{remark}

\begin{remark}\label{RemarkConcat}
The uniform measure $\mu$ on the path space $\Gamma$ satisfies the hierarchical symmetry
\begin{align*}
 \mu\,=\,\,\frac{1}{b}\bplus_{i=1}^b \prod_{j=1}^b \mu \hspace{.5cm}\text{under the identification}\hspace{.5cm} \Gamma\,\equiv\,\bigsqcup_{i=1}^b \bigtimes_{j=1}^b  \Gamma  \,.
\end{align*}
More precisely, if $\varphi: \bigsqcup_{i=1}^b \bigtimes_{j=1}^b  \Gamma  \rightarrow \Gamma$ is the canonical bijection of Remark~\ref{Remarkn=1}, then $\mu$ is equal to the pushfoward  of $\frac{1}{b}\bplus_{i=1}^b \prod_{j=1}^b \mu $ under $\varphi$.
\end{remark}

\subsection{A weak-disorder limit theorem for the random Gibbsian path measures}\label{SecLimitThmForMeasures}

The set of  directed paths, $\Gamma$, on the DHL is a compact metric space when equipped with the uniform metric $d_{\Gamma}:\Gamma\times \Gamma \rightarrow [0,1]$ defined by $d_{\Gamma}(p,q)=\sup_{t\in [0,1]}d_D\big(p(t),q(t)\big)$ for $p,q\in \Gamma$, where $d_D$ is a travel metric on $D$ defined in Section~\ref{SectionDHLConst}. We denote the set of finite Borel measures on $\Gamma$ by $\mathcal{M}_{\Gamma}$, and  we equip $\mathcal{M}_{\Gamma}$ with the  topology of weak convergence, i.e., for measures  $ \eta_{n}\in \mathcal{M}_{\Gamma}$ indexed by $n\in \overline{\mathbb{N}}$  we have $\eta_n\stackrel{\text{w}}{\rightarrow} \eta_{\infty}$ with large $n$ if and only if
$$ \int_{\Gamma}g(p)\eta_n(dp) \hspace{.5cm} \longrightarrow   \hspace{.5cm} \int_{\Gamma}g(p)d\eta_{\infty}(dp) $$
for every continuous function $g:\Gamma\rightarrow \R$.\footnote{Since the space $\Gamma$ is compact, we do not need to explicitly stipulate that the continuous function $g$ is  bounded.} The space $\mathcal{M}_{\Gamma}$ is Polish  under the weak topology, which is induced by the Prohorov metric.\footnote{See~\cite[Lemma 4.5]{Kallenberg}.} We  endow $\mathcal{M}_{\Gamma}$ with its Borel $\sigma$-algebra, $\mathcal{B}_{\mathcal{M}_{\Gamma}}$. A  \textit{random measure} on $\Gamma$ defined over a probability space $(\Omega, \mathscr{F}, \mathbb{P})$ is a kernel $\eta:\Omega\times \mathcal{B}_{\Gamma}\rightarrow [0,1]$ satisfying (i)-(ii) below:
\begin{enumerate}[(i)]
\item  $\eta(\omega,\cdot)$ defines a Borel measure on $\Gamma$ for every $\omega\in \Omega$.

\item   $\eta(\cdot ,A)$ is an $\mathscr{F}$-measurable function for every $A\in \mathcal{B}_{\Gamma}$.
\end{enumerate}
Conditions (i)-(ii)  are equivalent to requiring that $\omega\mapsto \eta(\omega,\cdot)$ is a measurable map from $\Omega$ to $\mathcal{M}_{\Gamma}$. In this subsection, we will attach $\omega$ as a superscript to random measures, i.e., identify $\eta^{\omega}(A)\equiv \eta(\omega, A)  $. 
 If $\{\eta_n^{\omega}\}_{n\in \overline{\mathbb{N}} }$ is a family of $\mathcal{M}_{\Gamma}$-valued random elements---not necessarily defined on the same probability space---, then the sequence $(\eta_n^{\omega})_{n\in \mathbb{N}}$ is said to \textit{converge in distribution} to $\eta_{\infty}^{\omega}$  provided that for any bounded, continuous function $f:\mathcal{M}_{\Gamma}\rightarrow \R$ we have
$$  \mathbb{E}\big[ f\big(\eta_{n}^{\omega}\big)  \big] \hspace{.5cm} \longrightarrow \hspace{.5cm} \mathbb{E}\big[ f\big(\eta_{\infty}^{\omega}\big)  \big] \,,  $$
where $\mathbb{E}$ denotes the expectation with respect to $\mathbb{P}$.  By standard theory~\cite[Theorem 4.11]{Kallenberg}, it suffices to show that for any continuous function $g:\Gamma \rightarrow \R   $ there is convergence in law of  random variables 
\begin{align}\label{ConvEquiv}
 \eta_{n}^{\omega}(g)\,:=\, \int_{\Gamma}g(p) \eta_{n}^{\omega}(dp) \hspace{.5cm} \underset{n\rightarrow \infty}{\overset{\text{d}}{\Longrightarrow}}  \hspace{.5cm} \eta_{\infty}^{\omega}(g)\,:=\,\int_{\Gamma}g(p) \eta_{\infty}^{\omega}(dp)\,.  
 \end{align}

 To formulate a distributional limit theorem for the sequence of random measures $\big(\mathbf{M}^{\omega}_{\beta_{r,n}, n}\big)_{n\in \mathbb{N}}$---where, recall,  $\mathbf{M}^{\omega}_{\beta_{r,n}, n}$ acts on the discrete path space $\Gamma_n$---we need  a simple operation for embedding finite measures on $\Gamma_n$ into the space of finite Borel measures on $\Gamma$.  
\begin{definition}\label{DefLocAvg}
Let $\varrho_n$ be a finite measure on  $\Gamma_n$.   We define  $\overline{\varrho}_n$  as the Borel measure on $\Gamma$ with
\begin{enumerate}[(i)]

\item $\varrho_n(\mathbf{p})= \overline{\varrho}_n(\mathbf{\overline{p}}) $ for every  $\mathbf{p}\in \Gamma_n$, and

\item  $\overline{\varrho}_n$ is uniform on each cylinder set $\mathbf{\overline{p}}\subset \Gamma$, i.e., the restriction of $\overline{\varrho}_n$ to $\mathbf{\overline{p}}$ is a multiple,  $\varrho_n(\mathbf{p})/\mu(\mathbf{\overline{p}})$, of the  restriction of  $\mu$ to  $\mathbf{\overline{p}}$.

\end{enumerate}

\end{definition}
 
\begin{remark} The convention in (ii) of defining $\overline{\varrho}_n$ as a locally uniform measure is unimportant, and we can also define the restriction of $\overline{\varrho}_n$ to $\mathbf{\overline{p}}\subset \Gamma$ to assign the full weight $\varrho_n(\mathbf{p})$ to a single path $p\in \mathbf{\overline{p}}$.
\end{remark}

The proof of the following limit theorem is in Section~\ref{SecProofLimitThmForMeasures}.

\begin{theorem}\label{ThmUniversality} Fix $r\in \R$ and $b\in \{2,3,\ldots\}$.  For each $n\in \mathbb{N}$, let the random measure $ \mathbf{M}^{\omega}_{\beta, n}$ on the discrete path space $\Gamma_n$  be defined as in~(\ref{DefEM}) and the sequence  $(\beta_{r,n})_{n\in \mathbb{N}}$ have the asmptotics in~(\ref{BetaForm}).  If the  measure $  \mathbf{M}^{\omega}_{r, n}\,:=\, \mathbf{\overline{M}}^{\omega}_{\beta_{r,n} , n}  $ on the  path space $\Gamma$ is constructed as in Definition~\ref{DefLocAvg}, then  the sequence of random measures  $\big(\mathbf{M}^{\omega}_{r, n}\big)_{n\in \mathbb{N}}$ converges in distribution to a limiting $\mathcal{M}_{\Gamma}$-valued random element  $\mathbf{M}_{r}^{\omega}$.  
\end{theorem}

\begin{remark}\label{RemarkSuff}  As mentioned above, proving Theorem~\ref{ThmUniversality} boils down to showing that for any continuous  $g:\Gamma\rightarrow\R$ there is convergence in distribution 
\begin{align*}
\mathbf{M}^{\omega}_{r, n}(g)\,:=\,  \int_{\Gamma}g(p)\mathbf{M}^{\omega}_{r, n}(dp) \hspace{.5cm} \underset{n\rightarrow \infty}{\overset{\text{d}}{\Longrightarrow}}  \hspace{.5cm} \mathbf{M}_{r}^{\omega}(g)\,:=\,  \int_{\Gamma}g(p)\mathbf{M}_{r}^{\omega}(dp) \,.
\end{align*}
The simple topology of the path space $\Gamma$ eases this task. 
\end{remark}

\section{The limiting RCPM  laws} \label{SecLimitLawChar}

We will now turn to a study of the   RCPM laws $\{ \mathbf{M}_r^{\omega}  \}_{r\in \R}$ arising as limits in Theorem~\ref{ThmUniversality}.  In Section~\ref{SubSecLimitLawChar},   we formulate a set of properties (Theorem~\ref{ThmExistence}) that uniquely determines the family of random path measure laws $\{ \mathbf{M}_r^{\omega}  \}_{r\in \R}$, and Section~\ref{SubSecLimitLawProp} contains a proposition listing some additional properties of the random measures.  Proposition~\ref{CorProp4} in Section~\ref{SubSecLimitLawInfDecomp} identifies a consequence of hierarchical symmetry for realizations of $\mathbf{M}_r^{\omega} $  that we refer to as \textit{infinite decomposability}.  In the sequel, we will omit the superscript $\omega\in \Omega$ from  $\mathbf{M}_{r}^{\omega} \equiv  \mathbf{M}_{r} $  that signifies  the random measure's dependence on an underlying probability space $(\Omega, \mathscr{F}, \mathbb{P})$.  As before, $\mathbb{E} $ denotes the expectation with respect to $\mathbb{P}$.

\subsection{A characterization of the limiting RCPM laws}\label{SubSecLimitLawChar}

To state Theorem~\ref{ThmExistence} below, we will need the measure $\upsilon_{r}$ on $\Gamma\times \Gamma$ defined in the next lemma, which we will show to be the correlation measure for $\mathbf{M}_{r}$. In other terms,  $\upsilon_{r}$ formally satisfies $ \upsilon_r(dp,dq)\,=\,\mathbb{E}\big[ \mathbf{M}_{r}(dp) \mathbf{M}_{r}(dq)  \big]$. As mentioned in the introduction,  the correlation measure $\upsilon_r$ would be absolutely continuous with respect to $\mathbb{E}[ \mathbf{M}_{r}]\times \mathbb{E}[ \mathbf{M}_{r}  ] =\mu\times \mu$ for any subcritical model. However,  the lemma below, which we prove in Section~\ref{SecProofLemCorrelate}, states that $\upsilon_r$ has a nontrivial Lebesgue decomposition with respect to $\mu\times \mu$. This Lebesgue decomposition will become more important in Section~\ref{SecTwoPaths}, where we focus on topics related to path intersections.

\begin{lemma}\label{LemCorrelate} Let $R:\R\rightarrow \R_+$ be defined as in Lemma~\ref{LemVar}. The  statements below hold for any $r\in \R$.
\begin{enumerate}[(i)]
\item   There is a unique measure $\upsilon_{r}$ on $\big(\Gamma\times \Gamma, \mathcal{B}_{\Gamma\times \Gamma}\big)$ such that for any $n\in \mathbb{N}_0$ and  $\mathbf{p},\mathbf{q}\in \Gamma_{n}$  
\begin{align}\label{UpsilonCond}
 \upsilon_{r}(\mathbf{\overline{p}}\times \mathbf{\overline{q}}   )\,=\, \frac{1}{|\Gamma_n|^2} \big(1+R(r-n)\big)^{\xi_n(\mathbf{p},\mathbf{q}) }  \,,
\end{align}
where $\xi_n(\mathbf{p},\mathbf{q}) $ is the number of edges  shared by the discrete paths $\mathbf{p}$ and $\mathbf{q}$, i.e., the cardinality of $\textup{Range}(\mathbf{p}) \cap \textup{Range}(\mathbf{q})$. The marginals of  $\upsilon_r$ are both equal to $\big(1+R(r)\big) \mu$.

\item  The Lebesgue decomposition of $\upsilon_{r}$ with respect to the product measure $\mu\times \mu$ has components $\mu\times \mu$  and $\upsilon_{r}-\mu\times \mu=R(r)\rho_r$, where  $\rho_r$ is a probability measure on $\Gamma\times \Gamma$  that is supported on the set of pairs $(p,q)$ such that $\xi_n(p,q ):=\xi_n\big([p]_n, [q]_n \big)>0$ for all $n$.  The marginals of $ \rho_r$ are both equal to $\mu$.

\end{enumerate}

\end{lemma}
\begin{remark}\label{RemarkTotalMass}
 Note that we get $\upsilon_r(\Gamma\times \Gamma  )=1+R(r)$ by applying~(\ref{UpsilonCond}) with $n=0$.  Thus $\widehat{\upsilon}_r:=\frac{1}{1+R(r)}\upsilon_r$ is a probability measure on $\Gamma\times \Gamma$.
\end{remark}

\begin{remark}\label{RemarkRHierSymm} Recall from Remark~\ref{RemarkCylinder} that $\Gamma$ is canonically bijective with $ \Gamma_n \times \bigtimes_{\ell=1}^{b^n}  \Gamma $.  If $\mathbf{p},\mathbf{q}\in \Gamma_n$, then the restriction of the correlation measure $\upsilon_r$ to the cylinder set $\mathbf{\overline{p}}\times \mathbf{\overline{q}}  $ has the product form 
 \begin{align*}
\upsilon_r\big|_{\mathbf{\overline{p}}\times \mathbf{\overline{q}} } \, = \, \frac{1}{|\Gamma_n|^2} \big(1+R(r-n)\big)^{\xi_n(\mathbf{p},\mathbf{q}) }   \prod_{\ell=1}^{b^n} \mathscr{U}^{\mathbf{p}(\ell),\mathbf{q}(\ell)}_{r-n}\hspace{.5cm}\text{for}\hspace{.5cm} \mathscr{U}^{\mathbf{e},\mathbf{f}}_t\,:=\,\begin{cases} \widehat{\upsilon}_{t}   &  \mathbf{e}=\mathbf{f} \\ \mu\times \mu & \mathbf{e}\neq \mathbf{f}  \end{cases}
\end{align*} 
under the identification $\mathbf{\overline{p}}\times \mathbf{\overline{q}}\,\equiv\, \bigtimes_{\ell=1}^{b^n}  \Gamma \times \Gamma $.  This is consistent with~(\ref{UpsilonCond}) because the product is a probability measure.  
\end{remark}

The next theorem, whose proof is split between Sections~\ref{ProofThmExistence} \&~\ref{ProofThmExistence2}, lists the essential properties of the limit laws  in Theorem~\ref{ThmUniversality}.  Statements (I) \& (II) in Theorem~\ref{ThmExistence} determine the first- \& second-order correlations of the random measure $\mathbf{M}_r$, statement (III) effectively gives some  control over the higher-order correlations in the regime $-r \gg 1$, and statement (IV) features the distributional hierarchical symmetry, which we extrapolate from in Section~\ref{SubSecLimitLawInfDecomp}.

\begin{theorem}\label{ThmExistence}
There is a unique one-parameter family of laws for random measures $\{\mathbf{M}_{r}\}_{r\in \R}$  on the path space, $\Gamma$, of $D$ satisfying properties (I)-(IV) below.  Moreover, this family of random measure laws is  the same as the family of limit laws   in Theorem~\ref{ThmUniversality}. 

\begin{enumerate}[(I)]

\item The expectation of the measure $\mathbf{M}_{r}$ with respect to the underlying probability space is the uniform measure on paths: $\mathbb{E}[\mathbf{M}_r ]=\mu  $.  More precisely,   $\mathbb{E}[ \mathbf{M}_{r}(A) ]=\mu(A)$ for any $A\in \mathcal{B}_{\Gamma}$.

\item For the measure   $ \upsilon_r $ on $\Gamma\times \Gamma  $ of Lemma~\ref{LemCorrelate}, we have that $\mathbb{E}[ \mathbf{M}_r  \times \mathbf{M}_r    ]= \upsilon_r $.   In particular, for any measurable function $g:\Gamma\times \Gamma\rightarrow [0,\infty)$
$$ \mathbb{E}\bigg[ \int_{ \Gamma \times \Gamma  } g(p,q) \mathbf{M}_r(dp) \mathbf{M}_r(dq)     \bigg]\,=\,\int_{ \Gamma\times \Gamma   }g(p,q)  \upsilon_r(dp,dq) \,. $$

\item For each $m\in \{3,4,\ldots \}$, the $m^{th}$ centered moment of the total mass of $\mathbf{M}_r$ is  equal to $R^{(m)}(r)$ for a continuous, increasing function $R^{(m)}:\R\rightarrow \R_+$ that vanishes with order $(\frac{1}{-r})^{\lceil m/2\rceil}$ as $r\rightarrow -\infty$ and diverges to $\infty$ as $r\rightarrow \infty$.\footnote{To be more explicit, we mean that $R^{(m)}(r):=\mathbb{E}\big[\big(\mathbf{M}_r(\Gamma)-1   \big)^m\big]$.}

\item If $\big\{\mathbf{M}_{r}^{(i,j)}\big\}_{1\leq i,j\leq b} $ is a family of independent copies of the random measure  $\mathbf{M}_{r} $,  then there is the following equality in distribution of  random measures:
\begin{align*}
\mathbf{M}_{r+1 }\,\stackrel{\textup{d}}{=}\, \frac{1}{b}\bplus_{i=1}^b  \prod_{j=1}^b  \mathbf{M}_{r }^{(i,j)}  \hspace{.5cm} \text{under the identification} \hspace{.5cm} \Gamma\,\equiv\, \bigsqcup_{i=1}^b  \bigtimes_{j=1}^b  \Gamma \,. 
\end{align*}
\end{enumerate}
For the  uniqueness, it suffices to replace (III) by the weaker condition (III') that the fourth centered moment of  $ \mathbf{M}_r(\Gamma)$ vanishes as $r\rightarrow -\infty$.
\end{theorem}

\begin{remark}
As a consequence of  (I) above, the expectation of  $\mathbf{M}_r(\Gamma)$ is $\mu(\Gamma)=1$. 
\end{remark}
\begin{remark} The second moment of $\mathbf{M}_r(\Gamma)$ is $1+R(r)$ since $\mathbb{E}\big[ \big(\mathbf{M}_r(\Gamma)\big)^2    \big]=\mathbb{E}\big[ \mathbf{M}_r  \times \mathbf{M}_r(\Gamma\times \Gamma)    \big]=\upsilon_r(\Gamma\times \Gamma)$ and $\upsilon_r$ has total mass $1+R(r)$.  Thus the variance of the total mass of $\mathbf{M}_{r}$ is equal to $R(r)$.
\end{remark}

\begin{remark}  Statement (IV) of Theorem~\ref{ThmExistence} should be interpreted in a similar way to Remark~\ref{RemarkConcat}, in other terms, as meaning that the random measure $\mathbf{M}_{r+1 }$ is equal in distribution to the pushforward of the random measure $\frac{1}{b}\bplus_{i=1}^b  \prod_{j=1}^b  \mathbf{M}_{r }^{(i,j)}$ by the canonical bijection  from $\bigsqcup_{i=1}^b \bigtimes_{j=1}^b  \Gamma $ to  $\Gamma$
\end{remark}

\subsection{Infinite decomposability for realizations of the limiting RCPM}\label{SubSecLimitLawInfDecomp}

The proposition below unfolds a structural consequence of property (IV) of Theorem~\ref{ThmExistence} held for a.e.\ realization of the random measure $\mathbf{M}_{r}$.  Its proof,  in Section~\ref{ProofThmExistence2}, merely relies on the Kolmogorov extension theorem. We will use  the following additional diamond graph notation:
\begin{notation}\label{NotationSub} \textup{For $e\in E_k$ and $(i,j)\in \{1,\ldots,b\}^2$, let $e\times (i,j)$ denote the element in $E_{k+1}$ corresponding to the $j^{th}$ segment along the $i^{th}$ branch on the subcopy of $D_1$  within $D_{k+1}$ associated to the edge $e\in E_k$.}   
\end{notation}

\begin{proposition}\label{CorProp4} For  $r\in \R$,  the random measure $\mathbf{M}_r$ from Theorem~\ref{ThmExistence} can be defined on the same probability space  as a family of random Borel measures  $\mathbf{M}^{e}$ on $\Gamma$ indexed by $e\in \bigcup_{k=0}^{\infty} E_k$ that a.s.\ satisfies  the properties (i)-(ii) below for every $k\in \mathbb{N}_0$, where $\mathbf{M}^{e}:=\mathbf{M}_r$ for $e\in E_0$.

\begin{enumerate}[(i)]

\item  $\big\{\mathbf{M}^{e}\big\}_{e\in E_k}$ is a family of independent  copies of the random measure $\mathbf{M}_{r-k}$.

\item  $\displaystyle \mathbf{M}^{e}= \frac{1}{b}\bplus_{i=1}^b  \prod_{j=1}^b  \mathbf{M}^{e\times(i,j)} $ under the identification $\displaystyle \Gamma\equiv \bigsqcup_{i=1}^b \bigtimes_{j=1}^b  \Gamma$ for any $e\in E_k$.

\end{enumerate}
In particular, applying (ii) recursively yields  that for any $n\in \mathbb{N}$ 
\begin{align}\label{PathProd}
\mathbf{M}_{r }\,=\, \frac{1}{|\Gamma_n|}\bplus_{\mathbf{p}\in \Gamma_n}  \prod_{\ell=1}^{b^n}  \mathbf{M}^{\mathbf{p}(\ell)} \hspace{.5cm}  \text{under the identification} \hspace{.5cm} \Gamma\,\equiv\, \bigsqcup_{ \mathbf{p}\in \Gamma_n } \bigtimes_{\ell=1}^{b^n}\Gamma\,.
\end{align}
 \end{proposition}

Let us coin a term for the property found in (ii) of Proposition~\ref{CorProp4}.

\begin{definition}\label{DefDecomp} We will refer to a measure $\eta\in \mathcal{M}_{\Gamma}$ as being \textit{infinitely decomposable} if there exists a family of   measures $ \eta^e\in \mathcal{M}_{\Gamma}$  indexed by $e\in  \bigcup_{k=0}^{\infty} E_k$ such that $\eta^e:=\eta$ when $e\in E_0$ and  the hierarchical relation $ \eta^{e}= \frac{1}{b}\bplus_{i=1}^b  \prod_{j=1}^b  \eta^{e\times(i,j)} $ holds under the identification $\Gamma\equiv \bigsqcup_{i=1}^b \bigtimes_{j=1}^b  \Gamma$ for all $e\in \bigcup_{k=0}^{\infty} E_k$. 
The set of infinitely decomposable measures in $\mathcal{M}_{\Gamma}$  will be denoted by $\mathcal{M}_{\Gamma}^{\diamond}$.
\end{definition}

\begin{remark} Obviously, the component measures $\{ \eta^e\}_{e \in  \cup_{k=0}^{\infty} E_k }$ of $\eta\in \mathcal{M}_{\Gamma}^{\diamond}$ are also in $\mathcal{M}_{\Gamma}^{\diamond}$.  In addition, $\mathcal{M}_{\Gamma}^{\diamond}$ is a closed subspace of $\mathcal{M}_{\Gamma}$ under the weak topology.
\end{remark}

\begin{remark}\label{RemarkProp44} Let $\eta$ be an infinitely decomposable measure on $\Gamma$ and   $\{\eta^e\}_{ e\in \cup_{k=0}^{\infty} E_k}  $ be its family of component measures. For any $n\in \mathbb{N}$ and  $\mathbf{p}\in \Gamma_n$, the restriction of $\eta$ to the cylinder set $\mathbf{\overline{p}}$ is the product measure $\frac{1}{|\Gamma_n|}\prod_{\ell=1}^{b^n}  \eta^{\mathbf{p}(\ell)}$ under the identification $\mathbf{\overline{p}}\equiv  \bigtimes_{\ell=1}^{b^n}  \Gamma$.  Note that this is just a different way of formulating what~(\ref{PathProd}) states for $\eta=\mathbf{M}_r$.
\end{remark}

 The following is an immediate corollary of Proposition~\ref{CorProp4}.
\begin{corollary}\label{CorInfDec} For any $r\in \R$,  the random measure $\mathbf{M}_r$ from Theorem~\ref{ThmExistence} takes values in $\mathcal{M}_{\Gamma}^{\diamond}$ with probability one.  Thus there is an $\mathcal{M}_{\Gamma}^{\diamond}$-valued version of $\mathbf{M}_r$.
\end{corollary}

\subsection{Some basic properties of the limiting RCPM laws}\label{SubSecLimitLawProp}

The next proposition addresses some elementary questions about the typical behavior of the  random path measures $\mathbf{M}_{r}$, and the proofs of parts (i)-(iv) are in Section~\ref{SecProofPropProperties}. Part (v), which  characterizes a transition to strong disorder as $r\rightarrow \infty$, is proved in~\cite[Section 5]{Clark4}. The properties listed below are not surprising because they are also present in the subcritical ($b<s$) counterpart of the RCPMs. 

\begin{proposition}\label{PropProperties} Let the family of random measures $\{\mathbf{M}_{r}\}_{r\in \R}$ be defined as in Theorem~\ref{ThmExistence}.  The statements (i)-(iii) below hold for a.e.\ realization of  $\mathbf{M}_r$.
\begin{enumerate}[(i)]

\item  $\mathbf{M}_r$ is  mutually singular to $\mu$.

\item  $\mathbf{M}_r$ is diffuse, i.e., has no atoms.

\item $\mathbf{M}_r$ assigns positive weight to every  open subset of $\Gamma$.

\item As $r\rightarrow -\infty$, the random measures $\mathbf{M}_r$ converge in distribution to the nonrandom measure  $\mu$.  In fact, the following slightly stronger statement holds: for any $g\in L^2(\Gamma  ,\mu)$, the random variables $\mathbf{M}_r(g) $ converge  to $\mu(g) $ in $L^2$ as $r\rightarrow -\infty$.\footnote{Recall from~(\ref{ConvEquiv}) that $ \mathbf{M}_r\stackrel{\textup{d} }{\rightarrow} \mu   $ if and only if  $ \mathbf{M}_r(g)\stackrel{\textup{d} }{\rightarrow} \mu(g)   $ for all $g\in C(\Gamma)$.}

\item The total mass, $\mathbf{M}_r(\Gamma)$, converges in probability to zero as $r\rightarrow \infty$.

\end{enumerate}

\end{proposition}

\begin{remark}In the language of~\cite{alberts2},
the \textit{continuum directed random polymer} (CDRP) on $D$ with parameter $r\in \R$ refers to  the random probability measure $Q_r(dp)=\mathbf{M}_r(dp)/\mathbf{M}_r(\Gamma)$.    This ratio is a.s.\  well-defined because the measure $\mathbf{M}_r$ is a.s.\ finite and $\mathbf{M}_r(\Gamma)>0$ by  (iii) of Proposition~\ref{PropProperties}.  An in-depth  study of $Q_r$ in the $r\gg 1$ regime  may yield interesting conclusions, but the current article does not provide the tools and insights needed for this pursuit.
\end{remark}

\section{The set of intersection times between two  paths}\label{SecTwoPaths}

In this section we will discuss the typical behavior of the set of intersection times 
$$I^{p,q}:=\{t\in [0,1]\,|\,p(t)=q(t)   \}$$
for two paths $p,q\in \Gamma$  chosen independently under a fixed realization of $\mathbf{M}_r$.  We introduce a kernel $T(p,q)$ in Section~\ref{SubSecLocTime} that effectively measures the size of the set $I^{p,q}$, shedding light on the Lebesgue decomposition of the correlation kernel  $\upsilon_r:=\mathbb{E}\big[ \mathbf{M}_r\times \mathbf{M}_r  ]$  in part (ii) of Lemma~\ref{LemCorrelate}.  In Section~\ref{SubSecTwoPaths}, we present a closely related Hausdorff-type dimensional analysis of  $I^{p,q}$.   The study of $T(p,q)$ is extended in Sections~\ref{SecLocalResults} \& \ref{SectionHilbert} through results related to the random linear operator, $\mathbf{T}_{\mathbf{M}_r}$, on the space $L^2(\Gamma, \mathbf{M}_r)$ having $T(p,q)$ as its integral kernel. These latter results are applied in~\cite{Clark4} to the study of the Gaussian multiplicative chaos structure of the family of random path measure laws $\{ \mathbf{M}_r \}_{r\in \R}$, which we summarized in the introduction.

\subsection{The correlation measure and the intersection time between two paths  }\label{SubSecLocTime}

Recall from Lemma~\ref{LemCorrelate} that the  Lebesgue decomposition of the correlation measure $\upsilon_r$  with respect to $\mu\times \mu$ simply has components $\mu\times \mu$ and $\upsilon_r-\mu\times \mu=R(r)\rho_r$ for a probability measure  $\rho_r$  supported on the set of pairs $(p,q)\in \Gamma\times \Gamma$ such that $\xi_n(p,q)>0$ for all $n$, where $\xi_n(p,q)$ denotes the number of graphical edges shared by the discrete paths $[p]_n,[q]_n\in \Gamma_n$. The following proposition implies that two paths $p,q\in \Gamma$ independently sampled   uniformly at random---in other words, using $\mu\times \mu$---will a.s.\ have a finite   set of intersection times, $I^{p,q}$. This contrasts with the DHL in the case $b<s$, for which there is a positive probability that the set of intersection times will have Hausdorff dimension $(\log s -\log b)/\log s$, and thus be uncountably infinite~\cite{Clark2}.  The proof is short and nontechnical, so we will present it here rather than delaying it to a later section.

\begin{proposition}\label{PropPathInter}
 If $p\in \Gamma$ is fixed and $q\in \Gamma$ is chosen uniformly at random, i.e.,  according to the  measure $\mu  $, then  the set of intersection times  $I^{p,q}$ is a.s.\ finite.  In particular,  the points of intersection  $p(r)=q(r)\in D$ are all contained in $V$.  
\end{proposition} 

\begin{proof} It suffices to show that the sequence $\big(\xi_n(p,q)\big)_{n\in \mathbb{N}}$  has only finitely many nonzero terms for $\mu$-a.e.\ $q\in \Gamma$. The sequence $\xi_n (p,q)$ can be understood as the number of members  at generation $n\in \mathbb{N}_0$ for a simple Markovian  population model  that begins with a single member ($\xi_0(p,q)=1 $), and where each member of the population  independently has either no children
with probability $\frac{b-1}{b}$ or $b$ children with probability $\frac{1}{b}$.  If $\frak{p}_{n}$ denotes the probability of extinction by generation $n$, then $\frak{p}_{0}=0$ and, by  hierarchical symmetry, the sequence $(\frak{p}_{n})_{n\in \mathbb{N}_0}$ satisfies the recursive relation $\frak{p}_{n+1}=\psi(\frak{p}_{n})$ for  $\psi:[0,1]\rightarrow [0,1]$ defined by  $\psi(x):=\frac{b-1}{b}+\frac{1}{b}x^b $.  The map $\psi$ has a unique attractive fixed point at $x=1$, and  thus $\frak{p}_{n}$ converges to $1$ with large $n$.  It follows that eventual extinction has probability $1$, i.e., $ \xi_n(p,q)=0$ for large enough $n$ for $\mu$-a.e.\ $q$.  Finally, when $ \xi_n(p,q)=0$, the intersection of the sets $\textup{Range(p)}$ and  $\textup{Range}(q)$ is a subset of $V_{n-1}\subset V$.
\end{proof}

\begin{corollary}\label{CorrMuMu}  The set, $\mathbf{S}_\emptyset  $, of pairs $(p,q)\in \Gamma\times \Gamma$ such that $ \xi_n(p,q)=0$ for large enough $n$ is a full-measure set for $\mu\times \mu$.  
\end{corollary}

\begin{remark}$\mathbf{S}_\emptyset  $ can also be described as the set of $(p,q)\in \Gamma\times \Gamma$ such that  $I^{p,q}$ is finite. 
\end{remark}

\begin{remark}\label{RemarkCond}  The measure $\rho_r$ on $\Gamma\times \Gamma$ from part (ii) of Lemma~\ref{LemCorrelate} can be defined through conditioning the  correlation measure $\upsilon_r$ on the event $\mathbf{S}_{\emptyset}^c$, in other terms,  $\rho_r(A):=\upsilon_r\big(A\cap \mathbf{S}_{\emptyset}^c\big)/\upsilon_r\big(\mathbf{S}_{\emptyset}^c\big) $ for  $A\in \mathcal{B}_{\Gamma\times\Gamma}$.  
\end{remark}

The following proposition defines the intersection-time kernel $T(p,q)$ mentioned at the beginning of the section and links $T(p,q)$ with the family of correlation measures $\{\upsilon_{r}\}_{r\in\R}$ through the Radon-Nikodym derivative formula $\frac{d\nu_t}{d\nu_r}(p,q)=  \textup{exp}\{(t-r)T(p,q)\}$ for $t,r\in \R$.  The limit in part (i) that defines $T(p,q)$ can be understood through a critical generation-inhomogeneous  Markovian population model discussed in Section~\ref{SecPopModel}.  The proof is in Section~\ref{SecPropLemCorr}.

\begin{proposition}\label{PropLemCorrelate}
Let the family of  measures $\{\upsilon_{r}\}_{r\in\R}$ on $\Gamma\times \Gamma$ be defined as in Lemma~\ref{LemCorrelate}.  The statements below hold for any $r\in \R$.

\begin{enumerate}[(i)]

\item  The sequence  $ \big(\frac{\kappa^2}{n^2} \xi_n(p,q)\big)_{n\in \mathbb{N}} $ converges $ \upsilon_r$-a.e.\   to a finite limit that we denote by $T(p,q)$.  We define $T(p,q)=\infty$ for pairs $(p,q)$ such that the sequence  $ \big(\frac{\kappa^2}{n^2} \xi_n(p,q)\big)_{n\in \mathbb{N}} $  is divergent.

\item  In particular, $\upsilon_r$ assigns full measure to the set of pairs $(p,q)\in \Gamma\times \Gamma$ such that  $ T(p,q)$ is finite.  Regarding the components $\mu\times \mu$ and $R(r)\rho_r$ from the Lebesgue decomposition of $\upsilon_r$  in Lemma~\ref{LemCorrelate}, the measure
 $\mu\times \mu$ is supported on the set of pairs  $(p,q)$ with $T(p,q)=0$ as a trivial consequence of Corollary~\ref{CorrMuMu}, and the measure $\rho_r$ is supported on the set of pairs $(p,q)$ with $0<T(p,q)<\infty$.

\item $\upsilon_t$ has  Radon-Nikodym derivative $\textup{exp}\{(t-r)T\}$   with respect to $\upsilon_r$ for any $t\in \R$.

\item The exponential moments of $T(p,q)$ under $\upsilon_r$ have the form 
$$  1+R(r+a)\,=\,\int_{\Gamma\times \Gamma } e^{aT(p,q)  }   \upsilon_{r}(dp,dq) \hspace{1cm}\text{for any}\,\, a\in \R  \,. $$

\end{enumerate}

\end{proposition} 

\begin{remark}\label{RemarkDerRII} By differentiating the formula in (iv) of Proposition~\ref{PropLemCorrelate} with respect to $a$ and setting $a=0$, we get  $R'(r)=\int_{\Gamma\times \Gamma}T(p,q)\upsilon_{r}(dp,dq)$. 
\end{remark}

\begin{remark}\label{RemarkMart}   The rescaling factor $\frac{\kappa^2}{n^2}$ in the limit definition of $T(p,q)$ in part (i) above arises from the sequence $ \big(\frac{R'(r-n)}{1+R(r-n)} \xi_n(p,q)\big)_{n\in \mathbb{N}} $ being a martingale on the measure space $\big(\Gamma\times \Gamma, \mathcal{B}_{ \Gamma\times \Gamma }, \upsilon_r\big)$ that converges $\upsilon_r$-a.e.\ to a limit $T(p,q)$---see Proposition~\ref{PropMartingales}---and from the large $n$  asymptotics $\frac{R'(r-n)}{1+R(r-n)}\sim \frac{\kappa^2}{n^2}$, which follows from (II) of Lemma~\ref{LemVar} and Remark~\ref{RemarkDerR}.
\end{remark}

\begin{remark}\label{RemarkIntTimeKernel}If for a given $n\in \mathbb{N}$  we identify paths $p,q\in \Gamma$ with  tuples $(\mathbf{p};  p_1,\ldots,  p_{b^n})$ and $(\mathbf{q};  q_1,\ldots,  q_{b^n})$ in the Cartesian product $\Gamma_n\times \bigtimes_{\ell=1}^{b^n}\Gamma$ as in Remark~\ref{RemarkCylinder}, we can express the hierarchical symmetry of the intersection-time kernel $T(p,q)$ through the equality
\begin{align}\label{TFormula}
T(p,q)\,=\,\sum_{\substack{ 1\leq  \ell \leq b^n \\ \mathbf{p}(\ell)=\mathbf{q}(\ell) } } T(p_\ell,q_\ell) \,,  \vspace{-.5cm}
\end{align}
which holds for $ \upsilon_r$-a.e.\ $(p,q)$.  The  absence of a rescaling factor $\lambda^n$ for some $0<\lambda <1$ in front of the sum on the right-hand side is a symptom of the model's criticality. 
\end{remark}

\begin{remark}\label{RemarkRHierSymmRe} By integrating out the `microscopic' path variables $p_1,\ldots,  p_{b^n}$ and $q_1,\ldots,  q_{b^n}$ in~(\ref{TFormula}) through the product measure  $\prod_{\ell=1}^{b^n} \mathscr{U}^{\mathbf{p}(\ell),\mathbf{q}(\ell)}_{r-n}$ from the decomposition of $\upsilon_r$ in Remark~\ref{RemarkRHierSymm}, we get
 \begin{align*}
&\int_{\prod_{\ell=1}^{b^n}\Gamma\times \Gamma } T(p,q)\prod_{\ell=1}^{b^n} \mathscr{U}^{\mathbf{p}(\ell),\mathbf{q}(\ell)}_{r-n}(dp_{b^\ell},dq_{b^\ell})\\ &\hspace{.3cm}\,=\, \sum_{\substack{ 1\leq  \ell \leq b^n \\ \mathbf{p}(\ell)=\mathbf{q}(\ell) } } \frac{1}{1+R(r-n)}\int_{\Gamma\times \Gamma}T(p_\ell,q_\ell)\upsilon_{r-n}(dp_\ell,dq_\ell)\,=\,\frac{R'(r-n)}{1+R(r-n)} \xi_n(p,q)\,,
\end{align*} 
where the second equality applies Remark~\ref{RemarkDerRII} and the definition of $\xi_n(p,q)$.  The above computation shows that $\frac{R'(r-n)}{1+R(r-n)} \xi_n(p,q)$ is the  conditional expectation of $T(p,q)$ under $\upsilon_r$ given the coarse-grainings $\mathbf{p}:=[p]_n$ and $\mathbf{q}:=[q]_n$, which offers some perspective on the martingale  in Remark~\ref{RemarkMart}.
\end{remark}

\subsection{A Hausdorff analysis of the set of intersection times between two paths}\label{SubSecTwoPaths}

By Corollary~\ref{CorrMuMu}, the measure $  \mu\times \mu$ assigns full weight to the set of pairs $(p,q)\in \Gamma\times \Gamma$ having finite intersection-times sets, $I^{p,q} $.  In contrast, parts (i) \& (iii) of Theorem~\ref{ThmPathMeasure} below imply that for a.e.\ realization of the random measure $\mathbf{M}_r$ the product $\mathbf{M}_r\times \mathbf{M}_r$ assigns positive weight to the set of pairs $(p,q)$ such that $I^{p,q}$ is uncountably infinite but of Hausdorff dimension zero.  The definitions below provide us with a framework for characterizing the size of the sets $I^{p,q}$.

\begin{definition} A \textit{dimension function} is a continuous, non-decreasing function $h:[0,\infty)\rightarrow [0,\infty)     $
satisfying $h(0)=0$. Given a dimension function $h$, the \textit{generalized Hausdorff outer  measure}  of a  set $S\subset \R^d$ is defined through the limit
\begin{align*}
 H^{h}(S)\,:=\,\lim_{\delta\searrow 0}H^{h}_{\delta}(S)  \hspace{.5cm}\text{for}\hspace{.5cm}    H^{h}_{\delta}(S)\,:=\,\inf_{\substack{S\subset \cup_{k}\mathcal{I}_{k}  \\  \textup{diam}(\mathcal{I}_k)\leq  \delta  }} \sum_{k}h\big( \textup{diam}(\mathcal{I}_k)\big),  
\end{align*}
where  the infimum is over all countable coverings of $S$ by sets $\mathcal{I}\subset \R^d$ of diameter bounded by $ \delta$.  A dimension function $h$ is said to be \textit{zero-dimensional}  if  $ x^{\alpha}\ll h(x)$ as  $x\searrow 0$  for any $\alpha>0$. 

\end{definition}

\begin{remark}\label{Remark}When the dimension function has the form $h_{\alpha}(x)=x^\alpha$ for $\alpha>0$, then $H^{h_{\alpha}}$ is the standard dimension-$\alpha$ Hausdorff outer measure.  The \textit{Hausdorff dimension} of a  set $S\subset \R^d$  is defined as the supremum of the set of $\alpha\in (0,d]$ such that $H^{h_{\alpha} }(S)=\infty $.\footnote{The supremum of the empty set is interpreted as zero in this context.}  
\end{remark}

\begin{definition}\label{DefLogHaus}
Let $S\subset \R^d$ be a set of Hausdorff dimension zero.  For $\frak{h} > 0$ and  $0<\delta<1$, define  $H^{\log}_{\frak{h},\delta}(S):= H^{h}_\delta(S) $ and $H^{\log}_{\frak{h}}(S) := H^{h}(S)  $ for the dimension function $h(x)\,=\, 1/ \log^{\frak{h}}(\frac{1}{x})$.  We define the \textit{log-Hausdorff exponent} of $S$ as the supremum  of the set of $\frak{h}$ such that   $H^{\log}_{\frak{h}}(S)=\infty $.
\end{definition}

The lemma below offers additional perspective on the Lebesgue decomposition found in part (ii) of Lemma~\ref{LemCorrelate}.    The main step of the proof is in Section~\ref{SubSectionPathIntRho}.
\begin{lemma}\label{LemIntSet}  Let the measure $ \upsilon_{r}$ on $\Gamma\times \Gamma$ be defined as in Lemma~\ref{LemCorrelate}.  The normalized measure $\rho_r=\frac{1}{R(r)}(\upsilon_{r}-\mu\times \mu)    $  assigns probability one to the set of pairs $(p,q)$ such that the intersection-times set $I^{p,q} $ has log-Hausdorff exponent $\frak{h}=1$.
\end{lemma}

The theorem below, which is proved in Section~\ref{SecPathIntMxM}, states our main results on the behavior of $I^{p,q}$ under $\mathbf{M}_r\times\mathbf{M}_r$ for typical realizations of $\mathbf{M}_r$.  Part (iv) states that the  cross sections of the set, $\mathbf{G}$, of path pairs $(p,q)$ for which  $I^{p,q}$ has log-Hausdorff exponent one are $\mathbf{M}_r$-a.e.\ not $\mathbf{M}_r$-null. 

\begin{theorem}\label{ThmPathMeasure}
Let the family of random measure laws $\{\mathbf{M}_r\}_{r\in \R}$ be defined as in Theorem~\ref{ThmExistence}.  The statements below hold for any $r\in \R$ and a.e.\ realization of  $\mathbf{M}_r$.

\begin{enumerate}[(i)]

\item The  set   $I^{p,q}$  has Hausdorff dimension zero for $\mathbf{M}_r\times \mathbf{M}_r  $-a.e.\ pair $(p,q)$ in $\Gamma\times \Gamma$.

\item   The product measure $\mathbf{M}_r\times \mathbf{M}_r  $ is supported on the set of  $(p,q)\in \Gamma\times \Gamma$ such that $T(p,q)<\infty$, i.e., where $\frac{1}{n^2}\xi_n(p,q)$ converges to a finite limit  as $n\rightarrow \infty$.  Moreover, the exponential moments of $T(p,q)$ with respect to $\mathbf{M}_r\times \mathbf{M}_r$ have expected value
$$\text{}\hspace{.1cm} \mathbb{E}\bigg[ \int_{\Gamma\times \Gamma}e^{aT(p,q)} \mathbf{M}_r(dp)\mathbf{M}_r(dq)      \bigg]\,=\,  1+R(r+a)\hspace{1cm}\text{for any $a\in \R$}\,.   $$

\item   The product measure  $\mathbf{M}_r\times \mathbf{M}_r  $ assigns full weight to the set of $(p,q)\in \Gamma\times \Gamma$ such that one of the following statements hold:\vspace{.2cm}\\
\text{}\hspace{1.5cm} (a)  $I^{p,q}$ is finite \hspace{.5cm} or\hspace{.5cm} (b) $I^{p,q}$ has log-Hausdorff exponent $\frak{h}=1$.\vspace{.2cm} \\
Moreover, $\mathbf{M}_r\times \mathbf{M}_r  $ assigns both of these events positive  weight. The above a.s.\ dichotomy remains  true when statement (b) is replaced by  `\,$0<T(p,q)<\infty$'.

\item Given $p\in \Gamma$, let $\mathbf{G}_p$ denote the set of $q\in  \Gamma$ such that the set of intersection times $I^{p,q}$ has log-Hausdorff exponent $\frak{h}=1$.  Then $\mathbf{M}_r$  satisfies that $\mathbf{M}_r(\mathbf{G}_p)>0$ for  $\mathbf{M}_r$-a.e.\ $p\in \Gamma$.  The analogous statement holds for the sets $\mathbf{\widehat{G}}_p:=\{q\in \Gamma\,|\, T(p,q)>0\}$.

\end{enumerate}

\end{theorem}

\begin{remark}Parts (iii) and~(iv) of Theorem~\ref{ThmPathMeasure} imply a form of locality for the disordered measures $\{\mathbf{M}_{r}\}_{r\in \R}$.  Two paths chosen independently according to  $\mathbf{M}_r$ may intersect nontrivially, whereas this is impossible under the pure measure $\mu$ by Corollary~\ref{CorrMuMu}.   This supports the perspective that the disordered environment induces the possibility of independently sampled paths  intersecting nontrivially by strongly restricting the set of points that paths are allowed to pass through. 
\end{remark}

\begin{remark}\label{RemarkInd} Let $\mathbf{M}_r$ and $\mathbf{M}_r'$ be independent copies of the random  measure from Theorem~\ref{ThmExistence}.  Then the set $\mathbf{\widehat{G}}$ of pairs $(p,q)$ such that  $T(p,q)>0$ is a.s.\ a null set for  $\mathbf{M}_r\times \mathbf{M}_r'$ since 
$$ \mathbb{E}\big[ \mathbf{M}_r\times \mathbf{M}_r' \big(\mathbf{\widehat{G}}\big)\big]\,=\,\mathbb{E}[ \mathbf{M}_r]\times \mathbb{E}[ \mathbf{M}_r'] \big(\mathbf{\widehat{G}}\big)\,=\,\mu\times \mu\big(\mathbf{\widehat{G}}\big)\,=\,0 \,, $$
where the second equality holds by (I) of Theorem~\ref{ThmExistence} and the third is from (ii) of Proposition~\ref{PropPathInter}. Thus, for realizations of  $\mathbf{M}_r$ and $\mathbf{M}_r'$, a path sampled using $\mathbf{M}_r$ and a path sampled using $\mathbf{M}_r'$ will have negligible overlap. 
\end{remark}

\begin{remark} Part (iv)  of  Theorem~\ref{ThmPathMeasure} implies that $\mathbf{M}_r$ and $\mathbf{M}_r'$ are a.s.\ mutually singular because $\mathbf{M}_r$ is supported on $\big\{p\in \Gamma\,\big|\,\mathbf{M}_r\big(\mathbf{\widehat{G}}_p\big)>0\big\}$, which  is a.s.\ a null set for $\mathbf{M}_r' $ since  by Remark~\ref{RemarkInd}
$$ 0\,=\,\mathbf{M}_r\times \mathbf{M}_r' \big(\mathbf{\widehat{G}}\big)\,=\,\int_{\Gamma }\mathbf{M}_r\big(\mathbf{\widehat{G}}_p\big) \mathbf{M}_r'(dp)\,, $$
where we have evaluated the product measure of $\mathbf{\widehat{G}}$ through an integral over its cross sections. 
\end{remark}

\subsection{The spatial concentration of path intersection time} \label{SecLocalResults}
 
Our formalism for analyzing  path intersections under the random product measure $\mathbf{M}_r\times \mathbf{M}_r$ can be extended  by considering a measure $\vartheta_{\mathbf{M}_r}$ induced on  $D$ by weighing each set $A\in \mathcal{B}_D$ through the `average' intersection time that pairs of paths spend intersecting within $A$ when independently sampled using $\mathbf{M}_r$.  In the next subsection, these measures are used in the definition of a linear operator $\hat{Y}_{\mathbf{M}_r}$ from  $L^2(D,\vartheta_{\mathbf{M}_r})$ to  $L^2(\Gamma,\mathbf{M}_r)  $ satisfying $\mathbf{T}_{\mathbf{M}_r}:=\hat{Y}_{\mathbf{M}_r}\hat{Y}_{\mathbf{M}_r}^*$, where $\mathbf{T}_{\mathbf{M}_r}$  is the linear operator on $L^2(\Gamma,\mathbf{M}_r) $  having integral kernel $T(p,q)$.

   First, we revisit  the intersection-time kernel $T(p,q)$ from Proposition~\ref{PropLemCorrelate}  by defining a canonical Borel measure $\tau^{p,q}$ on $[0,1]$ that has total mass $T(p,q)$ and assigns full weight to the  set of intersection times $I^{p,q}$.  The proof of the following proposition  is in Section~\ref{SubSecPropIntMeas}.  Recall that $\mathcal{M}_{[0,1]}$ denotes the set of finite Borel measures on $[0,1]$, which we equip with the weak topology and the associated Borel $\sigma$-algebra.  We use $\mathcal{V}$ to denote the set of $t\in [0,1]$ having  the form $t=\frac{\ell}{b^n}$ for some $\ell,n\in \mathbb{N}_0$, i.e., the set of times at which  paths pass through the vertex set $ V$.

\begin{proposition}\label{PropIntMeasure} There is a  measurable function from $ \Gamma\times \Gamma$ to $ \mathcal{M}_{[0,1]}$  sending  pairs  $(p,q)$  to   measures $\tau^{p,q}$ that are supported on the set of intersection times $I^{p,q}$, have $\mathcal{V}$ as a null set,  and satisfy the following property for  $\upsilon_r$-a.e.\ $(p,q)$: for each $0\leq x\leq 1$, the measure of the interval $[0,x]$ under $ \tau^{p,q}$ is given by the limit 
\begin{align}\label{TauForm}
\tau^{p,q}([0,x])=\lim_{n\rightarrow \infty}\frac{\kappa^2}{n^2}\left|\bigg\{ 1\leq \ell \leq b^n  \,\bigg|\, [p]_n(\ell)=[q]_n(\ell)\,\,\,\text{and}\,\,\,\Big[\frac{ \ell-1 }{b^n  }, \frac{ \ell }{ b^n }   \Big]\subset [0,x]\bigg\}\right| \,.
 \end{align}
For $\upsilon_r$-a.e.\ pair $(p,q)$, the measure $\tau^{p,q}$ is diffuse and has total mass $T(p,q)$. The measurable function $(p,q)\mapsto \tau^{p,q}$  is unique up to $\upsilon_r$-null subsets of $\Gamma\times \Gamma$.

\end{proposition}

\begin{remark}
Since $\upsilon_r=\mathbb{E}[\mathbf{M}_{r}\times \mathbf{M}_r]$,
any  $\upsilon_r$-null set must a.s.\ be a $\mathbf{M}_{r}\times \mathbf{M}_r$-null set. In particular, $\mathbf{M}_{r}\times \mathbf{M}_r$ must a.s.\ be supported on the set of  pairs $(p,q)$ for which the measure $\tau^{p,q}$ satisfies~(\ref{TauForm}).
\end{remark}

\begin{remark}\label{RemarkTau} We can generalize the observation on hierarchical symmetry in Remark~\ref{RemarkIntTimeKernel} 
as follows: for $\upsilon_r$-a.e.\ $(p,q)$, we have that for all $A\in \mathcal{B}_{[0,1]} $\vspace{-.1cm} $$\tau^{p,q}(A)\,=\,\sum_{\substack{1\leq \ell \leq b^n  \\  \mathbf{p}(\ell)=\mathbf{q}(\ell)  }} \tau^{ p_{\ell}, q_{\ell}}\big(A_{\ell}^n\big) \hspace{.5cm} \text{where} \hspace{.5cm} A_{\ell}^n\,:=\,(b^n A-\ell+1)\cap [0,1] \vspace{-.1cm}   $$
under the same identifications
 $p\equiv \big(\mathbf{p};  p_1,\ldots,  p_{b^n}\big)$ and $q\equiv\big(\mathbf{q};  q_1,\ldots,  q_{b^n}\big)$ as in Remark~\ref{RemarkIntTimeKernel}. 
\end{remark}

\begin{remark} Recall that (iv) of Theorem~\ref{ThmPathMeasure} states that  for a.e.\ realization of the random path measure $\mathbf{M}_r$, the measure is supported on a set of paths  that intersect nontrivially with a $\mathbf{M}_r$-positive subset of the other paths.  Using the hierarchical symmetry of $\tau^{p,q}$ in Remark~\ref{RemarkTau}, it is not difficult to prove a slightly stronger statement: the partially `averaged' measure $\tau_{\mathbf{M}_r }^{p}:=\int_{\Gamma} \tau^{p,q}\mathbf{M}_r(dq) $  a.s.\ assigns positive weight to all open subsets of $[0,1]$ for  $\mathbf{M}_r$-a.e.\ $p$.  In other terms, a path sampled using $\mathbf{M}_r $ intersects nontrivially with a $\mathbf{M}_r $-positive portion of the path space in each region of its trajectory.
\end{remark}

The following   immediate corollary of Proposition~\ref{PropIntMeasure} introduces measures $\gamma^{p,q}$ that assign weight in the space $D$ according to where the  paths $p$ and $q$ spend their intersection time. We use $\gamma^{p,q}$ in the definition of the  measure $\vartheta_{\eta}$ on $D$ below.   

\begin{corollary}\label{CorollaryGamma} For $p,q\in \Gamma$ and $\tau^{p,q}\in \mathcal{M}_{[0,1]}$ defined as in Proposition~\ref{PropIntMeasure}, let $\gamma^{p,q}$ denote the Borel measure on $D$ given by the  pushforward of $\tau^{p,q}$ by the path $p:[0,1]\rightarrow D$,\footnote{Since $\tau^{p,q}$ is supported on the  set of intersection times $I^{p,q}$, we can also define $\gamma^{p,q}:=\tau^{p,q}\circ q^{-1}$.} i.e., 
$$\gamma^{p,q}\,:=\,\tau^{p,q}\circ p^{-1}\,.  $$  Then the map from $\Gamma\times \Gamma$ to $\mathcal{M}_{D}$ defined by $(p,q)\mapsto \gamma^{p,q}$ is measurable, and  for every $(p,q)$ the measure $\gamma^{p,q}$ is supported on $\textup{Range}(p)\cap \textup{Range}(q)$ and has $V$ as a null set.  
 Moreover, for $\upsilon_r$-a.e.\ pair $(p,q)$, the measure $\gamma^{p,q}$ is diffuse and  has total mass $T(p,q)$. 
\end{corollary}

 In the definition below, $\tilde{\mathcal{M}}_D$  denotes the set of Borel measures on $D$---including infinite ones. 

\begin{definition}\label{DefVartheta} Let $\vartheta: \mathcal{M}_{\Gamma}\rightarrow \tilde{\mathcal{M}}_D$ denote the function that sends   a finite Borel measure $\eta $ on $\Gamma$ to the Borel measure $\vartheta_{\eta}$ on $D$ given by
$$  \vartheta_{\eta}\,:=\,\int_{\Gamma\times \Gamma}\gamma^{p,q}\eta(dp)\eta(dq) \,,   $$
where the measure $\gamma^{p,q}$ on $D$ is defined as in Corollary~\ref{CorollaryGamma}.  Accordingly,  $\vartheta_{\mathbf{M}_r}$ denotes the random   measure on $D$ given by evaluating the map above with the random measure  $\mathbf{M}_r$ from Theorem~\ref{ThmExistence}.
\end{definition}

\begin{remark}\label{RemarkVarthetaNull} For any $\eta\in  \mathcal{M}_{\Gamma}$, we have $\vartheta_{\eta}(V)=0$  since $\gamma^{p,q}(V)=0$ for all $p,q\in \Gamma$.
\end{remark}

\begin{remark}\label{RemarkVarthetaMass}  Since $\gamma^{p,q}$ has total mass $T(p,q)$, the total mass of $\vartheta_{\eta}$ is  $\int_{\Gamma\times \Gamma}T(p,q)\eta(dp)\eta(dq)$.  Thus the expectation of the  total mass of $\vartheta_{\mathbf{M}_r}$  is  $R'(r)$ as a consequence of part  (ii) of Theorem~\ref{ThmPathMeasure}.
\end{remark}

\begin{remark}\label{RemarkVarthetaMu} Since $T(p,q)=0$ for $\mu\times \mu$-a.e.\ pair $(p,q)$ by (ii) of Proposition~\ref{PropLemCorrelate}, it follows from Remark~\ref{RemarkVarthetaMass} that $\vartheta_{\mu}=0$.  If our model were subcritical ($b<s$), then the measure $\vartheta_{\mu}$ would be a positive multiple of the uniform measure $\nu$ on $D$.
\end{remark}

Recall that the diamond fractal $D$ has Hausdorff dimension two. Statements (iii) and (v) of the theorem below effectively provide lower bounds---in  Hausdorff dimensional analysis terms---for how small a subset of $D$ can be and still be assigned positive weight under a typical realization of the random  measure $\vartheta_{\mathbf{M}_r}$. 
As before, let $d_D:D\times D\rightarrow [0,1]$ denote the travel metric on $D$, to be defined in Section~\ref{SectionDHLConst}. For $x, y\in D$, we define $\frak{g}_D(x,y)$ as the smallest $n\in \overline{\mathbb{N}}$ such that $x$ and $y$ do not belong to the same generation-$n$ embedded copy of $D$. Said differently,  $n$ is the first value for which there does not exist $\mathbf{e}\in E_n$ with  $x,y\in \textup{cl}( \mathbf{\overline{e}})$, where $\textup{cl}( \mathbf{\overline{e}})$ denotes the closure of the cylinder set $\mathbf{\overline{e}}\subset E$ associated to $\mathbf{e}$.  The proof of the next theorem is in Section~\ref{SubSecLast6}.

\begin{theorem}\label{ThmVartheta}
Let the random Borel measure $\vartheta_{\mathbf{M}_r}$ on $D$ be defined from  $\mathbf{M}_r$ as in Definition~\ref{DefVartheta}. The following statements hold for any $r\in \R$.
\begin{enumerate}[(i)]

\item $\mathbb{E}[\vartheta_{\mathbf{M}_r}]=R'(r)\nu$

\item $\vartheta_{\mathbf{M}_r}$ a.s.\ assigns positive weight to every open subset of $D$.

\item $\vartheta_{\mathbf{M}_r}$ a.s.\ has Hausdorff dimension two, i.e., if $A\in \mathcal{B}_{D} $ and $\vartheta_{\mathbf{M}_r}(A)>0$, then $\textup{dim}_{H}(A)=2$.

\item The correlation measure $ \mathbb{E}[ \vartheta_{\mathbf{M}_r}\times \vartheta_{\mathbf{M}_r}   ]$ of the random measure $\vartheta_{\mathbf{M}_r}$  has Radon-Nikodym derivative $C_r:D\times D\rightarrow [0,\infty)$  with respect to $\nu \times \nu$  that obeys the following  asymptotics for some   $\mathbf{c}>0$ when $\frak{g}_D(x,y)\gg 1$:
$$  C_r(x,y)\,\sim \,\mathbf{c}\big( \frak{g}_D(x,y)  \big)^8\,. $$

\item Given $\lambda>0$, define $h_\lambda:[0,1)\rightarrow [0,\infty)$ as the dimension function   $h_\lambda(a):=a^2 |\log(1/a)|^{\lambda}    $.  For a Borel measure $\varrho$ on $D$, define the $h_\lambda$-energy 
\begin{align*}
\mathcal{Q}_{\lambda}(\varrho)\,:=\, \int_{D\times D} \frac{1}{ h_\lambda\big(d_D(x,y)\big)}\varrho(dx)\varrho(dy)\, .
\end{align*}
 For a.e.\ realization of $\mathbf{M}_r$, the energy  $\mathcal{Q}_{\lambda}(\vartheta_{\mathbf{M}_r})$ is finite for all $\lambda>9$, and  the expectation of $\mathcal{Q}_{\lambda}(\vartheta_{\mathbf{M}_r})$ is infinite for all $\lambda\leq 9$.

\end{enumerate}

\end{theorem}

\begin{remark} In part (iv) above, $\frak{g}_D(x,y)$ can be  roughly identified with $\log_b\big( \frac{1}{d_D(x,y)}  \big)$, so the correlation density $C_r(x,y)$   essentially has a power 8 of logarithmic  blow-up near the diagonal $x=y$. 
\end{remark}

  For $\lambda>0$, let $H_{\lambda}$ denote the Hausdorff outer measure on $D$ induced by  the dimension function $h_{\lambda}(a):=a^2|\log(1/a)|^{\lambda}$. The following  corollary of the energy bound in (v) of Theorem~\ref{ThmVartheta}  effectively gives a lower bound for `how large' a Borel subset of $D$ must be  to avoid being a $\vartheta_{\mathbf{M}_r}$-null set.  We omit the proof because it is analogous to the proof of Corollary~\ref{CorIntSet} in Appendix~\ref{AppendixHausdorff}.  
\begin{corollary}\label{CorHD}  With probability one, the random measure   $\vartheta_{\mathbf{M}_r}$ has the following property:  every Borel set $S\subset D$ with $\vartheta_{\mathbf{M}_r}(S)>0$ satisfies that $H_{\lambda}(S)=\infty$ for all $\lambda>9$.
\end{corollary}

The above result would be more complete if the following conjecture were established, which  states  `how small' $\vartheta_{\mathbf{M}_r}$-full measure sets can be.  
\begin{conjecture}\label{Conj} With probability one, the random measure   $\vartheta_{\mathbf{M}_r}$ has the following property:  there exists a Borel set $S\subset D$ such that $\vartheta_{\mathbf{M}_r}(S^c)=0$  and  $H_{\lambda}(S)=0$ for all $\lambda<9$.
\end{conjecture}

\begin{remark}\label{RemarkConj} The statement of the above conjecture implies, in particular, that the measure $\vartheta_{\mathbf{M}_r}$ is a.s.\ mutually singular to the uniform measure $\nu$ on $D$, which is a multiple of $H_{\lambda}$ when $\lambda=0$, and thus assigns any Borel set $S$ measure zero that satisfies $H_{\lambda}(S)<\infty$ for some $\lambda>0$.  This is a form of locality for the disordered measure $\mathbf{M}_r$ because $\vartheta_{\mathbf{M}_r}$ assigns weight in the space $D$ according to where  pairs of paths $(p,q)\in \Gamma\times \Gamma$  spend their intersection time when independently sampled from $\mathbf{M}_r$.  Nontrivial path intersections do not occur outside a set $S\subset D$ that is small in the sense given in the conjecture. 
\end{remark}

\subsection{The Hilbert-Schmidt operator defined by the intersection-time kernel}\label{SectionHilbert}

 Let us fix some $\eta\in \mathcal{M}_{\Gamma}^{\diamond}$ such that the intersection time kernel $T(p,q)$ has finite second moment under $\eta\times\eta$.  We next discuss   the linear operator $\mathbf{T}_{\eta}$ on the space $L^2(\Gamma,\eta)$ that is defined through integration against the intersection-time kernel $T(p,q)$, i.e., such that for $g\in L^2(\Gamma,\eta)$
$$  \big(\mathbf{T}_{\eta}g\big)(p)\,:=\,\int_{\Gamma}T(p,q)g(q)\eta(dq) \,.$$
 The results   stated in this subsection are applied in~\cite{Clark4}, where $\mathbf{T}_{\eta}$ is the correlation operator for a Gaussian random field $\{ \mathbf{W}_{ \eta }(p) \}_{p\in \Gamma}$ on the measure space $(\Gamma,\eta)$.  This means that each $g\in L^2(\Gamma,\eta)$ is mapped  to a centered normal random variable with variance  $\langle g, \mathbf{T}_{\eta}g\rangle_{L^2(\Gamma,\eta)  } $, which we  denote by $\int_{\Gamma}g(p)\mathbf{W}_{ \eta }(p)\eta (dp)  $.  Let $\{ \hat{W}_{ \eta }(x) \}_{x\in D}$ be a  white noise field on $D$ with variance measure $\vartheta_{\eta}$, i.e., a map sending each $f\in L^2\big(D,\vartheta_{\eta}\big)$ to a centered normal random variable $\int_{D}f(x)\hat{W}_{ \eta }(x)\vartheta_{\eta} (dx)  $ with variance $\|f\|^2_{L^2(D,\vartheta_{\eta})  }$.  The Gaussian field  $\{ \mathbf{W}_{ \eta }(p) \}_{p\in \Gamma}$ can be formally constructed from $\{ \hat{W}_{ \eta }(x) \}_{x\in D}$ as 
 \begin{align}\label{Formal}
   \mathbf{W}_{ \eta }(p)  \,=\,\int_{D}y^{p}_{\eta}(x)\hat{W}_{ \eta }(x)\vartheta_{\eta} (dx)\, =\,\big\langle  y^{p}_{\eta},\hat{W}_{ \eta }\big\rangle_{L^2(D,\vartheta_{\eta})} \,,
  \end{align}
where the family of distributions  $\big\{ y^{p}_{\eta}\big\}_{p\in \Gamma}$ on $D$  can be  expressed as $y^{p}_{\eta}(x):=\frac{\int_{\Gamma}\gamma^{p,q}(dx)\eta(dq)}{\vartheta_{\eta}(dx)}$; see  the beginning of Section~\ref{SectionLast} for an alternative definition of $y^{p}_{\eta}$.  Although the measure $\int_{\Gamma}\gamma^{p,q}\eta(dq)$ on $D$ should not be expected to be  absolute continuous with respect to $\vartheta_{\eta}$,  the signed measure $\int_{\Gamma\times \Gamma}g(p)\gamma^{p,q}\eta(dp)\eta(dq) $ is absolutely continuous with respect to $\vartheta_{\eta}$ for any test function $g\in L^2(\Gamma,\eta)$.
 The linear operator $\hat{Y}_{\eta}:L^2\big(D,\vartheta_{\eta}\big)\rightarrow L^2(\Gamma, \eta)$ defined in~(\ref{YDef}) below is formally given by  $\big(\hat{Y}_{\eta}f\big)(p)=\big\langle  y^{p}_{\eta},f\big\rangle_{L^2(D,\vartheta_{\eta})} $, and thus we can alternatively write the equality~(\ref{Formal}) as $\mathbf{W}_{ \eta }(p) =\big(\hat{Y}_{\eta}\hat{W}_{ \eta }\big)(p)$.

 The proof of the theorem below is in Section~\ref{SubSecLast4}.  Appendix~\ref{AppendixHier} contains lemmas on the hierarchical symmetry of the linear operator $\hat{Y}_{\eta}$ and a natural method for  approximating $\hat{Y}_{\eta}$  by finite-rank operators.
\begin{theorem}\label{ThmOperator}  For any $\eta\in \mathcal{M}_{\Gamma}^{\diamond}$ with $\int_{\Gamma\times \Gamma}\big(T(p,q)\big)^2\eta(dp)\eta(dq)<\infty$, the Hilbert-Schmidt operator $\mathbf{T}_{\eta}$ on $L^2(\Gamma,\eta)$ with integral kernel $T(p,q)$  can be written as $\mathbf{T}_{\eta}= \hat{Y}_{\eta}\hat{Y}_{\eta}^*  $ for the compact linear operator $\hat{Y}_{\eta}$ from $L^2(D,\vartheta_{\eta})$ to $ L^2(\Gamma,\eta)  $ defined by 
\begin{align}\label{YDef}
\big(\hat{Y}_{\eta}f\big)(p)\,:=\, \int_{D\times \Gamma}f(x)\gamma^{p,q}(dx)\eta(dq) \,, 
\end{align}
 where  $p\in \Gamma$ and $f\in L^2(D,\vartheta_{\eta}) $. In particular, $\mathbf{T}_{\eta}$ is a positive operator. 
\end{theorem}

\begin{remark} Theorem~\ref{ThmOperator}   applies with $\eta:=\mathbf{M}_r$ since  the random measure $\mathbf{M}_r$ a.s.\ takes values   in $\mathcal{M}_{\Gamma}^{\diamond}$ by Corollary~\ref{CorInfDec} and satisfies $\int_{\Gamma\times \Gamma}\exp\{T(p,q)\}\mathbf{M}_r(dp)\mathbf{M}_r(dq)<\infty $ by (ii) of Theorem~\ref{ThmPathMeasure}.
\end{remark}

\begin{remark} Note that the operator $\mathbf{T}_{\eta}$ is not trace class unless $\mathbf{T}_{\eta}=0$ since $T(p,q)=\infty$ when $p=q$ as an immediate consequence of the limit  formula defining $T$  in (i) of Proposition~\ref{PropLemCorrelate}.
\end{remark}

\begin{remark} Let  $\tilde{\mathcal{M}}_{\Gamma}^{\diamond}$ denote the set of $\eta\in \mathcal{M}_{\Gamma}^{\diamond}$ such that $\int_{\Gamma\times \Gamma}\big(T(p,q)\big)^2\eta(dp)\eta(dq)<\infty$. If we wish to view $\hat{Y}$ as defining a measurable map $\eta\mapsto \hat{Y}_{\eta}$ from $\tilde{\mathcal{M}}_{\Gamma}^{\diamond}$ into some Polish space, then we can follow the basic recipe below.  Let us equip $\R^{\infty}$ with the metric $d\big( (x_n)_{n\in \mathbb{N}},(y_n)_{n\in \mathbb{N}}\big)=\sum_{n=1}^{\infty}\frac{1}{2^n}\big(|x_n-y_n|\wedge 1\big)$, which induces the topology of pointwise convergence.  Since the metric spaces $\Gamma$ and $D$ are compact, the spaces of  continuous functions $C(\Gamma)$ and  $C(D)$ are separable with respect to their uniform norms.  Let $\mathbf{C}$ be a countable, dense  subset of $C(\Gamma)$, $\mathbf{C'}$ be a countable, dense  subset of $C(D)$,  and  $\big( (g_n,f_n) \big)_{n\in \mathbb{N}}$ be an enumeration of $\mathbf{C}\times \mathbf{C'}$.  Then we can identify $\hat{Y}$ with the measurable map from $\tilde{\mathcal{M}}_{\Gamma}^{\diamond}$ to $\R^{\infty} $ that sends $\eta\in \tilde{\mathcal{M}}_{\Gamma}^{\diamond}$ to the sequence $\hat{Y}\eta\equiv \big((\hat{Y}\eta)_n  \big)_{n\in \mathbb{N}}$ with $(\hat{Y}\eta)_n :=\int_{\Gamma} g_n(p)\big(\hat{Y}_{\eta}f_n\big)(p)\eta(dp) $. \end{remark}

The  minor technical lemma below is used in~\cite[Section 5]{Clark4}  to prove (v) of Proposition~\ref{PropProperties}.
\begin{lemma}\label{PropY} Given the random measure $\mathbf{M}_r$, define $\hat{Y}_{\mathbf{M}_r}:L^2(D,\vartheta_{\mathbf{M}_r})\rightarrow L^2(\Gamma,\mathbf{M}_r)  $ as in Proposition~\ref{ThmOperator}.  When $f=1_{D}$, the function $\hat{Y}_{\mathbf{M}_r}f$ is $\mathbf{M}_r$-a.e.\ positive for a.e.\ realization of $\mathbf{M}_r$.
\end{lemma}
\begin{proof} Notice that $\big(\hat{Y}_{\mathbf{M}_r}1_D\big)(p)=\int_{\Gamma}T(p,q)\mathbf{M}_r(dq)$ since the measure $\gamma^{p,q}$ has total mass $T(p,q)$ by Corollary~\ref{CorollaryGamma}.  The result then follows from (iv) of Theorem~\ref{ThmPathMeasure}.
\end{proof}

\begin{remark}\label{RemarkWN}  Recall from Remark~\ref{RemarkVarthetaMu} that $\vartheta_{\mu}=0$.  It follows that $\mathbf{T}_{\mu}=0$ and $\hat{Y}_{\mu}=0$.  In a subcritical version of the model, there would exist $\lambda,\lambda'\in (0,\infty)$ such that  $\vartheta_{\mu}=\lambda \nu   $ and   $Y_{\mu}:L^2(D,\nu)\rightarrow L^2(\Gamma,\mu)$ acts on $f\in L^2(D, \nu)$ as  $\big(Y_{\mu}f\big)(p)=\lambda'\int_{0}^{1}f\big(p(t)\big)dt  $.  
\end{remark}

\section{DHL construction and its basic structures}\label{SectionDHLConst}

The DHL construction that we sketch here was introduced in~\cite{Clark2}, and our presentation will be specialized to the  $b=s$ case.  A closely related perspective on diamond fractals that is oriented towards a discussion of diffusion  can be found in~\cite{Ruiz,Ruiz2}.  Diffusion has also been studied on critical percolation clusters constructed in a diamond graph setting~\cite{HamblyII}.   \vspace{.2cm}

\noindent \textbf{DHL construction through sequences:}  The recursive construction of the diamond graphs implies  an obvious one-to-one correspondence between the edge set, $E_n$, of the  diamond graph $D_n$ and the set of length-$n$ sequences, $\{(b_k,s_k)\}_{k\in \{1,\ldots ,n\}  }$, of pairs $(b_k,s_k)\in \{1,\ldots, b\}^2$.
%where $(b_k,s_k)$ specifies the $s_k^{th}$ segment of the $b_k^{th}$ for an edge of $D_1%{(b)$ 
 In other terms, $E_n$ is canonically identifiable with the product set $ \big(\{1,\ldots, b\}^2\big)^{n}$.  For $\mathcal{D}:=\big(\{1,\ldots, b\}^2\big)^{\infty}$, we define  the DHL  as a metric space $(D,  d_D)$ where
$$
D\, := \, \mathcal{D} /\big(x,y\in \mathcal{D} \text{ with }  d_D(x,y )=0\big)\,
$$
for a  semi-metric $d_D:\mathcal{D}\times \mathcal{D}\longrightarrow [0,1]$ to be defined below in~(\ref{DefSemiMetric}) that, intuitively, measures the traveling distance along paths.  \vspace{.2cm}

\noindent \textbf{The metric:} Define the map $\widetilde{\pi}: \mathcal{D}\rightarrow [0,1]$ such that a sequence $x= \{(b_k^x,s_k^x)\}_{k\in \mathbb{N}}$ is assigned the number, $ \widetilde{\pi}(x)$, with base $b$  decimal expansion  having  $s_k^x-1\in\{0,\ldots, b-1\}$  as its $k^{th}$ digit for each $k\in \mathbb{N}$:
$$\widetilde{\pi}(x)\,:=\, \sum_{k=1}^{\infty} \frac{s_{k}^{x}-1}{ b^{k}}  \,.$$
Define   $A:=\{x\in \mathcal{D}\,|\, \widetilde{\pi}(x)=0 \}$ and $B:=\{x\in \mathcal{D}\,|\, \widetilde{\pi}(x)=1 \}$, i.e., the opposing nodes.
   For $x,y\in\mathcal{D}$, we write $x\updownarrow y$ if  $x$ or $y$ belongs to one of the sets $A$, $B$ or if the sequences of pairs $\{(b_k^x,s_k^x)\}_{k\in \mathbb{N}}$ and  $\{(b_k^y,s_k^y)\}_{k\in \mathbb{N}}$ respectively defining $x$ and $y$ have their first disagreement at an $s$-component value, i.e., there exists an $n\in \mathbb{N}$ such that $  b_k^x=b_k^y$ for all $ 1\leq k \leq n $ and  $ s_n^x\neq s_n^y$.
   We define the semi-metric $d_D$ in terms of $\widetilde{\pi}$    as
\begin{align}\label{DefSemiMetric}
d_D(x,y)\,:=\,\begin{cases} \quad  \quad  \big|\widetilde{\pi}(x)-\widetilde{\pi}(y)\big|   & \quad\text{if } x\updownarrow y,  \\ \,\, \displaystyle \inf_{z \in \mathcal{D},\, z\updownarrow x,\, z\updownarrow y  }\big( d_D(x,z)+d_D(z,y) \big)  & \quad  \text{otherwise.} \end{cases}    
\end{align}
 The semi-metric $d_D(x,y)$  takes values $\leq 1$ since, by definition, $z\updownarrow x $  and $z\updownarrow y $ for any $z\in A$ or $z\in B$, and thus $d_D(x,y)\leq \min\big(\widetilde{\pi}(x)+\widetilde{\pi}(y),\, 2-\widetilde{\pi}(x)-\widetilde{\pi}(y)\big)$.\vspace{.2cm}

\noindent \textbf{Self-similarity:} The fractal decomposition of $D$ into shrunken, embedded subcopies of $D$ is easy to see through the family of shift maps  $S_{i,j}:\mathcal{D}\rightarrow \mathcal{D}$ for $(i,j)\in \{1,\ldots , b\}^2$ that send a sequence $x\in \mathcal{D}$ to a shifted sequence $y=S_{i,j}(x)$ having initial term $(i,j)$.  In other terms, $\{(b_k^x,s_k^x)\}_{k\in \mathbb{N}}$ is mapped to $\{(b_k^y,s_k^y)\}_{k\in \mathbb{N}}$ for $(b_1^y,s_1^y)=(i,j)$ and $(b_k^y,s_k^y)=(b_{k-1}^x,s_{k-1}^x)$ for $k\geq 2$.  The map $S_{i,j}$ is also well-defined as a function from $D$ into $D$  with the contractive property
$$d_D\big(S_{i,j}(x),S_{i,j}(y)\big)\,=\,\frac{1}{b} d_D(x,y)  \hspace{.5cm} \text{for} \hspace{.5cm} x,y\in D \,. $$
The family of sets $\big\{S_{i,j}(D)\big\}_{1\leq i,j\leq b}$ are the first-generation shrunken, embedded copies of $D$.   Thus the maps $S_{i,j}$ are the similitudes of the fractal $D$, and the above  property implies that the space $(D,d_D )$ has Hausdorff dimension two.  \vspace{.2cm}

\noindent \textbf{The vertex set:} The sets  $A, B \subset \mathcal{D}$ form equivalence classes under the metric $d_D  $ that correspond to the root nodes of $D$. Similarly, the higher-generation vertices, $V_n\backslash V_{n-1}$ for $n\in \mathbb{N}$, of the diamond graphs are identified with uncountably infinite  equivalence classes of $\mathcal{D}$ of the form
$$  \big( S_{ b_1, s_1 } \circ \cdots \circ S_{b_{n-1}, s_{n-1} }\circ S_{b_n, s_n }(B) \big)  \cup \big(S_{b_1,s_1 } \circ \cdots \circ S_{b_{n-1},s_{n-1}}\circ S_{b_n,s_n+1 }(A) \big)  \,  $$
for a length-$n$  sequence of pairs $(b_k,s_k)\in \{1,\ldots , b\}^2$ with $s_n<b$;
 see~\cite[Appendix A.1] {Clark2} for a more explicit construction of these vertex equivalence classes. In contrast, elements of $E:=D\backslash V$ for $V:=\bigcup_n V_n$ have unique representatives in $\mathcal{D}$.  The vertex set $V$ is dense in $D$.\vspace{.2cm}

\noindent \textbf{Measure theoretic structures on the DHL:} For $E:=D\backslash V$,  we can define a generation-$n$ cylinder subset of $E$ using  any length-$n$  sequence of pairs $( b_k ,s_k)\in \{1,\ldots , b\}^2$ through
\begin{align*}
C_{ (b_1,s_1), \ldots , (b_n,s_n)}  \,:=\, S_{b_1 ,s_1 } \circ \cdots \circ  S_{ b_n,s_n }(E 
)\,.
\end{align*}
The canonical one-to-one correspondence between the set of edges $E_n$  and the product set $\big(\{1,\ldots , b\}^2\big)^n$ can also be viewed as a canonical bijection between $E_n$ and the collection of generation-$n$ cylinder sets of $E$, and we denote the cylinder set associated to $\mathbf{e}\in E_n$ by $\mathbf{\overline{e}}$. The  Borel  $\sigma$-algebra, $\mathcal{B}_D$, of $D$ is  generated by the semi-algebra  $\mathcal{S}_{D}:=\mathcal{P}(V)\cup \bigcup_{n=0}^{\infty}\{\mathbf{\overline{e}}\,|\,\mathbf{e}\in E_n\}$, where  $\mathcal{P}(V)$ denotes the power set of $V$. There is a unique probability  measure $\nu$ on the measurable space $(D,\mathcal{B}_D )$ such that $\nu(V)=0$ and 
 $\nu(\mathbf{\overline{e}})= b^{-2n}$ for all $n\in \mathbb{N}_0$ and $\mathbf{e}\in E_n$.\vspace{.2cm}

\noindent \textbf{Directed paths:} A \textit{directed path}   on $D$ is a continuous function $p:[0,1]\rightarrow D$ such that $\widetilde{\pi}\big( p(r)\big)=r $ for all $r\in[0,1]$.  
Thus a path $p\in \Gamma$ moves progressively from $A$ to $B$  at a constant speed.  Paths pass through points in the vertex set $V_n\subset D$ at the  times $t=\frac{k}{b^n}$ for $k\in \{0,1,\ldots, b^n\}$.  We can measure the distance between paths using the uniform metric:
$$    d_\Gamma\big(p,q\big) \, := \, \max_{0\leq r\leq 1}d_D\big( p(r), q(r)    \big) \hspace{.5cm} \text{for} \hspace{.5cm} p,q\in \Gamma   \,.     $$
This definition  implies that distances on $\Gamma$ always have the discrete form  $ d_\Gamma(p,q)\, = \,  b^{-(n-1)}      $ for some $n\in \mathbb{N}$, where $n$ is the lowest generation of any vertex in the range of $p$ but not $q$.  It is easy to see that the metric space $(\Gamma, d_\Gamma  )$ is compact. \vspace{.2cm}

\noindent \textbf{Measure theoretic structures on the space of directed paths:} The set of directed paths, $\Gamma_n$, on the $n^{th}$ diamond graph
 defines an equivalence relation on $\Gamma$ in which $q\equiv_n p$ if and only if the coarse-grained paths $[p]_n$ and $[q]_n$ are equal. If $\mathbf{p}\in \Gamma_n$, then we denote its associated equivalence class, $\big\{p\in \Gamma\,\big|\,[p]_n=\mathbf{p}\big\}$, by  $\mathbf{\overline{p}}$.  The Borel $\sigma$-algebra,  $\mathcal{B}_\Gamma $, on $\Gamma$ is generated by the semi-algebra $\mathcal{S}_{\Gamma}:=\{\emptyset\}\cup \bigcup_{n=0}^\infty \big\{\mathbf{\overline{p}}\,\big|\,\mathbf{p}\in \Gamma_n\big\}$, and 
 there is a unique measure $\mu$ on the measurable space $(\Gamma,\mathcal{B}_\Gamma)$ satisfying $\mu(\mathbf{\overline{p}})=|  \Gamma_n |^{-1}  $ for all $n\in \mathbb{N}_0$ and  $\mathbf{p}\in \Gamma_n$.  In later sections,  we will apply the term  \textit{cylinder set}  more broadly  to unions  $\bigcup_{\mathbf{p}\in A}\mathbf{\overline{p}}$ for $A\subset \Gamma_n$ and  refer to  $\mathbf{\overline{p}}$ for $\textbf{p}\in E_n$ as a \textit{simple} cylinder set.

\section{The construction and properties of the correlation measure}\label{SecCorrMeas}

Our goal in this section is to prove results related to the correlation measure $\upsilon_{r}:=\mathbb{E}\big[\mathbf{M}_{r}\times  \mathbf{M}_{r}\big]$.  In Sections~\ref{SecProofLemCorrelate} and~\ref{SecPropLemCorr}, we will prove Lemma~\ref{LemCorrelate} and Proposition~\ref{PropLemCorrelate}, respectively. As a preliminary, in Section~\ref{SubsecAlg}, we discuss the $\sigma$-algebras $\mathcal{B}_{\Gamma}$ and $\mathcal{B}_{\Gamma\times \Gamma}$. In Section~\ref{SubsecMart}, we define a $\upsilon_r$-martingale that has a role in several proofs.

\subsection{Set algebras on $\mathbf{\Gamma}$ and $\mathbf{\Gamma\times \Gamma}$ }\label{SubsecAlg}
Recall that $\mathcal{B}_{\Gamma}$ denotes the Borel $\sigma$-algebra on $\Gamma$ and that discrete paths $\mathbf{p}\in \Gamma_n$ correspond to simple cylinder subsets $\mathbf{\overline{p}}$ of $\Gamma$.  Lemma~\ref{LemAlgebra} below is worth stating to avoid repetition in the proofs concerned with defining Borel measures on $\Gamma$ and $\Gamma\times \Gamma$.  The lemmas below easily generalize to the measurable space  $\big(\Gamma^d, \mathcal{B}_{\Gamma^d}\big)$ for any $d\in \mathbb{N}$, but we narrow our discussion to the cases $d=1$ and $d=2$ that we need.

\begin{definition}\label{DefAlgebra}
Define the following subalgebras of $\mathcal{B}_{\Gamma}$ and $\mathcal{B}_{\Gamma\times \Gamma}=\mathcal{B}_{\Gamma}\otimes \mathcal{B}_{\Gamma}$ for  $n\in \mathbb{N}_0$. 
\begin{itemize}

\item  $\mathcal{A}_{\Gamma}^{(n)}$ denotes the finite $\sigma$-algebra of generation-$n$ cylinder subsets of $\Gamma$, i.e., the collection of preimages of the map from $\Gamma$ to $\Gamma_n$ that sends a path $p$ to its generation-$n$ coarse-graining, $[p]_n$:\vspace{-.1cm}
 $$ \mathcal{A}_{\Gamma}^{(n)}\,:=\, \Big\{ \big\{ p\in \Gamma\,\big|\,[p]_n\in B\big\}\,\Big|\,\text{$B\subset \Gamma_n$}  \Big\}\,=\,\bigg\{ \bigcup_{\mathbf{p}\in B}\mathbf{\overline{p}} \, \bigg| \, B\subset \Gamma_n   \bigg\} \,. \vspace{-.35cm} $$

\item    $\mathcal{A}_{\Gamma\times \Gamma}^{(n)}=\mathcal{A}_{\Gamma}^{(n)}\otimes\mathcal{A}_{\Gamma}^{(n)}$ denotes the finite $\sigma$-algebra of generation-$n$ cylinder subsets of $\Gamma\times \Gamma$.

\item  Define  $\mathcal{A}_{\Gamma}:= \bigcup_{n=1}^\infty \mathcal{A}_{\Gamma}^{(n)}$ and $\mathcal{A}_{\Gamma\times \Gamma}:= \bigcup_{n=1}^\infty  \mathcal{A}_{\Gamma\times \Gamma}^{(n)}$.

\end{itemize}

\end{definition}

\begin{remark} \label{RemarkSubAlg} For $n\geq N$,  we have $\mathcal{A}_{\Gamma}^{(N)} \subset \mathcal{A}_{\Gamma}^{(n)}$ and $\mathcal{A}_{\Gamma\times \Gamma}^{(N)} \subset \mathcal{A}_{\Gamma\times \Gamma}^{(n)}$.
\end{remark}

The next trivial lemma  follows from the compactness of $\Gamma$  and  the metric $d_{\Gamma}$ being discrete-valued.
\begin{lemma}\label{RemarkOpenClosed} Every set in  $\mathcal{A}_{\Gamma}$ is both open and compact in the topology of the metric space $(\Gamma, d_\Gamma  )$. The analogous statement also holds for $\mathcal{A}_{\Gamma\times \Gamma}$.
\end{lemma}
\begin{proof} For any $A\in \mathcal{A}_{\Gamma}$,  there is an $n\in \mathbb{N}$ large enough such that $A$ is a finite disjoint union of cylinder sets $\mathbf{\overline{p}}$ for  $\mathbf{p}\in \Gamma_n$, and $\mathbf{\overline{p}} \,=\,\big\{ q\in \Gamma \,\big|\, d_\Gamma(p,q)\leq \delta   \big\}    $
for any element $p\in \mathbf{\overline{p}}$ and any choice of $\frac{1}{b^n}\leq \delta<\frac{1}{b^{n-1}}$. It follows that $A$ is both open and closed in the topology of $\Gamma$.  Furthermore,  $A$ is compact because $\Gamma$ is  compact.  This argument extends to $\mathcal{A}_{\Gamma\times \Gamma}$.
\end{proof}

\begin{lemma}\label{LemAlgebra}
Let $\mathcal{A}_{\Gamma}\subset \mathcal{B}_{\Gamma}$ and $\mathcal{A}_{\Gamma\times \Gamma}\subset \mathcal{B}_{\Gamma\times \Gamma}$ be defined as in Definition~\ref{DefAlgebra}.

\begin{enumerate}[(i)]

\item $\mathcal{A}_{\Gamma}$ is an algebra that generates $\mathcal{B}_{\Gamma}$.  Moreover, any finitely additive set function $\varrho:\mathcal{A}_{\Gamma}\rightarrow [0,\infty)$ must be a premeasure and thus extend uniquely to a measure on the measurable space $(\Gamma, \mathcal{B}_{\Gamma})$ through the Carath\'eodory extension procedure.

\item The analogous statement holds for the algebra $\mathcal{A}_{\Gamma\times \Gamma}$ and the measurable space $(\Gamma\times \Gamma, \mathcal{B}_{\Gamma\times \Gamma})$.

\end{enumerate}

\end{lemma}

\begin{proof} The only statement from (i) that is not obvious is that any  finitely additive function $\varrho:\mathcal{A}_{\Gamma}\rightarrow [0,\infty)$ must  be countably subadditive and thus be a premeasure.  The countable subadditivity of $\varrho$ holds vacuously in the following sense: if $A=\bigcup_{k=1}^\infty A_k  $ is a disjoint union with $A\in\mathcal{A}_{\Gamma}$ and $A_k\in\mathcal{A}_{\Gamma}$ for all $k\in \mathbb{N}$, then only finitely many of the sets $A_k$ are nonempty.  To see this, suppose to the contrary that  $\bigcup_{k=1}^n A_k  $ is a strict subset of $A$ for each $n$. Then the set  $\bigcap_{n=1}^\infty \big(A-  \bigcup_{k=1}^{n} A_k\big)$ must be nonempty because it is an intersection of a decreasing sequence of nonempty sets that are compact by Lemma~\ref{RemarkOpenClosed}. This contradicts our assumption that $A=\bigcup_{k=1}^\infty A_k  $, and therefore $A=\bigcup_{k=1}^{n} A_k$ for large enough $n$.  
We can apply the same reasoning   to prove statement (ii).
\end{proof}

The following lemma shows that the  algebra $\mathcal{A}_{\Gamma}$ defines a convenient collection of simple functions. 

\begin{lemma}\label{LemDense} Let $\mathbf{C}(\Gamma)$ denote the set of real-valued $\mathcal{A}_{\Gamma}$-measurable simple functions on $\Gamma$, i.e., $\psi \in  \mathbf{C}(\Gamma)$ if and only if it has the form $\psi=\sum_{j=1}^J\alpha_j1_{A_j}$ for some $\alpha_j\in \R$ and $A_j\in \mathcal{A}_{\Gamma}$. Then $\mathbf{C}(\Gamma)$ is a dense subset of $C(\Gamma)$ with respect to the uniform norm.  
\end{lemma}

\begin{proof} Functions in $\mathbf{C}(\Gamma)$ are continuous because sets in the algebra $\mathcal{A}_{\Gamma}$ are both open and closed.  Since $\Gamma $ is a compact metric space and $ \mathbf{C}(\Gamma)$ is an algebra of functions in $C(\Gamma)$ that separates points in $\Gamma$ and contains all constant functions, it follows that $ \mathbf{C}(\Gamma)$ is dense in $C(\Gamma)$ with respect to the uniform norm by the Stone-Weierstrass theorem.
\end{proof}

Corollary~\ref{CorollarySemiAlg} below rephrases Lemma~\ref{LemAlgebra} in terms of the semi-algebras of simple cylinder sets. 

\begin{definition}\label{DefAlgebraSemi} 
Let  $\mathcal{S}_{\Gamma}\subset \mathcal{B}_{\Gamma}$ 
and $\mathcal{S}_{\Gamma\times\Gamma}\subset \mathcal{B}_{\Gamma\times \Gamma}$ denote the  semi-algebras  below.
\begin{itemize}
\item $\mathcal{S}_{\Gamma}:=\{\emptyset\}\cup \bigcup_{n=0}^\infty \mathcal{S}_{\Gamma}^{(n)}$ for $ \mathcal{S}_{\Gamma}^{(n)}:= \big\{\mathbf{\overline{p}}\,\big|\,\mathbf{p}\in \Gamma_n\big\} $

\item $\mathcal{S}_{\Gamma\times \Gamma}:=\{\emptyset\}\cup \bigcup_{n=0}^\infty  \mathcal{S}_{\Gamma\times \Gamma}^{(n)} $ for $ \mathcal{S}_{\Gamma\times \Gamma}^{(n)}:= \big\{\mathbf{\overline{p}}\times \mathbf{\overline{q}}\,\big|\,\mathbf{p},\mathbf{q}\in \Gamma_n\big\}$

\end{itemize}

\end{definition}
\begin{corollary}\label{CorollarySemiAlg} The semi-algebra   $\mathcal{S}_{\Gamma}$  generates $\mathcal{B}_{\Gamma}$, and any finitely additive set function $\varrho:\mathcal{S}_{\Gamma}\rightarrow [0,\infty)$ extends uniquely to a Borel measure on $\Gamma$. The analogous statements hold for  $\mathcal{S}_{\Gamma\times \Gamma}$.
\end{corollary}
\begin{proof} The semi-algebra $\mathcal{S}_{\Gamma}$ generates the algebra $\mathcal{A}_{\Gamma}$, and so  an additive function on $\mathcal{S}_{\Gamma}$ extends uniquely to an additive function on $\mathcal{A}_{\Gamma}$.  The results then follow from Lemma~\ref{LemAlgebra}.
\end{proof}

\begin{remark}\label{RemarkSemiAlg}
 To prove that a set function $\varrho:\mathcal{S}_{\Gamma}\rightarrow [0,\infty)$ is finitely additive, by induction, it suffices merely to show that for all $n\in \mathbb{N}_0$ and $\mathbf{p}\in \Gamma_n$
$$ \varrho\big(\mathbf{\overline{p}}\big)\,=\,\sum_{\substack{ \mathbf{q}\in \Gamma_{n+1} \\ \mathbf{\overline{q}}\subset \mathbf{\overline{p}} }} \varrho\big(\mathbf{\overline{q}}\big)\,.$$
The analogous statement holds for a set function $\varrho:\mathcal{S}_{\Gamma\times \Gamma}\rightarrow [0,\infty)$.
For $\mathbf{q}\in \Gamma_{n+1}$, note that $\mathbf{\overline{q}}\subset \mathbf{\overline{p}} $ if and only if $[\mathbf{q}]_n=\mathbf{p}$, i.e., $\mathbf{p}$ is the generation-$n$ coarse-graining of   $\mathbf{q}$. The set $\{ \mathbf{q}\in \Gamma_{n+1}\,|\,\mathbf{\overline{q}}\subset \mathbf{\overline{p}} \}$ has a canonical one-to-one correspondence with $\{1,\ldots,b\}^{\{1,\ldots, b^n\}}$,  where each path $\mathbf{q}\in \Gamma_{N+1}$ with $\mathbf{\overline{q}}\subset\mathbf{\overline{p}}$ is determined by selecting one of $b$ possible branches for each embedded copy of $D_1$ within $D_{n+1}$ that is associated to an edge  $\mathbf{p}(\ell)\in E_n$ for some $1\leq \ell\leq b^n$. 
\end{remark}

\subsection{Proof of Lemma~\ref{LemCorrelate}}\label{SecProofLemCorrelate}

\begin{proof} Part (i): By Corollary~\ref{CorollarySemiAlg}, it suffices to show that the set function $\upsilon_{r}:\mathcal{S}_{\Gamma\times \Gamma}\rightarrow [0,\infty)$ defined by condition~(\ref{UpsilonCond}) is additive.  For $n\in \mathbb{N}_0$ and $\mathbf{p}, \mathbf{q}\in \Gamma_n$, we  have the first  equality below by definition of $\upsilon_{r}$ on  $\mathbf{\overline{p}}\times \mathbf{\overline{q}}$. 
\begin{align}
\upsilon_{r}\left(\mathbf{\overline{p}}\times \mathbf{\overline{q}} \right)\, &:= \,\frac{1}{|\Gamma_n|^2}  \big(1+R(r-n)\big)^{  \xi_n(\mathbf{p},\mathbf{q} ) } \nonumber \\ \nonumber
&\, = \,\frac{1}{|\Gamma_n|^2}  \prod_{\ell=1}^{b^n}\big(1+R(r-n)\big)^{ 1_{\mathbf{p}(\ell)=\mathbf{q}(\ell) } } \nonumber \\
&\,=\,\sum_{\substack{\mathbf{p_*},\mathbf{q_*}\in \Gamma_{n+1} \\ \mathbf{\overline{p}_*}\subset \mathbf{\overline{p}},\,\, \mathbf{\overline{q}_*}\subset \mathbf{\overline{q}}  }   }\frac{1}{|\Gamma_{n+1}|^2}  \prod_{l=1}^{b^{n+1}}\big(1+R(r-n-1)\big)^{ 1_{\mathbf{p_*}(l)=\mathbf{q_*}(l)}  } \nonumber  \\
  &\,=\,\sum_{\substack{\mathbf{p_*},\mathbf{q_*}\in \Gamma_{n+1} \\ \mathbf{\overline{p}_*}\subset \mathbf{\overline{p}},\,\, \mathbf{\overline{q}_*}\subset \mathbf{\overline{q}}  }   }\frac{1}{|\Gamma_{n+1}|^2}  \big(1+R(r-n-1)\big)^{\xi_{n+1}(\mathbf{p_*},\mathbf{q_*}) }\nonumber  \\
  &\,=:\,\sum_{\substack{\mathbf{p_*},\mathbf{q_*}\in \Gamma_{n+1} \\ \mathbf{\overline{p}_*}\subset \mathbf{\overline{p}},\,\, \mathbf{\overline{q}_*}\subset \mathbf{\overline{q}}  }  }\upsilon_{r}\left(\mathbf{\overline{p}_*}\times \mathbf{\overline{q}_*}\right) \label{UpsilonCalcII}
\end{align}
The second  equality uses that $\xi_n(\mathbf{p},\mathbf{q} )$ is defined as the number of $1\leq \ell \leq b^n$ such that $\mathbf{p}(\ell)=\mathbf{q}(\ell)$. For the third equality, we apply the recursive relation $1+ R(t+1)\,=\,\frac{1}{b}\big[ b-1+ \big( 1+R(t)\big)^b \big]  $ from (I) of Lemma~\ref{LemVar} with $t=r-n-1$ and   the canonical bijection between  $\big\{  \mathbf{p}_*\in \Gamma_{n+1} \,\big|\, \mathbf{\overline{p}_*}\subset \mathbf{\overline{p}}\big\}$ and $\{1,\ldots, b\}^{\{1,\ldots, b^n\}}$. Note, in particular, that this implies  the combinatorial identity $|\Gamma_{n+1}|=b^{b^n}|\Gamma_n|  $.  By Remark~\ref{RemarkSemiAlg}, the equality~(\ref{UpsilonCalcII}) implies that  $\upsilon_{r}$  is finitely additive, and therefore $\upsilon_{r}$ extends to a measure on $\mathcal{B}_{\Gamma\times \Gamma}$ by Corollary~\ref{CorollarySemiAlg}. \vspace{.3cm}

\noindent Part (ii): Let  $\mathbf{S}_\emptyset \subset\Gamma\times \Gamma $ be defined as in  Corollary~\ref{CorrMuMu}.  The set $\mathbf{S}_\emptyset$ can be written in terms of simple cylinder sets as follows:
\begin{align}\label{DefS}
  \mathbf{S}_{\emptyset}\,=\,\bigcup_{n=1}^{\infty}  \mathbf{S}^{(n)}_{\emptyset}  \hspace{.5cm}\text{for}\hspace{.5cm}    \mathbf{S}^{(n)}_{\emptyset}\,:=\,\bigcup_{\substack{\mathbf{p},\mathbf{q}\in \Gamma_n \\  \mathbf{\overline{p}}\cap \mathbf{\overline{q}}=\emptyset } } \mathbf{\overline{p}}\times\mathbf{\overline{q}} \,. 
  \end{align}
Observe that $\mathbf{S}^{(n)}_{\emptyset} $ is the set of pairs $(p,q)\in \Gamma\times \Gamma$ whose $n^{th}$ generation coarse-grainings $[p]_n$,$[q]_n\in \Gamma_n$ have no shared edges.  For  $\mathbf{p},\mathbf{q}\in \Gamma_n$ sharing no edges, the definition of $\upsilon_r$ in~(\ref{UpsilonCond}) reduces to $\upsilon_r\big(\mathbf{\overline{p}}\times\mathbf{\overline{q}} \big)=\frac{1}{|\Gamma_n|^2}=\mu\times \mu\big( \mathbf{\overline{p}}\times\mathbf{\overline{q}} \big)  $.   Hence $\mathbf{S}_{\emptyset}$ is a null set for the measure $\upsilon_r-\mu\times \mu$.   However, by Corollary~\ref{CorrMuMu}, the product measure $\mu\times \mu$ assigns full weight to $\mathbf{S}_{\emptyset}$, and therefore the measures $\mu\times \mu$ and $\upsilon_r-\mu\times \mu$ are mutually singular. Clearly, these are the components of  the Lebesgue decomposition of $\upsilon_r$ with respect to $\mu\times \mu$.  The measure $\upsilon_r-\mu\times \mu$ has total mass $R(r)$ since $\upsilon_r$ and $\mu\times \mu$ have total mass $1+R(r)$ and $1$, respectively.  Thus $\rho_r:=\frac{1}{R(r)}(\upsilon_r-\mu\times \mu)  $ is a probability measure supported on $\mathbf{S}_{\emptyset}^c$.
\end{proof}

\subsection{Some  martingales under the correlation measure}\label{SubsecMart}

For convenience, we will reduce one of our notations from Definition~\ref{DefAlgebra}:
\begin{definition}\label{DefSigmaAlgebra} Define $\mathcal{F}_n:=\mathcal{A}_{\Gamma\times \Gamma}^{(n)}$ for $n\in \mathbb{N}_0$. In other terms, $\mathcal{F}_n$ is the finite $\sigma$-algebra of subsets of $\Gamma\times \Gamma$ generated by the map from $ \Gamma\times \Gamma$ to $\Gamma_n\times \Gamma_n$ that sends $(p,q)$ to the coarse-graining  $([p]_n,[q]_n)$.

\end{definition}

The proof of Proposition~\ref{PropLemCorrelate} will follow easily once we have established the proposition below.

\begin{proposition}\label{PropMartingales}
Let the family of Borel measures $\{\upsilon_{r}\}_{r\in\R}$ on $\Gamma\times \Gamma$ be defined as in Lemma~\ref{LemCorrelate}.  For $n\in \mathbb{N}$ and $r,t\in \R$, define $\phi_n^{(r,t)}:\Gamma\times \Gamma\rightarrow (0,\infty) $ as
$$ \phi_n^{(r,t)}(p,q)\,=\,\bigg(\frac{ 1+ R(t-n)   }{ 1+ R(r-n)  }\bigg)^{\xi_n(p,q)}      \,,$$
where $0\leq \xi_n(p,q) \leq  b^n $ is the number of edges shared by the coarse-grained paths $[p]_n,[q]_n\in \Gamma_n$.  Then $\{\upsilon_{r}\}_{r\in \R}$ and $\big(\phi_n^{(r,t)}\big)_{n\in \mathbb{N}}$ satisfy the properties (i)-(iii) below for all $r,t\in \R$.

\begin{enumerate}[(i)]

\item  Under the probability measure $\widehat{\upsilon}_r :=\frac{1}{1+R(r)}\upsilon_r  $, the sequence $\big(\phi_n^{(r,t)}\big)_{n\in \mathbb{N}}$ is a nonnegative  martingale  with respect to the filtration $(\mathcal{F}_n)_{n\in \mathbb{N}} $.  The martingale $\big(\phi_n^{(r,t)}\big)_{n\in \mathbb{N}}$ converges  $\upsilon_r$-a.e.\ and in $L^2(\Gamma\times \Gamma, \upsilon_r)$  to a finite limit $\phi_{\infty}^{(r,t)}:=\textup{exp}\big\{(t-r)T(p,q)\big\}$, where $T(p,q)$ is $\upsilon_r$-a.e.\  equal to the limit of $\frac{\kappa^2}{n^2} \xi_n(p,q) $ as $n\rightarrow \infty$.

\item Similarly,  $\frac{d}{dt}\phi_n^{(r,t)}\big|_{t=r}= \frac{R'(r-n)   }{1+R(r-n)   }\xi_n(p,q)       $ is a nonnegative $\widehat{\upsilon}_r$-martingale   with respect to the filtration $(\mathcal{F}_n)_{n\in \mathbb{N}} $ and converges $\upsilon_r$-a.e.\ to $T(p,q)$ as $n\rightarrow \infty$.  Moreover, the exponential $\upsilon_r$-moments of $\frac{d}{dt}\phi_n^{(r,t)}\big|_{t=r}$ are uniformly bounded in $n\in \mathbb{N}$.

\item  $\phi_{\infty}^{(r,t)}$ is the Radon-Nikodym derivative of $\upsilon_t$ with respect to $\upsilon_r$ for any $r,t\in \R$.

\end{enumerate}

\end{proposition}

\begin{proof} Part (i): Let $\mathbbmss{E}_{\widehat{\upsilon}_r}$ denote the expectation with respect to the probability measure $\widehat{\upsilon}_r$. The $\sigma$-algebra $\mathcal{F}_n$ is  generated by the product sets $\mathbf{\overline{p}}\times \mathbf{\overline{q}}\subset \Gamma\times \Gamma$ for $\mathbf{p}, \mathbf{q}\in \Gamma_n   $.  To see that  $\big(\phi_{n}^{(r,t)} \big)_{n\in \mathbb{N}}$ is a $\widehat{\upsilon}_r$-martingale with respect to the filtration $(\mathcal{F}_n)_{n\in \mathbb{N}}$, we observe  that for any $\mathbf{p},\mathbf{q}\in \Gamma_{N}$ with $N\leq  n$  the conditional expectation of $\phi_{n}^{(r,t)}$ with respect to the event $\mathbf{\overline{p}}\times \mathbf{\overline{q}}$ is
\begin{align}
\mathbbmss{E}_{\widehat{\upsilon}_r}\big[ \phi_{n}^{(r,t)} \,\big|\, \mathbf{\overline{p}}\times \mathbf{\overline{q}}\big]  \,=\,& \, \frac{1}{ \upsilon_r(\mathbf{\overline{p}}\times\mathbf{\overline{q}} )  }  \int_{\mathbf{\overline{p}}\times \mathbf{\overline{q}}}  \phi_{n}^{(r,t)} (p,q) \upsilon_r(dp,dq) \nonumber \\
 \,=\,&\,\frac{1}{\upsilon_r(\mathbf{\overline{p}}\times\mathbf{\overline{q}} )  } \sum_{\substack{\mathbf{p_* },\mathbf{q_* }\in \Gamma_{n} \\ \mathbf{\overline{p}_*}\subset \mathbf{\overline{p}},\, \mathbf{\overline{q}_*}\subset \mathbf{\overline{q}}  }   } \int_{\mathbf{\overline{p}_*}\times \mathbf{\overline{q}_*}} \phi_{n}^{(r,t)} (p,q)\upsilon_r(dp, dq ) \,, \label{CondThing}
\end{align} 
where the first equality above uses that the normalizing constant $\frac{1}{1+R(r)}$ in the definition of $\widehat{\upsilon}_r$  cancels out, and the second equality holds because $\mathbf{\overline{p}}\times \mathbf{\overline{q}}$ is a disjoint union of product sets $\mathbf{\overline{p}_*}\times\mathbf{\overline{q}_*}  $ for  $\mathbf{p_*}, \mathbf{q_*}\in \Gamma_n $ with $\mathbf{\overline{p}_*}\subset \mathbf{\overline{p}}$ and $\mathbf{\overline{q}_*}\subset \mathbf{\overline{q}}$.
 Next, using that the integrand $\phi_{n}^{(r,t)}(p,q)$ in~(\ref{CondThing}) is constant and equal to  $\big(\frac{ 1+ R(t-n)   }{ 1+ R(r-n)  }\big)^{\xi_{n}(\mathbf{p_*},\mathbf{q_*}) }$ for $(p,q)\in \mathbf{\overline{p}_*}\times \mathbf{\overline{q}_*}  $ and that  $\upsilon_r(\mathbf{\overline{p}_*}\times \mathbf{\overline{q}_*}  )$ has the form~(\ref{UpsilonCond}), we can rewrite~(\ref{CondThing}) as
\begin{align} \mathbbmss{E}_{\widehat{\upsilon}_r}\big[ &\phi_{n}^{(r,t)} \,\big|\, \mathbf{\overline{p}}\times \mathbf{\overline{q}}\big] \nonumber \\ \,=\,&\,\frac{1}{ \upsilon_r(\mathbf{\overline{p}}\times\mathbf{\overline{q}} )  } \sum_{\substack{\mathbf{p_*},\mathbf{q_*}\in \Gamma_{n}\nonumber \\ \mathbf{\overline{p}_* }\subset \mathbf{\overline{p}},\, \mathbf{\overline{q}_* }\subset \mathbf{\overline{q}}  }   }\bigg(\frac{ 1+ R(t-n)   }{ 1+ R(r-n)  }\bigg)^{\xi_{n}(\mathbf{p_* },\mathbf{q_* })}  \frac{1}{|\Gamma_{n}|^2}  \big(1+R(r-n)\big)^{\xi_{n}(\mathbf{p_* }, \mathbf{q_* }  ) }\nonumber  \\ \,=\,&\,\frac{1}{ \upsilon_r(\mathbf{\overline{p}}\times\mathbf{\overline{q}} )  } \sum_{\substack{\mathbf{p_* },\mathbf{q_* }\in \Gamma_{n} \\ \mathbf{\overline{p}_* }\subset \mathbf{\overline{p}},\, \mathbf{\overline{q}_* }\subset \mathbf{\overline{q}}  }   }\frac{1}{|\Gamma_{n}|^2}  \big(1+R(t-n)\big)^{  \xi_{n}(\mathbf{p_* },\mathbf{q_* }) }  \nonumber  \\
  \,=\,&\,\frac{1}{\upsilon_r(\mathbf{\overline{p}}\times\mathbf{\overline{q}} )  } \frac{1}{|\Gamma_N|^2}  \big(1+R(t-N)\big)^{  \xi_{N}(\mathbf{p},\mathbf{q} ) }\,. \label{FindUpsilon}
 \end{align}
 The third equality above follows from the definition~(\ref{UpsilonCond}) of the measure $\upsilon_r$ and the finite additivity of $\upsilon_r$. 
 By applying~(\ref{UpsilonCond}) again to  $\upsilon_{r}(\mathbf{\overline{p}}\times \mathbf{\overline{q}})$ in (\ref{FindUpsilon}), we get that  for any $(p,q)\in \mathbf{\overline{p}}\times\mathbf{\overline{q}}$
\begin{align*}
\mathbbmss{E}_{\widehat{\upsilon}_r}\big[ \phi_{n}^{(r,t)} \,\big|\, \mathbf{\overline{p}}\times \mathbf{\overline{q}}\big]\,=\,&\,  \bigg(\frac{ 1+ R(t-N)   }{ 1+ R(r-N)  }\bigg)^{\xi_{N}(\mathbf{p},\mathbf{q})}\\ \,=\, &\, \bigg(\frac{ 1+ R(t-N)   }{ 1+ R(r-N)  }\bigg)^{\xi_{N}(p,q)} \,=:\,\phi_{N}^{(r,t)}(p,q)\,.
\end{align*}
  Since sets in $\mathcal{F}_N$ are finite unions of cylinder product sets $\mathbf{\overline{p}}\times \mathbf{\overline{q}}$, we have shown that $\mathbbmss{E}_{\widehat{\upsilon}_r}\big[ \phi_{n}^{(r,t)} \,\big|\, \mathcal{F}_N\big] =\phi_{N}^{(r,t)}$ for all $n\geq N$. Hence   $\big(\phi_{n}^{(r,t)}\big)_{n\in \mathbb{N}}$ is a nonnegative $\widehat{\upsilon}_r$-martingale with respect to the filtration $(\mathcal{F}_n)_{n\in \mathbb{N}} $, and   it converges $\upsilon_r$-a.e.\ to a nonnegative, $\upsilon_r$-integrable limit $\phi_{\infty}^{(r,t)}$. \vspace{.2cm}

Next we show that the log of $\phi_{\infty}^{(r,t)}(p,q)$ is $\upsilon_r$-a.e.\ equal to the limit of $(t-r)\frac{\kappa^2}{n^2}\xi_n(p,q)$ as $n\rightarrow \infty$. The log of  $ \phi_{n}^{(r,t)}\equiv \phi_{n}^{(r,t)}(p,q)$   has the large $n$ form
\begin{align}\label{ThisComp}
\log \big(\phi_{n}^{(r,t)}\big) \,=\,&\,\xi_n(p,q) \Big(  \log\big( 1+ R(t-n)    \big)\,-\, \log\big( 1+ R(r-n)    \big)     \Big) \nonumber \\ 
\,=\,&\,\xi_n(p,q) \bigg( \frac{\kappa^2}{n-t}\,+\,\frac{\eta\kappa^2 \log(n-t)}{(n-t)^2}\,-\, \frac{1}{2} \frac{\kappa^4}{(n-t)^2} \,  +\, \mathit{o}\Big(  \frac{1}{n^2}  \Big)   \bigg)\nonumber \\
&-\xi_n(p,q) \bigg( \frac{\kappa^2}{n-r}\,+\,\frac{\eta\kappa^2 \log (n-r)}{(n-r)^2}\,-\, \frac{1}{2} \frac{\kappa^4}{(n-r)^2} +\mathit{o}\Big(  \frac{1}{n^2}  \Big)   \bigg) \nonumber  \\
\,=\,&\,\xi_n(p,q)\bigg(  (t-r) \frac{\kappa^2}{n^2} +\mathit{o}\Big(  \frac{1}{n^2}  \Big)   \bigg)\,,
\end{align}
where the second equality applies the small $x$ asymptotics $\log(1+x)=x-\frac{x^2}{2}+\mathit{O}(x^3)$ combined with the asymptotics $R(r)=\frac{\kappa^2}{-r}+\frac{\eta\kappa^2\log(-r)}{r^2}+\mathcal{O}\Big(\frac{\log^2(-r)}{|r|^3}   \Big)   $ for $-r\gg 1$ from  (II) of Lemma~\ref{LemVar}.  Since the sequence $\big(\phi_{n}^{(r,t)}\big)_{n\in \mathbb{N}}$ converges $\upsilon_r$-a.e.\ to a finite limit $\phi_{\infty}^{(r,t)}$ for any $r,t\in \R$ with $r<t$, the sequence $\big(\frac{\kappa^2}{n^2}\xi_n  \big)_{n\in \mathbb{N}}$ converges $\upsilon_r$-a.e.\ to a nonnegative finite limit, which we denote by $T(p,q)$. \vspace{.1cm}

Next we argue that the martingale  $\big(\phi_{n}^{(r,t)}\big)_{n\in \mathbb{N}}$ has uniformly bounded positive  moments.  Notice that for $m\in \mathbb{N}$ the sequence $\big(\phi_{n}^{(r,t)}\big)^m$ is approximately equal to $ \phi_{n}^{(mr,mt)}$ by the computation in~(\ref{ThisComp}), and thus $\big(\phi_{n}^{(r,t)}\big)^m$ converges as $n\rightarrow \infty$ to the $\upsilon_r$-integrable function $\phi_{\infty}^{(mr,mt)}:=\exp\{m(t-r)T(p,q)   \}$.  It follows, in particular, that the sequence $\big(\phi_{n}^{(r,t)}\big)_{n\in \mathbb{N}}$ converges in $L^2(\Gamma\times \Gamma, \upsilon_r)$ to $\phi_{\infty}^{(r,t)}$ by uniform integrability.\vspace{.4cm}

\noindent Part (ii):  Since $\big(\phi_{n}^{(r,t)}\big)_{n\in \mathbb{N}}$ is a $\widehat{\upsilon}_r$-martingale with respect to $(\mathcal{F}_n)_{n\in \mathbb{N}}$  for any $t\in \mathbb{R}$ by part (i), the derivative  $\frac{d}{dt}\phi_n^{(r,t)}\big|_{t=r}= \frac{R'(r-n)   }{1+R(r-n)   }\xi_n(p,q)       $ is also a martingale.  However, $R(r-n)$ vanishes with large $n$ and $R'(r-n)=\frac{\kappa^2}{n^2}+\mathit{o}\big(\frac{1}{n^2}\big)$ by Remark~\ref{RemarkDerR}.  It follows that $ \frac{R'(r-n)   }{1+R(r-n)   }\xi_n(p,q)       $ becomes close to $\frac{\kappa^2}{n^2}\xi_n(p,q)$ with large $n$, and thus converges $\upsilon_r$-a.e.\ to $T(p,q)$ by part (i). Thus for $a>0$ and $n\gg 1$, the function $\textup{exp}\Big\{a\frac{d}{dt}\phi_n^{(r,t)}\big|_{t=r}\Big\}$ is close to $\textup{exp}\Big\{ a\frac{\kappa^2}{n^2}\xi_n(p,q) \Big\}$, which furthermore becomes  close to $\phi_n^{(r,r+a)}$  by our computation in~(\ref{ThisComp}).  Hence the exponential $\upsilon_r$-moments of   $\frac{d}{dt}\phi_n^{(r,t)}\big|_{t=r}$ are uniformly bounded in $n$ because the positive $\upsilon_r$-moments of $\phi_n^{(r,t)}$ are uniformly bounded.  \vspace{.3cm}

\noindent Part (iii):  For $N\in \mathbb{N}_0$ and  $\mathbf{p},\mathbf{q}\in \Gamma_N$, notice that the expression (\ref{FindUpsilon}) is equal to $\frac{\upsilon_t(\mathbf{\overline{p}}\times \mathbf{\overline{q}})}{\upsilon_r(\mathbf{\overline{p}}\times \mathbf{\overline{q}}) }$ by the definition~(\ref{UpsilonCond}) of $\upsilon_t$ on $\mathbf{\overline{p}}\times \mathbf{\overline{q}}$.  Thus, by canceling  $\upsilon_r(\mathbf{\overline{p}}\times \mathbf{\overline{q}})$, the equality between  (\ref{FindUpsilon}) and the right side of the first line of~(\ref{CondThing})  reduces to
\begin{align}\label{UpLim}
 \upsilon_t(\mathbf{\overline{p}}\times \mathbf{\overline{q}})\,=\,\int_{\mathbf{\overline{p}}\times \mathbf{\overline{q}}} \phi_n^{(r,t)}(p,q)  \upsilon_r(dp,dq)    
 \end{align}
for any $n\geq N$.  Since the functions $\phi_n^{(r,t)}$ converge $\upsilon_r$-a.e.\ to $\phi_\infty^{(r,t)}$ as $n\rightarrow \infty$ and are uniformly bounded in $L^2(\Gamma\times \Gamma, \upsilon_r)$ norm---and hence uniformly $\upsilon_r$-integrable---, the equality~(\ref{UpLim}) holds for the limit function $\phi_\infty^{(r,t)}=\textup{exp}\big\{(t-r)T(p,q)\big\} $.   We have verified that the Borel measures $\upsilon_t$ and $\phi_\infty^{(r,t)}\upsilon_r$ agree on sets in the semi-algebra $ \mathbf{S}_{ \Gamma \times \Gamma } $, which generates $\mathcal{B}_{\Gamma\times \Gamma}$.  By the uniqueness part of Carath\'eodory's extension theorem, we can conclude that $\upsilon_t =\phi_\infty^{(r,t)}\upsilon_r$.  In other terms, $\phi_\infty^{(r,t)}$ is the Radon-Nikodym derivative of $\upsilon_t$ with respect to $\upsilon_r$.
\end{proof}

\subsection{Proof of Proposition~\ref{PropLemCorrelate}}\label{SecPropLemCorr}

\begin{proof} Parts (i) and (iii) follow directly from  Proposition~\ref{PropMartingales}. The formula in part (iv) holds because $\textup{exp}\big\{aT(p,q)\big\} $ is the Radon-Nikodym derivative of $\upsilon_{r+a}$ with respect to $\upsilon_r$ by part (iii) of Proposition~\ref{PropMartingales}, and thus
\begin{align}\label{Sh}
 \int_{\Gamma\times \Gamma}e^{aT(p,q)}\upsilon_r(dp,dq)\,=\,\upsilon_{r+a}\big(\Gamma\times \Gamma\big)\,=\,1\,+\,R(r+a)\,.      
 \end{align}

The only  statement from part (ii) that does not follow directly from Lemma~\ref{LemCorrelate} and Proposition~\ref{PropMartingales} is the claim that $\upsilon_r-\mu\times \mu$ assigns full measure to the set of  pairs $(p,q)$ such that $T(p,q)>0$.  Define $\mathbf{\widehat{S}}=\{ (p,q)\,|\,T(p,q)=0      \}$, and let $\mathbf{S}_{\emptyset}\subset \Gamma\times \Gamma$ be defined as in Corollary~\ref{CorrMuMu}. As $a\rightarrow -\infty$, the left side of~(\ref{Sh})  converges to $\upsilon_r\big(\mathbf{\widehat{S}}\big)$ by the dominated convergence theorem, and the right side converges to $1$ since $R(r)$ vanishes as $r\rightarrow -\infty$. Hence we can deduce that $\upsilon_r\big(\mathbf{\widehat{S}}\big)=1$.  However, $\mathbf{S}_{\emptyset}\subset\mathbf{\widehat{S}} $ and $\upsilon_r(\mathbf{S}_{\emptyset})=\mu\times \mu(\mathbf{S}_{\emptyset})=1$ by Lemma~\ref{LemCorrelate} and Corollary~\ref{CorrMuMu}, so we get that
$$\upsilon_r\big(\mathbf{\widehat{S}}-\mathbf{S}_{\emptyset}   \big)\,=\,\upsilon_r\big(\mathbf{\widehat{S}}\big)-\upsilon_r(\mathbf{S}_{\emptyset})\,=\,1-1 \, =\,0\,. $$ Our results imply that $\mathbf{\widehat{S}}$ is a null set for the measure $\upsilon_r-\mu\times \mu$ since $\upsilon_r\big(\mathbf{\widehat{S}}\big)=\upsilon_r(\mathbf{S}_{\emptyset}   )=\mu\times \mu(\mathbf{S}_{\emptyset}) $.  In other terms, $\upsilon_r-\mu\times \mu$  is supported on the set of pairs $(p,q)$ such that $T(p,q)>0$. 
\end{proof}

\section{Construction, uniqueness, and basic properties of the RCPM laws}\label{SecCDRP}

We will now move on to the proofs of results stated in Sections~\ref{SecLimitThmForMeasures} \& \ref{SecLimitLawChar} that directly concern the random path measure laws $\{ \mathbf{M}_r \}_{r\in \R}$, namely, Theorem~\ref{ThmExistence}, Proposition~\ref{PropProperties}, and Theorem~\ref{ThmUniversality}. The proof of Theorem~\ref{ThmExistence} is split between Sections~\ref{ProofThmExistence} \& \ref{ProofThmExistence2}---except for the connection with the  random path measure laws that arise as limits in Theorem~\ref{ThmUniversality}, which follows trivially from the proof of Theorem 2.12 in Section~\ref{SecProofLimitThmForMeasures}.

\subsection{Proof of the existence part of Theorem~\ref{ThmExistence}}\label{ProofThmExistence}

In this subsection, we will prove the existence  claim in Theorem~\ref{ThmExistence} regarding a family  of random path measure laws $\{ \mathbf{M}_r \}_{r\in \R}$ satisfying properties (I)-(IV).   Our analysis below will rely on Theorem~\ref{ThmExist}, which is from~\cite[Theorem 5.16]{Clark3}.  The following clarifies the meaning of some notation  for arrays of numbers indexed by the edge sets $E_n$.

\begin{notation}\label{DefArray}\textup{For numbers $W^{e}\in\R$ labeled by $e\in E_k$, the notation $\{ W^e \}_{e\in E_{k}}$ denotes an element in $ \R^{E_k} \equiv \R^{b^{2k}  }$, which we refer to as an \textit{array}.   Similarly, $\{ W^e \}_{e\in \cup_{k=0}^{\infty}E_{k}}$ denotes an element in the sequence space $\R^{ \cup_{k=0}^{\infty}E_{k}  }\equiv \R^{\infty}$.\footnote{Here we use that  the set $\cup_{k=0}^{\infty}E_{k}$ is countable.}}
\end{notation}

\begin{theorem}\label{ThmExist} For any  $r\in \R$, there exists a unique law on sequences in $k\in \mathbb{N}_0$ of edge-labeled arrays of nonnegative random variables, $\{ \mathbf{W}^e\}_{e\in  E_{k}  }$,  holding  properties (I)-(IV) below for each $k\in \mathbb{N}_0$.\footnote{The variables  $\mathbf{W}^e$ for $e\in E_k$ are related to the variables  $\mathbf{X}_e^{(k)}$ in~\cite[Theorem~5.16]{Clark3} through $\mathbf{W}^e=1+\mathbf{X}_e^{(k)}$.}
\begin{enumerate}[(I)]

\item  The random variables in the array $\{ \mathbf{W}^e \}_{e\in E_{k}  }$ are i.i.d.

\item  The  random variables in the array $\{ \mathbf{W}^e\}_{e\in E_{k}  }$ have mean one and  variance  $R(r-k)$.

\item  For each  $m\in \{3, 4,\ldots\}$,  the random variables in the array $\{ \mathbf{W}^e \}_{e\in E_{k}  }$ have finite $m^{th}$ centered moment equal to $R^{(m)}(r-k)$, where  $R^{(m)}:\R\rightarrow \R_+$ is a continuous, increasing function with 
$ R^{(m)}(t)\,\propto\, (\frac{1}{-t})^{\lceil m/2  \rceil}$ as $t\rightarrow -\infty$ and $ R^{(m)}(t)$ diverges to $\infty$ as $t\rightarrow \infty$

\item The random variables in the array   $\{ \mathbf{W}^e \}_{e\in E_{k}  }$ are a.s.\ equal to $\displaystyle \mathbf{W}^e = \frac{1}{b}\sum_{i=1}^{b}\prod_{j=1}^b \mathbf{W}^{e\times (i,j)}$.\footnote{Recall the meaning of $e\times (i,j)$ for $e\in E_k$ and $(i,j)\in \{1,\ldots, b\}^2$ from Notation~\ref{NotationSub}.   }

\end{enumerate}
For the uniqueness, it suffices to replace (III) by the weaker condition (III') that the fourth centered moment of the random variables in the array $\{ \mathbf{W}^e \}_{e\in E_{k}  }$ vanishes as $k\rightarrow \infty$.\footnote{This weaker requirement for uniqueness is not stated in~\cite[Theorem~5.16]{Clark3} but follows from~\cite[Theorem~5.22]{Clark3}.}
\end{theorem}

The following lemma connects the decomposition property (IV) of Theorem~\ref{ThmExist} with the construction of Borel measures on $\Gamma$.  Let $\mathbf{H}\subset [0,\infty)^{ \cup_{k=0}^{\infty} E_k  } $ denote the space of functions $W:\bigcup_{k=0}^{\infty} E_k\rightarrow  [0,\infty)$ satisfying the hierarchical relation  $W^{e}= \frac{1}{b}\sum_{i=1}^{b}\prod_{j=1}^b W^{e\times (i,j)}$  for all $e\in \cup_{k=0}^{\infty} E_k$. We  equip $\mathbf{H} $ with the topology of pointwise convergence. Recall that the semi-algebra $\mathcal{S}_{\Gamma}\subset \mathcal{B}_{\Gamma}$ is defined  in Definition~\ref{DefSigmaAlgebra}.

\begin{lemma}\label{LemmaArray}   There is a unique  function $\Phi:  \mathbf{H}\rightarrow \mathcal{M}_{\Gamma} $  mapping each $W=\{W^e\}_{e\in \cup_{k=0}^{\infty} E_k}$ in $\mathbf{H}$ to a  measure $\Phi W\in \mathcal{M}_{\Gamma} $ that acts on the semi-algebra of simple cylinder sets  $\mathcal{S}_{\Gamma}$ as follows:  for any $n\in \mathbb{N}_0$ and $\mathbf{p}\in \Gamma_n$, the measure $\Phi W$ assigns $\mathbf{\overline{p}}\subset \Gamma$ weight
\begin{align}
 (\Phi W)(\mathbf{\overline{p}})\,:=\, \frac{1}{|\Gamma_n| }\prod_{\ell=1}^{b^n}  W^{\mathbf{p}(\ell)} \,. \label{ConI} 
 \end{align}
 Moreover, the function $\Phi$ is continuous when, as before, $\mathcal{M}_{\Gamma}$ is equipped with the weak topology.

\end{lemma}

\begin{proof}  By Corollary~\ref{CorollarySemiAlg}, to prove that the set function $\Phi W :\mathcal{S}_{\Gamma}\rightarrow [0,\infty) $ defined  by~(\ref{ConI}) extends uniquely to a Borel measure on $\Gamma$ for a given $W=\{W^e\}_{e\in \cup_{k=0}^{\infty} E_k}$ in $\mathbf{H}$, it suffices for us to show that $\Phi W$ is finitely additive. As a consequence of Remark~\ref{RemarkSemiAlg}, we only need to observe that   for any $n\in \mathbb{N}_0$ and $\mathbf{p}\in \Gamma_n$
\begin{align}
 (\Phi W)(\mathbf{\overline{p}}):=\frac{1}{|\Gamma_n| }\prod_{\ell=1}^{b^n}  W^{\mathbf{p}(\ell)}   = \sum_{  \substack{\mathbf{q}\in \Gamma_{n+1}  \\ \mathbf{\overline{q}} \subset \mathbf{\overline{p}}    } } \frac{1}{|\Gamma_{n+1}| }\prod_{l=1}^{b^{n+1}}  W^{\mathbf{q}(l)}  =: \sum_{  \substack{\mathbf{q}\in \Gamma_{n+1} \\  \mathbf{\overline{q}} \subset \mathbf{\overline{p}}    } } (\Phi W)(\mathbf{\overline{q}})\,,\label{ConII}
 \end{align} 
where the second equality uses the hierarchical relation $W^{e}= \frac{1}{b}\sum_{i=1}^{b}\prod_{j=1}^b W^{e\times (i,j)}$, the combinatorial identity $|\Gamma_{n+1}|=b^{b^n}|\Gamma_{n}|$, and the one-to-one correspondence between $\{ \mathbf{q}\in \Gamma_{n+1}\,|\, \mathbf{\overline{q}} \subset  \mathbf{\overline{p}}\}$ and $\{1,\ldots,b\}^{\{1,\ldots, b^n\}}$. Thus $\Phi W$ extends uniquely to a Borel measure on $\Gamma $.
 
Under the weak topology on $\mathcal{M}_{\Gamma}$, the function $\Phi:  \mathbf{H}\rightarrow \mathcal{M}_{\Gamma} $ is continuous if and only if for each $g\in C(\Gamma  ) $ the real-valued map  $W\mapsto \int_{\Gamma}g(p)(\Phi W)(dp) $  is continuous, and it suffices to establish this continuity for a dense subset of $C(\Gamma  ) $  with respect to the uniform norm.  Recall from Lemma~\ref{LemDense} that the set  $\mathbf{C}(\Gamma)$ of $\mathcal{A}_{\Gamma} $-measurable simple functions is a dense subset of $C(\Gamma)$.  Moreover,  if $\psi\in \mathbf{C}(\Gamma)$, then there exists an $n\in \mathbb{N}$ large enough such that $\psi=\sum_{j=1}^{J}\alpha_j 1_{A_j}$ for some $\alpha_j\in \R$ and $A_j\in \mathcal{A}_{\Gamma}^{(n)} $.  Hence by~(\ref{ConI})
$$ \int_{\Gamma}\psi(p)(\Phi W)(dp) \,=\,  \sum_{j=1}^{J}\alpha_j (\Phi W)(A_j)\,=\,  \sum_{j=1}^{J}\alpha_j\sum_{\substack{ \mathbf{p}\in \Gamma_n  \\ \mathbf{\overline{p}}\subset A_j  } }  \frac{1}{|\Gamma_n| }\prod_{\ell=1}^{b^n}  W^{\mathbf{p}(\ell)} \,. $$
The above is a multilinear polynomial in $\{ W^e  \}_{e\in \cup_{k=0}^n E_k }$ and is thus a continuous function of $W$.  Therefore $\Phi $ is continuous.
\end{proof}

\vspace{.2cm}

\begin{proof}[Proof of Theorem~\ref{ThmExistence}, existence part] 
Property (IV) in Theorem~\ref{ThmExist} implies that the $ [0,\infty)^{\cup_{k=0}^{\infty} E_k}$-valued random element $\mathbf{\hat{W}}_r:=\{\mathbf{W}^e\}_{e\in \cup_{k=0}^{\infty} E_k}  $ a.s.\ takes values in the space $\mathbf{H}$, and thus we can choose a version of $\mathbf{\hat{W}}_r $  that is $\mathbf{H}$-valued.  Define the random $\mathcal{M}_{\Gamma}$-valued element  $\mathbf{M}_r :=\Phi \mathbf{\hat{W}}_r$ for the function $\Phi:\mathbf{H}\rightarrow \mathcal{M}_{\Gamma}$ from Lemma~\ref{LemmaArray}. Next we prove properties (I)-(IV) in the statement of Theorem~\ref{ThmExistence} for this construction of the random path measure $\mathbf{M}_r$.\vspace{.3cm} \\
\noindent (I)  If $\mathbf{p}\in \Gamma_n$ for some $n\in \mathbb{N}_0$, then invoking the definition $\mathbf{M}_r:=\Phi \mathbf{\hat{W}}_r$ and~(\ref{ConI}) we get that
$$ \mathbb{E}\big[\mathbf{M}_r(\mathbf{\overline{p}})\big]\,=\,\mathbb{E}\left[\big(\Phi \mathbf{\hat{W}}_r\big)(\mathbf{\overline{p}})\right]\,=\,\frac{1}{|\Gamma_n| } \mathbb{E}\Bigg[\prod_{\ell=1}^{b^n}  \mathbf{W}^{\mathbf{p}(\ell)}\Bigg]\,=\, \frac{1}{|\Gamma_n| }\,=:\,\mu(\mathbf{\overline{p}})\,, $$
where the third equality uses that  $\{ \mathbf{W}^e \}_{e\in  E_n}$ is an independent family of  mean one random variables.  Hence the measures $\mathbb{E}[\mathbf{M}_r]$ and $\mu$ agree on sets in the semi-algebra $ \mathcal{S}_{\Gamma}$, and so   $\mathbb{E}[\mathbf{M}_r]=\mu$ on $\mathcal{B}_{\Gamma}$ by Corollary~\ref{CorollarySemiAlg}.
\vspace{.3cm} 
 
 \noindent (II)  If $\mathbf{p},\mathbf{q}\in \Gamma_n$ for some $n\in \mathbb{N}_0$, then
 the expectation of $\mathbf{M}_{r}\times \mathbf{M}_{r}(\mathbf{\overline{p}}\times \mathbf{\overline{q}})$ can be written in the form
\begin{align*}
\mathbb{E}\big[\mathbf{M}_{r}\times \mathbf{M}_{r}(\mathbf{\overline{p}}\times \mathbf{\overline{q}})   \big]\,&=\, \mathbb{E}\big[\mathbf{M}_{r}(\mathbf{\overline{p}}) \mathbf{M}_{r}( \mathbf{\overline{q}})   \big]\nonumber \\
\,&=\,\frac{1}{|\Gamma_n|^2 }\mathbb{E}\Bigg[ \prod_{l=1}^{b^n} \mathbf{W}^{\mathbf{p}(l)}  \prod_{\ell=1}^{b^n}  \mathbf{W}^{\mathbf{q}(\ell)}  \Bigg] \nonumber \\
 \,&=\, \frac{1}{|\Gamma_n|^2 }\big(1+R(r-n)   \big)^{\xi_n(\mathbf{p}, \mathbf{q})}\nonumber \\ \, &=:\,\upsilon_r(\mathbf{\overline{p}}\times \mathbf{\overline{q}})\,,
\end{align*}
where, recall, $\xi_n(\mathbf{p}, \mathbf{q})$ is the number of edges shared by the paths $\mathbf{p}, \mathbf{q}\in \Gamma_n$.  The first equality above merely uses the definition of a product measure;   the second  equality  uses  that $\mathbf{M}_r:=\Phi\mathbf{\hat{W}}_r$ and~(\ref{ConI});   the third equality follows from $\{ \mathbf{W}^e \}_{e\in  E_n}$ being an independent family of random variables with mean one and variance $R(r-n)$; and the last equality invokes the definition of $\upsilon_r$.  Thus the measures $\mathbb{E}\big[\mathbf{M}_{r}\times \mathbf{M}_{r}\big]$ and $\upsilon_r$ agree on the semi-algebra $\mathcal{S}_{\Gamma\times \Gamma}$, and it follows that the measures are equal by Corollary~\ref{CorollarySemiAlg}.  \vspace{.3cm} 
 
 \noindent (III) By definition, $\mathbf{M}_{r}(\Gamma)$ is equal to the random variable $\mathbf{W}^{e}$ with  $e\in E_0$. Thus the $m^{th}$ centered moment  of  $\mathbf{M}_{r}(\Gamma)$ is equal to $R^{(m)}(r)$ by  (III) of Theorem~\ref{ThmExist}.

\vspace{.3cm} 
 
 \noindent (IV) Define the random measure $\mathbf{M}_{r+1}:=\Phi \mathbf{\hat{W}}_{r+1}$  using the family of  variables $\mathbf{\hat{W}}_{r+1}=\{\mathbf{W}^e\}_{e\in \cup_{k=0}^{\infty} E_k}$ from Theorem~\ref{ThmExist} with parameter value $r+1$.  For each $(i,j)\in \{1,\ldots, b\}^2$, define the random $\mathbf{H}$-valued element $\mathbf{\hat{W}}^{(i,j)}_r:=\{\mathbf{W}^{(i,j)\times e}\}_{e\in \cup_{k=0}^{\infty} E_k}$, where  $(i,j)\times e$ for $e\in E_k$ denotes the element in $E_{k+1}$ corresponding to the edge $e$  on the $(i,j)$-labeled  copy of $D_k$ embedded within the diamond graph $D_{k+1}$.  Note that  $\{\mathbf{\hat{W}}^{(i,j)}_r\}_{1\leq i,j\leq b}$ is a family of independent copies of $\mathbf{\hat{W}}_r$ as a consequence of 
properties (I)-(IV) in Theorem~\ref{ThmExist}. Thus, if we define $\mathbf{M}_{r}^{(i,j)}:= \Phi \mathbf{\hat{W}}^{(i,j)}_r $, then $\{\mathbf{M}_{r}^{(i,j)}\}_{1\leq i,j\leq b}$ is a family of independent copies of $\mathbf{M}_{r}$.  In particular,  for any $n\in \mathbb{N}_0$ and $\mathbf{p}\in\Gamma_n$, we have that
  \begin{align*}
 \mathbf{M}_{r}^{(i,j)}(\mathbf{\overline{p}})\,:=\,\big(\Phi \mathbf{\hat{W}}^{(i,j)}_r \big)(\mathbf{\overline{p}})\,:=\,  \frac{1}{|\Gamma_n| }\prod_{l=1}^{b^n}  \mathbf{W}^{(i,j)\times \mathbf{p}(l)} \,. 
 \end{align*}
There is a one-to-one correspondence between paths $\mathbf{p}\in \Gamma_{n+1}$ and  $(b+1)$-tuples $ (i; \mathbf{p}_1,\ldots ,\mathbf{p}_b) \in \{1,\ldots,b\}\times\Gamma_{n}^b$, where $i\in \{1,\ldots, b\}$ labels the  branch and $\mathbf{p}_{j}(l):= \mathbf{p}\big((j-1)b^{n}+l   \big)    $ for $ j\in\{1,\ldots, b\}$ and $l\in \{1,\ldots, b^{n}\}  $.  Thus we have the identity $|\Gamma_{n+1}| = b|\Gamma_{n}|^b$, and the random variable $\mathbf{M}_{r+1}(\mathbf{\overline{p}})$ can be written as 
 \begin{align}\label{ContM}
 \mathbf{M}_{r+1}(\overline{\mathbf{p}})\,:=\,(\Phi \mathbf{\hat{W}}_{r+1})(\overline{\mathbf{p}}) \,:=\,  &\,\frac{1}{|\Gamma_{n+1}| }\prod_{\ell=1}^{b^{n+1}} \mathbf{W}^{\mathbf{p}(\ell)} \nonumber \\ \,=\,&\,\frac{1}{b}\prod_{ j=1}^b \frac{1}{|\Gamma_{n}|}\prod_{l=1}^{b^{n}}  \mathbf{W}^{\mathbf{p}_j(l)} \nonumber \\ \, =\,&\,\frac{1}{b}\prod_{j=1}^b \mathbf{M}_{r}^{(i,j)}(\mathbf{\overline{p}}_j)\,.
\end{align}
Define the measure $\mathbf{\overline{M}}_{r+1}:= \bplus_{\mathbf{i}=1}^b\prod_{\mathbf{j}=1}^b    \mathbf{M}_{r}^{(\mathbf{i},\mathbf{j})}$, which acts on the product space $\{1,\ldots,b\}\times\Gamma^b$. If  $U$ denotes the canonical bijection from $\{1,\ldots,b\}\times\Gamma^b$ to  $\Gamma$ in Remark~\ref{Remarkn=1}, then we can write~(\ref{ContM}) as 
$$\mathbf{M}_{r+1}(\overline{\mathbf{p}}) \, =\,\mathbf{\overline{M}}_{r+1}\Bigg(\{i\}\times \bigtimes_{j=1}^b \mathbf{\overline{p}}_j\Bigg)  \, =\,\left(\mathbf{\overline{M}}_{r+1}\circ U^{-1}\right)(\overline{\mathbf{p}})\,.
$$
The above shows that $\mathbf{M}_{r+1}$ and the pushforward measure $\mathbf{\overline{M}}_{r+1}\circ U^{-1}$ agree on sets in the semi-algebra $\mathcal{S}_{\Gamma}$.  Hence the Borel measures $\mathbf{M}_{r+1}$ and $\mathbf{\overline{M}}_{r+1}\circ U^{-1}$ are equal by Corollary~\ref{CorollarySemiAlg}, which completes the proof. 
\end{proof}

\subsection{Proof of the uniqueness part of Theorem~\ref{ThmExistence}}\label{ProofThmExistence2}

Before moving to the uniqueness part of Theorem~\ref{ThmExistence}, we will prove   Proposition~\ref{CorProp4}.

\begin{proof}[Proof of Proposition~\ref{CorProp4}] Recall that we endow the space of finite Borel measures,  $\mathcal{M}_\Gamma$, on $\Gamma$   with the weak topology, and that $\mathcal{B}_{\mathcal{M}_\Gamma}$ denotes the Borel $\sigma$-algebra on $\mathcal{M}_\Gamma$. Since $\Gamma$ is a Polish space,  $\mathcal{M}_\Gamma$ is also a Polish space; see~\cite[Lemma 4.5]{Kallenberg}. Thus the measurable space $\big(\mathcal{M}_\Gamma,   \mathcal{B}_{\mathcal{M}_\Gamma}\big)$ is standard Borel.  In particular, the   product spaces $\big(\mathcal{M}_\Gamma^{E_k}, \bigotimes_{e\in E_k}  \mathcal{B}_{\mathcal{M}_\Gamma}\big)$---where the product is indexed by an edge set $E_k$ for some $k\in \mathbb{N}_0$---are also standard Borel, which is what we require below to apply the Kolmogorov extension theorem.

 By applying (IV)  of Theorem~\ref{ThmExistence} inductively, for any $n\in \mathbb{N}_0$ we can construct a family of $ \mathcal{M}_\Gamma $-valued random elements, $\{\mathbf{M}^e\}_{e\in \cup_{k=0}^n E_k}$, defined over the same probability space and  satisfying property  (i) for $0\leq k \leq  n$ and  property (ii) for $0\leq k  <  n$. We can view $\{\mathbf{M}^e\}_{e\in \cup_{k=0}^n E_k}$ as a random element taking values in the  product space  $\bigtimes_{k=0}^n\mathcal{M}_{\Gamma}^{E_k}$.  As a consequence of properties (i) \& (ii), the distributional measures, $\varrho_n $, of the random elements $ \{\mathbf{M}^e\}_{e\in \cup_{k=0}^n E_k}$ form a consistent sequence  of probability measures    on the measurable spaces $\big( \bigtimes_{k=0}^n\mathcal{M}_{\Gamma}^{E_k}, \bigotimes_{k=0}^n \bigotimes_{e\in E_k}\mathcal{B}_{\mathcal{M}_{\Gamma}}\big)$.  By Kolmogorov's extension theorem, there is a probability measure  on the measurable space $\big( \bigtimes_{k=0}^{\infty}\mathcal{M}_{\Gamma}^{E_k}, \bigotimes_{k=0}^{\infty} \bigotimes_{e\in E_k}\mathcal{B}_{\mathcal{M}_{\Gamma}}\big)$ having the measures $\varrho_n$ as marginals.
 \end{proof}

The proof of the uniqueness part of Theorem~\ref{ThmExistence} will depend on Proposition~\ref{CorProp4} and the uniqueness part of Theorem~\ref{ThmExist}.

\begin{proof}[Proof of Theorem~\ref{ThmExistence}, uniqueness part] Let $\{\mathbf{\widehat{M}}_r\}_{r\in \R}$ be  a family of random measure laws satisfying properties (I), (II), (III'), \& (IV) of Theorem~\ref{ThmExistence}.  For fixed $r\in \R$, let the family of measures $\mathbf{\widehat{M}}^e$ for $e\in \cup_{k=0}^{\infty}E_k$ be defined in relation to $\mathbf{\widehat{M}}_r$ as in Proposition~\ref{CorProp4}.  If we define the random variables $\mathbf{\widehat{W}}^e:=\mathbf{\widehat{M}}^e(\Gamma)$ for  $e\in \cup_{k=0}^{\infty}E_k$, then the family of random variables $\{\mathbf{\widehat{W}}^e\}_{e\in \cup_{k=0}^{\infty} E_k}$ satisfies properties (I), (II), (III'), \& (IV) of Theorem~\ref{ThmExist}.  On the other hand, we can write $\mathbf{\widehat{M}}_r$ in terms of the random $\mathbf{H}$-valued element $\{\mathbf{\widehat{W}}^e\}_{e\in \cup_{k=0}^{\infty} E_k}$ as  \begin{align}\label{M}
\mathbf{\widehat{M}}_r\,=\,\Phi\{\mathbf{\widehat{W}}^e\}_{e\in \cup_{k=0}^{\infty} E_k}
\end{align} for the  map $\Phi:\mathbf{H}\rightarrow \mathcal{M}_{\Gamma}$ from Lemma~\ref{LemmaArray}. The equality~(\ref{M}) can be easily  verified  using   Corollary~\ref{CorollarySemiAlg} by showing that the measures $\mathbf{\widehat{M}}_r$ and $\Phi\{\mathbf{\widehat{W}}^e\}_{e\in \cup_{k=0}^{\infty} E_k}$ agree on sets in  the semi-algebra $\mathcal{S}_{\Gamma}$.  

Let $\mathbf{M}_r$ be the random measure   constructed in Section~\ref{ProofThmExistence} using the family of random variables $\{\mathbf{W}^e\}_{e\in \cup_{k=0}^{\infty} E_k}$ from Theorem~\ref{ThmExist}, in other terms, $\mathbf{M}_r=\Phi\{\mathbf{W}^e\}_{e\in \cup_{k=0}^{\infty} E_k}$.  By the uniqueness part of Theorem~\ref{ThmExist}, we must have the equality in distribution $\{\mathbf{\widehat{W}}^e\}_{e\in \cup_{k=0}^{\infty} E_k}\stackrel{\text{d}}{=}\{\mathbf{W}^e\}_{e\in \cup_{k=0}^{\infty} E_k}$, and thus the random measures $\mathbf{\widehat{M}}_r$ and $\mathbf{M}_r$ are equal in distribution. Therefore, the family of random measure laws in Theorem~\ref{ThmExistence} is unique.
\end{proof}

\subsection{Proof of Theorem~\ref{ThmUniversality}}\label{SecProofLimitThmForMeasures}

The proof of Theorem~\ref{ThmUniversality} relies on Theorem~\ref{ThmLimit} below, which is from~\cite[Theorem 5.22]{Clark3}.

\begin{definition}\label{DefInduct} For some fixed $r\in \R$, let $\beta_{r,n}>0$ be defined as in~(\ref{BetaForm}). Let the random variables $\{ \omega_h \}_{h\in E_n}$ be as in~(\ref{DefEM}).  
We  inductively  define the i.i.d.\ arrays of random variables $\big\{ W^{e}_{n} \big\}_{e\in E_k}  $   for  $k\in \{0,1,\ldots, n\}$ as follows:
\begin{enumerate}[(I)]
\item For each $h\in E_n$, define the random variable $ W^{h}_{n}:= \frac{  e^{\beta_{r,n}\omega_h } }{\mathbb{E}\left[e^{\beta_{r,n}\omega_h}   \right]   }  $.
\item For each $k<n$, define the random variables in the array $\big\{ W^{e}_{n}  \big\}_{e\in E_k}$ in terms of those in the array $\big\{ W^{f}_{n}  \big\}_{f\in E_{k+1}}$ through $  W^{e}_{n} := \frac{1}{b}\sum_{i=1}^b\prod_{j=1}^b  W^{e\times (i,j)}_{n}   $.

\end{enumerate}

\end{definition}

\begin{remark}\label{RemarkMeasure} Let the random measure $ \mathbf{M}_{r,n}^{\omega}$ on $\Gamma$ be defined as in Theorem~\ref{ThmUniversality}. For any $N\in \{0,1,\ldots, n\}$ and $\mathbf{q}\in \Gamma_N$,  the random variables $\big\{ W^{e}_n  \big\}_{e\in E_N}$ relate to the measures $\mathbf{M}_{r,n}^{\omega}$ through
\begin{align*}
\mathbf{M}_{r,n}^{\omega}(\mathbf{\overline{q}})\,=\,\frac{1}{|\Gamma_n|}\sum_{\substack{\mathbf{p}\in \Gamma_n \\ \mathbf{\overline{p}} \subset \mathbf{\overline{q}}   }  }  \prod_{l=1 }^{b^n}W^{\mathbf{p}(l)}_{n}
  \,=\,\frac{1}{|\Gamma_N|}\prod_{\ell=1}^{b^N}W^{\mathbf{q}(\ell)}_{n}\,,
\end{align*}
where the first equality follows  from the definition of $\mathbf{M}_{r,n}^{\omega}$, and the second equality is an identity that follows from inductive application of (II). 
\end{remark}

\begin{theorem}\label{ThmLimit} Fix $r\in \R$. For $k,n\in \mathbb{N}_0$ with $0\leq k\leq n$, let the arrays $\big\{ W^e_{n} \big\}_{e\in  E_{k} } $ and $\{ \mathbf{W}^e \}_{e\in  E_{k}  }  $  be defined as in Definition~\ref{DefInduct} and Theorem~\ref{ThmExist}, respectively.\footnote{The random variables $ W^{e}_{n} $ and $\mathbf{W}^e$ are related to the random variables $ X_{e}^{(k,n)} $  and $\mathbf{X}_e^{(k)}$ in~\cite[Section~5]{Clark3} through $ W^{e}_{n}=1+X_{e}^{(k,n)}  $ and $\mathbf{W}^e=1+\mathbf{X}_e^{(k)}$.  } 
\begin{enumerate}[(i)]
\item For all $k\in \mathbb{N}_0$, the array $\big\{ W^e_{n} \big\}_{e\in  E_{k} } $, viewed as an  $\R^{b^{2k }}$-valued random variable, converges in law as $n\rightarrow \infty$ to $\{ \mathbf{W}^e \}_{e\in  E_{k}  }  $.\footnote{Since these arrays of random variables are i.i.d., the  distributional convergence of the arrays stated here is equivalent to the distributional convergence of a single component of the array: $W_n^e\stackrel{\textup{d}}{\Longrightarrow} \mathbf{W}^e$.}

\item  For all $m\in\{2,3,\ldots\}$, $k\in \mathbb{N}_0$,  and  $e\in E_k$, the centered moment $\mathbb{E}\big[ \big(W^e_{n}-1\big)^m  \big]$ converges to $R^{(m)}(r-k)=\mathbb{E}\big[ \big(\mathbf{W}^e-1\big)^m  \big]   $ with large $n$.
\end{enumerate}

\end{theorem}

\begin{proof}[Proof of Theorem~\ref{ThmUniversality}]  By Remark~\ref{RemarkSuff}, it suffices to show that for any continuous function $g:\Gamma\rightarrow \R $ the sequence of random variables $\big\{\mathbf{M}_{r,n}^{\omega}(g)\big\}_{n\in \mathbb{N}}$ converges in distribution  to $\mathbf{M}_{r}(g)$.   Recall from Lemma~\ref{LemDense} that the set of  $ \mathcal{A}_{\Gamma}$-measurable simple functions, $ \mathbf{C}(\Gamma)$, is a dense subset of $C(\Gamma)$ with respect to the uniform norm. Thus for any $g\in C(\Gamma)$ and $\epsilon>0$, there exists a $\psi \in \mathbf{C}(\Gamma)  $ such that $|g(p)-\psi(p)|<\epsilon $ for all $p\in \Gamma$, and hence
\begin{align}
\mathbb{E}\Big[\big|\mathbf{M}_{r,n}^{\omega}(g)\,-\,\mathbf{M}_{r,n}^{\omega}(\psi) \big|^2 \Big]\,\leq \,&\,\epsilon^2 \mathbb{E}\big[\big( \mathbf{M}_{r,n}^{\omega}(\Gamma)    \big)^2\big]\nonumber \\ \,= \,&\,\epsilon^2 \mathbb{E}\big[\big( W_e^{(0,n)}\big)^2\big]\,\,\stackrel{n\rightarrow  \infty}{\longrightarrow} \,\,\epsilon^2\big(1+R(r)\big)\,,
\end{align}
where the convergence holds by (ii) of Theorem~\ref{ThmLimit}. It follows that the $L^2$ distance between  the random variables $\mathbf{M}_{r,n}^{\omega}(g)$ and $\mathbf{M}_{r,n}^{\omega}(\psi)$ is smaller than $2\epsilon (1+R(r))^{1/2}$ for large enough $n$. Moreover, the $L^2$ distance between the random variables $\mathbf{M}_{r}(g)$ and $\mathbf{M}_{r}(\psi)$ is bounded by $\epsilon (1+R(r))^{1/2}$.  In the analysis below, we will show that $\mathbf{M}_{r,n}^{\omega}(\psi) $ converges in distribution to $\mathbf{M}_{r}(\psi) $ as $n\rightarrow \infty$ for any $\psi\in \mathbf{C}(\Gamma)$. Since $\epsilon>0$ is arbitrary, this implies that $\mathbf{M}_{r,n}^{\omega}(g)$ converges in distribution to $\mathbf{M}_{r}(g)$.

If $\psi\in \mathbf{C}(\Gamma)$, then by definition there exist $\alpha_j\in \R$ and $A_j\in \mathcal{A}_{\Gamma} $ such that    $\psi =\sum_{j=1}^J\alpha_j 1_{A_j }  $.  Since $\mathcal{A}_{\Gamma}:=\bigcup_{k=1}^{\infty}\mathcal{A}_{\Gamma}^{(k)}$,  we have that $A_j\in \mathcal{A}_{\Gamma}^{(N)}$ for all $1\leq j\leq J$ when  $N\in \mathbb{N}$ is large enough. For any $n\geq N$,   the $\mathbf{M}_{r,n}^{\omega}$-integral of $\psi$ can be written in the form
\begin{align*}
\mathbf{M}_{r,n}^{\omega}(\psi)\,=\, \sum_{j=1}^{J}\alpha_j\mathbf{M}_{r,n}^{\omega}( A_j )   \,=\,&\,
 \sum_{j=1}^{J}\alpha_j\sum_{\substack{ \mathbf{p}\in \Gamma_n  \\ \mathbf{\overline{p}}\subset A_j  } }   \mathbf{M}_{r,n}^{\omega}(\mathbf{\overline{p}}) \nonumber \\  \,=\,&\,
 \sum_{j=1}^{J}\alpha_j\sum_{\substack{ \mathbf{p}\in \Gamma_n  \\ \mathbf{\overline{p}}\subset A_j  } }   \frac{1}{|\Gamma_n|} \prod_{l=1}^{b^n}W^{ \mathbf{p}(l) }_{n}\,, 
\end{align*}
where we have used the definition of $\mathbf{M}_{r,n}^{\omega}(\mathbf{\overline{p}})$ and  that $A_j\in  \mathcal{A}_{\Gamma}^{(N)}\subset  \mathcal{A}_{\Gamma}^{(n)}$ is a disjoint union of cylinder sets  $\mathbf{\overline{p}}$ with $ \mathbf{p}\in \Gamma_n $. Next using that each $A_j\in\mathcal{A}_{\Gamma}^{(N)}$ can be written as a disjoint union of cylinder sets $\mathbf{\overline{q}}$ with $ \mathbf{q}\in \Gamma_N $, we can rewrite the above as
\begin{align}\label{WW}
\mathbf{M}_{r,n}^{\omega}(\psi) 
\,=\,\sum_{j=1}^{J}\alpha_j\sum_{\substack{ \mathbf{q}\in \Gamma_N \\ \mathbf{\overline{q}}\subset A_j  } }  \sum_{\substack{ \mathbf{p}\in \Gamma_n  \\ \mathbf{\overline{p}}\subset \mathbf{\overline{q}}  } }  \frac{1}{|\Gamma_n|} \prod_{l=1 }^{b^n}W^{  \mathbf{p}(l) }_{n}\,=\,\sum_{j=1}^{J}\alpha_j\sum_{\substack{ \mathbf{q}\in \Gamma_N \\ \mathbf{\overline{q}}\subset A_j  } } \frac{1}{|\Gamma_N|}\prod_{\ell=1}^{b^N} W_{n}^{\mathbf{q}(\ell)}   \,,
\end{align}
where the second equality holds by Remark~\ref{RemarkMeasure}.  By  Theorem~\ref{ThmLimit},  the random array  $\big\{W_{n}^{e}\big\}_{e\in E_N  } $   converges in law with large $n$ to the random array   $\{\mathbf{W}^e\}_{e\in E_N  }$.   Therefore~(\ref{WW}) converges in law as $n\rightarrow \infty$ to  
\begin{align*}
\sum_{j=1}^{J}\alpha_j\sum_{\substack{ \mathbf{q}\in \Gamma_N \\ \mathbf{\overline{q}}\subset A_j  } } \frac{1}{|\Gamma_N|}\prod_{\ell=1}^{b^N}  \mathbf{W}^{\mathbf{q}(\ell) }  \,=:\,\sum_{j=1}^{J}\alpha_j\mathbf{M}_r(A_j )\,=\,\mathbf{M}_r(\psi )\,.
\end{align*}
The first equality  uses the definition of $\mathbf{M}_r:=\Phi\mathbf{\hat{W}}_r$ on $A_j\in \mathcal{A}_{\Gamma}^{(N)}$ in~(\ref{ConI}). The above shows that $\mathbf{M}_{r,n}^{\omega}(\psi) $ converges in distribution to $\mathbf{M}_{r}(\psi) $ with large $n$ for any $\psi\in \mathbf{C}(\Gamma)$, which completes the proof.
\end{proof}

\subsection{Proof of Proposition~\ref{PropProperties}}\label{SecProofPropProperties}

The following is a trivial consequence of Proposition~\ref{CorProp4} that we will use in the proof below.  
\begin{corollary}\label{RemarkProp4} Let $\mathbf{M}_r$ be the random measure on $\Gamma$ from Theorem~\ref{ThmExistence}.
 For any $n\in \mathbb{N}$ and paths $\mathbf{p},\mathbf{q}\in \Gamma_n$ sharing no edges, the restriction measures $\mathbf{M}_{r}|_{\mathbf{\overline{p}}}$ and $\mathbf{M}_{r}|_{\mathbf{\overline{q}}}$ are independent. 
\end{corollary}

\begin{proof}[Proof of Proposition~\ref{PropProperties}] Part (i): Let us write $\mathbf{M}_r=M_r+A_r$, where $M_r$ and $A_r$ are respectively the singular and absolutely continuous components  in the Lebesgue decomposition 
of $\mathbf{M}_r$ with respect to $\mu$. We seek to show that $A_r=0$ holds a.s.  By symmetry of the path space $\Gamma$, the expectations of $M_r$ and $A_r$ must be multiples of the uniform measure, in other terms,   $\mathbb{E}[M_r]=\alpha_r \mu$  and  $\mathbb{E}[A_r]= \beta_r \mu$ for constants $\alpha_r,\beta_r\in [0,1]$ with $\alpha_r+\beta_r=1$.   By property (II) of Theorem~\ref{ThmExistence} and part (ii) of Lemma~\ref{LemCorrelate}, respectively, $\upsilon_r=\mathbb{E}\big[\mathbf{M}_r\times \mathbf{M}_r  \big]$ and $\upsilon_r =\mu\times\mu\,+\,R(r)\rho_r $. Thus we have the equality
\begin{align}
 \mu\times\mu\,+\,R(r)\rho_r  \,=\,&\,\underbrace{\mathbb{E}\big[ M_r \times M_r  \big]}_{\geq \alpha^2_r\mu\times \mu  }\,+\,\underbrace{\mathbb{E}\big[ M_r \times A_r  \big]}_{\geq \alpha_r\beta_r\mu\times \mu  }\nonumber \\ &\,+\,\underbrace{\mathbb{E}\big[  A_r\times M_r   \big]}_{ \geq\alpha_r\beta_r\mu\times \mu  }\,+\,\underbrace{\mathbb{E}\bigg[  \frac{dA_r}{d\mu}(p)\frac{dA_r}{d\mu}(q) \bigg]}_{ \leq\beta_r^2\, \,\,\mu\times \mu-\text{a.e.}  } \mu \times \mu \,.\label{Compare}
\end{align}
The inequalities for the braced terms  will be explained below.

Our first task will be to show the lower bound for $\mathbb{E}\big[ M_r \times M_r  \big]$ by $\alpha^2_r\mu \times \mu$ in~(\ref{Compare}).
Let the set $\mathbf{S}_{\emptyset} \subset \Gamma\times \Gamma$ be defined as in  Corollary~\ref{CorrMuMu}.  Furthermore, let $\mathbf{S}^{(n)}_{\emptyset}\subset \Gamma\times \Gamma$ be defined as in~(\ref{DefS}) for $n\in \mathbb{N}$ and extend the definition to the vacuous case $n=0$ with  $\mathbf{S}^{(0)}_{\emptyset}:=\emptyset$.  Notice that $\mathbf{A}_n:=\mathbf{S}^{(n)}_{\emptyset}\backslash \mathbf{S}^{(n-1)}_{\emptyset} $ is a disjoint union of product cylinder sets $\mathbf{\overline{p}}\times \mathbf{\overline{q}}$ for paths $\mathbf{p},\mathbf{q}\in \Gamma_n$ that share no edges but that have intersecting generation-$(n-1)$ coarse-gainings $[\mathbf{p}]_{n-1}$ and $[\mathbf{q}]_{n-1}$.  Thus by~(\ref{DefS})  we can write $ \mathbf{S}_{\emptyset}$ as a disjoint union of product cylinder sets as follows:
\begin{align}\label{DefSII}
  \mathbf{S}_{\emptyset}\,=\,\bigcup_{n=1}^{\infty} \mathbf{A}_n \,=\,\bigcup_{n=1}^{\infty}  \bigcup_{\substack{\mathbf{p},\mathbf{q}\in \Gamma_n \\  \mathbf{\overline{p}}\times \mathbf{\overline{q}}\subset \mathbf{A}_n} } \mathbf{\overline{p}}\times\mathbf{\overline{q}} \, .
  \end{align}
Observe that for any $N\in \mathbb{N}_0$ and $\mathbf{p_*},\mathbf{q_*}\in \Gamma_N$ we have
\begin{align}\label{Y}
\mathbb{E}\Big[ M_r \times M_r\big(\mathbf{S}_{\emptyset}\cap (\mathbf{\overline{p}_*}\times\mathbf{\overline{q}_*}) \big)  \Big]\,=\,&\,\sum_{n=1  }^{\infty}\sum_{\substack{\mathbf{p},\mathbf{q}\in \Gamma_n \\  \mathbf{\overline{p}}\times \mathbf{\overline{q}}\subset \mathbf{A}_n } }\mathbb{E}\Big[ M_r\big(\mathbf{\overline{p}}\cap \mathbf{\overline{p}_*}\big)  M_r\left(\mathbf{\overline{q}}\cap \mathbf{\overline{q}_*}\right)  \Big] \nonumber \\
\,=\,&\,\alpha^2_r\sum_{n=1  }^{\infty}\sum_{\substack{\mathbf{p},\mathbf{q}\in \Gamma_n \\  \mathbf{\overline{p}}\times \mathbf{\overline{q}}\subset \mathbf{A}_n } }\mu\big(\mathbf{\overline{p}}\cap \mathbf{\overline{p}_*}\big)  \mu\left(\mathbf{\overline{q}}\cap \mathbf{\overline{q}_*}\right)   \nonumber \\
\,=\,&\,\alpha^2_r\mu \times \mu \big(\mathbf{S}_{\emptyset}\cap (\mathbf{\overline{p}_*}\times\mathbf{\overline{q}_*})   \big)\nonumber  \\ \,=\,&\,\alpha^2_r\mu \times \mu\big(\mathbf{\overline{p}_*}\times\mathbf{\overline{q}_*}\big) \,.
\end{align}
The first and third equalities above hold by countable additivity of the measures involved because $\mathbf{S}_{\emptyset} $ can be written as the disjoint union in~(\ref{DefSII}). For the second equality, we have used that  $\mathbb{E}[M_r]=\alpha_r\mu$ and that the random  restriction measures $M_r|_{\mathbf{\overline{p}}}$ and $M_r|_{\mathbf{\overline{q}}}$ are independent for $\mathbf{p},\mathbf{q}\in \Gamma_n$ with $\mathbf{\overline{p}}\cap \mathbf{\overline{q}}  =\emptyset$ since the restrictions $\mathbf{M}_r|_{\mathbf{\overline{p}}}$ and $\mathbf{M}_r|_{\mathbf{\overline{q}}}$ are independent by  Corollary~\ref{RemarkProp4}.
 The last equality above holds because $\mu\times \mu$ takes full measure on $\mathbf{S}_{\emptyset} $   by Corollary~\ref{CorrMuMu}.  We have shown that the measures $\mathbb{E}\big[ M_r \times M_r\big(\mathbf{S}_{\emptyset}\cap (\,\cdot\, ) \big)  \big]$ and  $\alpha^2_r\mu \times \mu$ agree on the semi-algebra $\mathcal{S}_{\Gamma\times\Gamma}$, and therefore the measures are equal by Corollary~\ref{CorollarySemiAlg}. Obviously, this implies the inequality $\mathbb{E}\big[ M_r \times M_r\big]\geq \alpha^2_r\mu \times \mu$.

 The lower bounds in~(\ref{Compare}) for $\mathbb{E}\big[ M_r \times A_r  \big]$ and $\mathbb{E}\big[ A_r  \times  M_r \big]$ by $\alpha_r\beta_r\mu\times \mu$ follow from the same argument as above.   Since $\rho_r$ is mutually singular to $\mu\times \mu$ and $1=\alpha_r+\beta_r$,  the rightmost term in (\ref{Compare})  must be $\leq \beta_r^2\mu\times \mu$,  and thus the function  $F_r(p,q):=\mathbb{E}\big[  \frac{dA_r}{d\mu}(p)\frac{dA_r}{d\mu}(q) \big]$ is $\mu\times \mu$-a.e.\  $\leq \beta_r^2$.  For any continuous function $g:\Gamma\rightarrow [0,\infty)$, we have
\begin{align*}
\textup{Var}\bigg(\int_{\Gamma}& g(p) A_r(dp)    \bigg)\\ \,=\,&\,\mathbb{E}\bigg[ \bigg( \int_{\Gamma} g(p) A_r(dp)   \bigg)^2 \bigg]\,-\,\bigg(\mathbb{E}\bigg[  \int_{\Gamma} g(p) A_r(dp)   \bigg]\bigg)^2  \nonumber  \\ 
\,=\,& \,\int_{\Gamma} g(p)g(q) \mathbb{E}\bigg[ \frac{dA_r}{d\mu}(p)  \frac{dA_r}{d\mu}(q) \bigg] \mu(dp)\mu(dq)\,-\,\beta_r^2 \bigg(\int_{\Gamma} g(p) \mu(dp)  \bigg)^2    
\end{align*}
since $\mathbb{E}\big[A_r(dp)\big]=\beta_r\mu(dp)$. However, the $\mu\times \mu$-a.e.\ inequality $\mathbb{E}\big[ \frac{dA_r}{d\mu}(p)  \frac{dA_r}{d\mu}(q) \big] \leq \beta_r^2 $ implies that the variance above is equal to zero. Therefore $\int_{\Gamma} g(p) A_r(dp) $  is a nonrandom constant. Since this holds for any nonnegative $g\in C(\Gamma)$, the measure $A_r $ must be nonrandom and equal to $\beta_r\mu$. 

Since  $A_r=\beta_r\mu$, equation~(\ref{Compare}) reduces to  
 \begin{align}\label{CompareII}
  \mu\times\mu\,+\,R(r)\rho_r  \,=\,      \mathbb{E}\big[ M_r\times M_r \big]\,+\,2\alpha_r\beta_r \mu\times\mu \,+\,\beta_r^2\mu\times\mu \,,  
  \end{align}
and  $\mathbf{M}_r= M_{r} + \beta_r\mu $  for $M_{r} $ with expectation $\alpha_r\mu$.   With the equality   $\mathbf{M}_r= M_{r} + \beta_r\mu $, it follows from Remark~\ref{RemarkConcat} and property (IV) of Theorem~\ref{ThmExistence} that $\beta_r$ must satisfy the recurrence relation $\beta_{r+1}\,=\,\beta_{r}^b$ for any $r\in \R$.  In particular, if $\beta_r > 0$ for some $r$, then $\beta_{r-n }= \beta_{r}^{1/b^n}$ converges  to $1$ as $n\rightarrow \infty$ with an error, $\alpha_{r-n}=1-\beta_{r-n}  $, of order $b^{-n}$.
  The first and second moments of $M_r(\Gamma)$ are respectively $\mathbb{E}\big[ M_r(\Gamma) \big]=\alpha_r $ and  $\mathbb{E}\big[ \big(M_r(\Gamma)\big)^2 \big]=\alpha_r^2+R(r) $,  where the  form for the second moment holds by evaluating both sides of~(\ref{CompareII}) by the set $\Gamma\times \Gamma$.   The third moment of $\mathbf{M}_{r-n}(\Gamma)$ has the lower bound
$$ \mathbb{E}\Big[ \big| \mathbf{M}_{r-n}(\Gamma) \big|^3\Big] \,\geq \, \mathbb{E}\Big[ \big| M_{r-n}(\Gamma) \big|^3\Big] \,\geq \,\frac{   \mathbb{E}\Big[ \big| M_{r-n}(\Gamma)   \big|^2\Big]^2   }{ \mathbb{E}\big[  M_{r-n}(\Gamma)   \big]    } \, >\,   \frac{  \big(R(r-n)\big)^2 }{ \alpha_{r-n}   } \,,    $$
where the first inequality uses that $\mathbf{M}_{r-n}\geq  M_{r-n}$, the second inequality is Cauchy-Schwarz, and the third inequality follows from our observations above.   The lower bound above implies that the third moment of $\mathbf{M}_{r-n}(\Gamma)$ grows exponentially as $n\rightarrow \infty$ 
since  $\alpha_{r-n}=1-\beta_{r-n}$ vanishes exponentially and $R(r-n) \propto \frac{1}{n} $ by Lemma~\ref{LemVar}.  This contradicts that the moments of  $ \mathbf{M}_{r}(\Gamma)   $ converge to $1$ as $r\rightarrow -\infty$ as a consequence of part (III) of Theorem~\ref{ThmExistence}. \vspace{.3cm}

\noindent Part (ii):  The a.s.\ absence of atoms for $\mathbf{M}_r$ follows trivially from part (i) of Theorem~\ref{ThmPathMeasure},
 which we will prove in Section~\ref{SecPathIntMxM}.

\vspace{.3cm}

\noindent Part (iii): We will first show that the total mass of $\mathbf{M}_{r}$ is a.s.\ positive. Define $x_r\in [0,1]$ as the probability that $\mathbf{M}_{r}(\Gamma)=0$.  Notice that 
$$x_r\,\leq \,\mathbb{E}\big[ | \mathbf{M}_{r}(\Gamma)       \,-\,1    |^2   \big]\,=\,\textup{Var}\big(\mathbf{M}_{r}(\Gamma)\big) \,=  \,R(r)\,, $$
and thus $x_r$ must vanish as $r\rightarrow -\infty$.  By the distributional recursive relation in property (IV) of Theorem~\ref{ThmExistence},  $\{x_r\}_{r\in \R}$ must satisfy the recursive relation $x_{r+1}=\psi(x_r)$ for the map  $\psi:[0,1]\rightarrow [0,1]$ given by $\psi(x)= \big(1-(1-x)^b   \big)^b$.  
Notice that $\psi$ is contractive towards $0$ for small $x>0$ because
$$\psi(x)\, = \, x^b\Bigg(\sum_{k=0}^{b-1}(1-x)^k \Bigg)^b\,\leq \, (bx)^b  \,. $$
Since $x_r\rightarrow 0$ as $r\rightarrow -\infty$ and $x_r$ contracts  towards zero through the operation of $\psi$, it follows that $x_r=0$ for all $r\in \R$.  Therefore, for any $r\in \R$ the measure $\mathbf{M}_{r}$ is a.s.\ nonzero.

Next we leverage this result using hierarchical symmetry to show that $\mathbf{M}_r(A)>0$ holds for every open set $A\subset \Gamma$ a.s.  There exists an $N\in \mathbb{N}$ and a  $\mathbf{p}\in\Gamma_N$ such  that $\mathbf{\overline{p}}\subset A$.  Thus $\mathbf{M}_r(A)$ has the  lower bound
 $$ \mathbf{M}_r(A)\,\geq \,   \mathbf{M}_r(\mathbf{\overline{p}})\,=\,\frac{1}{|\Gamma_N|}\prod_{\ell=1}^{b^N}\mathbf{M}^{\mathbf{p}(\ell)}(\Gamma)   \,, $$
where the  family of  random measures $\{ \mathbf{M}^{e} \}_{e\in E_N}$ is defined in relation to $\mathbf{M}_r$ as in Proposition~\ref{CorProp4}.  Since each  $\mathbf{M}^{e}$ is equal in distribution to $\mathbf{M}_{r-N}$, it follows that the measures in the family $\{ \mathbf{M}^{e} \}_{e\in E_N}$ are a.s.\ nonzero by our result above, and thus $ \mathbf{M}_r(A)$ is a.s.\ nonzero.  The collection of simple cylinder sets $\mathcal{S}_{\Gamma}$ is countable, so we can conclude that a.e.\ realization of $\mathbf{M}_r$ assigns all open sets positive weight. \vspace{.3cm}

\noindent Part (iv):  Let $g\in L^2(\Gamma,\mu)$.  Since $\mathbb{E}[  \mathbf{M}_{r}]=\mu$ and $\mathbb{E}[  \mathbf{M}_{r}\times \mathbf{M}_{r}]=\upsilon_r=\mu\times \mu+R(r)\rho_r$, we get
\begin{align*}
\mathbb{E}\Bigg[ \bigg( \int_{\Gamma} g(p) \mathbf{M}_{r}(dp)       \,-\,\int_{\Gamma} g(p) \mu(dp)        \bigg)^2   \Bigg]\,=\,&\, R(r)\int_{\Gamma} g(p)g(q)\rho_r(dp,dq) \nonumber  \\
\,\leq \,&\, R(r)\int_{\Gamma}\bigg(\frac{1}{2} \big|g(p)\big|^2\,+\,\frac{1}{2} \big|g(q)\big|^2 \bigg)\rho_r(dp,dq)\nonumber \\
\,=\,&\, R(r)\int_{\Gamma}\big|g(p)\big|^2 \mu(dp)  \,,
\end{align*}
where the last equality uses that the marginals of $\rho_r$ are both equal to $\mu$.  The result then follows since $R(r)=\mathit{O}(1/|r|)$  as $r\rightarrow -\infty$.
\end{proof}

\section{Path intersections under the correlation measure  }\label{SecPathIntCorrMeas}

In this section we prove Proposition~\ref{PropIntMeasure} and Lemma~\ref{LemIntSet}. The main step in the proof of Lemma~\ref{LemIntSet} is to show that for $\rho_{r}$-a.e.\ pair $(p,q)$ the set of intersection times $I^{p,q}=\{t\in[0,1]\,|\, p(t)=q(t)\}$ has log-Hausdorff exponent $\geq 1$.  This is achieved through an energy bound in Proposition~\ref{PropFrostman}.   The next subsection introduces some ideas that we use in the proof of Proposition~\ref{PropIntMeasure}.

\subsection{A generation-inhomogeneous Markovian population model}\label{SecPopModel}

Recall from Corollary~\ref{CorrMuMu} that the product measure $\mu\times \mu$ assigns full weight to the set, $\mathbf{S}_{\emptyset}$, of pairs $(p,q)\in \Gamma\times \Gamma$ such that the number, $\xi_n(p,q)$, of edges shared by the coarse-grained paths $[p]_n ,[q]_n\in \Gamma_n$  becomes zero for all large enough $n$. This is equivalent to the a.s.\ extinction of a population beginning with a single member where each member of the generation $n\in \mathbb{N}_0$ population independently has  either $b$ children with probability $\frac{1}{b}$ or no children at all. Under the normalized measure $\widehat{\upsilon}_r=\upsilon_r/(1+R(r))   $ on the space $\Gamma\times \Gamma$, the intersection number $\xi_n(p,q)$ has a similar, but generation-inhomogeneous, population interpretation wherein a  member of  generation $n$ has $b$ children with probability $\frac{1}{b}(1+R(r-n-1))^b/(1+R(r-n))$ and otherwise goes childless.\footnote{The consistency of the interpretation comes from the identity $R(t+1)=\frac{1}{b}[(1+R(t))^{b}-1]$ for  $t\in \R$.}  The following list summarizes our previous definitions/results concerning the three probability measures $\mu\times \mu$, $\widehat{\upsilon}_r$, and $\rho_r$  in the population model language.
\begin{itemize}
\item The events $\mathbf{S}_{\emptyset}$ and $\mathbf{S}_{\emptyset}^c $ correspond to eventual extinction and perpetual survival, respectively.%\footnote{As before, $\mathbf{S}_{\emptyset}^{c}$ denotes the complement of $\mathbf{S}_{\emptyset}$ in $\Gamma\times \Gamma$.}
\item Extinction is certain under $\mu\times \mu$  by Corollary~\ref{CorrMuMu}.

\item Perpetual survival occurs with probability $\frac{R(r)}{1+R(r)}$ under $\widehat{\upsilon}_r$ by  (ii) of Lemma~\ref{LemCorrelate}.

\item $\rho_r$ is the conditioning of $\widehat{\upsilon}_r$ on the event of perpetual survival    by Remark~\ref{RemarkCond}.

\item The population  grows quadratically with $n$ in the event of survival by  (ii) of Proposition~\ref{PropLemCorrelate}.

\end{itemize}
  Lemma~\ref{LemMargPop} below characterizes the asymptotic growth of the number,  $\widetilde{\xi}_n(p,q)$, of  members within the generation-$n$  population having progeny that will never go extinct in the event of indefinite survival of the total population. First, we require a few more definitions.  
Recall that $\mathcal{F}_n$ denotes the $\sigma$-algebra of generation-$n$ cylinder subsets of $\Gamma\times \Gamma$; see Definition~\ref{DefSigmaAlgebra}.

\begin{definition}\label{DefSurvProg} For paths $p,q\in \Gamma$ and $n\in \mathbb{N}$, let $ \big([p]_n;\, \langle p\rangle^1_n,\ldots,  \langle p\rangle^{b^n}_{n}    \big)$ and $ \big([q]_n;\, \langle q\rangle^1_{n},\ldots,  \langle q\rangle^{b^n}_{n}    \big)$ be their corresponding decompositions in $\Gamma_n\times \bigtimes_{\ell=1}^{b^n}\Gamma$  in the sense of Remark~\ref{RemarkCylinder}.  We define $\mathbf{I}_n^{p,q}$ as the set of  $\ell\in \{1,\ldots ,b^n\}$ such that $[p]_n(\ell)=[q]_n(\ell)$ and $\big(\langle p\rangle^{\ell}_{n},\langle q\rangle^{\ell}_{n}\big)\in \mathbf{S}_{\emptyset}^{c}$. We also define the  notations below.
\begin{itemize} 
 \item  $\widetilde{\xi}_{n}\equiv\widetilde{\xi}_{n}(p,q) $  denotes the number of elements in  $ \mathbf{I}_n^{p,q}$. 

\item $\mathpzc{F}_n$ denotes the $\sigma$-algebra   on $\Gamma\times \Gamma$ generated by sets in the $\sigma$-algebra $\mathcal{F}_n$ and  the map from $\Gamma\times \Gamma$ to $\mathcal{P}\big(\{1,\ldots,b^n\}  \big)$ defined by $(p,q)\,\mapsto  \mathbf{I}_n^{p,q}$.

\end{itemize}

\end{definition}

\begin{remark}\label{RemarkAncestors} Given paths $p,q\in \Gamma$, the variable $\widetilde{\xi}_{n}(p,q) \in \mathbb{N}_0$ is the number of edges $e\in E_n$ corresponding to generation-$n$ embedded subcopies of $D$ within which the paths $p$ and $q$  have nontrivial intersections---or, in the language of the population model, indefinitely surviving progeny.  The $\sigma$-algebra  $\mathpzc{F}_n$ contains information about  the coarse-grained paths $[p]_n, [q]_n\in \Gamma_n$ and also which edges $\mathbf{e}\in E_n$ shared by $[p]_n$, $[q]_n$ correspond to cylinder sets $\mathbf{\overline{e}}$ such that $\mathbf{\overline{e}}\cap \textup{Range}(p)\cap \textup{Range}(q)$ is nonempty. 
\end{remark}

\begin{lemma}\label{LemMargPop}   For $(p,q)\in \Gamma\times \Gamma$, let  $\widetilde{\xi}_{n}\equiv \widetilde{\xi}_{n}(p,q)$ be defined as in Definition~\ref{DefSurvProg}, and recall that  $\xi_n\equiv\xi_n(p,q)$ is the number of edges shared by the coarse-grained paths $[p]_n, [q]_n\in \Gamma_n$.

\begin{enumerate}[(i)]

\item Under the probability measure $\rho_r$, the sequence of random variables $(\widetilde{\xi}_{n})_{n\in \mathbb{N}_0}$ is a Markov chain starting with $\widetilde{\xi}_{0}=1$ and having transition law $ \widetilde{\xi}_{n+1}\,\stackrel{\textup{d}}{=}\,\sum_{k=1}^{\widetilde{\xi}_{n}}\mathbf{n}_n^{k} $,   where $\big\{\mathbf{n}^{k}_n\big\}_{k\in \mathbb{N}}$ is a family  of independent copies of a random variable $\mathbf{n}_n$ taking values  in $\{1,\ldots , b\}$ with probabilities
\begin{align}\label{FamilyLine}
\mathbbmss{P}_r\big[\mathbf{n}_n=\ell\big]\,:=\,\frac{1}{b}{b \choose \ell }\frac{ \big( R(r-n-1)  \big)^{\ell}   }{R(r-n)  }\,,
\end{align}
and $\big\{\mathbf{n}^{k}_n\big\}_{k\in \mathbb{N}}$ is independent of the random variable  $\widetilde{\xi}_{n}$.

\item Under the probability measure $\rho_r$, the sequence $\mathbf{\widetilde{m}}_{r,n}:= \frac{ R'(r-n)  }{R(r-n)  } \widetilde{\xi}_{n} $ is a uniformly integrable martingale with respect to the filtration $(\mathpzc{F}_n)_{n\in \mathbb{N}_0}$.  Moreover, $(\mathbf{\widetilde{m}}_{r,n})_{n\in \mathbb{N}_0}$  converges $\rho_r$-a.s.\ with large $n$ to $T(p,q)>0$. 
 
\item Under the probability measure $\widehat{\upsilon}_r$, the conditional expectation of $\mathbf{\widetilde{m}}_{r,n}$
with respect to the $\sigma$-algebra $\mathcal{F}_n$ is equal to $\mathbf{m}_{r,n}:= \frac{ R'(r-n)  }{1+R(r-n)  } \xi_{n} $, and  $\mathbf{m}_{r,n}$ converges $\widehat{\upsilon}_r$-a.s.\ with large $n$ to $T(p,q)$.

\end{enumerate}

\end{lemma}

\begin{remark}\label{Remark} Note that statements (ii) \& (iii) of Lemma~\ref{LemMargPop} have the following consequences.
\begin{itemize}
\item  (ii)  implies that $\widetilde{\xi}_{n}$ $\rho_r$-a.s.\ grows linearly with  $n$ since $\frac{ R'(r-n)  }{R(r-n)  }=\frac{1}{n}+\mathit{o}(\frac{1}{n})$. 
\item (iii) implies that $\xi_{n}$ $\rho_r$-a.s.\ grows quadratically with  $n$ since $\frac{ R'(r-n)  }{1+R(r-n)  }=\frac{\kappa^2}{n^2}+\mathit{o}(\frac{1}{n^2})$.
\end{itemize}
\end{remark}

\begin{remark}\label{RemarkMargPop}  It is interesting to compare the linear growth of the number, $\widetilde{\xi}_n$, of generation-$n$  members that have indefinitely surviving progeny with the quadratic growth of the total population, $\xi_n$.  In this critical population model, where there is neither inevitable  extinction nor the possibility of asymptotically exponential growth, a vanishing fraction of the population have unending family subtrees.  A member of the generation-$n$ population that is destined to have indefinitely surviving progeny will have $b$ direct descendants, but when $n\gg 1$ typically only one them  will have an  unending progeny subtree because
 $\mathbbmss{P}_r\big[\mathbf{n}_n=1\big]=1+\mathit{O}(\frac{1}{n})$ by~(\ref{FamilyLine}) and the asymptotics $R(t)\approx \frac{\kappa^2}{-t}$ for $-t\gg 1$.
\end{remark}

\begin{proof}[Proof of Lemma~\ref{LemMargPop}] The Markovian interpretation in (i) is made possible through the identity 
$$R(t+1)=\frac{1}{b}\big[(1+R(t))^b  -1 \big]=\sum_{\ell=1}^{b}\frac{1}{b}{b \choose \ell } \big( R(t)  \big)^{\ell}        $$ 
with $t=r-n-1$.  We address the various statements in (ii) and (iii)  below.  \vspace{.2cm}

 The martingale property  for $(\mathbf{\widetilde{m}}_{r,n})_{n\in \mathbb{N}_0} $ holds since 
$$\mathbbmss{E}_{\rho_r}\big[ \mathbf{\widetilde{m}}_{r,n+1} \,\big|\,\mathpzc{F}_n \big]\,=\, \frac{ R'(r-n-1)  }{R(r-n-1)  }\mathbbmss{E}_r[ \mathbf{n}_n ] \widetilde{\xi}_n \,=\,\frac{R'(r-n)   }{R(r-n)  }\widetilde{\xi}_n\,=:\,\mathbf{\widetilde{m}}_{r,n}\,, $$
where the second equality follows from the calculation below. Using part (i), we get that
\begin{align}\label{Expect}
 \frac{ R'(r-n-1)  }{R(r-n-1)  }\mathbbmss{E}_r[ \mathbf{n}_n ]\,=\,\,&\frac{ R'(r-n-1)  }{R(r-n-1)  }\sum_{\ell=1}^b  \frac{\ell}{b}{b \choose \ell }\frac{ \big( R(r-n-1)  \big)^{\ell}   }{R(r-n)  }\,.\nonumber  \\
 \,=\,\,&\frac{\frac{1}{b}\frac{d}{dr}\big[(1+R(r-n-1)\big)^b -1   \big]}{ R(r-n)   }\,=\,\frac{R'(r-n)   }{R(r-n)  }\,.
\end{align}
In the above, we have used the derivative formula $R'(t+1)=\big(1+R(t)\big)^{b-1}R'(t)$, which follows from the chain rule and the recursive identity $R(t+1)=\frac{1}{b}\big[(1+R(t)\big)^b -1   \big]$.

Next we shift our focus to the conditional expectation relation between $ \mathbf{m}_{r,n}   $ and $\mathbf{\widetilde{m}}_{r,n}$ in (iii). Note that we can rewrite 
$\mathbf{\widetilde{m}}_{r,n}$ in the form
\begin{align*}
\mathbf{\widetilde{m}}_{r,n}\,=\,\frac{R'(r-n)   }{R(r-n)  } \sum_{\substack{ 1\leq \ell\leq b^n  \\  [p]_n(\ell)=[q]_n(\ell)}} 1_{ \ell \in \mathbf{I}_n^{p,q} } \,.
\end{align*}
Under the measure $\hat{\upsilon}_r$, the probability that $\ell\in \mathbf{I}_n^{p,q}$ when conditioned on the event $[p]_n(\ell)=[q]_n(\ell)$  is equal to $\frac{R(r-n)}{1+R(r-n)   }$; this follows from the Markov property, hierarchical symmetry, and the third bullet point above Definition~\ref{DefSurvProg}.  Thus $ \mathbf{m}_{r,n}$ is the conditional expectation of $\mathbf{\widetilde{m}}_{r,n} $ given $\mathcal{F}_n$:
\begin{align*}
\mathbbmss{E}_{\upsilon_r}\big[ \mathbf{\widetilde{m}}_{r,n}  \,\big|\, \mathcal{F}_n \big]\,=\,\frac{R'(r-n)}{1+R(r-n)   } \xi_n \,=:\, \mathbf{m}_{r,n}\,.
\end{align*}

Since the sequence $(\mathbf{\widetilde{m}}_{r,n})_{n\in \mathbb{N}_0}$  is a nonnegative martingale, it  converges $\rho_r$-a.s.\ to a limit $\mathbf{\widetilde{m}}_{r,\infty}$ with finite expectation.  This a.s.\ convergence extends trivially to the measure $\widehat{\upsilon}_r=\frac{1}{1+R(r)}\big(\mu\times \mu+ R(r)\rho_r  \big)$ because---as a consequence of Corollary~\ref{CorrMuMu}---the measure $\mu\times \mu$ assigns full weight to the event that $\mathbf{\widetilde{m}}_{r,n}=0$ holds for large enough $n$.  By part (ii) of Proposition~\ref{PropMartingales}, the sequence $( \mathbf{m}_{r,n})_{n\in \mathbb{N}_0}$ is a $\widehat{\upsilon}_r$-martingale that converges $\widehat{\upsilon}_r$-a.e.\ to $T(p,q)$. The calculation below shows that the $L^2(\widehat{\upsilon}_r)  $ distance between  $\mathbf{\widetilde{m}}_{r,n}$ and $\mathbf{m}_{r,n}  $ vanishes with large $n$, and thus $\mathbf{\widetilde{m}}_{r,\infty}=T(p,q)$ for $\widehat{\upsilon}_r$-a.e.\ pair $(p,q)\in \Gamma\times \Gamma$:
\begin{align}\label{L2Approx}
\mathbbmss{E}_{\widehat{\upsilon}_r}\big[ (\mathbf{\widetilde{m}}_{r,n} -\mathbf{m}_{r,n}  )^2\big]\,=\,\,&\mathbbmss{E}_{\widehat{\upsilon}_r}\Big[\mathbbmss{E}_{\widehat{\upsilon}_r}\big[ (\mathbf{\widetilde{m}}_{r,n} -\mathbf{m}_{r,n}  )^2\,\big|\, \mathcal{F}_n \big]  \Big]\nonumber \\ \,=\,\,&\bigg(\frac{R'(r-n)   }{R(r-n)  } \bigg)^2\mathbbmss{E}_{\widehat{\upsilon}_r}[\xi_n] \textup{Var}_r\big( \mathbf{n}_n \big)\nonumber  \\
\,=\,\, & \big(\underbrace{1+R(r-n)}_{\approx\,1 \text{ for }n\gg 1}\big)\frac{ R'(r)   }{1+R(r)  }\underbrace{\frac{R'(r-n)   }{\big(R(r-n)\big)^2  }}_{ \approx\,\kappa^{-2} \text{ for }n\gg 1    }\underbrace{\textup{Var}_r( \mathbf{n}_n)}_{\,\mathit{O}( \frac{1}{n}) \text{ for }n\gg 1    }\,,
\end{align}
where the third equality holds because $\mathbf{m}_{r,n}=\frac{ R'(r-n)  }{1+R(r-n)   }\xi_n  $ is a $\widehat{\upsilon}_r$-martingale with expectation $\frac{ R'(r)  }{1+R(r)   }$.  Since  $R(t)\sim \frac{\kappa^2}{-t}$ and $R'(t)\sim \frac{\kappa^2}{t^2}$ with $-t\gg 1$ by Remark~\ref{RemarkDerR},  the variance of the random variable $\mathbf{n}_n$ is of order $\frac{1}{n}$ with large $n$, and thus the $L^2(\widehat{\upsilon}_r)  $ distance between  $\mathbf{\widetilde{m}}_{r,n}$ and $\mathbf{m}_{r,n}  $  vanishes with order $\frac{1}{\sqrt{n}}$.   Since $\sup_{n\in \mathbb{N}}\mathbbmss{E}_{\widehat{\upsilon}_r}\big[ \mathbf{m}_{r,n}^2\big]$ is finite as a consequence of (ii) of Proposition~\ref{PropMartingales}, it follows that  $\sup_{n\in \mathbb{N}}\mathbbmss{E}_{\widehat{\upsilon}_r}\big[ \mathbf{\widetilde{m}}_{r,n}^2\big]$ is finite, and hence the sequence of random variables $(  \mathbf{\widetilde{m}}_{r,n})_{n\in \mathbb{N}_0 }$ is uniformly integrable.
\end{proof}

\subsection{Proof of Proposition~\ref{PropIntMeasure}}\label{SubSecPropIntMeas}

We include a lemma after the proof  that makes a few  observations about the constructions involved. 

\begin{proof} We will split the proof into parts (a)\,-\,(e), where parts (a)\,-\,(d) construct the  measures $\tau^{p,q}\in \mathcal{M}_{[0,1]}$ for every pair $(p,q)\in \Gamma\times \Gamma$ and  part  (e) addresses their properties.  The uniqueness of the measures $\tau^{p,q}$ up to a $\upsilon_r$-null set of pairs $(p,q)$ follows  trivially from~(\ref{TauForm}) because $\tau^{p,q}$ is determined by its values on the intervals $[0,x]$.  It suffices to work with the probability measure $\rho_{r}=\frac{1}{R(r)}(\upsilon_r-\mu\times \mu    )$ rather than $\upsilon_r$ since $\mu\times \mu  $ assigns full measure to the set of pairs $(p,q)$ such that $T(p,q)=0$ by part (ii) of Proposition~\ref{PropLemCorrelate}, in which case we obviously define $\tau^{p,q}:=0$.\vspace{.2cm}

\noindent \textbf{(a) A convenient set algebra:} Recall that $\mathcal{V} $ denotes the set of $x\in [0,1]$ such that  $ x=\frac{k}{b^n}$   for some   $k,n\in \mathbb{N}_0$. All measures that we define below assign  $\mathcal{V}$ weight zero, so its complement $\mathcal{E}:=[0,1]\backslash\mathcal{V} $ will have a more direct role in our constructions. For $0\leq a<b\leq 1$, we will  use $(a,b)_{\mathcal{E}}$ to denote the $\mathcal{E}$-interval $(a,b)\cap \mathcal{E}  $. Define  $\mathcal{A}_\mathcal{E}:=\bigcup_{n=1}^\infty \mathcal{A}_\mathcal{E}^{(n)}$, where  $\mathcal{A}_\mathcal{E}^{(n)}$ denotes the collection of all unions of sets of the form $(\frac{\ell-1}{b^n},\frac{\ell}{b^n})_{\mathcal{E}}$  for $\ell \in \{1,\ldots, b^n\}$---including the empty union. The set collection $\mathcal{A}_\mathcal{E}$ is an algebra on $\mathcal{E}$, and we will define  an additive set function $\widehat{\tau}^{p,q}:\mathcal{A}_\mathcal{E}\rightarrow [0,\infty)$ in part (b) below  that we will extend  to our definition of $\tau^{p,q}\in \mathcal{M}_{[0,1]}$  in part (d). \vspace{.3cm}

\noindent \textbf{(b) A sequence of measures:} For $p,q\in \Gamma$ and $n\in \mathbb{N}$, define  $\mathbf{I}_{n}^{p,q}\subset \{1,\ldots,b^n\}$ as in Definition~\ref{DefSurvProg}, and let $S^{p,q}_{n}$ be the collection of $\mathcal{E}$-intervals $(\frac{\ell-1}{b^n},\frac{\ell}{b^n})_{\mathcal{E}}$ such that $\ell \in \mathbf{I}_{n}^{p,q}$.  In the language of Section~\ref{SecPopModel}, $S^{p,q}_{n}$ is the generation-$n$ subpopulation that have indefinitely surviving progeny.   We define the measure $\tau^{p,q}_{r, n}$ on $[0,1]$ to have density 
\begin{align}\label{TauN}
 \frac{d\tau^{p,q}_{r,n}}{dx}\,=\,\frac{R'(r-n)}{R(r-n)}b^{n}\sum_{e\in S^{p,q}_{n} } 1_{e} \,. 
 \end{align}
We will show that for $\rho_r$-a.e.\ pair $(p,q)$  the sequence $\big(\tau^{p,q}_{r, n}(A)\big)_{n\in \mathbb{N}}$ converges to a limit  $\widehat{\tau}^{p,q}(A)$ for every $A\in \mathcal{A}_\mathcal{E}$.   The limit defines a finitely additive set function $\widehat{\tau}^{p,q}:\mathcal{A}_\mathcal{E}\rightarrow [0,\infty)$.

Let $A\in \mathcal{A}_{\mathcal{E}}$ be arbitrary and pick $N\in \mathbb{N}$ large enough so that $A\in  \mathcal{A}_{\mathcal{E}}^{(N)}$.
The computation below is similar to~(\ref{Expect}) and shows that the sequence of random variables $\big(\tau^{p,q}_{r,n}(A)\big)_{n\geq N} $ forms a martingale on the probability space $(\Gamma\times \Gamma, \mathcal{B}_{\Gamma\times \Gamma}, \rho_r   )$ with respect to the filtration $(\mathpzc{F}_n)_{n\geq N}$.  Notice that
\begin{align}
\mathbbmss{E}_{\rho_r}\Big[ \tau^{p,q}_{r, n+1}(A)  \,\Big|\, \mathpzc{F}_n \Big]\,=\,&\,\frac{R'(r-n-1) }{R(r-n-1)   }  \mathbbmss{E}_{\rho_r}\Bigg[  \sum_{f\in S^{p,q}_{n+1}} 1_{f\subset A }\,\Bigg|\,  \mathpzc{F}_n\Bigg]\nonumber  \\  \,=\,&\, \frac{R'(r-n-1) }{R(r-n-1)   }   \mathbbmss{E}_{\rho_r}\Bigg[ \sum_{e\in S^{p,q}_{n}} 1_{e\subset A }\Big|\big\{  f\in S^{p,q}_{n+1} \,\big|\, f\subset e \big\}\Big|\,\Bigg|\,  \mathpzc{F}_n\Bigg]\nonumber \\ \,=\,&
\, \frac{ R'(r-n-1)  }{R(r-n-1)   }\sum_{ e\in S^{p,q}_{n}   }1_{e\subset A }\mathbbmss{E}_{\rho_r}\bigg[\Big|\big\{  f\in S^{p,q}_{n+1} \,\big|\, f\subset e \big\}\Big| \,\bigg|\,  e\in S^{p,q}_{n} \bigg] \,, \nonumber
\end{align}
where the third equality  uses the Markov property and that the $\sigma$-algebra $\mathpzc{F}_n$ contains information about the generation-$n$ population $ S^{p,q}_{n}$.  Since each $e\in S^{p,q}_{n}$ has $j\in \{1,\ldots, b\}$ children in $ S^{p,q}_{n+1}$ with probability $ \frac{1}{b} {b \choose j } \big(R(r-n-1)   \big)^{j} /  R(r-n) $,  the above is equal to 
 \begin{align*}
 \mathbbmss{E}_{\rho_r}\Big[ \tau^{p,q}_{r, n+1}(A)  \,\Big|\, \mathpzc{F}_n \Big]\,=\,&\,\frac{ R'(r-n-1)  }{R(r-n-1)   }\sum_{e\in S^{p,q}_{n}  }1_{ e\subset A }\sum_{j=1}^{b}\frac{j}{b} {b \choose j }  \frac{\big(R(r-n-1)   \big)^{j}  }{  R(r-n)}\nonumber \\
 \,=\,& \,\frac{\frac{d}{dr} \left[\frac{1}{b}\left(\big(1+  R(r-n-1)  \big)^b -1    \right)\right] }{R(r-n)   } \sum_{ e\in S^{p,q}_{n}   }1_{e\subset A}\nonumber
\\ \,=\,&\, \frac{R'(r-n) }{R(r-n)   }\sum_{e\in S^{p,q}_{n}    }1_{e\subset A}   \,=\, \tau^{p,q}_{r, n}(A) \,.
\end{align*}
The last two equalities respectively use the chain rule and the  identity $R(t+1)=\frac{1}{b}\big[(1+R(t))^b-1    \big]$ with $t=r-n-1$.

We have established that for every $A\in \mathcal{A}_{\mathcal{E}}$ the nonnegative sequence $\big(\tau^{p,q}_{r, n}(A)\big)_{n\in \mathbb{N}} $ is a  $\rho_r $-martingale when  $n\in \mathbb{N}$ is large enough. It follows that for $\rho_r$-a.e.\ pair $(p,q)$  the sequences $\big(\tau^{p,q}_{r, n}(A)\big)_{n\in \mathbb{N}} $ converge with large $n$ to limits $ \widehat{\tau}^{p,q}(A) $ for all $A$ in the countable collection  $\mathcal{A}_{\mathcal{E}} $.   Note that the limit $\widehat{\tau}^{p,q}$ does not depend on the parameter $r\in \R$ since $\frac{R'(r-n) }{R(r-n)   }\sim \frac{1}{n}$ for $n\gg 1$ as a consequence of Lemma~\ref{LemVar} and Remark~\ref{RemarkDerR}.  Moreover, the set  functions $\widehat{\tau}^{p,q}: \mathcal{A}_{\mathcal{E}}\rightarrow [0,\infty)$  inherit finite additivity   from the set functions $\tau^{p,q}_{r, n}: \mathcal{A}_{\mathcal{E}}\rightarrow [0,\infty)$. For the $\rho_r$-null set of pairs $(p,q)$ such that $\widehat{\tau}^{p,q}$ is not well-defined as a limit of $\tau^{p,q}_{r, n}$, we can  define $\widehat{\tau}^{p,q}:=0$.  \vspace{.3cm}

\noindent \textbf{(c) An open region around the vertex set:} Next we will show that for $\rho_r$-a.e.\ $(p,q)$ there exists an open subset, $O^{p,q}$, of $ [0,1]$ that contains  $\mathcal{V}$ as a subset  and for which $\widehat{\tau}^{p,q}$ is supported on $[0,1]\backslash O^{p,q}   $, meaning that $\widehat{\tau}_{p,q}(A)=0$ for each $A\in \mathcal{A}_{\mathcal{E}}$ with $A\subset O^{p,q}   $.  The main observation needed to construct  $O^{p,q}$ is the following:  for any $x\in \mathcal{V}$, there exist $\varepsilon_x, N_x>0$ such that the support of  $d\tau^{p,q}_{r,n}/dx$ is disjoint from $(x-\varepsilon_x,x+\varepsilon_x)$ for all $n\geq N_x$.\footnote{We are interpreting intervals $(x-\varepsilon,x+\varepsilon)$ as subsets of $[0,1]$, i.e., as $\big\{y\in [0,1]\,\big|\,|y-x|<\varepsilon\big\}.$ } To see this, note that if $n\gg 1$ and  $x\in \mathcal{V}$ is an endpoint of some $\mathcal{E}$-interval  $e=\big(\frac{k-1}{b^n},\frac{k}{b^n}   \big)_{\mathcal{E} }$ in  $S_{n}^{p,q}$, then there is roughly a $\frac{1}{b}$ probability that $x  $ is an endpoint of some offspring $f\in S^{p,q}_{n+1}$ of $e$. This holds because with probability close to one only one of the $b$ offspring of $e$ will be in $S^{p,q}_{n+1}$; see Remark~\ref{RemarkMargPop}. If $x$ is not an endpoint for one of the offspring $f \in S^{p,q}_{n+1}$ of  $e$, then $x$ will have a distance $\geq \frac{1}{b^{n+1}}$ from the set $\bigcup_{\mathsmaller{f\in S_{n+1}^{p,q}}}f$, i.e., the support of $d\tau^{p,q}_{r,n+1}/dx$.  Thus the probability that $x$ is not gapped away from the support of  $d\tau^{p,q}_{r,n}/dx$ vanishes exponentially with large $n$, and with probability one there exist $\varepsilon_x,N_x>0$ such that $\tau^{p,q}_{r,n}$ assigns the interval $(x-\varepsilon_x,x+\varepsilon_x)$ measure zero  for all $n\geq N_x$.  Since  $\mathcal{V}$ is countable, this property a.s.\  holds for all elements in $\mathcal{V}$.  The open set $O^{p,q}:=\bigcup_{x\in \mathcal{V}  }(x-\varepsilon_x,x+\varepsilon_x)$ obviously satisfies  $\mathcal{V}\subset O^{p,q}$,  and for each $A\in \mathcal{A}_{\mathcal{E}}$ with $A\subset O^{p,q}   $ we have that $\tau^{p,q}_{r,n}(A )=0$ for large enough $n$, and thus $\widehat{\tau}^{p,q}(A)=0$. \vspace{.3cm}

\noindent \textbf{(d) Extension to measures:} Let $\Xi \subset \Gamma\times \Gamma$ denote the set of pairs $(p,q)$ for which the sequence of measures $\big(\tau^{p,q}_{r,n}\big)_{n\in \mathbb{N}}$ converges set-wise on $\mathcal{A}_{\mathcal{E}}$ to a limit $\widehat{\tau}^{p,q}:\mathcal{A}_{\mathcal{E}}\rightarrow [0,\infty)$ and for which there exists an open set  $O^{p,q}\subset [0,1]$ of the type in part (c), i.e., such that $\mathcal{V}\subset O^{p,q}$ and  the premeasure $\widehat{\tau}^{p,q}$ is supported on $[0,1]\backslash O^{p,q}$.  By parts (b) and (c), the set $\Xi$ has full measure under $\rho_r$. We define $\tau^{p,q}:=0$ when $(p,q)\in \Xi^c$, and the definition of $\tau^{p,q}\in \mathcal{M}_{[0,1]}$ in the case $(p,q)\in\Xi$ is formulated below.

For $(p,q)\in  \Xi$, define $F^{p,q}:[0,1]\rightarrow [0,\infty)$ such that $F^{p,q}(1):=\widehat{\tau}^{p,q}( \mathcal{E}) $ and  for $x\in [0,1)$
\begin{align}
F^{p,q}(x)\,:=\,\inf_{a\in \mathcal{V}\cap (x,1]   }\widehat{F}^{p,q}(a)\hspace{.5cm}\text{for}\hspace{.5cm}\widehat{F}^{p,q}(a)\,:=\,\widehat{\tau}^{p,q}\big([0,a]\cap \mathcal{E}  \big)  \,.
\end{align} 
It follows immediately from the above definition that $F^{p,q}$  is nondecreasing and right-continuous.  Note that $\widehat{F}^{p,q}(a)$ is a nondecreasing function of $a\in \mathcal{V}$ because $\widehat{\tau}^{p,q}$ is finitely additive.   If $x\in [0,1)\cap \mathcal{V}$, then we can pick $\varepsilon>0$ small enough such  that $(x-\varepsilon, x+\varepsilon)\subset O^{p,q}$.  It follows that  $\widehat{\tau}^{p,q}\big([x,a]\cap \mathcal{E}  \big)=0$ for all $a\in (x,x+\varepsilon)\cap \mathcal{V}$, and   thus  we have $$\widehat{F}^{p,q}(a)\,=\, \widehat{F}^{p,q}(x)\,+\,\widehat{\tau}^{p,q}\big([x,a]\cap \mathcal{E}  \big)\,=\,  \widehat{F}^{p,q}(x)\,. $$
 Therefore, $F^{p,q}(x)= \widehat{F}^{p,q}(x)$ for all $x\in \mathcal{V}$.  In fact, $F^{p,q}(x)$ is constant and thus continuous on the interval $(x-\varepsilon, x+\varepsilon)$ by the same argument.
 
Since $F^{p,q}$ is nondecreasing and right-continuous, there is a unique Borel measure $\tau^{p,q}$ on $[0,1]$ having $F^{p,q}$ as its cumulative distribution function, i.e., $F^{p,q}(x)=\tau^{p,q}\big([0,x]  \big)$.  The measure $\tau^{p,q}$ assigns the countable set $\mathcal{V}$ weight zero since $F^{p,q}$ is continuous at points in $\mathcal{V}$.  For any $x\in \mathcal{V}$, we thus have agreement between $\tau^{p,q}$ and $\widehat{\tau}^{p,q}$ on the set $[0,x]\cap \mathcal{E}$:
 $$\tau^{p,q}\big([0,x]\cap \mathcal{E}  \big)\,=\,\tau^{p,q}([0,x])\,=:\,F^{p,q}(x)\,=\,\widehat{F}^{p,q}(x)\,:=\,\widehat{\tau}^{p,q}\big([0,x]\cap \mathcal{E}  \big)\,.$$ It follows, more generally, that $\tau^{p,q}(A )=\widehat{\tau}^{p,q}(A)$  for all $A\in \mathcal{A}_{\mathcal{E}}$.  In other terms, the measure $\tau^{p,q}:\mathcal{B}_{[0,1]}\rightarrow [0,\infty)$ is an extension of the additive set function $\widehat{\tau}^{p,q}:\mathcal{A}_{\mathcal{E}}\rightarrow [0,\infty)$.  As a consequence of this construction of $\tau^{p,q}$, if $A\in \mathcal{A}_{\mathcal{E}}$ and $B\subset \mathcal{V}$, then $\tau^{p,q}(A\cup B)$ is the limit of $\tau^{p,q}_{r,n}(A\cup B)$ since
\begin{align}\label{Tau2Tau}
\lim_{n\rightarrow \infty} \tau^{p,q}_{r,n}(A\cup B)\,=\,\lim_{n\rightarrow \infty} \tau^{p,q}_{r,n}(A)\,=:\,\widehat{\tau}^{p,q}(A) \,=\, \tau^{p,q}(A) \,=\, \tau^{p,q}(A\cup B)\,.
\end{align}

\vspace{.3cm}

\noindent \textbf{(e) Properties of the limit measure:} Next we will prove the properties of the measures $\tau^{p,q}$ stated in Proposition~\ref{PropIntMeasure}, except for the claim that $\tau^{p,q}$ is $\rho_r$-a.s.\ diffuse, which will follow as a trivial corollary of  Proposition~\ref{PropFrostman} in the next subsection. \vspace{.2cm}

$\bullet$ To see that the map from $\Gamma\times \Gamma$ to $\mathcal{M}_{[0,1]} $ sending $(p,q)$ to $\tau^{p,q}$ is measurable, first observe that for each $x\in \mathcal{V}$ the value $\tau^{p,q}\big([0,x]  \big)$ is a $\mathcal{B}_{\Gamma\times \Gamma}$-measurable function of $(p,q)$ because it is defined as zero on $\Xi^c$ and is defined as a pointwise limit    of the sequence of simple functions $\big(\tau^{p,q}_{r,n}\big([0,x]  \big)\big)_{n\in \mathbb{N}  }$ on $\Xi$, where  $\tau^{p,q}_{r,n}\big([0,x]  \big)$ is $\mathcal{A}_{\Gamma\times \Gamma}^{(n)}$-measurable and thus $\mathcal{B}_{\Gamma\times \Gamma}$-measurable for each $n$. Since the collection of intervals $\big\{[0,x]\,\big|\, x\in \mathcal{V} \big\}$ is a $\pi$-system that generates the $\sigma$-algebra $\mathcal{B}_{[0,1]}$,  a routine application of the $\pi$-$\lambda$ theorem yields that $\tau^{p,q}(A)$ is a $\mathcal{B}_{\Gamma\times \Gamma}$-measurable function for each $A\in \mathcal{B}_{[0,1]}$.   Since the metric space $[0,1]$ is Polish, it follows that the map $(p,q)\rightarrow \tau^{p,q}\in \mathcal{M}_{[0,1]} $ is $\mathcal{B}_{\Gamma\times \Gamma}$-measurable when  $\mathcal{M}_{[0,1]} $ is equipped with its Borel $\sigma$-field.\footnote{See~\cite[Lemma 4.7]{Kallenberg} for equivalent characterizations of the Borel $\sigma$-field of $\mathcal{M}_S$ for a Polish space $S$.}\vspace{.2cm}

$\bullet$ $\tau^{p,q}(\mathcal{V})=0$  for all $(p,q)$ since this holds when $(p,q)\in \Xi$ by the discussion in part (d) and   we define $\tau^{p,q}:=0$ when $(p,q)\notin \Xi$. \vspace{.2cm}

$\bullet$ Next we show that $I^{p,q}$ is a $\tau^{p,q}$-full measure set for all $(p,q)$. When $(p,q)\notin \Xi$ this holds vacuously since $\tau^{p,q}:=0$.  When $(p,q)\in \Xi$, we first notice that
the complement of $I^{p,q}$  has the form
$$ [0,1]\,\backslash \, I^{p,q}\,=\,\bigcup_{N=1}^\infty O_N   \hspace{.5cm} \text{for} \hspace{.5cm} O_N \,:=\, \bigcup_{\substack{ 1\leq \ell\leq b^N \\  [p]_N(\ell)\neq [q]_N(\ell)  }} \bigg(\frac{\ell-1}{b^N},\frac{\ell}{b^N}\bigg)    \,. $$ 
The definition of $\tau_{r, n}^{p,q}$ in~(\ref{TauN})  implies that  $\tau_{r,n}^{p,q}(O_N)=0$ for all $n\geq N$. It follows that $\tau^{p,q}(O_N)=0$ by~(\ref{Tau2Tau}) since $O_N=A\cup B$ for some $A\in \mathcal{A}_{\mathcal{E}}$ and $B\subset V$.   Therefore $\tau^{p,q}$ assigns the set $[0,1]\,\backslash \,I^{p,q}$ measure zero by countable additivity. \vspace{.2cm}

$\bullet$ Next we show the formula~(\ref{TauForm}) for the distribution function $F^{p,q}(x)=\tau^{p,q}\big([0,x]  \big) $. Since $\tau^{p,q}$ has no atoms, $F^{p,q}$ is continuous, and  it suffices to verify~(\ref{TauForm}) for $x$ in the dense set  $\mathcal{V}\subset [0,1]$.  If $x\in \mathcal{V}$, then
\begin{align}\label{DisplayDing}
 \tau^{p,q}\big([0,x]\big)\,=\,\widehat{\tau}^{p,q}\big([0,x]\cap \mathcal{E}\big)\,:=\,\lim_{n\rightarrow \infty}\frac{R'(r-n)}{R(r-n)}\big|\left\{e\in S^{p,q}_{n}\,\big|\,   e\subset [0,x] \right\}\big|\,. 
  \end{align}
Define  $\mathbf{m}_{r,n}^x:=\frac{R'(r-n)}{1+R(r-n)}\,\big|\,\big\{1\leq \ell \leq  b^n \,\big|\,[p]_n(\ell)=[q]_n(\ell) \text{ and } [\frac{\ell-1}{b^n},\frac{\ell}{b^n}]\subset [0,x]   \big\}\big| $.   
The proof of (iii) of Lemma~\ref{LemMargPop} can be superficially adjusted to show that $(\mathbf{m}_{r,n}^x)_{n\in \mathbb{N}}$ is an $\widehat{\upsilon}_r$-martingale for large enough $n$ that converges $\widehat{\upsilon}_r$-a.e.\ to~(\ref{DisplayDing}).  Since $R(t)\rightarrow 0$ and $R'(t)\sim \frac{k^2}{t^2}$ as $t\rightarrow -\infty$ by Remark~\ref{RemarkDerR},  the factor $\frac{R'(r-n)}{1+R(r-n)}$ can be replaced by $\frac{k^2}{n^2}$ in the limit as $n\rightarrow \infty$.
 \end{proof}

The  lemma below concerning the sequence of  random measures $\big(\tau^{p,q}_{r,n}\big)_{n\in \mathbb{N}}$ in the proof of Proposition~\ref{PropIntMeasure} will be useful in the next subsection.  Recall our notation $\mathcal{V}:=\big\{ \frac{k}{b^n}\in [0,1]\,\big|\,k,n\in \mathbb{N}_0  \big\}$,  $\mathcal{E}:=[0,1]\backslash   \mathcal{V}$, and  $(a,b)_{\mathcal{E} }:=(a,b)\cap \mathcal{E}$.  Part (ii), which uses that $\tau^{p,q}(\mathcal{V})=0$ for all $(p,q)$, will be unnecessarily restrictive once we establish that the random measure $\tau^{p,q}$ is $\rho_r$-a.e.\ diffuse. 
\begin{lemma}\label{LemmaCas} For $n\in \mathbb{N}_0$ and $p,q\in \Gamma$, let $\tau^{p,q}_{r,n}$ denote the Borel measure on $[0,1]$ with Lebesgue density~(\ref{TauN}) and $\tau^{p,q}$ be  defined as in Proposition~\ref{PropIntMeasure}.  Moreover, let the $\sigma$-algebra $ \mathpzc{F}_n \subset \mathcal{B}_{\Gamma\times \Gamma}$ be defined as in Definition~\ref{DefSurvProg}.
\begin{enumerate}[(i)]
\item  For $\rho_r$-a.e.\ $(p,q)\in \Gamma\times \Gamma$, the sequence of measures $\big(\tau^{p,q}_{r, n}\big)_{n\in \mathbb{N}}$ converges weakly to $\tau^{p,q}$. 

\item For $\rho_r$-a.e.\ $(p,q)\in \Gamma\times \Gamma$, the sequence  $\big(\tau^{p,q}_{r, n}(A)\big)_{n\in \mathbb{N}}$ converges to  $\tau^{p,q}(A)$ for every  set $A\subset [0,1]$ that is a union of $\mathcal{E}$-intervals of the form $(\frac{\ell-1}{b^n}, \frac{\ell}{b^n}  )_{\mathcal{E}}$ with $\ell,n\in \mathbb{N}_0$.

\item Fix $N\in \mathbb{N}$ and any set $A\subset [0,1]$  of the form $A=\bigcup_{\ell\in I}(\frac{\ell-1}{b^N}, \frac{\ell}{b^N}  )_{\mathcal{E}}$ for some $I\subset \{1,\ldots, b^N\} $.  The sequence of variables  $\big( \tau^{p,q}_{r,n}(A)\big)_{n\geq N} $ on the probability space $(\Gamma\times \Gamma,\mathcal{B}_{\Gamma\times\Gamma},\rho_r)$ is a martingale  with respect to the filtration  $( \mathpzc{F}_n)_{n\in \mathbb{N}} $ that satisfies  $\tau^{p,q}_{r,n}(A)=\mathbbmss{E}_{\rho_r}[ \tau^{p,q}(A)\, |\,\mathpzc{F}_n]$ for all $n\geq N$.

\item Fix $n\in \mathbb{N}$ and distinct $\ell,l\in \{1,\ldots, b^n\}$.  If we view $\tau^{p,q}$ as a random measure on $[0,1]$ defined  over the probability space  $(\Gamma\times \Gamma,\mathcal{B}_{\Gamma\times\Gamma},\rho_r)$, then   the restrictions $\tau^{p,q}|_{(\frac{\ell-1}{b^n}, \frac{\ell}{b^n} )_{\mathcal{E}}}$ and $\tau^{p,q}|_{(\frac{l-1}{b^n}, \frac{l}{b^n}  )_{\mathcal{E}}}  $ are independent random measures  when conditioned on $\mathpzc{F}_n$.

\end{enumerate}

\end{lemma}

\begin{proof}   We can prove (i) by showing that for all $(p,q)\in \Xi$ there is pointwise convergence of the cumulative distribution  functions $F^{p,q}_{r,n}(x)=\tau^{p,q}_{r,n}\big([0,x]\big)  $ to $F^{p,q}(x)=\tau^{p,q}\big([0,x]\big)  $ over the set $\mathcal{V}\subset [0,1]$, which is dense and contains $1$. When $x\in \mathcal{V}$, then $[0,x]=A\cup B$ for some $A\in \mathcal{A}_{\mathcal{E}}$ and $B\subset \mathcal{V}$, and thus the convergence holds by~(\ref{Tau2Tau}). \vspace{.2cm}

 From the construction of $\tau^{p,q}$ in part (d) of the proof of Proposition~\ref{PropIntMeasure}, we have that $\tau^{p,q}(\mathcal{V})=0$ for all $(p,q)$.  Thus (ii) follows from (i) because $\partial A\subset \mathcal{V}$.   \vspace{.2cm}

For statement (iii),  recall that we showed that the sequence of random variables  $\big( \tau^{p,q}_{r,n}(A)\big)_{n\geq N} $ is a $\rho_r$-martingale with respect to the filtration  $( \mathpzc{F}_n)_{n\in \mathbb{N}} $ in part (b) of the proof of Proposition~\ref{PropIntMeasure}.  Since  $\tau^{p,q}_{r,n}(A)\rightarrow \tau^{p,q}(A)$ for $\rho_r$-a.e.\ $(p,q)$ by (ii), it suffices to show that $\big( \tau^{p,q}_{r,n}(A)\big)_{n\geq N} $ is uniformly $\rho_r$-integrable.  However, for any $A\in \mathcal{B}_{[0,1]}$, we have  $\tau^{p,q}_{r,n}(A)\leq \tau^{p,q}_{r,n}\big([0,1]\big)= \mathbf{\widetilde{m}}_{r,n}$ for all $p,q\in \Gamma$. The sequence of random variables $ (\mathbf{\widetilde{m}}_{r,n})_{n\in \mathbb{N}}$  is uniformly $\rho_r$-integrable by (ii) of Lemma~\ref{LemMargPop}, so  $ \big(\tau^{p,q}_{r,n}(A)\big)_{n\in \mathbb{N}}$ is also uniformly $\rho_r$-integrable. \vspace{.2cm}

For statement (iv), recall from part (b) of the proof of Proposition~\ref{PropIntMeasure} that $S^{p,q}_{n}$ is a collection of intervals $(\frac{\ell-1}{b^n}, \frac{\ell}{b^n} )_{\mathcal{E}}$ with $\ell\in \mathbf{I}^{p,q}_{n}$, where $ \mathbf{I}^{p,q}_{n}\subset \{1,\ldots, b^n\}$ is defined as in Definition~\ref{DefSurvProg}.  Under the probability measure $\rho_r$, we can view the sequence $\big(\mathbf{I}^{p,q}_{n}\big)_{n\in \mathbb{N}_0}$ as defining a  random infinite branching graph, where generation-$n$ nodes on the graph independently have $\ell\in \{1,\ldots,b\}$ offspring with probability~(\ref{FamilyLine}), and the $\sigma$-algebra $\mathpzc{F}_n$ corresponds to information on the family lineages up to generation $n$.  Conditioned on $\mathpzc{F}_n$, two potential members, $e,f\in \{1,\ldots,b^n\}$,  of the generation-$n$ population generate progeny subtrees independently. Thus  the restrictions $\widehat{\tau}^{p,q}|_{e}$ and $\widehat{\tau}^{p,q}|_{f}$ of the set function $\widehat{\tau}^{p,q}:\mathcal{A}_{\mathcal{E}}\rightarrow [0,\infty)$ defined in part (b) of the proof of Proposition~\ref{PropIntMeasure} are independent when conditioned on  $\mathpzc{F}_n$, and this independence  is inherited by $\tau^{p,q}|_{e}$ and $\tau^{p,q}|_{f}$ through the construction of $\tau^{p,q}$ in part (d) of the proof  of Proposition~\ref{PropIntMeasure}.
\end{proof}

\subsection{A lower bound for the log-Hausdorff exponent of the set of intersection times    }\label{SubSectionPathIntRho}

The following is a corollary of Proposition~\ref{PropFrostman} below. 
Recall that $\rho_r$ is the probability measure on $\Gamma\times \Gamma$ from  (ii) of Lemma~\ref{LemCorrelate} and that $I^{p,q}$ denotes the set of intersection times between   $p,q\in \Gamma$. 

\begin{corollary}\label{CorIntSet}
The  set of intersection times $I^{p,q}$ has log-Hausdorff exponent  $\geq 1$ for  $\rho_r$-a.e.\ pair $(p,q)\in \Gamma\times \Gamma$.
\end{corollary}

 The proof of Corollary~\ref{CorIntSet} is placed in Appendix~\ref{AppendixHausdorff} because it proceeds from the energy bound in Proposition~\ref{PropFrostman}  without change from the analogous method for obtaining a lower bound for the  Hausdorff dimension of a set using an energy bound.   Note that the integrand in~(\ref{Energy}) below is the reciprocal of the dimension function $h(a)=1/\log^{\frak{h}}(\frac{1}{a})$ in Definition~\ref{DefLogHaus} with $a=|x-y|$.

\begin{proposition}\label{PropFrostman} Let the map $(p,q)\mapsto \tau^{p,q}$ from $\Gamma\times \Gamma$ to $\mathcal{M}_{[0,1]}$ be defined as in Proposition~\ref{PropIntMeasure}.  For $\rho_r$-a.e.\  $(p,q)$ and any  $\frak{h}\in [0,1)$, we have
\begin{align}\label{Energy}
Q^\frak{h}\left(\tau^{p,q}\right)\,:=\,\int_{[0,1]\times [0,1]}  \log^\frak{h} \Big(\frac{1}{|x-y|}\Big)   \tau^{p,q}(dx) \tau^{p,q}(dy)\,\,<\,\,\infty \,.
\end{align}
 
\end{proposition}

\begin{proof} Our analysis will be divided into parts (a)\,-\,(e)  below, but first we recollect some notation and previous observations.  Recall  that the $\sigma$-algebra $\mathpzc{F}_n\subset \mathcal{B}_{\Gamma\times \Gamma}$ and the set $\mathbf{I}_{n}^{p,q}\subset \{1,\ldots, b^n\}$ for $p,q\in \Gamma$ are defined in Definition~\ref{DefSurvProg}.    As in the proof of Proposition~\ref{PropIntMeasure},  we define 
\begin{itemize}

\item $S_{n}^{p,q}$ as the collection of $\mathcal{E}$-intervals $ \big(\frac{\ell-1}{b^n},\frac{\ell}{b^n}  \big)_{\mathcal{E}}\subset [0,1]$ such that $\ell\in \mathbf{I}_{n}^{p,q}$ and

\item  $\tau^{p,q}_{r,n}$ as the measure on $[0,1]$ with Lebesgue density~(\ref{TauN}). 

\end{itemize}
Also, recall from Lemma~\ref{LemmaCas}  that   $\big( \tau^{p,q}_{r,n}(A)\big)_{n\geq N} $ is a uniformly integrable  $\rho_r$-martingale with respect to the filtration  $( \mathpzc{F}_n)_{n\in \mathbb{N}} $  that converges  $\rho_r$-a.e.\ to $\tau^{p,q}(A)$ for any set $A\subset [0,1]$  of the form $A=\bigcup_{\ell\in I}\big(\frac{\ell-1}{b^N}, \frac{\ell}{b^N}  \big)_{\mathcal{E}}$ for some $I\subset \{1,\ldots, b^N\} $. \vspace{.3cm}

\noindent \textbf{(a) Energy estimate:}   It suffices to prove that $\mathbbmss{E}_{\rho_r}\big[ Q^\frak{h}( \tau^{p,q})  \big]<\infty$ for any $\frak{h}\in [0,1)$.  We will define a slightly different energy function $\widetilde{Q}^\frak{h}$ below that fits conveniently with the hierarchical symmetry of our model and that can be used to bound $Q^\frak{h}$.  For $x,y\in \mathcal{E}$, define $\frak{g}(x,y)$ as the smallest value $n\in \overline{\mathbb{N}}$  such that $x$ and $y$ do not belong to the same interval $ (\frac{\ell-1}{b^n},\frac{\ell}{b^n})_{\mathcal{E}}$ for some $\ell\in \{1,\ldots, b^n\}$.  For a Borel measure $ \varrho $ on $[0,1]$, we define
\begin{align*}
\widetilde{Q}^\frak{h}( \varrho)\,:=\,\int_{\mathcal{E}\times \mathcal{E}} \big(\frak{g}(x,y) \big)^\frak{h} \varrho(dx)\varrho(dy) \,,
\end{align*}
and we define $\widetilde{Q}^\frak{h}_{n}\big( \varrho\big)$ analogously with $\frak{g}$ replaced by its  cut-off version  $\frak{g}_n:=\min(  \frak{g}, n )$ for $n\in \mathbb{N}$.  For the constant $\mathbf{c}:=\log^\frak{h} b+ \int_0^1\int_0^1 \big|\log \big(\frac{1}{s+t}\big)\big|^{\frak{h}}dsdt     $,  our analysis will be split between showing the inequalities (I) and (II) below:
\begin{align}\label{Q2Q}
\mathbbmss{E}_{\rho_r}\left[ Q^\frak{h}( \tau^{p,q})  \right]\,\underbrace{\leq}_{\textup{(I)}} \,\mathbf{c}\mathbbmss{E}_{\rho_r}\left[ \widetilde{Q}^\frak{h}( \tau^{p,q})  \right]\, \leq \,\mathbf{c}\liminf_{n\rightarrow\infty}\mathbbmss{E}_{\rho_r}\left[ \widetilde{Q}^\frak{h}_{n}\big( \tau^{p,q}_{r,n}\big)  \right] \,\underbrace{<}_{ \textup{(II)} }\,\infty \,.
\end{align}
The second inequality above holds because $\widetilde{Q}^\frak{h}( \tau^{p,q})\leq \liminf_{n\rightarrow\infty}\widetilde{Q}^\frak{h}_{n}\big( \tau^{p,q}_{r,n}\big)  $ $\rho_r$-a.s.\ by a version of  Fatou's lemma since the functions $\frak{g}_n$ converge pointwise up to $\frak{g}$  and the measures $\tau^{p,q}_{r,n}\times \tau^{p,q}_{r,n}$  converge weakly to $\tau^{p,q}\times \tau^{p,q}$ for $\rho_r$-a.e.\ pair $(p,q)$ as a consequence of (i) of Lemma~\ref{LemmaCas}. Note that the set of discontinuity points of $\frak{g}$ is a subset of the union of $\mathcal{V}\times [0,1]$ and $[0,1]\times \mathcal{V}$, which are null sets for $\tau^{p,q}\times \tau^{p,q}$. \vspace{.3cm}

\noindent \textbf{(b) Proof of (I):} Note that the integration domain $[0,1]\times [0,1]$  in~(\ref{Energy}) can be replaced by $\mathcal{E}\times \mathcal{E}$ since $\tau^{p,q}(\mathcal{V})=0$ for all pairs $(p,q)$.  We can express the expectation of $Q^\frak{h}( \tau^{p,q}) $  in  terms of  nested conditional expectations as follows:
\begin{align*}
\mathbbmss{E}_{\rho_r}\left[ Q^\frak{h}( \tau^{p,q})  \right]\,=\,&\,\sum_{n=1}^\infty \mathbbmss{E}_{\rho_r}\bigg[ \int_{n=\frak{g}(x,y)}  \log^\frak{h} \Big(\frac{1}{|x-y|}\Big)       \tau^{p,q}(dx) \tau^{p,q}(dy)\bigg]\\
\,=\,&\,\sum_{n=1}^\infty \mathbbmss{E}_{\rho_r}\Bigg[ \underbrace{\mathbbmss{E}_{\rho_r}\bigg[ \int_{n=\frak{g}(x,y)}  \log^\frak{h} \Big(\frac{1}{|x-y|}\Big)       \tau^{p,q}(dx) \tau^{p,q}(dy)\,\bigg|\, \mathpzc{F}_n\bigg]}_{\mathbf{Q}^\frak{h}_{r,n}(p,q)  }\Bigg]\,, 
\end{align*}
where `$n=\frak{g}(x,y)$' refers to the set $\big\{(x,y)\in \mathcal{E}^2\,\big|\,n=\frak{g}(x,y)  \big\} $ and  $\mathbf{Q}^\frak{h}_{r,n}(p,q) $ denotes the conditional expectation.  Similarly, we can write
\begin{align*}
  \,\mathbbmss{E}_{\rho_r}\left[ \widetilde{Q}^\frak{h}( \tau^{p,q})  \right]\,=\,\sum_{n=1}^\infty \mathbbmss{E}_{\rho_r}\Bigg[\underbrace{\mathbbmss{E}_{\rho_r}\bigg[ \int_{n=\frak{g}(x,y)}  \big(\frak{g}(x,y)\big)^\frak{h}     \tau^{p,q}(dx) \tau^{p,q}(dy)\,\bigg|\, \mathpzc{F}_n\bigg]}_{\mathbf{\widetilde{Q}}^\frak{h}_{r,n}(p,q)  }\Bigg]  \,.
\end{align*}
It suffices to show that $\mathbf{Q}^\frak{h}_{r,n}(p,q) $ is bounded from above by $\mathbf{c}\mathbf{\widetilde{Q}}^\frak{h}_{r,n}(p,q)$. The expression $\mathbf{Q}^\frak{h}_{r,n}(p,q)$ can be written as
\begin{align}
\mathbf{Q}^\frak{h}_{r,n}(p,q)  \,=\,& \int_{n=\frak{g}(x,y)}  \log^\frak{h} \Big(\frac{1}{|x-y|}\Big)       \tau^{p,q}_{r,n}(dx) \tau^{p,q}_{r,n}(dy) \,.\label{ConditionedQ}
\end{align}
To see the above equality, first recall from (iii) of Lemma~\ref{LemmaCas} that  $\tau^{p,q}_{r,n}(A)$ is the conditional $\rho_r$-expectation of $\tau^{p,q}(A)$  with respect to the $\sigma$-algebra $\mathpzc{F}_n$ for any set $A$ that is a union of intervals $(\frac{\ell-1}{b^n}, \frac{\ell}{b^n})_{\mathcal{E}}$. The set $\{(x,y)\in \mathcal{E}^2\,|\,  n=\frak{g}(x,y)\}$ is a union of products $(\frac{\ell_1 -1}{b^n}, \frac{\ell_1}{b^n})_{\mathcal{E}}\times (\frac{\ell_2 -1}{b^n}, \frac{\ell_2}{b^n})_{\mathcal{E}}$ for $1\leq \ell_1,\ell_2\leq b^n$ with $\ell_1\neq \ell_2$; however, the measure $\tau^{p,q}$ acts independently on the intervals  $ (\frac{\ell_1 -1}{b^n}, \frac{\ell_1}{b^n})_{\mathcal{E}}$  and  $ (\frac{\ell_2 -2}{b^n}, \frac{\ell_2}{b^n})_{\mathcal{E}}$ when conditioned on $\mathpzc{F}_n $ by part (iv) of Lemma~\ref{LemmaCas}.  Using the definition of $ \tau^{p,q}_{r,n}$ in~(\ref{TauN}), we can rewrite~(\ref{ConditionedQ}) as follows:
\begin{align}
\mathbf{Q}^\frak{h}_{r,n}(p,q)  \,=\,&\,\sum_{\substack{e_1, e_2 \in S_{n}^{p,q}\\ \frak{g}(e_1,e_2)=n  } } \bigg(\frac{R'(r-n)}{R(r-n)  } \bigg)^2 \underbracket{b^{2n}\int_{e_1\times e_2}      \log^\frak{h} \Big(\frac{1}{|x-y|}\Big)     dx dy}   \nonumber \\
\,\leq \,&\,\mathbf{c}\sum_{\substack{e_1, e_2 \in S_{n}^{p,q}\\ \frak{g}(e_1,e_2)=n  } } \bigg(\frac{R'(r-n)}{R(r-n)  } \bigg)^2 n^\frak{h}   \label{CritInq}\\
 =\, &\,\mathbf{c} \int_{n=\frak{g}(x,y)}  \big(\frak{g}(x,y)\big)^\frak{h}     \tau^{p,q}_{r,n}(dx) \tau^{p,q}_{r,n}(dy)\, =:\, \mathbf{c}\mathbf{\widetilde{Q}}^\frak{h}_{r,n}(p,q) \,. \nonumber 
\end{align}
The inequality~(\ref{CritInq}) follows through bounding the bracketed expression by $\mathbf{c}n^\frak{h}$ using the computation in~(\ref{prelim}) below.  The second equality invokes the definition  of $\tau^{p,q}_{r,n}$ again, and the last equality follows by the same argument as for~(\ref{ConditionedQ}).  Thus (I) follows once the inequality~(\ref{CritInq}) is justified.    

To see~(\ref{CritInq}), recall that the sets $e_1,e_2 \in S_{n}^{p,q}$ have the forms $(\frac{\ell_1-1}{b^{n}},\frac{\ell_1}{b^{n}})_{\mathcal{E}}$ and $(\frac{\ell_2-1}{b^{n}},\frac{\ell_2}{b^{n}})_{\mathcal{E}}$ for some $\ell_1 , \ell_2\in \{1,\ldots, b^n\}$ with $\ell_1 \neq \ell_2$.  Without loss of generality, we can assume $\ell_1 < \ell_2$.  The left side below  is maximized when $\ell_2=\ell_1+1$, i.e., the intervals are adjacent, so we have the inequality
\begin{align}
b^{2n}\int_{e_1\times e_2}  \log^\frak{h} \Big(\frac{1}{|x-y|}\Big) dxdy\,\leq \,&\,b^{2n}\int_{(\frac{\ell_1-1}{b^{n}},\frac{\ell_1}{b^{n}})_{\mathcal{E}}\times (\frac{\ell_1}{b^{n}},\frac{\ell_1+1}{b^{n}})_{\mathcal{E}}}  \log^\frak{h}  \Big(\frac{1}{|x-y|}\Big) dxdy  \nonumber\\
 \, = \,&\,\int_0^1\int_0^1\bigg( n\log b+ \log \Big(\frac{1}{s+t}\Big)\bigg)^\frak{h} dsdt  \,, \label{prelim} 
\end{align}
where the equality makes a change of integration variables to $s=\ell_1- b^{n}x  $ and  $t=b^{n}y-\ell_1$.  Finally, the above is less than $\mathbf{c}n^\frak{h}$.

\vspace{.3cm}

\noindent \textbf{(c) A first step towards proving (II):}  Note that $\widetilde{Q}^\frak{h}_{n}\big( \tau^{p,q}_{r,n}\big)$ can be expressed as
\begin{align*}
\widetilde{Q}^\frak{h}_{n}\big( \tau^{p,q}_{r,n}\big)\,=\,&\,\int_{\mathcal{E}\times \mathcal{E}} \big(\frak{g}_n(x,y) \big)^\frak{h} \tau^{p,q}_{r,n}(dx)\tau^{p,q}_{r,n}(dy) \\ \,=\,&\, \bigg(\frac{ R'(r-n) }{R(r-n)  }\bigg)^2  \sum_{
e_1,e_2\in S^{p,q}_{n} }  \big( \frak{g}_n(e_1,e_2)\big)^\frak{h} \,,
\end{align*}
where $ \frak{g}_n(e_1,e_2):= \frak{g}_n(x,y)$ for representatives $x\in e_1$ and $y\in e_2$.  The conditional expectation of  $\widetilde{Q}_{n+1}^\frak{h}\big( \tau^{p,q}_{r,n+1}\big)$ with respect to $ \mathpzc{F}_n $ has the form
\begin{align}
\mathbbmss{E}_{\rho_r}\Big[  \widetilde{Q}_{n+1}^\frak{h}&\big( \tau^{p,q}_{r,n+1}\big) \,\Big|\,\mathpzc{F}_n  \Big]\nonumber \\  \,=\,&\,\bigg(\frac{ R'(r-n-1) }{R(r-n-1)  }\bigg)^2\sum_{
e_1,e_2\in S^{p,q}_{n} } \mathbbmss{E}_{\rho_r}\Bigg[   \sum_{\substack{
f_1, f_2\in S^{p,q}_{n+1} \nonumber  \\ f_1 \subset e_1, f_2\subset e_2   } }  \big( \frak{g}_{n+1}(f_1,f_2) \big)^\frak{h}  \,\Bigg|\, \mathpzc{F}_n  \Bigg] \nonumber \\
\,=\,& \,\bigg(\frac{ R'(r-n) }{R(r-n)  }\bigg)^2  \sum_{\substack{
e_1 ,e_2\in S^{p,q}_{n}\\ e_1\neq e_2  }}  \big(\frak{g}_n(e_1,e_2) \big)^\frak{h} \nonumber 
\\
& \,+\,\bigg(\frac{ R'(r-n-1) }{R(r-n-1)  }\bigg)^2  \sum_{
e\in S^{p,q}_{n}} \mathbbmss{E}_{\rho_r}\Bigg[   \sum_{ \substack{
f_1, f_2\in S^{p,q}_{n+1} \\ f_1,f_2\subset e  } }  \big(\frak{g}_{n+1}(f_1,f_2) \big)^\frak{h}  \,\Bigg|\, \mathpzc{F}_n \Bigg] \,,\label{Diagonal}
\end{align}
where we have  split the sum over $e_1,e_2\in S^{p,q}_{n}$ into the  cases $e_1\neq e_2$ and $e_1 = e_2$.  We have applied the identity~(\ref{Expect}) twice to rewrite the sum over the $e_1\neq e_2$ terms.
\vspace{.3cm}

\noindent \textbf{(d) A single term in the sum~(\ref{Diagonal}):} Next we will show that terms $e\in S^{p,q}_{n}  $  from the sum~(\ref{Diagonal}) satisfy the large $n$ order equality 
\begin{align}\label{Cond}
\bigg(\frac{ R'(r-n-1) }{R(r-n-1)  }\bigg)^2 & \mathbbmss{E}_{\rho_r}\Bigg[   \sum_{ \substack{
f_1, f_2\in S^{p,q}_{n+1} \\ f_1,f_2\subset e  } }  \big(\frak{g}_{n+1}(f_1,f_2) \big)^\frak{h}  \,\Bigg|\, \mathpzc{F}_n \Bigg]\nonumber \\ &\,=\, \bigg(\frac{R'(r-n) }{ R(r-n) }\bigg)^2\Big(n^\frak{h}\,+\,\mathit{O}\big(n^{\frak{h}-1}\big)\Big)\,.
\end{align}
 Notice that the left side of~(\ref{Cond}) can be rewritten as
\begin{align*}
   \bigg(\frac{ R'(t) }{R(t)  }\bigg)^2  \Bigg(n^\frak{h}\sum_{\ell=1}^{b}\ell(\ell-1)\frac{\frac{1}{b}{ b\choose \ell} \big(R(t) \big)^{\ell}}{R(t+1)  }   \,+\,(n+1)^\frak{h}\sum_{\ell=1}^{b}\ell \frac{\frac{1}{b}{ b\choose \ell} \big(R(t) \big)^{\ell}}{R(t+1)  }   \Bigg)\Bigg|_{t=r-n-1}\,,
 \end{align*}
 where the two  terms above respectively correspond to when $f_1\neq f_2$ and $f_1=f_2$.  The factor $\ell(\ell-1)$ appears in the first term above because if $e$ has $\ell\in \{1,\ldots, b\}$ children then there are $\ell(\ell-1)$ ways to choose $ f_1, f_2\subset e$ such that $f_1\neq f_2$.  We can rearrange terms and use the chain rule to express the above as   
 \begin{align*}
& \frac{R'(t)}{R(t+1)}\Bigg(n^\frak{h}    \frac{d}{dt}\sum_{\ell=1}^{b}\frac{\ell}{b}{ b\choose \ell} \big(R(t) \big)^{\ell-1}   \,+\, (n+1)^\frak{h} \frac{\frac{d}{dt}\sum_{\ell=1}^{b}\frac{1}{b}{ b\choose \ell} \big(R(t) \big)^{\ell}  }{R(t)   }\Bigg) \Bigg|_{t=r-n-1}\,,
\end{align*}
and once more with the chain rule the above can be written as 
\begin{align*}
& \frac{R'(t)}{R(t+1)}\Bigg(n^\frak{h}  \frac{d}{dt} \frac{  \frac{d}{dt} \frac{1}{b}\big[\big(1+R(t)\big)^b-1     \big]  }{ R'(t)  } \,+\, (n+1)^\frak{h} \frac{\frac{d}{dt}\frac{1}{b}\big[\big(1+R(t)\big)^b-1     \big]}{R(t) }\Bigg)\Bigg|_{t=r-n-1}\,.
\end{align*}
From the identity $R(t+1)=\frac{1}{b}\big[\big(1+R(t)\big)^b-1     \big]$, the above is equal to
\begin{align*}
 \Bigg(n^\frak{h}& \frac{ R'(t)}{R(t+1)    }  \frac{d}{dt}\bigg[ \frac{R '(t+1)   }{ R'(t) }  \bigg] \,+\,(n+1)^\frak{h}\frac{ R'(t) }{R(t)  } \frac{ R'(t+1)}{ R(t+1)  }\Bigg) \Bigg|_{t=r-n-1} \\
 &\quad \,=\, n^\frak{h} \frac{R'(t+1) }{ R(t+1) }\Bigg[ \frac{d}{dt}\log\bigg(\frac{ R'(t+1)}{R'(t)  }  \bigg) \,+\, \Big(1+\frac{1}{n}\Big)^\frak{h}\frac{ R'(t) }{R(t)  } \Bigg]\Bigg|_{t=r-n-1} \\
&\quad  \,=\, n^\frak{h} \frac{R'(t+1) }{ R(t+1) }\Bigg[ \frac{d}{dt}\log\Big( \big(1+R(t)\big)^{b-1}  \Big) \,+\, \Big(1+\frac{1}{n}\Big)^\frak{h}\frac{ R'(t) }{R(t)  } \Bigg] \Bigg|_{t=r-n-1}\,,
\end{align*}
where the second equality uses that $ R'(t+1)= \big(1+ R(t)  \big)^{b-1}R'(t)$.  By computing the derivative with the chain rule and factoring out  $\frac{R'(t+1) }{ R(t+1) }$, we get 
\begin{align*}
& n^\frak{h} \bigg(\frac{R'(t+1) }{ R(t+1) }\bigg)^2\Bigg[ (b-1)  \frac{R(t+1) }{R'(t+1)  }\frac{ R'(t)  }{ 1+R(t)   }  \,+\,\Big(1+\frac{1}{n}\Big)^\frak{h}\frac{R(t+1)  }{ R'(t+1)  }\frac{R'(t) }{R(t) } \Bigg]\Bigg|_{t=r-n-1} \\
&\,\, \,=\, n^\frak{h} \bigg(\frac{R'(t+1) }{ R(t+1) }\bigg)^2\Bigg[\underbrace{   (b-1)\frac{ R(t+1)  }{ \big(1+R(t) \big)^b  }  \,+\,\Big(1+\frac{1}{n}\Big)^\frak{h}\frac{ R(t+1) }{R(t) \big(1+R(t) \big)^{b-1}  }}_{ 1+\frac{\frak{h}}{n}+\mathit{o}( \frac{1}{n} )   }\Bigg] \Bigg|_{t=r-n-1}\,.
\end{align*}
The equality above applies the identity  $ R'(t+1)= \big(1+ R(t)  \big)^{b-1}R'(t)$.
The asymptotic $R(t)=-\frac{\kappa^2}{t}+\frac{ \kappa^2\eta \log(-t) }{  t^2} +\mathit{O}\Big(\frac{ \log^2(-t) }{  |t|^3}\Big)$ for $-t\gg 1$ implies that the braced expression is $1+\frac{\frak{h}}{n}+\mathit{o}\big( \frac{1}{n} \big) $ with large $n$.
 \vspace{.3cm}

\noindent \textbf{(e) Returning to~(\ref{Diagonal}).} As a consequence of~(\ref{Cond}), there is a $C>0$ such that for all $n\in \mathbb{N}$
\begin{align}\label{FinExp}
\mathbbmss{E}_{\rho_r}\Big[  \widetilde{Q}^\frak{h}_{n+1}\big( \tau^{p,q}_{r,n+1}\big) \,\Big|\, \mathpzc{F}_n  \Big]\,-\,\widetilde{Q}^\frak{h}_{n}\big( \tau^{p,q}_{r,n}\big) \,\leq \,Cn^{\frak{h}-1} \bigg(\frac{R'(r-n) }{ R(r-n) }\bigg)^2\widetilde{\xi}_{n}(p,q)\,,
\end{align}
where $\widetilde{\xi}_{n}(p,q)$ is the number of elements in $S^{p,q}_{n}$.  The expectation of $\mathbbmss{E}_{\rho_r}\big[  \widetilde{Q}^\frak{h}_{n}\big( \tau^{p,q}_{r,n}\big)  \big]$ can be written in terms of a telescoping sum as 
\begin{align*}
\mathbbmss{E}_{\rho_r}\left[  \widetilde{Q}^\frak{h}_{n}\big( \tau^{p,q}_{r,n}\big)  \right]\,=\,\mathbbmss{E}_{\rho_r}\Big[ \widetilde{Q}^\frak{h}_{1}\big( \tau^{p,q}_{r,1}\big)   \Big]\,+\,  \sum_{k=1}^{n-1}\bigg( \mathbbmss{E}_{\rho_r}\Big[ \widetilde{Q}^\frak{h}_{k+1}\big( \tau^{p,q}_{r,k+1}\big)   \Big]\,-\,\mathbbmss{E}_{\rho_r}\Big[ \widetilde{Q}^\frak{h}_{k}\big( \tau^{p,q}_{r,k}\big)   \Big]\bigg)\,. 
\end{align*}
Inserting nested conditional expectations and applying~(\ref{FinExp})  yields the first two relations below.
\begin{align*}
\mathbbmss{E}_{\rho_r}\Big[  \widetilde{Q}^\frak{h}_{n}\big( \tau^{p,q}_{r,n}\big)  \Big]
\,=\,&\,\bigg(\frac{R'(r)  }{R(r)}\bigg)^2 \,+\,  \sum_{k=1}^{n-1} \mathbbmss{E}_{\rho_r}\Big[ \mathbbmss{E}_{\rho_r}\Big[  \widetilde{Q}^\frak{h}_{k+1}\big( \tau^{p,q}_{r,k+1}\big) \,\Big|\, \mathpzc{F}_k \Big]\,-\,\widetilde{Q}^\frak{h}_{k}\big( \tau^{p,q}_{r,k}\big)   \Big]\\ \,\leq \,&\,\bigg(\frac{R'(r)  }{R(r)}\bigg)^2 \,+\,  \sum_{k=1}^{n-1}  Ck^{\frak{h}-1} \bigg(\frac{R'(r-k) }{ R(r-k) }\bigg)^2\mathbbmss{E}_{\rho_r}\Big[\widetilde{\xi}_{k}(p,q)   \Big]\,\\
  \, = \,&\,\bigg(\frac{R'(r)  }{R(r)}\bigg)^2 \,+\,\frac{R'(r)}{R(r)}  \sum_{k=1}^{n-1}  Ck^{\frak{h}-1} \frac{R'(r-k) }{ R(r-k) }\\
\, \leq  \,&\,\bigg(\frac{R'(r)  }{R(r)}\bigg)^2 \,+\,\frac{R'(r)}{R(r)} \mathbf{C} \sum_{k=1}^{n-1}  k^{\frak{h}-2}  
\end{align*}
The second equality uses that $\mathbbmss{E}_{\rho_r}\big[\widetilde{\xi}_{k}(p,q)\big] = \frac{R'(r)}{R(r)} \frac{R(r-k)}{R'(r-k)} $, which follows from $\mathbf{\widetilde{m}}_{r,k}:=\frac{ R'(r-k)  }{R(r-k)  }\widetilde{\xi}_{k}(p,q) $ being a $\rho_r$-martingale---see  part (ii) of Proposition~\ref{LemMargPop}---with expectation $\frac{R'(r)}{R(r)}$. The last inquality holds for large enough $\mathbf{C}>0$ because $\frac{R'(r-k) }{ R(r-k)}=\frac{1}{k}+\mathit{o}\big(\frac{1}{k}\big)$ for $k\gg 1$ by Lemma~\ref{LemVar} and Remark~\ref{RemarkDerR}.  Since the series $\sum_{k=1}^{\infty}  k^{\frak{h}-2}$ is summable for $\frak{h}\in (0,1)$,  the limit of $\mathbb{E}\big[  \widetilde{Q}^\frak{h}_{n}\big( \tau^{p,q}_{r,n}\big)  \big]$ as $n\rightarrow \infty$ is finite. 
\end{proof}

\subsection{Proof of Lemma~\ref{LemIntSet}}

\begin{proof}
For $p,q\in \Gamma$,  the set of  intersection times $I^{p,q}=\big\{t\in [0,1]\,\big|\,p(t)=q(t)   \big\}$ can be expressed as
$$ I^{p,q}\,=\,\bigcap_{n=1}^{\infty} I^{p,q}_{n} \hspace{.5cm}\text{for}\hspace{.5cm} I^{p,q}_{n}\,=\,[0,1]\,\big\backslash\,\Bigg(\bigcup_{\substack{1\leq \ell \leq b^n \\  [p]_n(\ell)\neq [q]_n(\ell)   }}\Big(\frac{\ell-1}{b^n}, \frac{\ell}{b^n}  \Big)\Bigg)\,.   $$
By~Corollary~\ref{CorIntSet},  the log-Hausdorff exponent of $I^{p,q}$ is  $\geq 1$ for $\rho_r$-a.e.\ $(p,q)$, and  thus we only need to prove that the log-Hausdorff exponent of $I^{p,q}$ is $\leq 1$ for $\rho_r$-a.e.\  $(p,q)$.  It suffices to show that  $H^{\log}_{1}( I^{p,q} )$ is $\rho_r$-a.e.\, finite, where $H^{\log}_{1}=\lim_{\delta\searrow 0} H^{\log}_{1,\delta}$ is the outer measure from Definition~\ref{DefLogHaus}.

Recall that $\mathcal{V}$ is defined as the set of $x\in [0,1]$ of the form $x=\frac{\ell}{b^n}$ for  $\ell,n\in \mathbb{N}_0$ and that $\mathcal{E}:=[0,1]-\mathcal{V}$. Since  $\mathcal{V}$ is countable, we have  $ H^{\log}_{1}\big( \mathcal{V}\big)=0  $, and thus we can focus on showing that $H^{\log}_{1}\big(I^{p,q}\cap \mathcal{E} \big)$ is $\rho_r$-a.e.\ finite.   Fix $\delta>0$, and pick $n\in \mathbb{N}$ large enough so that $ b^{-n}\leq \delta $.  For   $\widetilde{\xi}_n(p,q)$  defined as in Definition~\ref{DefSurvProg},  the set $I^{p,q}\cap \mathcal{E}$ can be covered by $\widetilde{\xi}_n(p,q)$ intervals $\big(\frac{\ell-1}{b^n},\frac{\ell}{b^n}  \big)$ with $\ell\in \{1,\ldots, b^n\}$.  It follows that for all $(p,q)$ we have the bound
$$H^{\log}_{1,\delta}\big(I^{p,q}\cap \mathcal{E} \big) \,\leq \, \frac{\widetilde{\xi}_n(p,q)}{ \log ( b^n) }\,=\,\frac{  \widetilde{\xi}_n(p,q)   }{n \log b  }   \,.   $$ 
However, by part (ii) of Lemma~\ref{LemMargPop}, the sequence $\big(\frac{\kappa^2    }{n   }\widetilde{\xi}_n(p,q) \big)_{n\in \mathbb{N}}$ converges $\rho_r$-a.e.\ to $T(p,q)$.  Thus for $\rho_r$-a.e.\ pair $(p,q)$ we get that
\begin{align}
H^{\log}_{1}\big( I^{p,q}\cap \mathcal{E}\big)\,=\,\lim_{\delta\searrow 0} H^{\log}_{1,\delta}\big( I^{p,q}\cap \mathcal{E}\big)\,\leq \,\liminf_{n\rightarrow \infty}\frac{  \widetilde{\xi}_n(p,q)   }{n \log b  }\, =\,\frac{ T(p,q)}{ \kappa^2 \log b  }\,.
\end{align}
Since $T(p,q)$ is $\rho_r$-a.e.\ finite by part (ii) of Proposition~\ref{PropLemCorrelate}, the random variable $H^{\log}_{1}( I^{p,q}\cap \mathcal{E})$ is $\rho_r$-a.e.\ finite.  Therefore the log-Hausdorff exponent of $I^{p,q}$ is $\leq 1$ for $\rho_r$-a.e.\ pair $(p,q)$. 
\end{proof}

\section{Proof of Theorem~\ref{ThmPathMeasure}  }\label{SecPathIntMxM}

\begin{proof} Part (i)  is a corollary of (iii), which is proved below. \vspace{.3cm}

\noindent (ii) By property (II) of Theorem~\ref{ThmExistence} and part (iv) of Proposition~\ref{PropLemCorrelate},
\begin{align}\label{TComp}
\mathbb{E}\bigg[ \int_{\Gamma\times \Gamma}e^{aT(p,q)} \mathbf{M}_r(dp) \mathbf{M}_r(dq)     \bigg]\,=\,\int_{\Gamma\times\Gamma}e^{aT(p,q)}\upsilon_r(dp,dq)\,=\,1\,+\,R(r+a)\,.
\end{align}
It follows that $\mathbf{M}_{r}\times \mathbf{M}_r$ a.s.\ assigns full measure to the set of  pairs $(p,q)$ s.t.\ $ \displaystyle T(p,q)$ is finite.  \vspace{.3cm}

\noindent (iii)  Let $\mathbf{G}$ denote  the set of $(p,q)\in \Gamma\times \Gamma$ such that the intersection-times set $I^{p,q}$  has log-Hausdorff exponent one, and let $\mathbf{\widehat{G}}$
 denote the set of $(p,q)$ such that $T(p,q)>0$.   The events $ \mathbf{G}$ and $\mathbf{\widehat{G}}$ differ by sets of $\upsilon_r$-measure zero since
\begin{align}\label{UpsilonZero} 
  \upsilon_r\big( \mathbf{G}\Delta\mathbf{\widehat{G}}   \big)\,=\,\big(\mu\times \mu\,+\,R(r)\rho_r     \big)\big( \mathbf{G}\Delta\mathbf{\widehat{G}}   \big)\,=\, R(r) \rho_r \big( \mathbf{G}\Delta\mathbf{\widehat{G}}   \big)\,=\,0\,,  
  \end{align}
where $\mathbf{G}\Delta\mathbf{\widehat{G}}$ denotes the symmetric difference $\big(\mathbf{G}\backslash\mathbf{\widehat{G}}\big)\cup \big(\mathbf{\widehat{G}}\backslash\mathbf{G}\big) $.  The first equality above holds by part (ii) of Lemma~\ref{LemCorrelate}, the second equality is a consequence of Corollary~\ref{CorrMuMu},  and the third equality holds  because the  measure $\rho_r$ assigns full weight to both $ \mathbf{G}$ and $\mathbf{\widehat{G}}$ by  Lemma~\ref{LemIntSet} and Proposition~\ref{PropLemCorrelate}, respectively. Applying  (II) of Theorem~\ref{ThmExistence} with $g=1_{ \mathbf{G}\Delta\mathbf{\widehat{G}} }$ and~(\ref{UpsilonZero}) gives us the two equalities below.
 \begin{align}\label{EM}
 \mathbb{E}\big[ \mathbf{M}_r\times \mathbf{M}_r\big( \mathbf{G}\Delta\mathbf{\widehat{G}} \big)  \big]\,=\,\upsilon_r\big(\mathbf{G}\Delta\mathbf{\widehat{G}}\big) \,=\,0 
\end{align}
Thus $\mathbf{M}_r\times \mathbf{M}_r$ a.s.\ assigns the set $\mathbf{G}\Delta\mathbf{\widehat{G}}$ measure zero. Define $\mathbf{\widehat{S}}:= \mathbf{\widehat{G}}^c$, i.e., as the set of pairs such that $T(p,q)=0$, and recall that $\mathbf{S}_{\emptyset}$ denotes the set of pairs $(p,q)$ such that the intersection-times set $I^{p,q}$ is finite.  Then $\mathbf{S}_{\emptyset}\subset \mathbf{\widehat{S}}$, and
 \begin{align*} 
 \mathbb{E}\big[ \mathbf{M}_r\times \mathbf{M}_r\big( \mathbf{\widehat{S}}\backslash\mathbf{S}_{\emptyset} \big)  \big]\,=\,\upsilon_r\big(\mathbf{\widehat{S}} \backslash\mathbf{S}_{\emptyset}\big)\,=\,&\,\big(\mu\times \mu+R(r)\rho_r \big)\big(\mathbf{\widehat{S}} \backslash\mathbf{S}_{\emptyset}\big) \nonumber \\ \,=\,&\,\mu\times \mu\big(\mathbf{\widehat{S}} \backslash\mathbf{S}_{\emptyset}\big)\,=\,0\,, 
\end{align*}
where the third equality holds since $\mathbf{\widehat{S}}$ is a $\rho_r$-null set by part (ii) of Proposition~\ref{PropLemCorrelate}, and the fourth equality uses that $ \mathbf{S}_{\emptyset}$ is a full measure set for $\mu\times \mu$ by Corollary~\ref{CorrMuMu}.  It follows that $\mathbf{\widehat{S}} \backslash\mathbf{S}_{\emptyset}$ is a.s.\ a null set for $\mathbf{M}_r\times \mathbf{M}_r$.   Since $\Gamma\times \Gamma= \mathbf{\widehat{S}}\cup \mathbf{\widehat{G}} $,  the above results imply that the random measure $\mathbf{M}_r\times \mathbf{M}_r$  a.s.\ assigns full weight to the set $ \mathbf{S}_{\emptyset}\cup \mathbf{G}$, which was the desired conclusion. \vspace{.3cm}

\noindent (iv):   Given $p\in \Gamma$, recall that  $\mathbf{\widehat{G}}_{p}$  denotes the set of $q\in \Gamma$ such that $T(p,q)>0$, which can be expressed as the  cross section $\mathbf{\widehat{G}}_{p}=\big\{q\in \Gamma \,|\,(p,q)\in   \mathbf{\widehat{G}}\}$ of the set $\mathbf{\widehat{G}}$ defined in  part (iii). Define the random set $\mathbf{S}_{\mathbf{M}_r}:=\big\{p\in \Gamma\,\big|\, \mathbf{M}_r( \mathbf{\widehat{G}}_{p})=0  \big\}$, and let $1_{\mathbf{S}_{\mathbf{M}_r}}  $ denote its indicator function.   By definition, we must show that the random measure $\mathbf{M}_r$ a.s.\ satisfies  $\mathbf{M}_r(\mathbf{S}_{\mathbf{M}_r})=0$.  We can write $\mathbf{M}_r=\mathbf{A}_r+\mathbf{B}_r$, where the random measures $\mathbf{A}_r$ and $\mathbf{B}_r$  are defined to  have Radon-Nikodym derivatives with respect to $\mathbf{M}_r$ given by
$$ \frac{d\mathbf{A}_r}{d\mathbf{M}_r  }\,=\,  1_{\mathbf{S}_{\mathbf{M}_r}}\hspace{1cm}\text{and} \hspace{1cm} \frac{d\mathbf{B}_r}{d\mathbf{M}_r  }\,=\, 1- 1_{\mathbf{S}_{\mathbf{M}_r}} \,.   $$
The following gives us a lower bound for the second moment of the total mass of  $ \mathbf{B}_r$:
\begin{align}\label{LowerRBound}
R(r)\,=\,\upsilon_r\big(\mathbf{\widehat{G}}  \big)  \,=\,&\,\mathbb{E}\big[\mathbf{M}_r\times \mathbf{M}_r\big( \mathbf{\widehat{G}}  \big)   \big] \nonumber \\
\,=\,&\,\mathbb{E}\big[\mathbf{B}_r\times \mathbf{B}_r\big( \mathbf{\widehat{G}}  \big)   \big]\,+\,\mathbb{E}\big[\mathbf{A}_r\times \mathbf{M}_r\big( \mathbf{\widehat{G}}  \big)   \big]\,+\,\mathbb{E}\big[\mathbf{B}_r\times \mathbf{A}_r\big( \mathbf{\widehat{G}}  \big)   \big] \nonumber \\
\,=\,&\,\mathbb{E}\big[\mathbf{B}_r\times \mathbf{B}_r\big( \mathbf{\widehat{G}}  \big)   \big]\nonumber \\ \,\leq \,&\,\mathbb{E}\big[\mathbf{B}_r\times \mathbf{B}_r\big(\Gamma\times\Gamma\big)  \big] \,=\,\mathbb{E}\big[ |\mathbf{B}_r(\Gamma)|^2  \big]\,.
\end{align}
The first equality holds because $\upsilon_r=\mu\times \mu+R(r)\rho_r$, the probability measure $\rho_r$ assigns full weight to $\mathbf{\widehat{G}}$, and  $\mu\times \mu(\mathbf{\widehat{G}} )=0$. The second equality is by (II) of Theorem~\ref{ThmExistence}, and the third simply uses that $\mathbf{M}_{r}=\mathbf{A}_r+\mathbf{B}_r$.  The fourth equality above follows closely from the  definition of $\mathbf{A}_r$ because
\begin{align*}
\mathbf{A}_r\times \mathbf{M}_r\big( \mathbf{\widehat{G}}  \big)\,=\,&\,\int_{\Gamma} \mathbf{M}_r(\mathbf{\widehat{G}}_{p}  )     \mathbf{A}_r(dp)\\
 \,=\,&\,\int_{\Gamma}1_{\mathbf{S}_{\mathbf{M}_r}}(p)\underbrace{\mathbf{M}_r(\mathbf{\widehat{G}}_{p}   )}_{=\,0 \text{ for }  p\in\mathbf{S}_{\mathbf{M}_r}  }     \mathbf{M}_r(dp)\,=\,0 \,, 
\end{align*}
where the first equality holds by Fubini's theorem since the sets $\mathbf{\widehat{G}}_{p} $ are the cross sections of $\mathbf{\widehat{G}}$.  The same reasoning yields that $\mathbf{M}_r\times \mathbf{A}_r\big( \mathbf{\widehat{G}}  \big)=0$, and thus $\mathbf{B}_r\times \mathbf{A}_r\big( \mathbf{\widehat{G}}  \big)=0$. 

Since $\mathbb{E}[\mathbf{M}_r]=\mu$ and $ \mu$ is a probability measure, 
 the numbers $\alpha_r:=\mathbb{E}[ \mathbf{A}_r(\Gamma) ]$ and $\beta_r:=\mathbb{E}[ \mathbf{B}_r(\Gamma) ]$ sum to $1$. The distributional recursive relation in (IV) of Theorem~\ref{ThmExistence} implies that $\alpha_r$ satisfies $\alpha_{r+1}=\alpha_r^b$ for all $r\in \R$ because two paths would need to have trivial intersections in all $b$ components of the concatenation decomposition to avoid having nontrivial intersections.    Suppose to reach a contradiction that $\alpha_r >0$ for some $r\in \R$.  Then $\alpha_{r-N}=\alpha_r^{b^{-N}}$ converges to 1 exponentially quickly as $N\rightarrow \infty$.  Moreover, the third moment of $\mathbf{M}_{r-N}(\Gamma)$ has the lower bound
\begin{align}
\mathbb{E}\Big[ \big|\mathbf{M}_{r-N}(\Gamma)\big|^3 \Big]\,\geq \,\mathbb{E}\Big[ \big|\mathbf{B}_{r-N}(\Gamma) \big|^3 \Big]\,\geq\,\frac{ \mathbb{E}\Big[ \big|\mathbf{B}_{r-N}(\Gamma)\big|^2 \Big]^2   }{  \mathbb{E}\big[ \mathbf{B}_{r-N}(\Gamma) \big]   }\,\geq \,\frac{   \big(R(r-N)\big)^2    }{  \beta_{r-N}   }\,,
\end{align}
where the last inequality uses that  $\beta_{r-N}= \mathbb{E}\big[ \mathbf{B}_{r-N}(\Gamma) \big]$ and  the lower bound~(\ref{LowerRBound}).
Since $R(r-N) \approx \frac{\kappa^2}{N} $ as $N\rightarrow \infty$ by Lemma~\ref{LemVar} and $ \beta_{r-N}=1-\alpha_{r-N}$ vanishes exponentially quickly by the observation above,  the third moment of   $\mathbf{M}_{r-N}(\Gamma)$ must diverge to $\infty$ as $N\rightarrow \infty$, which brings us to a contradiction with (III) of Theorem~\ref{ThmExistence}.  Therefore $\alpha_r=\mathbb{E}[ \mathbf{A}_{r}(\Gamma)  ]=\mathbb{E}\big[ \mathbf{M}_r\big( \mathbf{S}_{ \mathbf{M}_r }  \big) \big]$ is zero for all $r\in \R$, and so the set $\mathcal{S}_{\mathbf{M}_r}\subset \Gamma$ must a.s.\   have $\mathbf{M}_r$-measure zero.

The same argument applies with   $\mathbf{G}_{p}$ in place of $\mathbf{\widehat{G}}_{p}$.
\end{proof}

\section{Proofs of results from Sections~\ref{SecLocalResults} \& \ref{SectionHilbert}}\label{SectionLast}

Our main goals in this section are to prove Theorems~\ref{ThmOperator} \&~\ref{ThmVartheta}, which we do in Sections~\ref{SubSecLast4} \& \ref{SubSecLast6}, respectively.   In the next subsection, we introduce a family of probability measures $\big\{\Theta^{\updownarrow x}_{\eta}\big\}_{x\in D}$  on the path space $\Gamma$ that, in intuitive terms, result from conditioning $\eta  \in \mathcal{M}_{\Gamma}$ to the event that the path passes through a point $x\in D$.  The measures $\Theta^{ \updownarrow x}_{\eta}$ appear in the following helpful equality involving $\eta\in \mathcal{M}_{\Gamma}^{\diamond} $ and $\vartheta_{\eta}\in \mathcal{M}_{D}$ between two measures on the space $D\times \Gamma\times \Gamma$:
\begin{align}\label{Ident}
\gamma^{p,q}(dx)\eta(dp)\eta(dq)\,=\,\Theta^{ \updownarrow x}_{\eta}(dp)\Theta^{ \updownarrow x}_{\eta}(dq)\vartheta_{\eta}(dx) \,,\hspace{.75cm}x\in D\,,\, p,q\in \Gamma\,,
\end{align}
which we prove in Section~\ref{SubSecLast3}. 

We can use~(\ref{Ident}) to derive an alternative expression for the distributions $y^p_{\eta}(x):=\frac{ \int_{\Gamma}\gamma^{p,q}(dx)\eta(dq)  }{\vartheta(dx) }  $  through which the linear operator $Y_{\eta}$ in Theorem~\ref{ThmOperator} is formally given by $\left(Y_{\eta}f\right)(p)=\big\langle y^p_{\eta} , f\big\rangle_{L^2(D,\vartheta_{\eta})}$.
Since $\Theta^{ \updownarrow x}_{\eta}$ is a probability measure, integrating out $q\in \Gamma$ in~(\ref{Ident}) and rearranging yields the equation $y^p_{\eta}(x)=\frac{ \Theta^{ \updownarrow x}_{\eta}(dp) }{\eta(dp)}  $.  By dividing both sides of~(\ref{Ident}) by $\eta(dp)\eta(dq)$, we get the identity $\gamma^{p,q}(dx)=y^p_{\eta}(x)y^q_{\eta}(x)\vartheta_{\eta}(dx)   $. Since the measure $\gamma^{p,q}$ has total mass $T(p, q)$, integrating out $x$ yields
$$ T(p,q)\,=\,\big\langle  y^p_{\eta}, y^q_{\eta} \big\rangle_{L^2(D,\vartheta_{\eta})  }\,.  $$ 
The proof  that  $\mathbf{T}_{\eta}=Y_{\eta}Y_{\eta}^*$ in Section~\ref{SubSecLast4} proceeds through rigorous versions of these observations.

\subsection{Conditioning paths  to pass through a point  }\label{SubSecLast1}

 For $\eta\in \mathcal{M}_{\Gamma}$, recall from Definition~\ref{DefVartheta} that the measure $\vartheta_{\eta}$ weights a set $S\subset D$ through the `average' intersection time that two paths $p,q\in \Gamma$ spend intersecting within $S$ when  sampled independently using the measure $\eta$. Since the behavior of  $\vartheta_{\eta}(dx)$ formally  involves the set of paths passing through a point $x\in D$, we introduce the following specialized notations:
\begin{align*}
&\Gamma^{\updownarrow x}  \hspace{1.23cm}  \text{Set of paths $p\in \Gamma$ with $x\in\textup{Range}(p)$, i.e., passing through $x$}  & \\
&\Theta^{ \updownarrow x}_{\eta}     \hspace{1.15cm}  \text{Probability measure on $\Gamma$ induced by conditioning $\eta$ on the event $\Gamma^{\updownarrow x}$}& \\
&E^{\updownarrow x}   \hspace{1.17cm}  \text{Set of points in $E$ that share a path with $x\in D$, i.e., the \textit{path horizon} of $x$}& 
\end{align*}
More precisely, $E^{\updownarrow x} $ denotes the set of $y\in E$ such that there exists a $p\in \Gamma$ with $x,y\in \textup{Range}(p)$.  
The set $\Gamma^{\updownarrow x}  $  will typically have measure zero under  $\eta\in \mathcal{M}_{\Gamma}$, and hence a precise definition of  $\Theta^{ \updownarrow x}_{\eta} $  requires special consideration.  For $\mathbf{e}\in E_n$, let $\Gamma^{\updownarrow \mathbf{e}}$ denote the set of $p\in \Gamma$ that pass through the cylinder set $ \mathbf{\overline{e}}$, i.e., such that $\textup{Range}(p)\cap \mathbf{\overline{e}}\neq \emptyset $.  We begin by  defining   coarse-grained versions of $\Theta^{ \updownarrow x}_{\eta} $, which we  take a  weak limit of in Definition~\ref{DefThetaII}.
\begin{definition}\label{DefThetaPre}   For $\mathbf{e}\in E_n$ and $\eta\in\mathcal{M}_{\Gamma}$ with $\eta\big(\Gamma^{\updownarrow \mathbf{e}}  \big)\neq 0$, we define $\Theta^{ \updownarrow \mathbf{e}}_{\eta}$ as the Borel probability measure on $\Gamma$ given by conditioning $\eta$ to the event $\Gamma^{\updownarrow \mathbf{e}}$, in other terms, as $\Theta^{ \updownarrow \mathbf{e}}_{\eta}(A):=\frac{1}{\eta(\Gamma^{\updownarrow \mathbf{e}})}\eta\big(A\cap  \Gamma^{\updownarrow \mathbf{e}}\big)  $ for $A\in \mathcal{B}_{\Gamma}$. Moreover, we define $\Theta^{ \updownarrow \mathbf{e}}_{\eta} :=\Theta^{ \updownarrow \mathbf{e}}_{\mu}$ in the degenerate case $\eta\big(\Gamma^{\updownarrow \mathbf{e}}  \big)= 0$.
\end{definition}

\begin{definition}\label{DefThetaII}    For $x\in E$ and $\eta\in\mathcal{M}_{\Gamma}$,  we define $\Theta^{ \updownarrow x}_{\eta} $ as the weak limit of the sequence of measures $\big(  \Theta^{ \updownarrow [x]_n}_{\eta}\big)_{n\in \mathbb{N}}$ provided that the limit exists.  If $x\in V$,  we define $\Theta^{ \updownarrow x}_{\eta}:=\mu$.\footnote{For our puposes, the convention for defining $\Theta^{ \updownarrow x}_{\eta}$ in the case $x\in V$ is unimportant.}
\end{definition}
\begin{remark} For $x\in E$,  the measure $\Theta^{ \updownarrow x}_{\eta} $ is supported on $\Gamma^{\updownarrow x}$ when it exists  since  $\Theta^{ \updownarrow [x]_n}_{\eta}  \stackrel{\textup{w}}{\rightarrow } \Theta^{ \updownarrow x}_{\eta}$  as $n\rightarrow \infty$ and the measures $\Theta^{ \updownarrow [x]_n}_{\eta} $ are supported on the sets $\Gamma^{\updownarrow [x]_n}$, which converge down to $\Gamma^{\updownarrow x}$.
\end{remark}

Lemma~\ref{LemmaTheta} below states a decomposition for $\Theta^{ \updownarrow x}_{\eta} $ in the case that $\eta\in \mathcal{M}^{\diamond}_{\Gamma}$.  As a preliminary, the next four remarks build up some additional notation and observations. 

\begin{remark} The path horizon  of $x\in E$
can be decomposed into a countable union of disjoint cylinder sets as $E^{\updownarrow x}\,=\,\bigcup_{k=1}^\infty  \bigcup_{e\in E_k^{\updownarrow x}  } \overline{e}$, where  $E_k^{\updownarrow x}$ is a  $(b-1)$-element subset of $E_k$ satisfying: \vspace{-.2cm}
\begin{itemize}
\item   $x\notin \overline{e}$
\item  $x$ \& $ e$ have the same generation-$(k-1)$ coarse-graining, i.e., $\{x\}, \overline{e}\subset\mathbf{\overline{e}}$ for some $\mathbf{e}\in E_{k-1}$.
\item $[x]_k\updownarrow e$, in other terms, there is a path $\mathbf{p}\in \Gamma_k$ such that $  [x]_k, e \in\textup{Range}(\mathbf{p}) $.
\end{itemize}
The set  $E^{\updownarrow x}$ has $\nu$-measure $\sum_{k=1}^{\infty}\frac{b-1}{b^{2k}}=\frac{1}{b+1}  $. 
\end{remark}

\begin{remark}\label{RemarkDecomp} Given $x\in E$, each path $p\in \Gamma^{\updownarrow x}$  can be uniquely decomposed into a sequence of $ p_e \in \Gamma$ indexed by the set $  \cup_{k=1}^\infty E_k^{\updownarrow x} $, where $   p_e $ is a dilation of the part of the path $p$ passing through the embedded subcopy of $D$ corresponding to $e\in E_k^{\updownarrow x} $. In other terms,  there is a canonical bijection 
\begin{align}\label{Timbuk}
\Gamma^{\updownarrow x} \,\,\,\longrightarrow \,\,\,\bigtimes_{k=1}^\infty \bigtimes_{ e\in E_k^{\updownarrow x} } \Gamma \hspace{.4cm}  \text{defined by sending} \hspace{.4cm} p\,\,\,\mapsto\,\,\, \left\{   p_e \right\}_{ e\in \cup_{k=1}^\infty E_k^{\updownarrow x}  }   \,.
   \end{align}
\end{remark}

\begin{remark}\label{RemarkDecompRe} It will be helpful to have a coarse-grained analog of~(\ref{Timbuk}). Fix $n\in \mathbb{N}$ and $\mathbf{e}\in E_n$.   When  $1\leq k\leq n$, the set $E_k^{\updownarrow x}$ is independent  of $x\in \mathbf{\overline{e}} $, and thus we can define $E_k^{\updownarrow \mathbf{e}}:=E_k^{\updownarrow x}$ for any representative $x\in \mathbf{\overline{e}} $. There is a canonical bijection 
\begin{align}\label{TimbukII}
\Gamma^{\updownarrow \mathbf{e}}\,\,\,\longrightarrow \,\,\,\Bigg(\bigtimes_{k=1}^n \bigtimes_{ e\in E_k^{\updownarrow \mathbf{e}} } \Gamma\Bigg)\times \Gamma  \hspace{.4cm} \text{defined by sending}\hspace{.4cm} p \,\,\, \mapsto \,\,\, \Big(\left\{   p_e \right\}_{ e\in \cup_{k=1}^n E_k^{\updownarrow \mathbf{e}}  } ,\, p_{\mathbf{e}}\Big) \,.
   \end{align}
\end{remark}   
 \begin{remark}\label{RemarkDecompReII} Suppose that $\eta\in \mathcal{M}_{\Gamma}^{\diamond}$ has  component measures    $\{\eta^{e}\}_{ e\in \cup_{k=0}^{\infty}E_k }$ that are  nonzero.  Under the identification of $\Gamma^{\updownarrow \mathbf{e}} $ with the product space in~(\ref{TimbukII}), the restriction of $\eta$ to $\Gamma^{\updownarrow \mathbf{e}}$ has the  form
$  \frac{1}{b^n}\big(\prod_{k=1}^n \prod_{e\in E_k^{\updownarrow \mathbf{e}} }\eta^{e}\big)\times \eta^{\mathbf{e}}  $, which follows from an $n$-fold application of the decomposition property $\eta^{f}  =\frac{1}{b} \bplus_{i=1}^b \prod_{j=1 }^b \eta^{f\times (i,j)}$. Thus the restriction of the measure $\Theta^{ \updownarrow \mathbf{e}}_{\eta}$ to $\Gamma^{\updownarrow \mathbf{e}}$ can be decomposed as
\begin{align}\label{ThetaToN}
 \Theta^{ \updownarrow \mathbf{e}}_{\eta}\big|_{\Gamma^{\updownarrow \mathbf{e}}}=\Bigg(\prod_{k=1}^n \prod_{e\in E_k^{\updownarrow \mathbf{e}} }\frac{1}{ \eta^{e}(\Gamma)  }\eta^{e}\Bigg)\times \bigg(\frac{1}{ \eta^{\mathbf{e}}(\Gamma)  }\eta^{\mathbf{e}}\bigg) \hspace{.25cm} \text{by identifying}\hspace{.25cm} \Gamma^{\updownarrow \mathbf{e}}\equiv\Bigg(\bigtimes_{k=1}^n \bigtimes_{ e\in E_k^{\updownarrow \mathbf{e}} } \Gamma\Bigg)\times \Gamma \,.
\end{align}    
\end{remark}

\begin{lemma}\label{LemmaTheta} If $\eta\in \mathcal{M}_{\Gamma}^{\diamond} $, then $\Theta^{ \updownarrow x}_{\eta}$ exists in the sense of Definition~\ref{DefThetaII} for every $x\in E$.
If the  component measures, $\{\eta^{e}\}_{ e\in \cup_{k=0}^{\infty}E_k }$, of $\eta$ are nonzero, then the restriction of $\Theta^{ \updownarrow x}_{\eta}$  to $\Gamma^{\updownarrow x}$  has the  decomposition 
\begin{align}\label{ThetaThingPre}
   \Theta^{ \updownarrow x}_{\eta}\big|_{\Gamma^{\updownarrow x}}\,=\ \prod_{k=1}^\infty \prod_{e\in E_k^{\updownarrow x} } \frac{1}{ \eta^{e}(\Gamma)  }\eta^{e} \hspace{.5cm}\text{under the identification}\hspace{.5cm}\Gamma^{\updownarrow x}\,\equiv\,\bigtimes_{k=1}^\infty \bigtimes_{ e\in E_k^{\updownarrow x}} \Gamma \,.  
   \end{align}
In the  case that $\eta^{e}=0$ for some $e\in \bigcup_{k=1}^{\infty} E_k^{\updownarrow x}$, the above holds with  $\eta^{e}$  replaced by $\mu$ for all $e$. 
\end{lemma}
\begin{proof} Let $\widetilde{\Theta}^{ \updownarrow x}_{\eta}$ denote the product measure  on $\widetilde{\Gamma}^{\updownarrow x}:=\bigtimes_{k=1}^\infty \bigtimes_{ e\in E_k^{\updownarrow x}} \Gamma $ in~(\ref{ThetaThingPre}), which exists by   Kolmogorov's extension theorem since the measurable space $\big(\Gamma,\mathcal{B}_{\Gamma}\big)$ is standard Borel.  Let $\mathcal{U} $ be the canonical injection from $ \widetilde{\Gamma}^{\updownarrow x}$ to  $\Gamma  $ whose range is equal  to $\Gamma^{\updownarrow x}$.  We need to show that the sequence of measures $\big(  \Theta^{ \updownarrow [x]_n}_{\eta}\big)_{n\in \mathbb{N}}$ converges weakly to $\widetilde{\Theta}^{ \updownarrow x}_{\eta}\circ  \mathcal{U}^{-1}$, i.e., the pushforward of $\widetilde{\Theta}^{ \updownarrow x}_{\eta}$ under $\mathcal{U}$.  By the same argument as in the proof of Theorem~\ref{ThmUniversality} (using Lemma~\ref{LemDense} and Corollary~\ref{CorollarySemiAlg}), it suffices to show that $\Theta^{ \updownarrow [x]_n}_{\eta}(\mathbf{p})$ converges to  $\big(\widetilde{\Theta}^{ \updownarrow x}_{\eta}\circ  \mathcal{U}^{-1}\big)(\mathbf{p})$ with large $n$ for each simple cylinder set $\mathbf{p}\in \mathcal{S}_{\Gamma}$. For $n\in \mathbb{N}$ and $x\in E$, let $\langle x \rangle_n\in E$  defined in Remark~\ref{RemarkCylinderE}.  If $\mathbf{e}:=[x]_n $, then the restriction of $\widetilde{\Theta}^{ \updownarrow x}_{\eta}\circ  \mathcal{U}^{ -1}$ to $\Gamma^{\updownarrow \mathbf{e}}$ can be written in the form
\begin{align}\label{TimbukIII}
 \widetilde{\Theta}^{ \updownarrow x}_{\eta}\circ  \mathcal{U}^{ -1}\big|_{\Gamma^{\updownarrow \mathbf{e}}}\,=\, \Bigg(\prod_{k=1}^n \prod_{e\in E_k^{\updownarrow \mathbf{e}} } \frac{1}{ \eta^{e}(\Gamma)  }\eta^{e}\Bigg)\times  \Big(\widetilde{\Theta}^{ \updownarrow \langle x\rangle_n}_{\eta^{\mathbf{e}}}\circ  \mathcal{U}^{ -1}\Big) 
\end{align}
under the identification of $\Gamma^{\updownarrow \mathbf{e}} $ with $ \big(\bigtimes_{k=1}^n \bigtimes_{ e\in E_k^{\updownarrow \mathbf{e}}} \Gamma\big)\times \Gamma $. Since $\frac{1}{ \eta^{\mathbf{e}}(\Gamma)  }\eta^{\mathbf{e}}$ and $\widetilde{\Theta}^{ \updownarrow \langle x\rangle_n}_{\eta^{\mathbf{e}}}\circ  \mathcal{U}^{ -1}$ are probability measures, the decompositions~(\ref{ThetaToN}) and~(\ref{TimbukIII}) imply that $\Theta^{ \updownarrow [x]_n}_{\eta}(\mathbf{p}) =\big(\widetilde{\Theta}^{ \updownarrow x}_{\eta}\circ  \mathcal{U}^{ -1}\big)(\mathbf{p})   $ for all $n\geq N$ when $\mathbf{p}\in \mathcal{S}_{\Gamma}^{(N)}$.  Hence the convergence of $\Theta^{ \updownarrow [x]_n}_{\eta}(\mathbf{p})$ to $\big(\widetilde{\Theta}^{ \updownarrow x}_{\eta}\circ  \mathcal{U}^{ -1}\big)(\mathbf{p})   $
holds trivially.  Therefore $\Theta^{ \updownarrow x}_{\eta}$ exists, and the restriction of  $\Theta^{ \updownarrow x}_{\eta}$ to $\Gamma^{\updownarrow x}$ has the decomposition~(\ref{ThetaThingPre}).
\end{proof}

The following  corollary of Lemma~\ref{LemmaTheta} states  the hierarchical symmetry~(\ref{TimbukIII}) in terms of $\Theta^{ \updownarrow x}_{\eta}$. 
\begin{corollary}\label{RemarkThetaHS}  Suppose that $\eta\in \mathcal{M}_{\Gamma}^{\diamond}$ has a nonzero 
family of component measures $\{\eta^{e}\}_{ e\in \cup_{k=0}^{\infty}E_k }$. If $\mathbf{e}\in E_n$ and $x\in \mathbf{\overline{e}}$, then the restriction of   $\Theta^{ \updownarrow x}_{\eta}$ to the set $\Gamma^{\updownarrow \mathbf{e}}$ can be decomposed as
\begin{align}\label{DisplayThetaHS}
   \Theta^{ \updownarrow x}_{\eta}\big|_{\Gamma^{\updownarrow \mathbf{e}}}\,= \Bigg(\prod_{k=1}^n \prod_{e\in E_k^{\updownarrow \mathbf{e}} } \frac{1}{ \eta^{e}(\Gamma)  }\eta^{e}\Bigg)\times  \Theta^{ \updownarrow \langle x\rangle_n}_{\eta^{\mathbf{e}}} \hspace{.4cm}\text{by identifying}\hspace{.4cm}\Gamma^{\updownarrow \mathbf{e}} \,\equiv \Bigg(\bigtimes_{k=1}^n \bigtimes_{ e\in E_k^{\updownarrow \mathbf{e}}} \Gamma\Bigg)\times \Gamma\,.
   \end{align}
\end{corollary}

Next we  address the continuity of 
$ \Theta^{\updownarrow x}_{\eta}$ as a function of $x\in E$ and $\eta\in \mathcal{M}_{\Gamma}^{\diamond}$.  We equip the space of continuous functions  $ C(E,\mathcal{M}_{\Gamma})$ with the uniform weak topology, i.e., if $\{F^{n}\}_{n\in \overline{\mathbb{N}}}$ is a family of elements in $ C(E,\mathcal{M}_{\Gamma})$, then $F^n\rightarrow F^{\infty}$ means that  for all $g\in C(\Gamma)$ 
$$\sup_{x\in E}\bigg|\int_{\Gamma}g(p)  F^n(x,dp)\,-\,\int_{\Gamma}g(p) F^{\infty}(x,dp)\bigg|\quad \stackrel{n\rightarrow \infty  }{\longrightarrow} \quad  0\,.$$

\begin{proposition}\label{PropThetaMeasure} Fix  $\eta \in \mathcal{M}_{\Gamma}^{\diamond}$ and $n\in \mathbb{N}$. The function $\Theta_{\eta}:E\rightarrow \mathcal{M}_{\Gamma}$ defined by  $x\mapsto \Theta^{ \updownarrow x}_{\eta}$ is continuous. Moreover, the function $\Theta:\mathcal{M}_{\Gamma}^{\diamond}\rightarrow  C(E,\mathcal{M}_{\Gamma}) $ defined by $\eta\mapsto \Theta_{\eta}$ is continuous at points  $\eta\in  \mathcal{M}_{\Gamma}^{\diamond}$ having nonzero component measures $\{\eta^{e}\}_{ e\in \cup_{k=0}^{\infty}E_k }$.
\end{proposition}

\begin{remark}
In later proofs and discussions, for any $\eta\in \mathcal{M}_{\Gamma}^{\diamond} $ we will assume that the component measures in the family $\{\eta^{e}\}_{ e\in \cup_{k=0}^{\infty}E_k }$ are all nonzero to avoid degenerate special cases.  Note that $\eta=\mathbf{M}_r$ a.s.\ satisfies this property  by (i) of Proposition~\ref{CorProp4} and (iii) of Proposition~\ref{PropProperties}.
\end{remark}

\begin{proof}
 For $\eta \in \mathcal{M}_{\Gamma}^{\diamond}$, let the function $F_{\eta}:E\rightarrow  \mathcal{M}_{\Gamma}$ be   defined by   $F_{\eta}(x):= \Theta_{\eta}^{\updownarrow x}  $.  We use the notations $F_{\eta}(x, A):=\Theta_{\eta}^{\updownarrow x}(A)$ and $F_{\eta}(x, g):=\int_{\Gamma}g(p)\Theta_{\eta}^{\updownarrow x}(dp)$ for $x\in E$, $A\in \mathcal{B}_{\Gamma}$, and $g\in C(\Gamma)$. Moreover, define the function $F_{\eta}^{n}:E\rightarrow  \mathcal{M}_{\Gamma}$ by  $F_{\eta}^{n}(x):= \Theta_{\eta}^{\updownarrow  [ x ]_n}  $.   We observe the following:
\begin{itemize}
\item $F_{\eta}^{n}$ is continuous because it is piecewise constant over cylinder sets in the collection $\{\mathbf{\overline{e}}\,|\,\mathbf{e}\in E_n\}  $, which are both open and closed in the topology on $E$ inherited from $D$.

\item $F_{\eta}^{n}(x,A)=F_{\eta}(x,A)$ for any $x\in E$ and  $A\in \mathcal{A}_{\Gamma}^{(n)}$ due to  the agreement of the decompositions~(\ref{ThetaToN})  and~(\ref{DisplayThetaHS}) up to generation $n$.

\end{itemize}
Recall that $\mathbf{C}(\Gamma)$ denotes the collection of simple functions on $\Gamma$ that are measurable with respect to the algebra $\mathcal{A}_{\Gamma}:=\bigcup_{n=1}^{\infty}\mathcal{A}_{\Gamma}^{(n)}$.  The second bullet point above implies that the functions $ F_{\eta}^{n}(\cdot ,\psi)$ and  $   F_{\eta}(\cdot,\psi)$   are equal for  any fixed  $\psi\in \mathbf{C}(\Gamma)$ when $n$ is large enough, and hence $F_{\eta}(\cdot,\psi)$ is continuous.  Since $\mathbf{C}(\Gamma)   $ is a dense subset of $C(\Gamma)   $ with respect to the uniform norm by Lemma~\ref{LemDense}, it follows that  for any $g\in C(\Gamma)$ the function $ F_{\eta}^{n}(\cdot,g)$ converges uniformly  with large $n$ to  $ F_{\eta}(\cdot,g)$.  Thus the function $F_{\eta}(\cdot,g)$ is continuous for every $g\in C(\Gamma)$.

The continuity claim about the function   $\eta\mapsto\Theta_{\eta}$ follows from an argument combining the above observations  with the fact that for any $A\in \mathcal{A}_{\Gamma}$ the map $\eta\mapsto \eta(A)$  
is  continuous because the cylinder set $A\subset \Gamma$ is both open and closed by Lemma~\ref{RemarkOpenClosed}.
\end{proof}

\subsection{An identity for infinitely decomposable path measures}\label{SubSecLast3}

The  proposition below provides an identity involving all of the measures in our current discussion and is key for analyzing the linear operator $\mathbf{T}_{\eta}$.   Recall that the measures $\gamma^{p,q}$, $\vartheta_{\eta}$, and $\Theta^{ \updownarrow x}_{\eta}$   are defined  in Corollary~\ref{CorollaryGamma},  Definition~\ref{DefVartheta}, and Definition~\ref{DefThetaII}, respectively. 
\begin{proposition}\label{PropDecomp} For $\eta \in \mathcal{M}_{\Gamma}^{\diamond}$, let the  measures $\Phi_{\eta}$ and $\Psi_{\eta}$ on the measurable space $\big(D\times \Gamma\times \Gamma,\,\mathcal{B}_{D}\otimes  \mathcal{B}_{\Gamma}\otimes \mathcal{B}_{\Gamma}    \big)$ be defined as follows: 
\begin{align*}
\Phi_{\eta}(dx,dp,dq)\,:= \,\gamma^{p,q}(dx)\eta(dp)\eta(dq) \hspace{.5cm} \text{and} \hspace{.5cm} 
\Psi_{\eta}(dx,dp,dq)\,:= \,\Theta^{ \updownarrow x}_{\eta}(dp)\Theta^{ \updownarrow x}_{\eta}(dq)\vartheta_{\eta}(dx)\,,
\end{align*}
wherein $x\in D$ and $p,q\in \Gamma$.  If   $\int_{\Gamma\times \Gamma}T(p,q)\eta(dp)\eta(dq)<\infty$, then $\Phi_{\eta}= \Psi_{\eta}$.
\end{proposition}

Before going to the proof, we will define some notation and state two trivial lemmas. 
\begin{definition} For $A\in \mathcal{B}_{D}$, let $\mathcal{B}_{A}$ denote the restriction of the $\sigma$-algebra $\mathcal{B}_{D}$ to $A$.  
\begin{itemize}

\item Define $\Upsilon:=E\times \Gamma\times \Gamma$ and the $\sigma$-algebra $\mathcal{B}_{\Upsilon}:=\mathcal{B}_{E}\otimes \mathcal{B}_{\Gamma}\otimes \mathcal{B}_{\Gamma}  $. 

\item Define $\mathcal{S}_{\Upsilon}:=\{\emptyset\}\cup \bigcup_{n=0}^{\infty} \mathcal{S}_{\Upsilon}^{(n)}  $ for  $\mathcal{S}_{\Upsilon}^{(n)}:=\big\{  \mathbf{\overline{e}}\times \mathbf{\overline{p}}\times\mathbf{\overline{q}} \,\big|\, \mathbf{e}\in E_n \text{ and } \mathbf{p},\mathbf{q}\in \Gamma_n  \big\}$.
\end{itemize}
\end{definition}

\begin{lemma}\label{LemAgree} If $\Phi$ and $\Psi $ are $\sigma$-finite measures on the measurable space $\big(D\times \Gamma\times \Gamma,\,\mathcal{B}_{D}\otimes  \mathcal{B}_{\Gamma}\otimes \mathcal{B}_{\Gamma}    \big)$ that agree on sets in $\mathcal{S}_{\Upsilon}$ and assign $V\times \Gamma\times \Gamma$ measure zero, then $\Phi =\Psi $.
\end{lemma}
\begin{proof} This follows from the uniqueness part of Carath\'eodory's extension theorem  since  $\mathcal{S}_{\Upsilon}$ is a semi-algebra on $\Upsilon$ that generates the $\sigma$-algebra $\mathcal{B}_{\Upsilon}$ on $\Upsilon$.
\end{proof}

The following lemma characterizes the hierarchical symmetry  of the measures $\vartheta_{\eta}$ from Definition~\ref{DefVartheta} when $\eta\in \mathcal{M}_{\Gamma}^{\diamond}$.  The proof is in Appendix~\ref{AppendixVartheta}.

\begin{lemma}\label{RemarkVarthetaSymm2} If $\eta\in \mathcal{M}_{\Gamma}^{\diamond}$, then the  following decomposition  holds for every $n\in \mathbb{N}$:
\begin{align}\label{GenDef}
\vartheta_{ \eta }\,=\,\frac{1}{b^{2n}}\bplus_{\mathbf{e}\in E_n}\Bigg( \prod_{k=1}^{n}\prod_{ e\in E_k^{\updownarrow \mathbf{e}} }  \eta^{e}( \Gamma)\Bigg)^2 \vartheta_{ \eta^{\mathbf{e}} }\hspace{.5cm} \text{under the identification} \hspace{.5cm} D\,\equiv \,\bigsqcup_{\mathbf{e}\in E_n}D\,.
\end{align}
\end{lemma}

\begin{proof}[Proof of Proposition~\ref{PropDecomp}] Recall that $\gamma^{p,q}$ is supported on the set of points in $D$ at which the paths $p$ and $q$ intersect,  and $\Theta^{ \updownarrow x}_{\eta}$ is supported on the set of 
paths passing through $x$.  It follows that the measures  $\Phi_{\eta}$ and $\Psi_{\eta}$  both assign full weight to the set of  triples $(x,p,q)\in D\times \Gamma\times \Gamma$ such that $x\in \textup{Range}(p)\cap  \textup{Range}(q)$.
Moreover,  the total masses of the measures $\Phi_{\eta}$ and $\Psi_{\eta}$ agree since $\Theta_{ \eta }^{\updownarrow x}$ is a probability measure and $\gamma^{p,q}$ has total mass $T(p,q)$:
\begin{align}\label{TotMass}
\Psi_{\eta}\big(D\times \Gamma\times \Gamma   \big)\,=\,\vartheta_{ \eta }(D)\,=\,\int_{\Gamma\times \Gamma}T(p,q)\eta (dp)\eta(dq) \,=\,\Phi_{\eta}\big(D\times \Gamma\times \Gamma   \big)\,,
\end{align}
where the second  equality holds by Remark~\ref{RemarkVarthetaMass}.   The above combined with our assumption on $\eta$ implies that the measures $\Psi_{\eta}$ and $\Phi_{\eta}$ are finite---and thus $\sigma$-finite. Our proof will  leverage~(\ref{TotMass}) using the hierarchical symmetry of the model. 

The measures $\Phi_{\eta}$ and $\Psi_{\eta}$ both  assign $V\times \Gamma\times \Gamma$ measure zero because  $\gamma^{p,q}(V)=0$ for all $p,q\in \Gamma$ and $\vartheta_{\eta}(V)=0$  by Remark~\ref{RemarkVarthetaNull}.   By Lemma~\ref{LemAgree}, it is sufficient for us to show that $\Phi_{\eta}$ and $\Psi_{\eta}$ agree on every cylinder set $\mathbf{\overline{e}}\times \mathbf{\overline{p}}\times\mathbf{\overline{q}}$ where $\mathbf{e}\in E_n$ and $\mathbf{p},\mathbf{q}\in \Gamma_n$  for some $n\in \mathbb{N}_0$.  When the edge  $\mathbf{e}$ does not lie at an intersection between the coarse-grained paths $\mathbf{p}$ and $\mathbf{q}$, then  $\mathbf{\overline{e}}\times \mathbf{\overline{p}}\times\mathbf{\overline{q}}$ is a null set for both $\Phi_{\eta}$ and $\Psi_{\eta}$ by the observation at the beginning of the proof. Hence we will focus on the case when $\mathbf{e}\in \textup{Range}(\mathbf{p})\cap \textup{Range}(\mathbf{q})$. Let  $\{\eta^e\}_{e\in \cup_{k=0}^{\infty} E_k}$ denote  the family of component measures for $\eta\in \mathcal{M}_{\Gamma}^{\diamond}$.  If  $\mathbf{p},\mathbf{q}\in \Gamma_n$ and $\mathbf{e}\in \textup{Range}(\mathbf{p})\cap \textup{Range}(\mathbf{q})$, then using Remark~\ref{RemarkProp44} we can write
\begin{align}\label{here}
\Phi_{\eta}(\mathbf{\overline{e}}\times \mathbf{\overline{p}}\times\mathbf{\overline{q}})\,=\,\frac{1}{|\Gamma_n|^2}
\Bigg(\prod_{\substack{\mathbf{f}\in \mathbf{p}\\ \mathbf{f}\neq \mathbf{e} }  }\eta^\mathbf{f}(\Gamma)\Bigg)\Bigg(\prod_{\substack{\mathbf{f} \in \mathbf{q}\\ \mathbf{f}\neq \mathbf{e} }  }\eta^\mathbf{f}(\Gamma)\Bigg)\underbrace{\int_{\Gamma\times \Gamma} T(p,q)   \eta^\mathbf{e}(dp)  \eta^\mathbf{e}(dp)}_{ \Phi_{\eta^{\mathbf{e}}}(D\times \Gamma\times\Gamma)  }  \,,
\end{align}
where, recall, $\mathbf{f}\in \mathbf{p}$ is shorthand for $\mathbf{f}\in \textup{Range}(\mathbf{p})$.  Next we will work towards arriving at the same expression for $\Psi_{\eta}(\mathbf{\overline{e}}\times \mathbf{\overline{p}}\times\mathbf{\overline{q}})$.

  For the map $\langle \cdot \rangle_n :E\rightarrow E$  from Remark~\ref{RemarkCylinderE}, define the set function $\pi_{\mathbf{e}}:\mathcal{B}_{\mathbf{\overline{e}}}\rightarrow \mathcal{B}_{E} $ by $\pi_{\mathbf{e}}(A):=\{\langle x\rangle_n\,|\, x\in A  \}$ for $A\in \mathcal{B}_{\mathbf{\overline{e}}}$. By Lemma~\ref{RemarkVarthetaSymm2}, we can write the restriction of $\vartheta_{ \eta }$ to $\mathbf{\overline{e}}$ in the form
\begin{align}\label{var}
\vartheta_{ \eta }\big|_{\mathbf{\overline{e}}  }\,=\,&\frac{1}{b^{2n}}\Bigg( \prod_{k=1}^{n}\prod_{ e\in E_k^{\updownarrow \mathbf{e}} }  \eta^{e}( \Gamma)\Bigg)^2 \vartheta_{ \eta^{\mathbf{e}} }\circ \pi_{\mathbf{e}}\, .
\end{align}
Moreover, by Corollary~\ref{RemarkThetaHS}, the measure $ \Theta_{ \eta }^{\updownarrow x}$ has the  hierarchical symmetry below  for $x\in \mathbf{\overline{e}}$.
\begin{align}\label{the}
 \Theta_{ \eta}^{\updownarrow x}(dp)\,=\,&\Bigg(\prod_{k=1}^{n}\prod_{e\in E_k^{\updownarrow \mathbf{e}} }\frac{1}{\eta^{e}(\Gamma)  }  \eta^{e}( d p_e)\Bigg)\Theta_{ \eta^{\mathbf{e}} }^{\updownarrow \langle x\rangle_{n}}\big(dp_{\mathbf{e}}\big)
\end{align}
 Plugging in the forms~(\ref{var}) and~(\ref{the}) results in a cancellation of the factors $\eta^{e}(\Gamma)    $, yielding the second equality below:
\begin{align}\label{Sishy}
\Psi_{\eta}&(\mathbf{\overline{e}}\times \mathbf{\overline{p}}\times\mathbf{\overline{q}})\nonumber  \\  \,= \,& \, \int_{\mathbf{\overline{e}}\times \mathbf{\overline{p}}\times\mathbf{\overline{q}}} \Theta_{ \eta }^{\updownarrow x}(dp)  \Theta_{ \eta }^{\updownarrow x}(dq)  \vartheta_{ \eta }(dx)  \nonumber  \\
\,=\,&\,\frac{1}{b^{2n}}
\Bigg(\prod_{k=1}^{n}\prod_{e\in E_k^{\updownarrow\mathbf{e}} }\eta^{e}( \mathbf{\overline{p}}_{e})\Bigg)\Bigg(\prod_{k=1}^{n}\prod_{e\in E_k^{\updownarrow\mathbf{e}} }\eta^{e}( \mathbf{\overline{q}}_{e})\Bigg) \underbracket{\int_{\mathbf{\overline{e}}\times \Gamma\times \Gamma} \Theta_{ \eta^{\mathbf{e}} }^{\updownarrow \langle x\rangle_n}(dp)  \Theta_{ \eta^{\mathbf{e}} }^{\updownarrow \langle x\rangle_n}(dq)  \vartheta_{ \eta^{\mathbf{e}} }\circ \pi_{\mathbf{e}}( dx) }\,,
\end{align}
where for $e\in E_k^{\updownarrow\mathbf{e}}$,   we define  $\mathbf{\overline{p}}_{e}\subset \Gamma$ as the cylinder set corresponding to $\mathbf{p}_{e}:=[p_e]_{n-k}$ for any  representative $p\in \mathbf{\overline{p}}$. The bracketed integral is equal to $\Psi_{\eta^{\mathbf{e}}}(D\times \Gamma\times\Gamma) $ since through the change of integration variables $y=\langle x\rangle_n$ it becomes
\begin{align*}
\int_{E\times \Gamma\times \Gamma} \Theta_{ \eta^{\mathbf{e}} }^{\updownarrow y}(dp)  \Theta_{ \eta^{\mathbf{e}} }^{\updownarrow y}(dq)  \vartheta_{ \eta^{\mathbf{e}} }( dy)  \,=\,\Psi_{\eta^{\mathbf{e}}}(E\times \Gamma\times\Gamma)\,=\, \Psi_{\eta^{\mathbf{e}}}(D\times \Gamma\times\Gamma)\,,
\end{align*}
where the second equality uses that $V\times \Gamma\times\Gamma$ is a null set for $\Psi_{\eta^{\mathbf{e}}}$.  Applying  Remark~\ref{RemarkProp44} to the terms $\eta^{e}( \mathbf{\overline{p}}_{e})$ and $\eta^{e}( \mathbf{\overline{q}}_{e})$ in~(\ref{Sishy}), we get $\eta^{e}( \mathbf{\overline{p}}_{e})=\frac{1}{|\Gamma_{n-k}|}\prod_{\substack{\mathbf{f}\in \mathbf{p}\\ \mathbf{f} \subset e  } }\eta^\mathbf{f}(\Gamma)$ and  $\eta^{e}( \mathbf{\overline{q}}_{e})=\frac{1}{|\Gamma_{n-k}|}\prod_{\substack{\mathbf{f}\in \mathbf{q}\\ \mathbf{f} \subset e  } }\eta^\mathbf{f}(\Gamma)$, so 
\begin{align*}
\Psi_{\eta}(\mathbf{\overline{e}}\times \mathbf{\overline{p}}\times\mathbf{\overline{q}})\,=\,&\,\frac{1}{b^{2n}}\Bigg(\prod_{k=1}^n \frac{1}{|\Gamma_{n-k}| }   \Bigg)^{2(b-1)}
\Bigg(\prod_{\substack{\mathbf{f} \in  \mathbf{p}\\ \mathbf{f}\neq \mathbf{e} }  }\eta^\mathbf{f}(\Gamma)\Bigg)\Bigg(\prod_{\substack{\mathbf{f} \in \mathbf{q}\\ \mathbf{f}\neq \mathbf{e} }  }\eta^\mathbf{f}(\Gamma)\Bigg)\Psi_{\eta^{\mathbf{e}}}(D\times \Gamma\times\Gamma)\\
\,=\,&\,\frac{1}{|\Gamma_n|^2}
\Bigg(\prod_{\substack{\mathbf{f} \in \mathbf{p} \\ \mathbf{f}\neq \mathbf{e} }  }\eta^\mathbf{f}(\Gamma)\Bigg)\Bigg(\prod_{\substack{\mathbf{f} \in \mathbf{q}\\ \mathbf{f}\neq \mathbf{e} }  }\eta^\mathbf{f}(\Gamma)\Bigg)\Psi_{\eta^{\mathbf{e}}}(D\times \Gamma\times\Gamma)\,.
\end{align*}
The second equality uses  the combinatorial identity $|\Gamma_n|=b^n \prod_{k=1}^{n}|\Gamma_{k-1}|^{b-1}  $, which follows from the formula $|\Gamma_k|=b^{\frac{b^k -1 }{b-1}}$. Hence the above agrees with~(\ref{here}) by the observation~(\ref{TotMass}). The proof is complete since we have shown that the measures $\Phi_{\eta}$ and $\Psi_{\eta}$ are equal on sets in  $\mathcal{S}_{\Upsilon}$. 
\end{proof}

\subsection{Proof of Theorem~\ref{ThmOperator}}\label{SubSecLast4}

Before going to the proof of  Theorem~\ref{ThmOperator}, we  will use Proposition~\ref{PropDecomp} to  prove the following proposition concerning formulas for the linear operator $Y_{\eta}$ and its adjoint. 
\begin{proposition}\label{PropOperatorRe} Suppose  $\eta\in \mathcal{M}_{\Gamma}^{\diamond}$ satisfies that $\int_{\Gamma\times \Gamma}\big(T(p,q)\big)^2\eta(dp)\eta(dq)<\infty$.  For $p\in \Gamma$ and $f\in L^2(D,\vartheta_{\eta})$, define $\big(Y_{\eta}f\big)(p):=\int_{D\times \Gamma}f(x)\gamma^{p,q}(dx)\eta(dq) $.  Then $Y_{\eta}$ defines a bounded linear map from $L^2(D,\vartheta_{\eta})$ to $L^2(\Gamma,\eta)$, and statements (i)-(iii) below hold. 
\begin{enumerate}[(i)]

\item For any $f\in L^2(D,\vartheta_{\eta})$, the signed measure $\Delta_f(dp):=\int_{D}f(x)\Theta^{\updownarrow x}_{ \eta }(dp) \vartheta_{\eta}(dx)$ is absolutely continuous with respect to $ \eta$ and $\big(Y_{\eta}f\big)(p)=\frac{  d\Delta_f }{  d\eta }(p)$ holds for $\eta$-a.e.\ $p\in \Gamma$.

\item  For any $g\in L^2(\Gamma,\eta)$,  the signed measure $\widehat{\Delta}_g(dx):=\int_{\Gamma\times \Gamma}g(p)\gamma^{p,q}(dx)\eta(dp)\eta(dq)$ is absolutely continuous with respect to  $\vartheta_{\eta}$ and $\big(Y_{\eta}^{*}g\big)(x)= \frac{ d\widehat{\Delta}_g  }{d\vartheta_{\eta}  }(x) $ holds for $\vartheta_{\eta}$-a.e.\ $x\in D$.

\item For any $g\in L^2(\Gamma,\eta)$, the equality $\big(Y_{\eta}^{*}g\big)(x)=\int_{\Gamma}g(p)\Theta^{\updownarrow x}(dp) $ holds for $\vartheta_{\eta}$-a.e.\ $x\in D$.

\end{enumerate}

\end{proposition}

\begin{proof} Let $g:\Gamma \rightarrow [0,\infty)$ and $f :D\rightarrow [0,\infty)$ be Borel measurable functions. Observe that  multiplying  the equality between  $\gamma^{p,q}(dx) \eta(dq)  \eta(dp)$ and $\Theta^{\updownarrow x}_{ \eta}(dq) \Theta^{\updownarrow x}_{ \eta}(dp) \vartheta_{\eta}(dx)$ from Proposition~\ref{PropDecomp} by $f(x)$ and integrating out $x$ and $q$ yields 
\begin{align}\label{ForY}
\bigg(\int_{D\times \Gamma}f(x)\gamma^{p,q}(dx)\eta(dq)\bigg)\eta(dp)\,=\, \int_{D}f(x)\Theta^{\updownarrow x}_{ \eta}(dp) \vartheta_{\eta}(dx)\,=:\, \Delta_f(dp) \,, 
\end{align}
where we have used that $\Theta^{\updownarrow x}_{ \eta}$ is a probability measure.  Similarly, if we multiply the   equality from Proposition~\ref{PropDecomp}  by $g(p)$ and integrate out $p$ and $q$, we get the second equality below
\begin{align}
\widehat{\Delta}_g(dx)\,:=\,\int_{\Gamma\times \Gamma}g(p)\gamma^{p,q}(dx)\eta(dp)\eta(dq)\,=\,\bigg( \int_{\Gamma}g(p)\Theta^{\updownarrow x}_{ \eta }(dp)\bigg) \vartheta_{\eta}(dx)   \,.\label{ForY*}
\end{align}
In particular,~(\ref{ForY}) implies that the measure $\Delta_f$ is absolutely continuous with respect to $\eta$, which gives us the following alternative representation for  $\hat{Y}_{\eta}f$:
\begin{align*}
\big(\hat{Y}_{\eta}f\big)(p)\,:=\, \int_{D\times \Gamma}f(x)\gamma^{p,q}(dx)\eta(dq) \,=\,\frac{  d\Delta_f }{  d\eta }(p)\,.
\end{align*}
Similarly, equation~(\ref{ForY*}) implies that the measure  $\widehat{\Delta}_g$  is absolutely continuous with respect to  $\vartheta_{\eta}$  and  $\frac{ d\widehat{\Delta}_g  }{d\vartheta_{\eta}  }(x)=\int_{\Gamma}g(p)\Theta^{\updownarrow x}_{ \eta }(dp)$. Note that $g\mapsto \frac{ d\widehat{\Delta}_g  }{d\vartheta_{\eta}  } $ defines a bounded linear map from $ L^2(\Gamma,\eta)$ to  $ L^2(D,\vartheta_{\eta})  $ since 
 \begin{align*}
  \int_{ D }\bigg|\frac{ d\widehat{\Delta}_g  }{d\vartheta_{\eta}  }(x)\bigg|^2 \vartheta_{\eta} (dx) \,= \, & \, \int_{\Gamma \times \Gamma \times D}g(p)g(q) \Theta^{\updownarrow x}_{ \eta }(dp)\Theta^{\updownarrow x}_{ \eta }(dq)\vartheta_{\eta}(dx)  \\
  \,= \, &  \, \int_{\Gamma \times \Gamma }g(p)g(q) T(p,q)\eta(dp)\eta(dq) \\
  \, \leq  \, & \, \|g\|^2_{L^2(\Gamma,\eta)}\bigg(\int_{\Gamma \times \Gamma } \big(T(p,q)\big)^2\eta(dp)\eta(dq)\bigg)^{1/2} \,,
  \end{align*}
where the second equality uses Proposition~\ref{PropDecomp} and that $\gamma^{p,q}$ has total mass $T(p,q)$ for $\eta\times \eta$-a.e.\ $(p,q)$, and  the  inequality applies Cauchy-Schwarz. The following  calculation  shows that $\big\langle  g \big|\, \hat{Y}_{\eta}f  \big\rangle_{L^2(\Gamma,\eta)  }$ is equal to $\big\langle \frac{ d\widehat{\Delta}_g  }{d\vartheta_{\eta}  }\,\big|\,f \big\rangle_{L^2(D,\vartheta_{\eta})  }$: 
\begin{align*}
\big\langle  g \big|\, \hat{Y}_{\eta}f \big\rangle_{L^2(\Gamma,\eta)  }\,:=\, &   \,\int_{\Gamma}g(p)\bigg(\int_{D\times \Gamma}f(x)\gamma^{p,q}(dx)\eta(dq)\bigg)\eta(dp) \nonumber  \\  \,=\,&\,\int_{D}f(x)\int_{\Gamma\times \Gamma}g(p)\gamma^{p,q}(dx) \eta(dp) \eta(dq) \nonumber  \\
\,=\,&\,\int_{D}f(x)\frac{ d\widehat{\Delta}_g  }{d\vartheta_{\eta}  }(x)\vartheta_{\eta}(dx) \,=\,\bigg\langle \frac{ d\widehat{\Delta}_g  }{d\vartheta_{\eta}  }\,\bigg|\,f  \bigg\rangle_{L^2(D,\vartheta_{\eta})  } 
\end{align*}
where we  invoked the  definitions of $\hat{Y}_{\eta}f$ and $\widehat{\Delta}_g $  in the first  and third equalities, respectively, and we permuted the integrations in the second equality. Since the above holds for all nonnegative $f\in L^2(D,\vartheta_{\eta})$ and $g\in  L^2(\Gamma,\eta)    $  and the linear map $g\rightarrow \frac{ d\widehat{\Delta}_g  }{d\vartheta_{\eta}  }   $ is bounded,  it follows  that $\hat{Y}_{\eta}$ defines a bounded linear map from  $ L^2(D,\vartheta_{\eta}) $ to $  L^2(\Gamma,\eta)  $ such that  $\hat{Y}_{\eta}^*g$ is $\vartheta_{\eta}$-a.e.\ equal to  $\frac{ d\widehat{\Delta}_g  }{d\vartheta_{\eta}  } $ and thus also equal to $\int_{\Gamma}g(p)\Theta^{\updownarrow (\cdot)}_{ \eta }(dp)$.
\end{proof} 

\begin{proof}[Proof of Theorem~\ref{ThmOperator}] We can deduce that the linear operator $\hat{Y}_{\eta}\hat{Y}_{\eta}^*$ on $L^2(\Gamma,\eta)$ is equal to $\mathbf{T}_{\eta  }$ by showing that  $\hat{Y}_{\eta}\hat{Y}_{\eta}^*$  has integral kernel $T(p,q)$.  Parts (i) and (iii) of Proposition~\ref{PropOperatorRe} give us the equalities below for any $g\in L^2(\Gamma,\eta)$.
\begin{align*}
\big(\hat{Y}_{\eta}\hat{Y}_{\eta}^*g\big)(p)\,=\,\frac{  \int_{D}\big(\hat{Y}_{\eta}^*g\big)(x)  \Theta^{\updownarrow x}_{\eta }(dp) \vartheta_{\eta}(dx) }{ \eta(dp) } \,=\,\frac{  \int_{D}\left(\int_{\Gamma}g(q)\Theta^{\updownarrow x}_{ \eta}(dq) \right)  \Theta^{\updownarrow x}_{\eta }(dp) \vartheta_{\eta}(dx) }{ \eta(dp) }
\end{align*}
Next using that the measure $ \Theta^{\updownarrow x}_{ \eta}(dq) \Theta^{\updownarrow x}_{ \eta}(dp) \vartheta_{\eta}(dx)$ is equal to $\gamma^{p,q}(dx) \eta(dq)  \eta(dp) $ by Proposition~\ref{PropDecomp}, we can rewrite the above as
\begin{align*}
 \big(\hat{Y}_{\eta}\hat{Y}_{\eta}^*g\big)(p) \,=\,\int_{D\times \Gamma}\gamma^{p,q}(dx)g(q)\eta(dq) \,=\,\int_{ \Gamma}T(p,q)g(q)\eta(dq)\,=:\big(\mathbf{T}_{\eta }g\big)(p)\,,
\end{align*}
where the second equality holds because $\gamma^{p,q}$ has total mass $T(p,q)$ for a.e.\ $(p,q)$.  Since $g\in L^2(\Gamma,\eta)$ is arbitrary, we have shown the desired result. 
\end{proof}

\subsection{Proof of Theorem~\ref{ThmVartheta}}\label{SubSecLast6}

 \begin{proof} Part (i): The symmetry of the model implies that the expectation of $\vartheta_{\mathbf{M}_r}$ must be a multiple $c>0$ of the uniform measure $\nu$.  Moreover, the expectation of the total mass of $\vartheta_{\mathbf{M}_r}$ is
 $$  \mathbb{E}\big[\vartheta_{\mathbf{M}_r}(D)\big]\,=\,\mathbb{E}\bigg[\int_{\Gamma\times \Gamma} T(p,q) \mathbf{M}_r(dp)  \mathbf{M}_r(dq)   \bigg]\,=\,\int_{\Gamma\times \Gamma} T(p,q)  \upsilon_r(dp,dq)\,=\,R'(r)\,,  $$
 where the second equality follows from (ii) of Theorem~\ref{ThmPathMeasure}, and the third equality holds by differentiating (iv) of Proposition~\ref{PropLemCorrelate} at $a=0$.  Therefore, $c=R'(r)$.\vspace{.3cm}

\noindent Part (ii):  Recall from Corollary~\ref{CorInfDec} that $\mathbf{M}_r $ is a.s.\ infinitely decomposable.  Let $\big\{\mathbf{M}^{e}\big\}_{e\in \cup_{k=1}^{\infty}E_k} $ be the family of component measures of  $\mathbf{M}_r$.  The measures $\mathbf{M}^{e}$ and $\vartheta_{\mathbf{M}^{e}}$ are a.s.\ nonzero for every $e\in  \cup_{k=1}^{\infty}E_k$ as a consequence of part (iv) of Theorem~\ref{ThmPathMeasure} and Remark~\ref{RemarkVarthetaMass}.  For any open set $A\subset D$, there exists an $\mathbf{e}\in E_n$ for large enough $n\in \mathbb{N}$ such that $ \mathbf{\overline{e}}\subset A $.     Using Lemma~\ref{RemarkVarthetaSymm2}, we can bound $\vartheta_{\mathbf{M}_r}(A)$ from below by $\vartheta_{\mathbf{M}_r}(\mathbf{\overline{e}})=\frac{1}{b^{2n}}\big( \prod_{k=1}^{n}\prod_{ e\in E_k^{\updownarrow \mathbf{e}} }  \mathbf{M}^{e}( \Gamma)\big)^2 \vartheta_{ \mathbf{M}^{\mathbf{e}} }(E)  $, which is positive by our previous observations since $E=D\backslash V$ is a full-measure set for $\vartheta_{ \mathbf{M}^{\mathbf{e}} }$ by Remark~\ref{RemarkVarthetaNull}.  Therefore, $\vartheta_{\mathbf{M}_r}$ a.s.\ assigns all open subsets of $D$ positive measure.  \vspace{.3cm}

\noindent Part (iii): Suppose to reach a contradiction that with positive probability   there exists an $A\in \mathcal{B}_{D}$ and an $\frak{h}\in [0,2)$ such that $\textup{dim}_{H}(A)=\frak{h}$ and $\vartheta_{\mathbf{M}_r}(A)>0$. If the aforementioned event occurs, then  the energy defined by
$$ \hat{\mathcal{Q}}_{\alpha}(\vartheta_{\mathbf{M}_r})\,=\,\int_{D\times D}  \frac{1}{\big(d_{D}(x,y)   \big)^{\alpha}}   \vartheta_{\mathbf{M}_r}(dx,dy)   $$
must be infinite for any  $\alpha\in (\frak{h},2)$.  This, however, contradicts part (v) below, which shows that the analogous energy is a.s.\ finite when the dimension function  $x^{\alpha}$  is replaced by the  generalized dimension function $h_\lambda(x):=x^2\big(\log(1/x)\big)^{\lambda}$ with $\lambda>9$, which vanishes faster as $x\searrow 0$ than $x^{\alpha}$ for any fixed $\alpha<2$.\vspace{.3cm}

\noindent Part (iv):   Since $\vartheta_{\eta}(V)=0$ by Remark~\ref{RemarkVarthetaNull}, we can treat the measures $\vartheta_{\eta}$ as acting on the measurable space $(E,\mathcal{B}_{E})$.  Recall from 
 Remark~\ref{RemarkCylinderE} that if $x\in \mathbf{\overline{e}}$ for some $\mathbf{e}\in E_n$, then $\langle x\rangle_n \in E$ denotes the position of  $x$ within the  subcopy of $D$ corresponding to  $\mathbf{e}$. As in the proof of Proposition~\ref{PropDecomp}, define $\pi_{\mathbf{e}}:\mathcal{B}_{ \mathbf{\overline{e}} }\rightarrow \mathcal{B}_E$ as the set map induced by the function from $\mathbf{\overline{e}}$ to $E$ defined by $x \mapsto \langle x\rangle_n$, i.e., such that $\pi_{\mathbf{e}}(A)=\{ \langle x\rangle_n\,|\,x\in A \} $ for $A\in \mathcal{B}_{ \mathbf{\overline{e}} }$. Lemma~\ref{RemarkVarthetaSymm2} is equivalent to stating that for any $\mathbf{e}\in E_n$ the restriction of $\vartheta_{\mathbf{M}_r}$ to the cylinder set $\mathbf{\overline{e}}$ has the form
\begin{align}
\vartheta_{\mathbf{M}_r}\big|_{\mathbf{\overline{e}}}\,=\,\frac{1}{b^{2n}}\Bigg(\prod_{k=1}^n\prod_{e\in E_k^{\updownarrow\mathbf{e}}}\mathbf{M}^{e}(\Gamma)   \Bigg)^2\vartheta_{\mathbf{M}^\mathbf{e}}\circ \pi_{\mathbf{e}}\,,
\end{align}
where the family of random measures $\{\mathbf{M}^{e}  \}_{e\in \cup_{k=0}^{\infty} E_k}$ are defined in relation to $\mathbf{M}_r$ as in Proposition~\ref{CorProp4}. Recall that  the random measure $\mathbf{M}^{e} $  is equal in distribution to $\mathbf{M}_{r-k} $ when $e\in E_k$, and  the random measures $\mathbf{M}^{e} $ and $\mathbf{M}^{f} $ are independent provided that $f\in  E_l$ satisfies $\overline{e}\cap \overline{f}=\emptyset$. 

 The set of points $(x,y)\in E\times E$  satisfying $\frak{g}_D(x,y)=n$ can be written as a union of  product cylinder sets $\mathbf{\overline{e}}\times \mathbf{\overline{f}}$ for distinct $\mathbf{e},\mathbf{f}\in E_n$.  Given $\mathbf{e},\mathbf{f}\in E_n$,    we write $\mathbf{e}\updownarrow \mathbf{f}$  if there is a path in $\Gamma_n$ passing through both $\mathbf{e}$ and $\mathbf{f}$, and we write $\mathbf{e}\diamond \mathbf{f}$ otherwise. \vspace{.2cm}

\noindent \textbf{Case $\mathbf{e\diamond f}$:}
If $\mathbf{e}\diamond \mathbf{f}$, then  the restriction of $\mathbb{E}\big[\vartheta_{\mathbf{M}_r}\times\vartheta_{\mathbf{M}_r} \big]$ to  $\mathbf{\overline{e}}\times\mathbf{\overline{f}} $ can be written in the form
\begin{align}
\mathbb{E}\big[\vartheta_{\mathbf{M}_r}\times\vartheta_{\mathbf{M}_r} \big]\big|_{ \mathbf{\overline{e}}\times\mathbf{\overline{f}}  }\nonumber \, &\,=\,\frac{1}{b^{4n}} \Bigg(\prod_{k=1}^{n-1}\prod_{e\in E_k^{\updownarrow\mathbf{e}}}\mathbb{E}\Big[ \big( \mathbf{M}^{e}(\Gamma)     \big)^4\Big]\Bigg)\mathbb{E}\big[ \vartheta_{\mathbf{M}^\mathbf{e}}\circ \pi_{\mathbf{e}}  \big]\times\mathbb{E}\big[ \vartheta_{\mathbf{M}^\mathbf{f}}\circ \pi_{\mathbf{f}} \big]\nonumber \\ & \,=\,\frac{1}{b^{4n}}\big( R'(r-n)\big)^2\Bigg(\prod_{k=1}^{n-1}\mathbb{E}\Big[ \big( \mathbf{M}_{r-k}(\Gamma)     \big)^4\Big]^{b-1}  \Bigg)(\nu\circ \pi_{\mathbf{e}})  \times (\nu\circ \pi_{\mathbf{f}})  \nonumber \\ & \,=\, \underbrace{\big( R'(r-n)\big)^2\Bigg(\prod_{k=1}^{n-1}\mathbb{E}\Big[ \big( \mathbf{M}_{r-k}(\Gamma)     \big)^4\Big]^{b-1}\Bigg)}_{=:\,C_{r,n}^{\diamond} }\nu \times \nu|_{ \mathbf{\overline{e}}\times\mathbf{\overline{f}}  } \label{Furr} \,.
\end{align}
The first equality  follows from the independence of $ \mathbf{M}^{e}$ and $\mathbf{M}^{f}$ when $\overline{e}\cap \overline{f}=\emptyset$ and that $\mathbb{E}\big[\mathbf{M}^e(\Gamma)\big]=1$. The second equality uses that the random measure $ \mathbf{M}^{e} $ is equal in distribution to $\mathbf{M}_{r-k}$ when $e\in E_k$ and that $\mathbb{E}\big[ \vartheta_{\mathbf{M}_t } \big]=R'(t)\nu$ by part (i).
The third equality holds because the dilation map $\pi_{\mathbf{e}}:\mathcal{B}_{ \mathbf{\overline{e}} }\rightarrow \mathcal{B}_{ E }$ satisfies $  \nu\circ \pi_{\mathbf{e}}=b^{2n}\nu|_{ \mathbf{\overline{e}}} $.

 We need to show that $C_{r,n}^{\diamond}$ grows in proportion to $n^8$ as $n\rightarrow \infty$.  Recall that we denote the moment  $\mathbb{E}\big[ \big( \mathbf{M}_{t}(\Gamma)     \big)^m\big]$ by $R^{(m)}(t)$. Through writing $ \mathbf{M}_{r-k}(\Gamma)=1+\big( \mathbf{M}_{r-k}(\Gamma)-1\big)$ and foiling inside the expectations in~(\ref{Furr}), we get 
\begin{align}
C_{r,n}^{\diamond}\,=\, & \,\big( R'(r-n)\big)^2 \prod_{k=1}^{n-1}\Big(1\,+\,6R(r-k)\,+\,4R^{(3)}(r-k)\,+\,  R^{(4)}(r-k)    \Big)^{b-1} \nonumber \\
 \, =\,&\, \underbrace{\big( R'(r-n)\big)^2}_{ \sim \,\frac{\kappa^4}{n^4}  } \textup{exp}\Bigg\{ \underbrace{(b-1)\sum_{k=1}^{n-1} \log\Big(1+6R(r-k)+4R^{(3)}(r-k)+R^{(4)}(r-k) \Big)}_{ =\, 12\log n \,+\,\mathit{O}(1)  }    \Bigg\} \,. \label{egret}
\end{align}
  The asymptotics for the  braced expressions above hold because for $-t\gg 1$ and  $\kappa^2:=\frac{2}{b-1}$ we have
\begin{align}\label{Rstuff} 
  R'(t)\,=\,\frac{\kappa^2}{t^2}\big(1+\mathit{o}(1)\big)\,,\hspace{.5cm} R(t)\,=\,\frac{ \kappa^2 }{ -t }+\mathit{O}\bigg(  \frac{ \log(-t) }{ t^2 }\bigg)\,,\hspace{.5cm} R^{(m)}(t)\,=\,\mathit{O}\bigg( \frac{1}{ |t|^{\lceil m/2\rceil} } \bigg)\,, 
  \end{align} 
where we have used Remark~\ref{RemarkDerR}, Lemma~\ref{LemVar}, and (III) of Theorem~\ref{ThmExistence}, and hence
$$  (b-1) \log\left(1+6R(t)+4R^{(3)}(t)+R^{(4)}(t) \right) \,\stackrel{ -t\gg 1 }{ = } \,\frac{12}{-t}\,+\,\mathit{O}\Big( \frac{1}{t^2} \Big)\,. $$
It follows that the difference between $12\sum_{k=1}^{n-1}\frac{1}{k}$ and the exponent on the second line of~(\ref{egret}) is a bounded sequence in $n$, which is how we obtain the asymptotic $12\log n +\mathit{O}(1)$ above. Hence $C_{r,n}^{\diamond}$ is asymptotic to a multiple  $\mathbf{c}>0$ of $n^8$ with large $n$.\vspace{.2cm}

\noindent \textbf{Case $\mathbf{e\updownarrow f}$:}    A similar argument as in~(\ref{Furr}) shows that when $\mathbf{e}\updownarrow\mathbf{f}$ the restriction of $\mathbb{E}\big[\vartheta_{\mathbf{M}_r}\times\vartheta_{\mathbf{M}_r} \big]$ to the product set $\mathbf{\overline{e}}\times\mathbf{\overline{f}} $ has the form
\begin{align}
\mathbb{E}\big[\vartheta_{\mathbf{M}_r}\times\vartheta_{\mathbf{M}_r}  \big]\big|_{ \mathbf{\overline{e}}\times\mathbf{\overline{f}}  } \,=\,&\,\Bigg(\mathbb{E}\Big[ \big( \mathbf{M}_{r-n}(\Gamma)     \big)^4\Big]^{b-2} \prod_{k=1}^{n-1}\mathbb{E}\Big[ \big( \mathbf{M}_{r-k}(\Gamma)     \big)^4\Big]^{b-1}\Bigg)\nonumber \\
&\,\,\cdot \underbracket{\mathbb{E}\Big[\big(\mathbf{M}_{r-n}(\Gamma)  \big)^2    \vartheta_{\mathbf{M}_{r-n}  }\Big]}\times \underbracket{\mathbb{E}\Big[ \big(\mathbf{M}_{r-n}(\Gamma)  \big)^2   \vartheta_{\mathbf{M}_{r-n}  }\Big]}\, \bigg|_{ \mathbf{\overline{e}}\times\mathbf{\overline{f}}  }\nonumber \\ \,=\,& \, C_{r,n}^{\updownarrow}\nu\times \nu \big|_{ \mathbf{\overline{e}}\times\mathbf{\overline{f}}  }  \,,\label{Furr2}
\end{align}
where the second equality holds for some $C_{r,n}^{\updownarrow}>0$ satisfying $C_{r,n}^{\updownarrow}\sim C_{r,n}^{\diamond}$ with large $n$ as a consequence of the discussion below. By essentially the same computation as below~(\ref{egret}), the parenthesized expression on the right side above is asymptotic to the exponential of the braced term in~(\ref{egret}), i.e., the same asymptotic multiple of $n^{12}$. Unlike the $x\diamond y$ case, we have to contend with the bracketed terms, which are a correlation of the square of the random variable $\mathbf{M}_{r-n}(\Gamma)$ with the random  measure $\vartheta_{\mathbf{M}_{r-n}  }$.  We will show that the bracketed measures have the form $B_{r-n}\nu$ for $B_t >0 $  asymptotic to $\frac{\kappa^2 }{t^{2}}$ as $t\rightarrow -\infty$, which implies that $C_{r,n}^{\updownarrow}\sim C_{r,n}^{\diamond}$ as $n\rightarrow \infty$.

By the spatial symmetry of the model, the expectation of the measures $ \mathbf{M}_t(\Gamma) \vartheta_{\mathbf{M}_t}  $ and $\big(\mathbf{M}_t(\Gamma)\big)^2 \vartheta_{\mathbf{M}_t} $ must be  constant multiplies $A_t,B_t>0$ of the uniform measure $\nu$, i.e.,
\begin{align}\label{DefAB}
\mathbb{E}\big[ \mathbf{M}_t(\Gamma) \vartheta_{\mathbf{M}_t}\big]   \,=\,A_t \nu\hspace{.5cm}\text{and}\hspace{.5cm}\mathbb{E}\Big[ \big(\mathbf{M}_t(\Gamma)\big)^2 \vartheta_{\mathbf{M}_t} \Big]   \,=\,B_t \nu\,.
\end{align}
We will first use the hierarchical symmetry to 
derive a closed expression for $A_{t}$.  With Proposition~\ref{CorProp4} and Lemma~\ref{RemarkVarthetaSymm2} with $n=1$, we get the second equality below 
\begin{align}
A_t \nu  \,=\,&\,\mathbb{E}\big[ \mathbf{M}_t(\Gamma) \vartheta_{\mathbf{M}_t} \big] \nonumber  \,\\
\,=\,&\,\mathbb{E}\Bigg[\Bigg(\frac{1}{b}\sum_{1\leq i\leq b} \prod_{1\leq j \leq b}\mathbf{M}_{t-1}^{(i,j)} (\Gamma)\Bigg)\Bigg(\frac{1}{b^2} \sum_{1\leq I,J \leq b}\Bigg( \prod_{\substack{ 1\leq \ell\leq b \\  \ell\neq J  }}\mathbf{M}_{t-1}^{(I,\ell)}(\Gamma)  \Bigg)^2\vartheta_{\mathbf{M}_{t-1}^{(I,J)}}\circ \pi_{I,J}\Bigg) \Bigg]\label{expand} \\
\,=\,&\, \frac{b-1}{b}\mathbb{E}\Big[\big( \mathbf{M}_{t-1}(\Gamma)  \big)^2  \Big]^{b-1}\mathbb{E}\big[\vartheta_{\mathbf{M}_{t-1}} \big] \,+\,\frac{1}{b}\mathbb{E}\Big[\big( \mathbf{M}_{t-1}(\Gamma)  \big)^3  \Big]^{b-1}\mathbb{E}\big[\mathbf{M}_{r-1}(\Gamma) \vartheta_{\mathbf{M}_{t-1}} \big] \nonumber \\
\,=\,&\, \frac{b-1}{b}\big(1+R(t-1)\big)^{b-1}R'(t-1)\nu \,+\,\frac{1}{b}\mathbb{E}\Big[\big( \mathbf{M}_{t-1}(\Gamma)  \big)^3  \Big]^{b-1}A_{t-1}\nu\,,\nonumber 
\end{align}
where $\pi_{I,J}:=\pi_{\mathbf{e}}$ for the edge $\mathbf{e}\in E_1$ corresponding to the $J^{th}$ segment along the $I^{th}$ branch of $D$.  The third equality above follows from foiling the sums, and the two terms on the third line correspond to products of terms in which $i\neq I$ and $i=I$,  respectively.  The fourth equality uses that $\mathbf{M}_{t}(\Gamma)$ has variance $R(t)$ and that $\mathbb{E}\big[\vartheta_{\mathbf{M}_{t}} \big]=R'(t)\nu  $ by part (i). 

  Thus $A_{t}$ can be expressed through the  series 
\begin{align*}
A_{t}\,=\,\frac{b-1}{b}\sum_{k=1}^\infty\big(1+R(t-k)\big)^{b-1}R'(t-k)\frac{1}{b^{k-1}}\prod_{\ell=1}^{k-1}\mathbb{E}\Big[\big( \mathbf{M}_{t-\ell}(\Gamma)  \big)^3  \Big]^{b-1} \,.\nonumber 
\end{align*}
Since the third moment of $ \mathbf{M}_{t-\ell}(\Gamma) $ can be written in terms of its centered moments as $1+3R(t-\ell)+R^{(3)}(t-\ell) $,  applying the asymptotics~(\ref{Rstuff}) when $-t\gg 1$ yields  
\begin{align*}
A_{t}
\,\sim\,\frac{b-1}{b}\sum_{k=1}^\infty \frac{\kappa^2}{t^2}\frac{1}{b^{k-1}}\,=\,\frac{\kappa^2}{t^2}\,.\nonumber 
\end{align*}

A similar analysis that begins by  expanding $\mathbb{E}\big[ \big(\mathbf{M}_t(\Gamma)\big)^2 \vartheta_{\mathbf{M}_t}\big]$  in analogy to~(\ref{expand}) leads to a  recursion relation in which $B_t$ is a function of  $B_{t-1}$ and  $A_{t-1}$.  The resulting series representation for $B_t$ again yields $B_{t}\sim \frac{\kappa^2}{t^2}$ for $-t\gg 1$.  Thus $C_{r,n}^{\updownarrow}$ is asymptotic to  $C_{r,n}^{\diamond}\sim \mathbf{c}n^8$ with large $n$.  It follows that  the  correlation function $C_r(x,y)$ is asymptotic to $\mathbf{c} \big(\frak{g}_D(x,y)\big)^{8} $ as $\frak{g}_D(x,y)\rightarrow \infty$.  \vspace{.3cm}

\noindent Part (v): By a similar argument as in the proof of Proposition~\ref{PropFrostman}, it suffices  to work with a modified version of the energy $\mathcal{Q}_\lambda $ having a form that fits with the hierarchical structure of the model:
 \begin{align*}
\widetilde{\mathcal{Q}}_\lambda \big(\vartheta_{\mathbf{M}_r}\big)\,:=\, \int_{D\times D} \frac{b^{2\frak{g}_D(x,y)}}{ \big( \frak{g}_D(x,y)  \big)^{\lambda}}\vartheta_{\mathbf{M}_r}(dx)\vartheta_{\mathbf{M}_r}(dy)\,,
\end{align*}
where   $\frak{g}_D(x,y)$ is defined as in part (iv).  Since $\vartheta_{\mathbf{M}_r}(V)=0$, we can replace the integration domain $D\times D$ by $E\times E$. For $x,y\in E$, we write $x\updownarrow  y $ if there is a path $p\in \Gamma$ passing through both $x$ and $y$ and $x\diamond  y $ otherwise.  Define $U_n^{\updownarrow    }$ as the set of $(x,y)\in E\times E$ such that  $\frak{g}(x,y)=n$ and  $x\updownarrow  y $, and define $U_n^{ \Diamond } $ analogously.  The expectation of $\widetilde{\mathcal{Q}}_\lambda(\vartheta_{\mathbf{M}_r})$ can be written as
  \begin{align*}
&\mathbb{E}\Big[\widetilde{\mathcal{Q}}_\lambda\big(\vartheta_{\mathbf{M}_r}\big)\Big]\\
& \,=\, \sum_{n=1}^\infty \left(\mathbb{E}\left[\int_{U_n^{\updownarrow    }} \frac{b^{2\frak{g}_D(x,y)}}{ \big( \frak{g}_D(x,y)  \big)^{\lambda}}\vartheta_{\mathbf{M}_r}(dx)\vartheta_{\mathbf{M}_r}(dy)\right]\,+\,\mathbb{E}\left[\int_{U_n^{\Diamond  }} \frac{b^{2\frak{g}_D(x,y)}}{ \big( \frak{g}_D(x,y)  \big)^{\lambda}}\vartheta_{\mathbf{M}_r}(dx)\vartheta_{\mathbf{M}_r}(dy)\right]\right)\\
& \,=\, \sum_{n=1}^\infty  \frac{b^{2n}}{ n^{\lambda}}\left(\mathbb{E}\left[  \vartheta_{\mathbf{M}_r}\times \vartheta_{\mathbf{M}_r}\big( U_n^{\updownarrow    }\big) \right]\,+\,\mathbb{E}\left[  \vartheta_{\mathbf{M}_r}\times \vartheta_{\mathbf{M}_r}\big( U_n^{\Diamond    }\big) \right]\right)\\
&  \,=\, \sum_{n=1}^\infty  \frac{b^{2n}}{ n^{\lambda}} \left(\int_{U_n^{\updownarrow    }}C_r(x,y)\nu(dx)\nu(dy)  \,+\,   \int_{U_n^{\Diamond    }}C_r(x,y)\nu(dx)\nu(dy)\right)\,,
\end{align*}
where, recall, $C_r$ is the Radon-Nikodym derivative of $\mathbb{E}\big[  \vartheta_{\mathbf{M}_r}\times \vartheta_{\mathbf{M}_r} \big]$ with respect to $\nu\times \nu$.  It follows from our discussion in the proof of (iv) that $C_r(x,y)=C_{r,n}^{\updownarrow}$  on $U_n^{\updownarrow    }$ and $C_r(x,y)=C_{r,n}^{\diamond}$ on $U_n^{\Diamond    }$. Thus we have 
\begin{align*}
\mathbb{E}\Big[\widetilde{Q}_{\lambda}\big(\vartheta_{\mathbf{M}_r}\big)\Big]\,=\,&\, \sum_{n=1}^\infty  \frac{b^{2n}}{ n^{\lambda}}C_{r,n}^{\updownarrow} \nu\times \nu\big( U_n^{\updownarrow    }\big) \,+\,\sum_{n=1}^\infty \frac{b^{2n}}{ n^{\lambda}}C_{r,n}^{\diamond} \nu\times \nu\big( U_n^{\diamond    }\big)\\ \,=\,&\,\sum_{n=1}^\infty  \bigg(\frac{b(b-1)}{ n^{\lambda}}C_{r,n}^{\updownarrow} \,+\, \frac{b-1}{ n^{\lambda}}C_{r,n}^{\diamond}   \bigg)\,,
\end{align*}
where the last equality holds because $\nu\times \nu\big( U_n^{\updownarrow    }\big)=\frac{b-1}{b^{2n-1}}   $  and  $\nu\times \nu\big( U_n^{\diamond    }\big)=\frac{b-1}{b^{2n}}   $.  The constants  $C_{r,n}^{\updownarrow}$ and $C_{r,n}^{\diamond}   $ are asymptotically proportional to $n^8$ for $n\gg 1$ by part (iv), and therefore the above series converge if and only if $\lambda>9$.
\end{proof}

\begin{appendix}

\section{Proof of Corollary~\ref{CorIntSet}}\label{AppendixHausdorff}

\begin{proof}
As a consequence of Proposition~\ref{PropFrostman}, for $\rho_r$-a.e.\ pair $(p,q)$ there is a nonzero measure $\tau^{p,q}$ that assigns full measure to $I^{p,q}$ and for which the energy  $Q^\frak{h}(\tau^{p,q})$ is finite for all $0\leq \frak{h}<1$.   For a fixed $\frak{h}\in [0,1)$, define $h(a):=1/\log^\frak{h} (\frac{1}{a})$ for $a\in (0,1)$. Notice that 
\begin{align*}
Q^\frak{h}\left(\tau^{p,q}\right)\,:=\,\int_{[0,1]}\bigg( \int_{[0,1]} \frac{1}{ h\big(|x-y|\big)}   \tau^{p,q}(dx) \bigg) \tau^{p,q}(dy)   \,\geq \,\int_{[0,1]} F_{h}^{p,q}(y)  \tau^{p,q}(dy)\,,
\end{align*}
where  $F^{p,q}_{h}(y):=\sup_{\delta >0} \big[\frac{1}{h(\delta)}  \tau^{p,q}(y-\delta,y+\delta) \big]$.  Given $M>0$, let $A^{p,q}_{M}$ denote the set of $y\in I^{p,q}$ such that  $F^{p,q}_{h}(y)\leq M  $.  Since $Q^\frak{h}(\tau^{p,q})<\infty$, there is an $M$ large enough so that $\tau^{p,q}\big( A^{p,q}_{M}\big)>\frac{1}{2}\tau^{p,q}([0,1])   $.

Next we focus on bounding $\inf_{ \mathcal{C}  }\sum_{ I\in  \mathcal{C}  }   h(|I|) $ from below, where the infimum is  over all countable coverings, $\mathcal{C}$, of $I^{p,q}$ by closed intervals $I$, where $|I|$ denotes the interval's diameter.  Given  such a collection  $\mathcal{C}$,   let  $\mathcal{C}_{M} $ be the subcollection of $\mathcal{C}$ consisting of intervals $ I$ such that $I\cap A^{p,q}_{M}\neq \emptyset  $. Of course, $\mathcal{C}_{M}$ forms a covering of $A^{p,q}_{M}$, and for  each interval $I\in \mathcal{C}_{M}$, we can pick a representative $y_I \in I\cap  A^{p,q}_{M}$.  Using that  $\mathcal{C}_{M}$ is a subcollection of $\mathcal{C}$ and the definition of $F^{p,q}_{h}$, we have the following inequalities:
\begin{align*}
\sum_{I\in \mathcal{C}} h(|I|) \,\geq \,\sum_{ I\in \mathcal{C}_{M}}  h\big(|I|\big)      \,\geq \,\sum_{I\in \mathcal{C}_{M} }\frac{   1}{ F^{p,q}_{h}(y_I)}\tau^{p,q}\big(y_I-|I|, y_I+|I|\big)\,.
\end{align*} 
Since $F^{p,q}_{h}(y_I)  \leq M$ when $y_I\in A^{p,q}_{M}$, the above is bounded from below by
\begin{align*} 
  \sum_{I\in \mathcal{C}_{M} }\frac{   1}{ M}\tau^{p,q}\big(y_I-|I|, y_I+|I|\big)     \,\geq \,&\,\frac{   1}{ M} \tau^{p,q}\big(A^{p,q}_{M}\big) \,\geq \,\frac{   1}{ 2M} \tau^{p,q}([0,1])>0 \,,
\end{align*}
 where the first inequality uses the subadditivity of the measure $ \tau^{p,q}$, and  the second inequality  follows from how we chose $M$. Thus we have shown that $\sum_{I\in \mathcal{C}} h(|I|)$ is uniformly bounded from below by $\frac{   1}{ 2M} \tau^{p,q}([0,1])$ for  all coverings $\mathcal{C}$ of $I^{p,q}$.  It follows that $H^{\textup{log}}_\frak{h}(I^{p,q})>0$, and so $H^{\textup{log}}_{ \frak{h}'}(I^{p,q})=\infty $ for all $0\leq \frak{h}' < \frak{h}$.   Since $\frak{h}$ is an arbitrary element in $  [0,1)$, we have shown that the log-Hausdorff exponent of $I^{p,q}$ is $\rho_r$-a.e. $\geq 1$.
\end{proof}

\section{Proof of Lemma~\ref{RemarkVarthetaSymm2}}\label{AppendixVartheta}

\begin{proof} For $p\in \Gamma$ and $e\in E_k$ with $\textup{Range}(p)\cap \overline{e}\neq \emptyset $, let the path $p_{e} \in \Gamma $  denote the dilation of the portion of the path $p$ crossing through the embedded copy of  $D$ corresponding to the edge $e$.  The hierarchical symmetry for the intersection-time measure $\tau^{p,q}$  in Remark~\ref{RemarkTau} translates to the following for the  pushforward measure $\gamma^{p,q}:= \tau^{p,q}\circ p^{-1}$ on $D$:
\begin{align}\label{GammaHierarchy}
\gamma^{p,q}\,=\,\bplus_{\mathbf{e}\in E_n} 1_{\mathbf{e}\in [p]_n}   1_{\mathbf{e}\in [q]_n}\gamma^{ p_\mathbf{e} , q_\mathbf{e} } \hspace{.7cm} \text{by identifying} \hspace{.7cm} D\,\equiv \,\bigsqcup_{\mathbf{e}\in E_n}D\,.
\end{align}
  Thus we can write $\vartheta_{ \eta }$ in the form 
\begin{align}\label{VVV}
\vartheta_{ \eta }\,:=\, \int_{\Gamma\times \Gamma}\gamma^{p,q} \eta(dp) \eta(dq)
\,=\, \bplus_{\mathbf{e}\in E_n} \int_{\Gamma\times \Gamma}1_{\mathbf{e}\in [p]_n}   1_{\mathbf{e}\in [q]_n}\gamma^{ p_\mathbf{e} , q_\mathbf{e} }  \eta(dp) \eta(dq) 
\end{align}
 under the identification $D\,\equiv \,\bigsqcup_{\mathbf{e}\in E_n}D$.  By Remark~\ref{RemarkDecompReII}, the  measure $\eta$ restricted to $\Gamma^{\updownarrow \mathbf{e}}$ is equal to  $\frac{1}{b^n}\big(\bigtimes_{k=1}^n\bigtimes_{ e\in E_k^{\updownarrow \mathbf{e}} }  \eta^{e} \big)\times \eta^{\mathbf{e}}  $  under the identification of $\Gamma^{\updownarrow \mathbf{e}}$ with $\big(\bigtimes_{k=1}^n \bigtimes_{ e\in E_k^{\updownarrow \mathbf{e}} } \Gamma\big)\times \Gamma  $, and thus a single term from the direct sum in~(\ref{VVV}) is equal to
\begin{align*}
&\int_{\Gamma\times \Gamma}1_{\mathbf{e}\in [p]_n}   1_{\mathbf{e}\in [q]_n}\gamma^{ p_\mathbf{e} , q_\mathbf{e} }  \eta(dp) \eta(dq)\\ &\,=\, \frac{1}{b^{2n}}\int_{\Gamma^{\updownarrow \mathbf{e}}\times \Gamma^{\updownarrow \mathbf{e}}}  \gamma^{ p_\mathbf{e} , q_\mathbf{e} } \Bigg(\prod_{k=1}^n\prod_{ e\in E_k^{\updownarrow \mathbf{e}} }  \eta^{e}( dp_e) \Bigg)\Bigg(\prod_{k=1}^n\prod_{ e\in E_k^{\updownarrow \mathbf{e}} }  \eta^{e}( dq_e) \Bigg)\eta^{\mathbf{e}}( dp_\mathbf{e})  \eta^{\mathbf{e}}( dq_\mathbf{e}) 
\\ &\,=\, \frac{1}{b^{2n}}\Bigg(\prod_{k=1}^n\prod_{ e\in E_k^{\updownarrow \mathbf{e}} }  \eta^{e}( \Gamma) \Bigg)^2\underbrace{\int_{\Gamma\times \Gamma }  \gamma^{ p_\mathbf{e} , q_\mathbf{e} } \eta^{\mathbf{e}}( dp_\mathbf{e})  \eta^{\mathbf{e}}( dq_\mathbf{e})}_{ \vartheta_{ \eta^{\mathbf{e}} } } \,,
\end{align*}
The above is equal to the right side of~(\ref{GenDef}). 
\end{proof}

\section{The operator from Theorem~\ref{ThmOperator}} \label{AppendixHier}

The next proposition characterizes the hierarchical symmetry of the linear operator $\hat{Y}_{\eta}$ appearing in  Theorem~\ref{ThmOperator}.

\begin{proposition}\label{PropHierarchy} Suppose that $\eta\in  \mathcal{M}_{\Gamma}^{\diamond}$ satisfies $\int_{\Gamma\times \Gamma}\big(T(p,q)\big)^2\eta(dp)\eta(dq)<\infty$ and  has component measures $ \{ \eta^{e}\}_{e\in \cup_{k=0}^{\infty}E_k }$. 
 Let the functions  $f_{\mathbf{e}}\in L^2\big(D,\vartheta_{ \eta^{\mathbf{e}} }\big)$ be the components of $f\in L^2(D,\vartheta_{ \eta})$ under the identification $D\equiv \bigsqcup_{\mathbf{e}\in E_n }D$.  Then $\hat{Y}_{\eta}$ decomposes in terms of the family  $\big\{\hat{Y}_{\eta^{\mathbf{e}}}\big\}_{\mathbf{e}\in E_n}$ as 
\begin{align}\label{Yhat}
\big(\hat{Y}_{\eta}f\big)(p)\,=\,\frac{1}{b^n}\sum_{\mathbf{e}\in [p]_n  }\Bigg(\prod_{k=1}^n\prod_{e\in E^{\updownarrow \mathbf{e} }_k}  \eta^{e}(\Gamma)   \Bigg)\big(\hat{Y}_{\eta^{\mathbf{e}}}f_{\mathbf{e}}\big) ( p_{\mathbf{e}})  \,,
\end{align}
where $ p_{\mathbf{e}}\in \Gamma$ denotes the dilation of the part of the path $p\in \Gamma$ in the generation-$n$ subcopy of $D$ corresponding to the edge $\mathbf{e}\in E_n$.
\end{proposition}

\begin{proof} 
The definition of $\hat{Y}_{\eta}$ and the hierarchical identity~(\ref{GammaHierarchy}) give us  the  equalities below.
\begin{align*}
\big(\hat{Y}_{\eta}f\big)(p)\,:=\, \int_{D\times \Gamma}f(x)\gamma^{p,q}(dx)\eta(dq)  
\,=\,&\,\sum_{\mathbf{e}\in [p]_n  } \int_{D\times \Gamma^{\updownarrow \mathbf{e}}}f_{\mathbf{e}}(x)\gamma^{p_\mathbf{e},q_\mathbf{e}}(dx)\eta(dq)
\end{align*}
By Remark~\ref{RemarkDecompRe}, the restriction of   $\eta$ to $\Gamma^{\updownarrow \mathbf{e}}$ is equal to the product measure  $\frac{1}{b^n} \prod_{k=1}^n\prod_{e\in E^{\updownarrow \mathbf{e} }_k}\eta^{e}$ under the identification of  $\Gamma $ with the product space $\big(\bigtimes_{k=1}^n \bigtimes_{ e\in E_k^{\updownarrow \mathbf{e}} } \Gamma\big)\times \Gamma$.  Hence the above is equal to 
\begin{align*}
\big(\hat{Y}_{\eta}f\big)(p)\,=\,&\,\frac{1}{b^n}\sum_{\mathbf{e}\in [p]_n  }\int_{D\times \big(\bigtimes_{k=1}^n \bigtimes_{ e\in E_k^{\updownarrow \mathbf{e}} } \Gamma\big)\times \Gamma}f_{\mathbf{e}}(x)\gamma^{p_\mathbf{e},q_\mathbf{e}}(dx)\Bigg(\prod_{k=1}^n\prod_{e\in E^{\updownarrow \mathbf{e} }_k} \eta^{e}(dq_e) \Bigg) \eta^{\mathbf{e}}(dq_\mathbf{e})
 \\ \,=\,&\,\frac{1}{b^n}\sum_{\mathbf{e}\in [p]_n  }\Bigg( \prod_{k=1}^n\prod_{e\in E^{\updownarrow \mathbf{e} }_k}  \eta^{e}(\Gamma)   \Bigg)\underbracket{\int_{D\times \Gamma}f_{\mathbf{e}}(x)\gamma^{p_\mathbf{e},q_\mathbf{e}}(dx)\eta^{\mathbf{e}}(dq_\mathbf{e})}\,.
\end{align*}
  Since the bracketed term is equal to $\big(\hat{Y}_{\eta^\mathbf{e}}f_{\mathbf{e}}\big)(p_\mathbf{e})$ by definition of $\hat{Y}_{\eta^\mathbf{e}}$, we have the desired form.
\end{proof}

\begin{remark} We can write the hierarchical symmetry~(\ref{Yhat}) in a slightly more decomposed form.  Let $\mathcal{U}_{ \eta }$ be the unitary operator from $\bigoplus_{\mathbf{e}\in E_n} L^2\big(D,\vartheta_{\eta^{\mathbf{e}}}\big)$  to  $L^2(D,\vartheta_{\eta})$ whose adjoint acts as
$$ \mathcal{U}_{ \eta }^*f\,=\,\bigoplus_{\mathbf{e}\in E_n}\frac{1}{b^n}\Bigg(\prod_{k=1}^n\prod_{e\in E^{\updownarrow \mathbf{e} }_k}  \eta^{e}(\Gamma)  \Bigg)    f_{\mathbf{e}}\,.$$
In other terms,  $\mathcal{U}_{ \eta }^*$ multiplies the component $f_{\mathbf{e}}$ of $f$   by the value  $ \frac{1}{b^n} \prod_{k=1}^n\prod_{e\in E^{\updownarrow \mathbf{e} }_k}  \eta^{e}(\Gamma)  $ and places the result into the direct sum. The unitarity of $ \mathcal{U}_{ \eta }^* $---and thus of $\mathcal{U}_{ \eta }$---follows from the  decomposition of $ \vartheta_{\eta}  $ in  Lemma~\ref{RemarkVarthetaSymm2}. For $\phi =\bigoplus_{\mathbf{e}\in E_n}\phi_{\mathbf{e}} $ in the space  $\bigoplus_{\mathbf{e}\in E_n} L^2\big(D,\vartheta_{\eta^{\mathbf{e}}}\big) $, we have the additive form
\begin{align*}
\big(\hat{Y}_{\eta}\mathcal{U}_{ \eta}\phi\big)(p)\,=\,\sum_{\mathbf{e}\in [p]_n  }\big(\hat{Y}_{\eta^{\mathbf{e}}}\phi_{\mathbf{e}}\big) (p_\mathbf{e})  \,.
\end{align*}
\end{remark}

The trivial lemma below presents a simple and natural approximation method for $\hat{Y}_{\eta}$ and $\mathbf{T}_{\eta}$  by finite-rank operators.  For the purpose of (III), we assume that $T(p,q)$ has finite exponential moments under $\eta\times \eta$, which a.s.\ holds when $\eta=\mathbf{M}_r$ by (ii) of Theorem~\ref{ThmPathMeasure}. 
\begin{lemma}\label{LemApprox} Suppose that $\eta\in \mathcal{M}_{\Gamma}^{\diamond}$ has  nonzero component measures $\{ \eta^{e} \}_{e\in \cup_{k=0}^{\infty}E_k }$ and that the intersection-time kernel  $T(p,q)$ has finite exponential moments under $\eta\times \eta$.  Let the linear operator  $\hat{Y}_{\eta}:L^2(D,\vartheta_{\eta})\rightarrow L^2(\Gamma,\eta)  $ be defined as in Theorem~\ref{ThmOperator}.
The sequence of finite-rank operators $\hat{Y}_{\eta}^{(n)}:L^2(D,\vartheta_{\eta})\rightarrow L^2(\Gamma,\eta)  $ defined through generation-$n$ coarse-graining as $\big(\hat{Y}_{\eta}^{(n)}f\big)(p):=\frac{1}{\eta([p]_n)  }\int_{[p]_n} \big(\hat{Y}_{\eta}f\big)(q) \eta(dq)$  has the following properties as $n\rightarrow \infty$:
\begin{enumerate}[(I)]
\item $\hat{Y}_{\eta}^{(n)}$ converges to $ \hat{Y}_{\eta}$ in operator norm,

\item the  kernels $\mathbf{T}_{\eta}^{(n)}(p,q)$ of $\mathbf{T}_{\eta}^{(n)}:=\hat{Y}_{\eta}^{(n)}\big(\hat{Y}_{\eta}^{(n)}\big)^*$ converge  $\eta\times \eta$-a.e.\ to $T(p,q)$, and

\item  the exponential moments of $\mathbf{T}_{\eta}^{(n)}(p,q)$ converge up to those of $T(p,q)$, i.e., for any $a\in \R$
$$   \int_{\Gamma\times \Gamma}e^{a\mathbf{T}_{\eta}^{(n)}(p,q)}\eta(dp) \eta(dq)\hspace{.5cm}\stackrel{n\rightarrow \infty}{\nearrow} \hspace{.5cm} \int_{\Gamma\times \Gamma}e^{aT(p,q)}\eta(dp)\eta(dq)\,<\, \infty\,. $$

\end{enumerate}

\end{lemma}

\begin{proof}    The operator $Y^{(n)}_{\eta  }$ can be written as $Y^{(n)}_{\eta  }=\mathbf{P}_{\eta }^{(n)}Y_{\eta  }$ for the orthogonal projection $\mathbf{P}_{\eta }^{(n)}$ on $L^2(\Gamma,\eta)$ defined by generation-$n$ coarse-graining: $ \big(\mathbf{P}_{\eta }^{(n)}g\big)(p)=\frac{1}{\eta([p]_n) } \int_{[p]_n}g(p')\eta(dp') $.  This is equivalent to defining $\mathbf{P}_{\eta }^{(n)}g=\mathbbmss{E}_{\eta}\big[g   \,\big|\,\mathcal{A}^{(n)}_{\Gamma} \big]$ for the $\sigma$-algebra $\mathcal{A}^{(n)}_{\Gamma}$ from Definition~\ref{DefAlgebra}.   When $g\in L^2(\Gamma,\eta)$, the sequence $\big(\mathbf{P}_{\eta }^{(n)}g\big)_{n\in \mathbb{N}}$ is an $\eta$-martingale with respect to  the filtration $\big( \mathcal{A}^{(n)}_{\Gamma} \big)_{n\in \mathbb{N}}$ that converges $\eta$-a.e.\ and in $L^2(\Gamma,\eta)$ to $g$ by the $L^p$ convergence theorem. 
In particular, $\mathbf{P}_{\eta }^{(n)}$ converges strongly to the identity operator on $L^2(\Gamma,\eta)$, and it follows that the linear operator $Y^{(n)}_{\eta  }=\mathbf{P}_{\eta }^{(n)}Y_{\eta  }$ converges in operator norm  to $Y_{\eta }$ with large $n$ because  $Y_{\eta  }$ is compact.  The kernel of $\mathbf{T}_{\eta}^{(n)} $ has the form 
\begin{align*}
\mathbf{T}_{\eta}^{(n)}(p,q)\,=\,&\,\frac{1}{\eta([p]_n)\eta([q]_n) }\int_{[p]_n\times [q]_n   } T(p',q')   \eta(dp')\eta(dq') \,=\,\mathbbmss{E}_{\eta\times \eta}\Big[T(p,q)   \,\Big|\,\mathcal{A}^{(n)}_{\Gamma\times \Gamma}\Big]\,.
\end{align*}
   Thus the sequence $\big( \mathbf{T}_{\eta}^{(n)}(p,q) \big)_{n\in \mathbb{N}}$ is a nonnegative $\eta\times \eta$-martingale with respect to the filtration $\big(\mathcal{A}^{(n)}_{\Gamma\times \Gamma}\big)_{n\in \mathbb{N}}$ that converges $\eta\times \eta$-a.e.\ to $T(p,q)$.   Moreover, for any $a>0$, the conditional version of Jensen's inequality implies that 
\begin{align}\label{Jen}
  \int_{\Gamma\times\Gamma} \textup{exp}\left\{ a  \mathbf{T}_{\eta}^{(n)}(p,q)\right\} \eta(dp)\eta(dq)\,\leq \,  \int_{\Gamma\times\Gamma} \textup{exp}\left\{ a  T(p,q)\right\} \eta(dp)\eta(dq) \,.
  \end{align}
 The convergence of the exponential moments holds by Fatou's lemma combined with~(\ref{Jen}).
\end{proof}

\end{appendix}


\begin{thebibliography}{99}


\bibitem{US} T.\ Alberts, J.\ Clark, S.\ Kocic: \emph{The intermediate disorder regime for a directed polymer model on a hierarchical lattice}, Stoch.\ Process.\ Appl.\ \textbf{127}, 3291-3330 (2017).   



\bibitem{alberts} T.\ Alberts, K.\ Khanin, J.\ Quastel: \emph{The intermediate disorder regime for directed polymers in dimension $1+1$}, Ann.\ Probab.\ \textbf{42}, No.\ 3, 1212-1256 (2014).



\bibitem{alberts2} T.\ Alberts, K.\ Khanin, J.\ Quastel: \emph{The continuum directed random polymer}, J.\ Stat.\ Phys.\ \textbf{154}, No.\ 1-2, 305-326 (2014).


%\bibitem{BC} L.\ Bertini and N.\ Cancrini: \emph{The two-dimensional stochastic heat equation: renormalizing a multiplicativenoise}, J.\ Phys.\ A: Math. Gen.\ \textbf{31}, 615, (1998).


\bibitem{CSZ0} F.\ Caravenna, R.\ Sun, and N.\ Zygouras: \emph{Polynomial chaos and scaling limits of disordered systems}, J.\ Eur.\ Math.\ Soc.\ \textbf{19}, 1-65 (2017).



\bibitem{CSZ1} F.\ Caravenna, R.\ Sun, and N.\ Zygouras: \emph{Universality in marginally relevant disordered systems}, Ann.\ Appl.\ Probab. \textbf{27}, No.\ 5, 3050-3112 (2017).



\bibitem{CSZ2} F.\ Caravenna, R.\ Sun, N.\ Zygouras, \emph{The Dickman subordinator, renewal theorems, and disordered
systems},  Elect.\ Journ.\  Prob.\ \textbf{24}, 1-48 (2019).




\bibitem{CSZ3} F.\ Caravenna, R.\ Sun, N.\ Zygouras: \textit{Scaling limits of disordered systems and disorder relevance},
to appear in Proceedings of XVIII International Congress on Mathematical Physics,
\textup{arXiv:1602.05825}.


\bibitem{CSZ4}  F.\ Caravenna, R.\ Sun, N.\ Zygouras: \emph{On the moments of the (2+1)-dimensional directed polymer
and stochastic heat equation in the critical window}, Commun.\ Math.\ Phys.\ \textbf{372}, 385-440 (2019).



\bibitem{CSZ5}  F.\ Caravenna, R.\ Sun, N.\ Zygouras: \emph{The critical 2d stochastic heat flow}, preprint (2021), \textup{arXiv:2109.03766}.




\bibitem{Clark1} J.T.\ Clark: \emph{High-temperature scaling limit for directed polymers on a hierarchical lattice with bond disorder}, J.\ Stat.\ Phys.\ \textbf{174}, No.\ 6, 1372-1403 (2019).




\bibitem{Clark2} J.T.\ Clark: \emph{Continuum directed random polymers on disordered hierarchical diamond lattices},  Stoch.\ Process.\ Appl.\ \textbf{130}, 1643-1668 (2020). 


\bibitem{Clark3} J.T.\ Clark: \emph{Weak-disorder limit at criticality for random directed polymers on  hierarchical graphs}, Commun.\,  Math.\, Phys.\, \textbf{386}, 651-710 (2021).


\bibitem{Clark4} J.T.\ Clark: \emph{The conditional Gaussian multiplicative chaos   structure underlying a critical  continuum random  polymer model on a diamond fractal},  preprint (2019), \textup{arXiv:1908.08192}.   



\bibitem{Hausdorff} F.\ Hausdorff: \emph{Dimension und äußeres Maß}, Math.\ Ann.\ \textbf{79}, 157-179 (1918).  




\bibitem{Giacomin} G. Giacomin: \emph{Disorder and Critical Phenomena through Basic Probability Models}, Lecture Notes in Mathematics Vol.\ \textbf{2025}, Springer 2010.



%\bibitem{GLT}  G.\ Giacomin, H.\ Lacoin, F.L.\ Toninelli: \emph{Hierarchical pinning models, quadratic maps, and quenched disorder}, Probab. Theor. Rel. Fields \textbf{145}, (2009).


\bibitem{GQT} Y.\ Gu, J.\ Quastel, L.\ Tsai: \textit{Moments of the 2D SHE at Criticality}, Prob.\ Math.\ Phys.\ \textbf{2}, 179-219  (2021).





\bibitem{HamblyII}  B.M.\ Hambly, T.\ Kumagai: \emph{Diffusion on the scaling limit of the critical percolation cluster in the diamond hierarchical lattice}, Adv.\ Appl.\ Prob., \textbf{36}, 824-838 (2004).


\bibitem{Kallenberg}  O.\ Kallenberg: \emph{Random Measures, Theory and Applications}, Springer 2017.


\bibitem{Rela} U.\ Molter, E.\ Rela:  \emph{Improving dimension estimates for Furstenberg-type sets}, Adv.\ Math.\  \textbf{223}, No.\ 2, 672-688 (2010).


\bibitem{Rela2} U.\ Molter, E.\ Rela:  \emph{Furstenberg sets for a fractal set of directions}, Proc.\ Math.\ Soc.\ \textbf{140},  2753-2765 (2012).


%\bibitem{Royden} H.L.\ Royden: \emph{Real Analysis}, third edition, Macmillan 1988.


\bibitem{Ruiz} P.A.\ Ruiz:  \emph{Explicit formulas for heat kernels on diamond fractals}, Commun.\ Math.\ Phys.\ \textbf{364}, 1305-1326 (2018).  

\bibitem{Ruiz2} P.A.\ Ruiz:  \emph{Heat kernel analysis on diamond fractals}, Stoch.\ Process.\ Appl.\ \textbf{131}, 51-72 (2021). 






\end{thebibliography}
\end{document}